\documentclass[11pt,a4paper]{article}
\usepackage[a4paper]{geometry}
\usepackage{amssymb,latexsym,mathtools,amsfonts,amsthm}
\usepackage{graphicx}
\usepackage{amsbsy}
\usepackage{hyperref}\hypersetup{hidelinks}
\usepackage{fullpage}
\usepackage{enumitem}
\usepackage{graphicx}
\usepackage{color}
\usepackage{tikz}
\usepackage{subfigure}
\usepackage{booktabs,longtable}
\usepackage[alphabetic,initials]{amsrefs}

\newcommand*{\dif}{\mathop{}\!\mathrm{d}}

\DeclareMathOperator*{\res}{Res}
\DeclareMathOperator*{\im}{Im}
\DeclareMathOperator*{\re}{Re}

\newtheorem{theorem}{Theorem}[section]
\newtheorem{lemma}[theorem]{Lemma}
\newtheorem{Proposition}[theorem]{Proposition}

\newtheorem{definition}[theorem]{Definition}

\newtheorem{RHP}[theorem]{RH problem}
\newtheorem{Dbarproblem}[theorem]{$\bar{\partial}$-Problem}
\newtheorem{dbar-RHP}[theorem]{$\bar{\partial}$-RH problem}
\newtheorem{assumption}[theorem]{Assumption}
\theoremstyle{remark}
\newtheorem{remark}[theorem]{Remark}
\theoremstyle{definition}
\numberwithin{equation}{section}
\allowdisplaybreaks[4]
\renewcommand{\thefootnote}{\fnsymbol{footnote}}

\begin{document}

\title{Transient asymptotics of the modified Camassa-Holm equation}

\author{Taiyang Xu\footnotemark[1],\quad Yiling Yang\footnotemark[3], \quad Lun Zhang\footnotemark[1]~\footnotemark[2]}

\renewcommand{\thefootnote}{\fnsymbol{footnote}}
\footnotetext[1]{School of Mathematical Sciences, Fudan University, Shanghai 200433, China.
E-mail: \texttt{\{tyxu19,lunzhang\}@fudan.edu.cn.}}
\footnotetext[2]{Shanghai Key Laboratory for Contemporary Applied Mathematics, Fudan University, Shanghai 200433, China.}
\footnotetext[3]{College of Mathematics and Statistics, Chongqing University, Chongqing 401331, China. E-mail: \texttt{ylyang19@fudan.edu.cn.}}

\date{\today}

\maketitle
%
%\begin{center}
%	\begin{Large}
%	\textbf{
%		On the Cauchy problem of mCH equation: long-time asymptotics in Painlev\'e regions and collisionless shock region}
%	\end{Large}
%	\end{center}
%
%	\begin{center}
%		Taiyang Xu$^1$\footnote{tyxu19\symbol{'100}fudan.edu.cn},
%		Yiling Yang$^1$\footnote{19110180006\symbol{'100}fudan.edu.cn},
%		Lun Zhang$^1$\footnote{lunzhang\symbol{'100}fudan.edu.cn}
%		\bigskip
%		\begin{minipage}{0.7\textwidth}
%		\begin{small}
%		\begin{enumerate}
%		\item [$1$] {\it School of Mathematical Sciences and Shanghai Key Laboratory for Contemporary Applied Mathematics, Fudan University, Shanghai 200433, China.}
%		\end{enumerate}
%		\end{small}
%		\end{minipage}
%		\vspace{0.5cm}
%	\end{center}
	
\begin{abstract}
We investigate long time asymptotics of the modified Camassa-Holm equation in three transition 
zones under a nonzero background. The first transition zone lies between the soliton region and 
the first oscillatory region, the second one lies between  the second oscillatory region and the fast 
decay region, and possibly, the third one, namely, the collisionless shock region, that bridges the 
first transition region and the first oscillatory region. Under a low regularity condition on the 
initial data, we obtain Painlev\'e-type asymptotic formulas in the first two transition regions, 
while the transient asymptotics in the third region involves the Jacobi theta function. We establish 
our results by performing a $\bar{\partial}$ nonlinear steepest descent analysis to the associated 
Riemann-Hilbert problem.
%We investigate the long-time asymptotics of the modified Camassa-Holm (mCH) equation under a nonzero background,
%		where initial data belongs to weighted Sobolev space.
%		By associating the solution of mCH equation to a solvable and unique Riemann-Hilbert (RH) problem, asymptotic formulas are obtained in three different transition regions, which are
%		$\mathcal{R}_{I}:=\{(x,t):0\leqslant|\xi-2|t^{2/3}\leqslant C\}$, $\mathcal{R}_{II}:=\{(x,t):0\leqslant|\xi+1/4|t^{2/3}\leqslant C\}$ and $\mathcal{R}_{III}:=\{(x,t):2\cdot 3^{\frac{1}{3}}<(2-\xi)t^{2/3}< C(\log t)^{2/3}\}$
%		with $\xi:=x/t$. For $\xi\in\mathcal{R}_{I}$ and $\mathcal{R}_{II}$, both the asymptotic formulas are associated to the distinguished Painlev\'e II (PII) transcendents.
%		With different value of reflection coefficient $r$, the asymptotic formulae for $\xi\in\mathcal{R}_{I}$ can be expressed in terms of the Hastings-McLeod solution
%		of Painlev\'e II equation in generic case, while Ablowitz-Segur solution in non-generic case. The asymptotic formulae for $\xi\in\mathcal{R}_{II}$ can be given in terms of Ablowitz-Segur solution.
%		When $|r(1)|=1$, a new asymptotic sector which is called ``collisionless shock region'' $\mathcal{R}_{III}$ occurs. The asymptotic formulae for $\xi\in\mathcal{R}_{III}$ is associated to Jacobi theta function.
%		The error terms of the asymptotic formulas for $\xi\in\mathcal{R}_{I}$ and $\mathcal{R}_{II}$ are explicitly derived via $\bar{\partial}$ nonlinear steepest descent technique.
	\end{abstract}
	
\setcounter{tocdepth}{2} \tableofcontents

%%%%%%%%%%%%%%%%%%%%%%%%%%%%%%%%%%%%%%%%%%%%%%%%%%%%%%%%%%%%%%%%
%%%%%%%%%%%%%%%%%%%%%%%%%%%%%%%%%%%%%%%%%%%%%%%%%%%%%%%%%%%%%%%%
\section{Introduction}
In this paper, we are concerned with the Cauchy problem for the modified Camassa-Holm (mCH) equation
\begin{subequations}\label{mcho2}
\begin{align}
	&m_{t}+\left(m\left(u^{2}-u_{x}^{2}\right)\right)_{x}=0, \quad m=u-u_{xx}, &x\in\mathbb{R}, \ t>0, \label{def:mCHeq}\\
	&u(x,0)=u_0(x), &x\in\mathbb{R},
\end{align}
\end{subequations}
with nonzero boundary condition
\begin{equation}\label{eq:bdry}
u_0(x)\to 1, \qquad |x|\rightarrow\infty.
\end{equation}

In an equivalent form, the mCH equation \eqref{def:mCHeq} might be traced back to the work of Fokas \cite{Fok95}. It is nowadays usually recognized as an
integrable modification of the celebrated Camassa-Holm (CH) equation \cites{Camassa19931661,CHH94}
\begin{equation}\label{def:CH}
m_t+(um)_{x}+u_xm=0, \qquad m=u-u_{xx},
\end{equation}
which was proposed by Fuchsstiener and Fokas in \cite{Fuchssteiner198147} as an integrable equation. This equation was firstly introduced by Camassa and Holm as
a model for shallow water waves \cite{Camassa19931661}; see also \cite{Constantin2009TheHR} for a rigorous justification. Since the CH equation turns out to be a different-factorization equation of the Korteweg-de Vries (KdV) equation, the mCH equation was introduced in \cites{FuchPhyD, POlverPhyRE} by applying the general method of tri-Hamiltonian duality to the bi-Hamiltonian representation of the modified KdV equation; see also \cite{Qiao2006ANI}. It is among the list of Novikov in the classification of generalized Camassa-Holm-type equations which possess infinite hierarchies of higher symmetries \cite{Nov09} as well, whereas a Miura-type map from \eqref{def:CH} to \eqref{def:mCHeq} was established in \cite{KLOQ16}. There exist relations between the mCH equation and other well-known integrable PDEs. For example,
by applying the scaling transformation and taking suitable limits, the mCH equation degenerates to the short pulse equation \cite{Schfer2004PropagationOU}
\begin{align}
	&u_{xt}=u+\frac{1}{6}(u^3)_{xx}.\nonumber
\end{align}	
Moreover, it has recently been shown in \cite{Chen2022TheSM} that the mCH equation \eqref{def:mCHeq} provides a model for the unidirectional propagation of shallow water waves of mildly large amplitude over a flat bottom, where the function $u(x,t)$ is interpreted as  the horizontal velocity in certain level of fluid.

%It is known that the mCH equation \eqref{mcho2} is an

Due to the aforementioned connections and applications, the mCH equation has attracted great 
interests and significant progresses have been achieved over the past few years.  
Similar to the CH equation, the mCH equation possesses a bi-Hamiltonian structure \cite{POlverPhyRE} 
and the wave-breaking, peakon solutions \cite{Gui2013WaveBreakingAP}; see also \cite{Hou2017TheAS} 
for the quasi-periodic algebro-geometric solutions. The B\"acklund transformation and the related 
nonlinear superposition formulae were presented in \cite{Wang2020TheMC} and the existence of 
global solutions were analyzed in \cite{YFLGlobalmCH}.

As an integrable equation, the mCH equation particularly admits a 
Lax pair \cites{Qiao2006ANI,Schiff1996}. Based on this feature, an Riemann-Hilbert (RH) formalism of 
the Cauchy problem \eqref{mcho2} has been developed in \cite{Monvel2020ARA}. 
It comes out that, analogous to the case of the CH equation \cite{Monvel2009LongtimeAF} or the Korteweg-de Vries (KdV) equation \cite{Deift1994TheCS}, large time asymptotics of $u(x,t)$ shows qualitatively different behaviors in different regions of the $(x,t)$-half plane. More precisely, these regions are
\begin{itemize}
	\item[$\bullet$]a soliton region: $\{(x,t):\xi>2\}$,
	\item[$\bullet$]a fast decay region: $\{(x,t):\xi<-1/4\}$,
	\item[$\bullet$]two oscillatory regions: $\{(x,t):0<\xi<2\}\cup\{(x,t):-1/4<\xi<0\}$,
\end{itemize}
where $\xi:=x/t$; see Figure \ref{fig: asymptotic sectors} for an illustration.
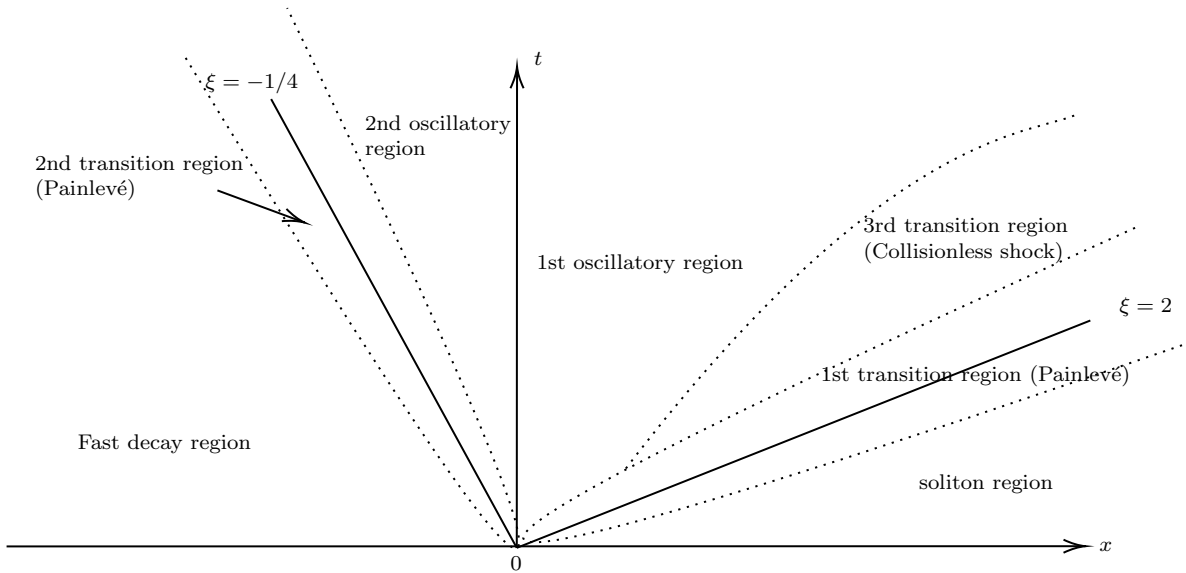
\begin{figure}[htbp]
\begin{center}
	\tikzset{every picture/.style={line width=0.75pt}} %set default line width to 0.75pt
	\begin{tikzpicture}[x=0.75pt,y=0.75pt,yscale=-1,xscale=1]
	%uncomment if require: \path (0,300); %set diagram left start at 0, and has height of 300
	%Straight Lines [id:da4061624042908565]
	\draw    (49.14,274.71) -- (585.42,274.6) ;
	\draw [shift={(587.42,274.6)}, rotate = 179.99] [color={rgb, 255:red, 0; green, 0; blue, 0 }  ][line width=0.75]    (10.93,-3.29) .. controls (6.95,-1.4) and (3.31,-0.3) .. (0,0) .. controls (3.31,0.3) and (6.95,1.4) .. (10.93,3.29)   ;
	%Straight Lines [id:da9744305664924162]
	\draw    (303.83,35.28) -- (303.7,275.43) ;
	\draw [shift={(303.83,33.28)}, rotate = 90.03] [color={rgb, 255:red, 0; green, 0; blue, 0 }  ][line width=0.75]    (10.93,-3.29) .. controls (6.95,-1.4) and (3.31,-0.3) .. (0,0) .. controls (3.31,0.3) and (6.95,1.4) .. (10.93,3.29)   ;
	%Straight Lines [id:da006840541635086961]
	\draw    (589.87,161.06) -- (303.7,275.43) ;
	%Straight Lines [id:da9583133523710743]
	\draw    (181.14,49.67) -- (303.7,275.43) ;
	%Shape: Parabola [id:dp6174517537316881]
	\draw  [dash pattern={on 0.84pt off 2.51pt}] (138.49,28.95) .. controls (341.88,365.95) and (358.77,357.59) .. (189.14,3.86) ;
	%Shape: Parabola [id:dp3245813990117892]
	\draw  [dash pattern={on 0.84pt off 2.51pt}] (612.18,114.43) .. controls (196.85,303.66) and (205.34,323.1) .. (637.67,172.78) ;
	%Curve Lines [id:da5098861839156059]
	\draw  [dash pattern={on 0.84pt off 2.51pt}]  (358.14,235.5) .. controls (493.43,57.57) and (567.43,67.57) .. (584.43,56.57) ;
	%Straight Lines [id:da3155167395588392]
	\draw    (154.34,95.61) -- (195.99,111.12) ;
	\draw [shift={(197.87,111.82)}, rotate = 200.43] [color={rgb, 255:red, 0; green, 0; blue, 0 }  ][line width=0.75]    (10.93,-3.29) .. controls (6.95,-1.4) and (3.31,-0.3) .. (0,0) .. controls (3.31,0.3) and (6.95,1.4) .. (10.93,3.29)   ;
	% Text Node
	\draw (592.99,270.78) node [anchor=north west][inner sep=0.75pt]  [font=\scriptsize,rotate=-359.99]  {$x$};
	% Text Node
	\draw (310.93,24.71) node [anchor=north west][inner sep=0.75pt]  [font=\scriptsize,rotate=-359.99]  {$t$};
	% Text Node
	\draw (602.9,147.88) node [anchor=north west][inner sep=0.75pt]  [font=\scriptsize,rotate=-359.99]  {$\xi =2$};
	% Text Node
	\draw (146.83,34.37) node [anchor=north west][inner sep=0.75pt]  [font=\scriptsize,rotate=-359.99]  {$\xi =-1/4$};
	% Text Node
	\draw (83.26,216.68) node [anchor=north west][inner sep=0.75pt]  [font=\scriptsize,rotate=-359.99] [align=left] {Fast decay region};
	% Text Node
	\draw (502.89,236.79) node [anchor=north west][inner sep=0.75pt]  [font=\scriptsize,rotate=-359.99] [align=left] {soliton region};
	% Text Node
	\draw (312.24,126.04) node [anchor=north west][inner sep=0.75pt]  [font=\scriptsize,rotate=-359.99] [align=left] {1st oscillatory region};
	% Text Node
	\draw (226.54,56.28) node [anchor=north west][inner sep=0.75pt]  [font=\scriptsize,rotate=-359.99] [align=left] {2nd oscillatory \\region};
	% Text Node
	\draw (61.91,76.54) node [anchor=north west][inner sep=0.75pt]  [font=\scriptsize,rotate=-359.99] [align=left] {2nd transition region \\(Painlev\'e)};
	% Text Node
	\draw (453.88,181.8) node [anchor=north west][inner sep=0.75pt]  [font=\scriptsize,rotate=-359.99] [align=left] {1st transition region (Painlev\'e)};
	% Text Node
	\draw (475.16,107.89) node [anchor=north west][inner sep=0.75pt]  [font=\scriptsize,rotate=-359.99] [align=left] {3rd transition region\\(Collisionless shock)};
	% Text Node
	\draw (299,278.4) node [anchor=north west][inner sep=0.75pt]  [font=\scriptsize]  {$0$};
	\end{tikzpicture}	
\caption{ The different asymptotic regions of the $(x,t)$-half plane, where $\xi=x/t$.}
\label{fig: asymptotic sectors}
\end{center}
\end{figure}
Under the condition that the initial value belongs to the Schwartz space, i.e., $m(x,0)$ is sufficiently
smooth and is rapidly decaying as $|x|\rightarrow\infty$, large time asymptotics of the mCH equation in the regular asymptotic zones are presented in \cite{Monvel2020TheMC}; we also refer to
\cites{Karpenko2022ARA,Yang2022OnTL} for the studies with step-like boundary conditions.

It is the aim of the present work to fulfill asymptotic picture of the mCH equation by focusing on the asymptotics in three transition zones as illustrated in Figure \ref{fig: asymptotic sectors}. The first one lies between the soliton region and the first oscillatory region, the second one lies between  the second oscillatory region and the fast decay region, and possibly, the third one, namely, the collisionless shock region -- a terminology due to Gurevich and Pitaevski \cite{Gurevich1973NonstationarySO}, that bridges the first transition region and the first oscillatory region. Under a low regularity condition on the initial value, we obtain Painlev\'e-type asymptotic formulas of the mCH equation in the first two transition regions, while the transient asymptotics in the third region involves the Jacobi theta function. This particularly reveals rich mathematical structures of the mCH equation. Our results are stated in the next section.

\section{Main results}\label{subsec:results}
Besides the boundary condition \eqref{eq:bdry}, we assume that the initial data satisfies the following conditions.
\begin{assumption}\label{Assumption}
\hfill
\begin{itemize}
	\item[$\bullet$]$m_0(x):=m(x,0)>0$ for $x\in\mathbb{R}$.
	\item[$\bullet$]$m_0(x)-1\in H^{2,1}(\mathbb{R})\cap H^{1,2}(\mathbb{R})$, where
    \begin{equation}\label{def:SobSpace}
		H^{k,s}(\mathbb{R})=\left\{f\in L^{2}(\mathbb{R}): \langle \cdot\rangle^s\partial_x^{j}f\in L^2(\mathbb{R}), j=0,1,\dots,k\right\}, \quad s \geqslant 0,
	\end{equation}
	with $\langle x\rangle:=(1+x^2)^{1/2}$ is the weighted Sobolev space.
\end{itemize}
\end{assumption}
The global well-posedness of the Cauchy problem for the mCH equation under the above assumption has recently been established in \cite{YFLGlobalmCH}. Indeed, it was also shown therein that
$m -1\in C([0, +\infty),  H^{2,1}(\mathbb{R}))\cap H^{1, 2} (\mathbb{R})$.

%The existence of global solution is the basis to investigate the long-time asymptotics. In the recent work \cite{YFLGlobalmCH},
%the global well-posedness is exhibited as the following theorem.
%\begin{theorem}[\cite{YFLGlobalmCH}]\label{thm1}
%Assume that $ m_0(x)>0, \forall  x\in \mathbb{R}$ and $m_0 -1\in  H^{2,1}(\mathbb{R})\cap H^{1, 2} (\mathbb{R})$. Then
%there exists a unique global solution $m$ of the Cauchy  problem  \eqref{mcho2} with $ m(0) = m_0 $ such that $m -1\in C([0, +\infty),  H^{2,1}(\mathbb{R}))\cap H^{1, 2} (\mathbb{R})$.
%\end{theorem}
%
%Regions in the $(x,t)$-half plane with qualitatively different large-time asymptotics of $u(x,t)$ have been distinguished before. These regions are divided into four asymptotic zones

We next give precise definitions of the three transition zones.
\begin{definition}\label{def:transzone}
For any positive constant $C$, we define
\begin{itemize}
	\item[\rm (a)]the first transition region (Painlev\'e) $\mathcal{R}_{I}:=\{(x,t):0\leqslant|\xi-2|t^{2/3}\leqslant C\}$,
	\item[\rm (b)]the second transition region (Painlev\'e) $\mathcal{R}_{II}:=\{(x,t):0\leqslant|\xi+1/4|t^{2/3}\leqslant C\}$,
	\item[\rm (c)]the third transition region (collisionless shock) $\mathcal{R}_{III}:=\{(x,t):2\cdot 3^{1/3}(\log t)^{2/3}<(2-\xi)t^{2/3}< C(\log t)^{2/3},~C>2\cdot 3^{1/3}\}$,
\end{itemize}
where recall that $\xi=x/t$.
\end{definition}
Large time asymptotics of $u(x,t)$ in these regions are main results of the present work.

%The main results given below
%
%for any constants $C>0$. In our present work, the asymptotic formulas for the  first two transition regions can be expressed in terms of the solution of the well-known homogeneous Painlev\'e II (PII) equation
%\begin{equation}\label{equ:standard PII equ}
%	v_{ss}(s)=sv(s)+2v^2(s),
%\end{equation}
%while the asymptotic formulae in the third transition region can be expressed in terms of

%The main results list as follows.
\begin{theorem}\label{mainthm}
Let $u(x,t)$ be the global solution of the Cauchy problem \eqref{mcho2} for the mCH equation over the real line under Assumption \ref{Assumption}, and denote by $r(z)$ and $\{z_n\}_{n=1}^{2\mathcal{N}}$, $|z_n|=1$, the reflection coefficient and the discrete spectrum associated to the initial data $u_0(x)$ in the lower half plane. As $t \to +\infty$, we have the following asymptotics of $u(x,t)$ in the regions $\mathcal{R}_{I}-\mathcal{R}_{III}$ given in Definition \ref{def:transzone}.
%. Then the asymptotics for $u(x,t)$ of the mCH equation can be described in terms of the Painlev\'e II
%	transcendents for $\xi\in\mathcal{R}_{I}\cup\mathcal{R}_{II}$, and Jacobi theta function for $\xi\in\mathcal{R}_{III}$ respectively. More precisely,
	\begin{itemize}
		\item[\rm (a)]
       For $\xi\in\mathcal{R}_{I}$,
		\begin{subequations}\label{result:1stTran}
		\begin{align}
			u(x,t)=1-(2/81)^{-1/3}t^{-2/3}v'(s) +\mathcal{O}(t^{-\min\{1-4\delta_1,1/3+9\delta_1\}}),
		\end{align}
		where $\delta_1$ is any real number belonging to $(1/24,1/18)$,
		\begin{align}\label{equ:s1stran}
		s=6^{-2/3}\left(\frac{x}{t}-2\right)t^{2/3},
		\end{align}
		\end{subequations}
		and $v(s)$ is the unique solution of the Painlev\'e II equation
        \begin{equation}\label{equ:standard PII equ}
	     v''(s)=sv(s)+2v^3(s)
        \end{equation}
        characterized by
		\begin{align*}
		v(s)\sim r(1)\textnormal{Ai}(s), \qquad s \to +\infty,
		\end{align*}
		with $\textnormal{Ai}$ being the classical Airy function and $r(1)\in [-1,1]$.
       %In generic case, $|r(1)|=1$,  $v(s)$ is the Hastings-McLeod solution of the Painlev\'e II equation, while in non-generic case, $|r(1)|<1$,  $v(s)$
%		is the  Ablowitz-Segur solution of the Painlev\'e II equation.
		\item[\rm (b)]For $\xi\in\mathcal{R}_{II}$,
		\begin{subequations}\label{result:2ndTran}
		\begin{align}
			u(x,t)=1+3^{-2/3}t^{-1/3}f_{II}(s)v_{II}(s)+\mathcal{O}\left(t^{\textnormal{max}\{{-2/3+4\delta_2, -1/3-5\delta_2}\}}\right),
		\end{align}
		where $\delta_2$ is any real number belonging to $(0,1/15)$,
		\begin{align}
			s&=-\left(\frac{8}{9}\right)^{1/3}\left(\frac{x}{t}+\frac{1}{4}\right)t^{2/3}, \label{equ:s2ndtran}\\ f_{II}(s)&=2\sqrt{2+\sqrt{3}}\left(\sin\psi_a\left(s,t\right)\cos\gamma_a-\frac{iT_1}{T(i)}\cos\psi_a\left(s,t\right)\sin\gamma_a\right)\nonumber\\
			&~~~+2\sqrt{2-\sqrt{3}}\left(\sin\psi_b\left(s,t\right)\cos\gamma_b-\frac{iT_1}{T(i)}\cos\psi_b(s,t)\sin\gamma_b\right)
\nonumber\\ &~~~+\sqrt{3}\cos\left(\frac{\Lambda_a+\Lambda_b}{2}\right)\sin\left(\frac{\Lambda_a+\Lambda_b}{2}\right)
		\end{align}
	\end{subequations}
		with
		\begin{subequations}
		\begin{align}
			\psi_a(s,t)&=\frac{3^{7/6}}{2}st^{1/3}+\frac{3\sqrt{3}}{4}t+\Lambda_a,\qquad \psi_b(s,t)=-\frac{3^{7/6}}{2}st^{1/3}-\frac{3\sqrt{3}}{4}t+\Lambda_b,\label{equ:psi_a}\\
			\Lambda_a & =\arg r(2+\sqrt{3})+4\sum_{n=1}^{\mathcal{N}}\arg(2+\sqrt{3}-z_n) \nonumber \\
&~~~ -\frac{1}{\pi}\int_{-\infty}^{+\infty}\frac{\log(1-|r(\zeta)|^2)}{\zeta-(2+\sqrt{3})}\dif\zeta-2\sqrt{3}\log\left(T\left(i\right)\right),\label{equ:Lambda_a}\\
			\Lambda_b & =\arg r(2-\sqrt{3})+4\sum_{n=1}^{\mathcal{N}}\arg(2-\sqrt{3}-z_n)
\nonumber \\
&~~~-\frac{1}{\pi}\int_{-\infty}^{+\infty}\frac{\log(1-|r(\zeta)|^2)}{\zeta-(2-\sqrt{3})}\dif\zeta+2\sqrt{3}\log\left(T\left(i\right)\right),\label{equ:Lambda_b}\\
			\gamma_a&=\arctan{(2+\sqrt{3})}, \qquad  \gamma_b=\arctan{(2-\sqrt{3})},\\
			T(i)&=\prod_{n=1}^{2\mathcal{N}}\left(\frac{\bar{z}_n-i}{z_n-i}\right)\exp\left\{-\frac{1}{2\pi i}\int_{-\infty}^{+\infty}\frac{\log(1-|r(\zeta)|^2)}{\zeta-i}\dif\zeta  \right\},\\
			T_1&=-\prod_{n=1}^{2\mathcal{N}}\left(\frac{\bar{z}_n-i}{z_n-i}\right)\cdot\frac{1}{2\pi i}\int_{-\infty}^{+\infty}\frac{\log(1-|r(\zeta)|^2)}{(\zeta-i)^2}\dif\zeta  \cdot\exp\left\{-\frac{1}{2\pi i}\int_{-\infty}^{+\infty}\frac{\log(1-|r(\zeta)|^2)}{\zeta-i}\dif\zeta  \right\}\nonumber\\
			&~~~ +\sum_{n=1}^{2\mathcal{N}}\left(\frac{\bar{z}_n-i}{(z_n-i)^2}\prod_{l\neq n, l=1}^{2\mathcal{N}}\frac{\bar{z}_l-i}{z_l-i}\right)\exp\left\{-\frac{1}{2\pi i}\int_{-\infty}^{+\infty}\frac{\log(1-|r(\zeta)|^2)}{\zeta-i}\dif\zeta  \right\},
		\end{align}
		\end{subequations}
		and $v_{II}(s)$ is the unique solution of the Painlev\'e II equation \eqref{equ:standard PII equ} characterized by
		\begin{align}
			v_{II}(s)\sim-\vert r(2+\sqrt{3})\vert \textnormal{Ai}(s), \qquad s \to +\infty,
		\end{align}
		with $\vert r(2+\sqrt{3})\vert<1$.
		\item[\rm (c)]For $\xi\in\mathcal{R}_{III}$, if $|r(\pm 1)|=1$, $r\in H^{s}$ with $s>5/2$,  we have
		\begin{align}\label{result: u(x,t); part (c)}
      u(x,t)&=1-\frac{(2-\xi)(a-b)q}{12p}\left(-i\cdot\frac{\Theta\left(A(\infty)-\frac{\varkappa}{4}\right)}{\Theta\left(A(\infty)-\frac{\varkappa}{4}+\frac{\phi}{\pi}\right)}\cdot\left( \frac{\Theta\left(-A(\hat{k})-\frac{\varkappa}{4}+\frac{\phi}{\pi}\right)}{\Theta\left(-A(\hat{k})-\frac{\varkappa}{4}\right)}\right)'\Big|_{\hat{k}=\infty} e^{i\phi}\right.\nonumber\\	&~~~\left.+\frac{\Theta\left(A(\infty)-\frac{\varkappa}{4}\right)\Theta\left(-A(\infty)-\frac{\varkappa}{4}+\frac{\phi}{\pi}\right)}{\Theta\left(-A(\infty)-\frac{\varkappa}{4}\right)\Theta\left(A(\infty)-\frac{\varkappa}{4}+\frac{\phi}{\pi}\right)}e^{i\phi}\right)\left(1+o(1)\right),
		\end{align}
		where $p$ and $q$ are any two fixed positive constants, $\Theta(z)$ is the Jacobi theta function defined by \eqref{equ:Jacobi Theta func} and
		\begin{align}\label{def:Ahatk}
			A(\hat{k})=\left(2 \int_b^a\frac{1}{w_+(\zeta)}\dif\zeta\right) ^{-1}\int_b^{\hat{k}}\frac{1}{w(\zeta)}\dif\zeta,\qquad \hat{k}\in\mathbb{C}\setminus [-b,b],
		\end{align}
		with $w^2(\hat{k})=(\hat{k}^2-a^2)(\hat{k}^2-b^2)$ is the Abel integral. The quantities $\phi$ and $\varkappa$ are defined in  \eqref{equ: varkappa expression}  and \eqref{equ: phi expression} below, and
		the parameters $a$ and $b$ are determined by the equations \eqref{equ: a,b con1} and \eqref{equ:a,b con2}.
	\end{itemize}
\end{theorem}

The Painelv\'{e} II transcendents in \eqref{result:1stTran} and \eqref{result:2ndTran} are real-valued, non-singular on the real line, and behave like $k\textnormal{Ai}(s)$, $-1 \leqslant k \leqslant 1$, for large positive $s$. This one-parameter family of solutions are known as the Hastings-McLeod solutions \cite{Hastings1980ABV} for $k=\pm 1$ and the Ablowitz-Segur solutions \cite{Segur1981AsymptoticSO} for $-1< k < 1$. It is worthwhile to see that the Painlev\'e transcendents and their higher order analogues play important roles in asymptotic studies of integrable PDEs. Among others, we mention their appearances in the KdV, mKdV equations and its hierarchy \cites{Segur1981AsymptoticSO,Deift1992ASD,CharlierLenellsJLMS,Huang20225291}, in the CH equation \cite{Monvel2010PainlevTypeAF}, in the Sasa-Satsuma equation \cite{Huang20207480}, in the defocusing nonliear Schr\"{o}dinger equation under nonzero background \cite{wang2023defocusing} and in the small dispersion limit of nonlinear integrable PDEs \cites{Bertola2009UniversalityIT,Bertola13,Claeys2008UniversalityOT,Claeys2008PainlevIA,Lu22}

The collisionless shock region $\mathcal{R}_{III}$ only occurs in the generic case $|r(\pm 1)|=1$. The disappearance of the shock region when $|r(\pm 1)|< 1$ is closely related to the significantly different asymptotics between the Hastings-McLeod solution and the Ablowitz-Segur solution of Painelv\'{e} II equation as $s\to -\infty$; cf. \cite{FIKN}. Similar phenomenon has also been observed in the KdV equation \cites{MSasycollision, Deift1994TheCS} and the CH equation on the half-line \cite{Monvel2009LongTA}. By Remarks \ref{remark a,b} and \ref{RK:indeppq} below, the asymptotic formula \eqref{result: u(x,t); part (c)} is independent of the choice of $p$ and $q$.

Finally, by adding a linear dispersion term $\kappa u_x$  to  \eqref{def:mCHeq} will lead to the following form of the
mCH equation
\begin{align}\label{mch}
	&m_{t}+\left(m\left(u^{2}-u_{x}^{2}\right)\right)_{x}+\kappa u_{x}=0,
\end{align}	
where $\kappa >0$ characterizes the effect of the linear dispersion. The strategy used in this paper also works to the equation \eqref{mch} after proper modifications. We leave the details of the calculation to the interested readers.

\paragraph{Outline of the paper}\label{subsec:outline}
The rest of this paper is organized as follows.  In Section \ref{sec:prework}, we recall an RH characterization of the mCH equation. 
After a preliminary transformation, this RH problem is asymptotically equivalent to a regular (holomorphic) RH problem, which serves as the starting point of performing $\bar{\partial}$ 
nonlinear steepest descent analysis \cites{McLaughlin2006,Cuccagna2016OnTA,Borghese2018LongTA,Dieng2019}. 
In Sections \ref{sec: 1st transition zone} and \ref{sec:2nd transition zone}, we prove parts (a) and (b) of Theorem \ref{mainthm}, respectively. 
In both cases, we need to open $\bar{\partial}$ lenses to construct a mixed $\bar{\partial}$-RH problem, which can be decomposed into a pure RH problem and a pure $\bar{\partial}$-problem. The Painelv\'{e} II parametrix introduced in Appendix \ref{appendix: RHP for PII} plays an important role in the analysis of the pure RH problem. In Section \ref{sec: collisionless shock region}, we carry out the analysis in the third transition zone. A key step in this case is to apply the $g$-function mechanism \cite{Deift1994TheCS} to arrive at a model RH problem solvable in terms of the Jacobi theta function. The outcome of our analysis is the proof of part (c) of Theorem \ref{mainthm}, which is presented at the end of the section.

\paragraph{Notation} Throughout this paper, the following notation will be used.
%\label{subsec:notation}
\begin{itemize}
	\item[$\bullet$]The norm in the classical space $ L^p (\mathbb{R})$, $ 1 \leqslant p \leqslant \infty $,  is written as $ \|\cdot \|_{p}$.
	\item[$\bullet$]The weighted Sobolev space $H^{k,s}$ defined in \eqref{def:SobSpace} is a Banach space
	equipped with the norm
	\begin{equation*}
		\Vert f \Vert_{H^{k,s}}=\sum_{j=0}^{k}\Vert \langle \cdot \rangle^s\partial_x^{j}f \Vert_{2}.
	\end{equation*}
	\item[$\bullet$]
    For a complex-valued function $f(z)$, we use
	\begin{equation*}
		f^{*}(z):=\overline{f(\bar{z})}, \qquad z\in\mathbb{C},
	\end{equation*}
	to denote its Schwartz conjugation.
	\item[$\bullet$]For a region $U\subseteq \mathbb{C}$, we use $U^*$ to denote the conjugated region of $U$.
	\item[$\bullet$]As usual, the classical Pauli matrices $\{\sigma_j\}_{j=1,2,3}$ are defined by
	\begin{equation}\label{def:PauliM}
		\sigma_1:=\begin{pmatrix}0 & 1 \\ 1 & 0\end{pmatrix}, \quad
		\sigma_2:=\begin{pmatrix}0 & -i \\ i & 0\end{pmatrix}, \quad
		\sigma_3:=\begin{pmatrix}1 & 0 \\ 0 & -1\end{pmatrix}.
	\end{equation}
   For a $2\times 2$ matrix $A$, we also define
	$$e^{\hat{\sigma}_j}A:=e^{\sigma_j}Ae^{-\sigma_j}, \quad j=1,2,3.$$
	\item[$\bullet$]For any smooth oriented curve $\Sigma$, the Cauchy operator $\mathcal{C}$ on $\Sigma$ is defined  by
	\begin{align*}
		\mathcal{C}f(z)=\frac{1}{2\pi i}\int_{\Sigma}\frac{f(\zeta)}{\zeta-z}\dif\zeta, \qquad  z\in\mathbb{C}\setminus \Sigma.
	\end{align*}
	Given a function $f \in L^p(\Sigma)$, $1\leqslant p<\infty$,
	\begin{align}\label{def:Cpm}
		\mathcal{C}_\pm f(z):=\lim_{z'\to z\in\Sigma}\frac{1}{2\pi i}\int_{\Sigma}\frac{f(\zeta)}{\zeta-z'}\dif\zeta
	\end{align}
   stands for the positive/negative (according to the orientation of $\Sigma$) non-tangential boundary value of $\mathcal{C}f$.
	\item[$\bullet$]We write $a\lesssim b$ to denote the inequality $a\leqslant Cb$ for some constant $C>0$.
    \item[$\bullet$] If $A$ is a matrix, then $(A)_{ij}$ stands for its $(i,j)$-th entry.
\end{itemize}

\section{Preliminaries}\label{sec:prework}

%\subsection{Scattering data and basic RH problem $M$}\label{subsec:basic RH M}
\subsection {An RH characterization of the mCH equation}\label{subsec:RH characterization}
The RH problem associated to the Cauchy problem \eqref{mcho2} for the mCH equation is based on the delicate analysis for the properties of Jost functions, reflection coefficient as well as the eigenvalues,
which form the direct theory for mCH equation.

Recall the following Lax pair (see \cites{Qiao2006ANI,Schiff1996}) of the mCH equation:
\begin{equation}\label{lax0}
	\Phi_x = X \Phi,\hspace{0.5cm}\Phi_t =T \Phi,
\end{equation}
where
\begin{align}
	&X=\frac{1}{2}\left(\begin{array}{cc}
		-1 &  {\lambda} m \\
		-\lambda  m & 1
	\end{array}\right), \qquad \qquad   \lambda=-\frac{1}{2}(z+z^{-1}), \nonumber\\
	&T=\left(\begin{array}{cc}
		\lambda ^{-2}+\frac{1}{2}(u^2-u_x^2) & -\lambda ^{-1} (u-u_x+1)-\frac{ 1}{2}\lambda (u^2-u_x^2)m \\[5pt]
		\lambda ^{-1} (u+u_x+1)+\frac{ 1}{2}\lambda (u^2-u_x^2)m  & -\lambda ^{-2}-\frac{1}{2}(u^2-u_x^2)
	\end{array}\right).\nonumber
\end{align}
The eigenvalues of spectral functions are distributed on the unit circle except on the real and imaginary axes, which are simple and finite. From \eqref{lax0}, an RH problem was proposed in \cites{Monvel2020ARA}, which is meromorphic in $\mathbb{C}\setminus \mathbb{R}$ and admits simple poles at $\pm 1$ and at the discrete spectrum. This RH problem, however, is not related to the solution of the mCH equation in a direct way. Due to this reason, the following RH problem was firstly introduced in \cite{Monvel2020TheMC}, which removes the singularities at $z=\pm 1$ on $\mathbb{R}$. To proceed,
let $\zeta_1,\cdots,\zeta_{\mathcal{N}}$ and $r(k)$ be the eigenvalues in the fourth quadrant and the reflection coefficient associated to
the Cauchy problem \eqref{mcho2}. For convenience, we use the notation $z_j$ instead of $\zeta_j$ to represent the discrete spectrum in such way that $z_l=\zeta_j$, $j\in\{1,\cdots,\mathcal{N}\}\leadsto l\in\{1,\cdots, \mathcal{N}\}$ and $z_l=-\bar{\zeta}_j$, $j\in\{1,\cdots, \mathcal{N}\}\leadsto l\in\{\mathcal{N}+1,\cdots,2\mathcal{N}\}$. Thus, $\mathcal{Z}=\{z_j\}_{j=1}^{2\mathcal{N}}$$(|z_j|=1)$ stands for the collection of those discrete spectrums located in the lower half plane. The RH problem related to the Cauchy problem of the mCH equation reads as follows.

\begin{RHP}\label{RHP:M^{(1)}}
	\hfill
	\begin{itemize}
		\item[$\bullet$] $M^{(1)}(z)$ is meromorphic for $z\in\mathbb{C}\setminus \mathbb{R}$ with simple poles in the set $\mathcal{Z}\cup {\mathcal{Z}}^*$, and satisfies the symmetry relations
		 $$M^{(1)}(z)=\sigma_1\overline{M^{(1)}(\bar{z})}\sigma_1=\sigma_2M^{(1)}(-z)\sigma_2,$$
    where $\sigma_j$, $j=1,2$, is the Pauli matrix defined in \eqref{def:PauliM}.
    %=\sigma_1M^{(1)}(0)^{-1}M^{(1)}(z^{-1})\sigma_1$.
		\item[$\bullet$] For $z\in \mathbb{R}$, we have
		\begin{equation}
			M^{(1)}_+(z)=M^{(1)}_-(z)V^{(1)}(z)
		\end{equation}
       where
		\begin{equation}
			V^{(1)}(z)=\left(\begin{array}{cc}
				1-|r(z)|^2 & r(z)e^{2i\theta(z)}\\
				-\bar{r}(z)e^{-2i\theta(z)} & 1
			\end{array}\right)  \label{V}
		\end{equation}
		with
		\begin{align}
			&\theta(z):=\theta(z;y,t)=-\frac{1}{4}\left(z-\frac{1}{z} \right) \left( y-\frac{8t}{(z+1/z)^2}\right) \label{theta},
%       \\
			%&r(z)=-\overline{r(-z)}=\overline{r(z^{-1})}.
		\end{align}
		and where
		\begin{equation}
			y(x)=x- \int_{x}^{+\infty}\left(m(\zeta)-1\right)\dif\zeta.
		\end{equation}
        is a space scaling variable
		\item[$\bullet$] For each $z_j \in \mathcal{Z}$, $j=1,\ldots,2\mathcal{N}$, we have the following residue conditions:
                          \begin{align}
			&\res_{z=z_j}M(z)=\lim_{z\to z_j}M(z)\left(\begin{array}{cc}
				0 & c_je^{2i\theta(z_j)}\\
				0 & 0
			\end{array}\right),\\
			&\res_{z=\bar{z}_j}M(z)=\lim_{z\to \bar{z}_j}M(z)\left(\begin{array}{cc}
				0 & 0\\
				\bar{c}_je^{-2i\theta(\bar{z}_j)} & 0
			\end{array}\right),
		\end{align}
       where $c_j$ is the norming constant associated with $z_j$.
\item[$\bullet$] As $z\rightarrow\infty$ in $\mathbb{C}\setminus \mathbb{R}$, we have $M^{(1)}(z)=I+\mathcal{O}(z^{-1})$.

	\end{itemize}
\end{RHP}
Suppose the initial data satisfies Assumption \ref{Assumption}, it has been shown in \cite{YFLGlobalmCH} that the scattering data  $\left(r,\{z_j,c_j\}_{j=1}^{2\mathcal{N}}\right)$
belongs to $\left(H^{1,2}(\mathbb{R})\cap H^{2,1}(\mathbb{R})\right) \otimes\mathbb{C}^{2\mathcal{N}}\otimes\mathbb{C}^{2\mathcal{N}}$, and the above RH problem admits a unique solution. Moreover, we can use the local behaviors of RH problem \ref{RHP:M^{(1)}} at $z=0$ and $z=i$ to characterize the solution for the Cauchy problem of the mCH equation in the following way.
%\eqref{mcho2}
%$M(z;y,t)$ and $M^{(1)}(z;y,t)$ satisfy the relation
%\begin{align}\label{BPTrans}
%	\left( I-\frac{\sigma_1}{z}\right) M(z;y,t)=\left( I-\frac{\sigma_1M^{(1)}(0;y,t)^{-1}}{z} \right) M^{(1)}(z;y,t).
%\end{align}
% Now we can use the RH problem $M^{(1)}$ to characterize the solution
%for the mCH equation \eqref{mcho2}, which is presented in the next proposition:
\begin{Proposition}\label{prop: M^{(1)}}
	Let $M^{(1)}(z;y,t)$ is the unique solution of RH problem \ref{RHP:M^{(1)}}, we have
	\begin{align}
		M^{(1)}(0 ;y,t)=\left(\begin{array}{cc}
			\alpha(y,t) & i\beta(y,t) \\
			-i\beta(y,t) & \alpha(y,t)
		\end{array}\right),
	\end{align}
	where $\alpha(y,t)$, $\beta(y,t)$ are real functions satisfying $\alpha^2-\beta^2=1$.
	If $\beta\neq0$, we have
	\begin{align}
		M^{(1)}(z;y,t)=&\left(\begin{array}{cc}
			f_1(y,t) &\frac{ i\beta}{\alpha+1}f_2(y,t) \\
			-\frac{ i\beta}{\alpha+1}f_1(y,t) & f_2(y,t)
		\end{array}\right)\nonumber\\
		&+\left(\begin{array}{cc}
			\frac{ i\beta}{\alpha+1}g_1(y,t) & g_2(y,t) \\
			g_1(y,t) & -\frac{ i\beta}{\alpha+1}g_2(y,t)
		\end{array}\right)(z-i)+\mathcal{O}\left(\left(z-i\right)^2\right), \quad z\to i,
	\end{align}
	where $g_1(y,t)$, $g_2(y,t)$, $f_1(y,t)$, $f_2(y,t)$ are real functions.

	The solution $u(x,t)=u(x(y,t),t)$ of the Cauchy problem \eqref{mcho2} can then be expressed in the following parametric form:
	\begin{subequations}\label{recovering u}
		\begin{align}
			x(y,t)&= y+2\log \left( \alpha_1(y,t)\right) ,\label{res2} \\
           u(y,t)&= 1-\alpha_2(y,t)\alpha_1(y,t)-\alpha_3(y,t)\alpha_1(y,t) ^{-1},\label{res1}
		\end{align}
	\end{subequations}
	where
	\begin{align}
		&\alpha_1(y,t)=\left( 1-\frac{\beta}{\alpha+1}\right) f_1,\ \ 	\alpha_2(y,t)=\frac{\beta}{\alpha+1}f_2+\left( 1-\frac{\beta}{\alpha+1}\right) g_2,\label{al1}\\
		&\alpha_3(y,t)=\frac{-\beta}{\alpha+1}f_1+\left( 1-\frac{\beta}{\alpha+1}\right) g_1.\label{al2}
	\end{align}
\end{Proposition}

%\subsection{From Transform to a regular RH problem: $M^{(1)}\rightarrow M^{(2)}\rightarrow M^{(3)}$}\label{subsec:deform RH problem}
\subsection{From $M^{(1)}$ to a regular RH problem }\label{subsec:deform RH problem}
It is the aim of this section to construct a regular (holomorphic) RH problem as our basis for further asymptotic analysis. The regular RH problem should satisfy the following conditions.
%The $\bar{\partial}$-nonlinear steepest descent method for oscillatory RH problems consists in a series of transformations, in
%order to reach a model RH problem which can be explicitly solved. Main aim in this section is to
%construct a regular (holomorphic) RH problem as our basis for the subsequent analysis. The regular RH problem should satisfy conditions as %follows:
\begin{itemize}
	\item[(i)] It admits a similar jump structure as in the RH problem \ref{RHP:M^{(1)}} for $M^{(1)}$ on the real line.
	\item[(ii)] Different factorizations of the jump matrix \eqref{V} should be taken for different transition zones.
	\item[(iii)] The residue conditions for $M^{(1)}$ should be converted into the jump conditions on the auxiliary contours such that the jump matrices
	vanish rapidly to the identity matrix as the parameter $t\rightarrow+\infty$.
\end{itemize}
In particular, the three items above are closely related to the growth or decay of the exponential function $e^{\pm2i\theta}$ (more specifically, the sign of $\im\theta(z)$) encountered in  both the jump condition and the residue conditions of $M^{(1)}$, where $\theta(z)$ defined in \eqref{theta} can be rewritten as
\begin{align}\label{def:phasefunc}
	\theta(z)=-\frac{t}{4}(z-z^{-1})\left[\hat{\xi}-\frac{8}{(z+z^{-1})^2} \right], \qquad \hat{\xi}:=\frac{y}{t}.
\end{align}
The signature tables for $\hat{\xi}=2$ and $\hat{\xi}=-1/4$ are depicted in Figure \ref{fig:critfigSigntable}.
\begin{figure}[htbp]
	\begin{center}
\tikzset{every picture/.style={line width=0.75pt}} %set default line width to 0.75pt
\begin{tikzpicture}[x=0.75pt,y=0.75pt,yscale=-1,xscale=1]
%uncomment if require: \path (0,300); %set diagram left start at 0, and has height of 300
%Straight Lines [id:da45914606591309615]
\draw    (3,146.28) -- (287.7,146.28) ;
\draw [shift={(289.7,146.28)}, rotate = 180] [color={rgb, 255:red, 0; green, 0; blue, 0 }  ][line width=0.75]    (10.93,-3.29) .. controls (6.95,-1.4) and (3.31,-0.3) .. (0,0) .. controls (3.31,0.3) and (6.95,1.4) .. (10.93,3.29)   ;
%Straight Lines [id:da9397656467654982]
\draw    (137.48,25.35) -- (137.48,260) ;
	\draw (128.48,265) node [anchor=north west][inner sep=0.75pt]  [font=\scriptsize]  {$(a)$};
\draw [shift={(137.48,23.35)}, rotate = 90] [color={rgb, 255:red, 0; green, 0; blue, 0 }  ][line width=0.75]    (10.93,-3.29) .. controls (6.95,-1.4) and (3.31,-0.3) .. (0,0) .. controls (3.31,0.3) and (6.95,1.4) .. (10.93,3.29)   ;
%Shape: Arc [id:dp5923123047735206]
\draw  [draw opacity=0] (136.99,90.34) .. controls (137.42,84.69) and (139.2,79.55) .. (142.47,75.44) .. controls (153.41,61.67) and (176.8,64.63) .. (194.71,82.05) .. controls (212.62,99.48) and (218.27,124.76) .. (207.33,138.53) .. controls (204.06,142.64) and (199.68,145.26) .. (194.68,146.45) -- (174.9,106.98) -- cycle ; \draw   (136.99,90.34) .. controls (137.42,84.69) and (139.2,79.55) .. (142.47,75.44) .. controls (153.41,61.67) and (176.8,64.63) .. (194.71,82.05) .. controls (212.62,99.48) and (218.27,124.76) .. (207.33,138.53) .. controls (204.06,142.64) and (199.68,145.26) .. (194.68,146.45) ;
%Shape: Arc [id:dp7944447446655349]
\draw  [draw opacity=0] (194.09,146.27) .. controls (199.31,147.16) and (203.94,149.59) .. (207.45,153.64) .. controls (219.03,166.97) and (214.19,192.88) .. (196.66,211.51) .. controls (179.12,230.14) and (155.53,234.42) .. (143.95,221.09) .. controls (140.44,217.04) and (138.44,211.83) .. (137.85,206.01) -- (175.7,187.36) -- cycle ; \draw   (194.09,146.27) .. controls (199.31,147.16) and (203.94,149.59) .. (207.45,153.64) .. controls (219.03,166.97) and (214.19,192.88) .. (196.66,211.51) .. controls (179.12,230.14) and (155.53,234.42) .. (143.95,221.09) .. controls (140.44,217.04) and (138.44,211.83) .. (137.85,206.01) ;
%Shape: Arc [id:dp2297804988363059]
\draw  [draw opacity=0] (82.51,146.99) .. controls (77.63,145.92) and (73.31,143.46) .. (69.97,139.56) .. controls (58.62,126.27) and (63.07,101.21) .. (79.92,83.58) .. controls (96.77,65.96) and (119.64,62.46) .. (130.99,75.75) .. controls (134.33,79.65) and (136.3,84.57) .. (136.99,90.02) -- (100.48,107.65) -- cycle ; \draw   (82.51,146.99) .. controls (77.63,145.92) and (73.31,143.46) .. (69.97,139.56) .. controls (58.62,126.27) and (63.07,101.21) .. (79.92,83.58) .. controls (96.77,65.96) and (119.64,62.46) .. (130.99,75.75) .. controls (134.33,79.65) and (136.3,84.57) .. (136.99,90.02) ;
%Shape: Arc [id:dp981214153296603]
\draw  [draw opacity=0] (137.8,205.74) .. controls (137.05,211.3) and (134.99,216.28) .. (131.53,220.16) .. controls (119.93,233.18) and (96.97,228.96) .. (80.26,210.72) .. controls (63.54,192.49) and (59.39,167.15) .. (70.99,154.13) .. controls (74.45,150.25) and (78.92,147.9) .. (83.93,147) -- (101.26,187.14) -- cycle ; \draw   (137.8,205.74) .. controls (137.05,211.3) and (134.99,216.28) .. (131.53,220.16) .. controls (119.93,233.18) and (96.97,228.96) .. (80.26,210.72) .. controls (63.54,192.49) and (59.39,167.15) .. (70.99,154.13) .. controls (74.45,150.25) and (78.92,147.9) .. (83.93,147) ;
%Curve Lines [id:da3391805629424751]
\draw    (137.01,90.18) .. controls (139,107.31) and (167.53,120.34) .. (193.96,146.25) ;
%Curve Lines [id:da09814097665553656]
\draw    (82.65,147.02) .. controls (83,154.89) and (147,88.5) .. (137.01,90.18) ;
%Curve Lines [id:da3506372504827182]
\draw    (137.83,205.87) .. controls (142,211.32) and (114,169.27) .. (84.56,147.77) ;
%Curve Lines [id:da7447805404334058]
\draw    (138.28,205.9) .. controls (137,202.47) and (189,141.61) .. (194.44,146.5) ;
%Straight Lines [id:da9497888703100359]
\draw    (318,147.15) -- (632.53,147.15) ;
\draw [shift={(634.53,147.15)}, rotate = 180] [color={rgb, 255:red, 0; green, 0; blue, 0 }  ][line width=0.75]    (10.93,-3.29) .. controls (6.95,-1.4) and (3.31,-0.3) .. (0,0) .. controls (3.31,0.3) and (6.95,1.4) .. (10.93,3.29)   ;
%Straight Lines [id:da26778935520690506]
\draw    (466.47,24) -- (466.47,261) ;
	\draw (458.47,266)node [anchor=north west][inner sep=0.75pt]  [font=\scriptsize]  {$(b)$};
\draw [shift={(466.47,22)}, rotate = 90] [color={rgb, 255:red, 0; green, 0; blue, 0 }  ][line width=0.75]    (10.93,-3.29) .. controls (6.95,-1.4) and (3.31,-0.3) .. (0,0) .. controls (3.31,0.3) and (6.95,1.4) .. (10.93,3.29)   ;
%Shape: Arc [id:dp7259645330673126]
\draw  [draw opacity=0] (467.19,47.12) .. controls (475.35,43.48) and (484.27,41.44) .. (493.62,41.37) .. controls (533.09,41.07) and (565.35,75.87) .. (565.68,119.1) .. controls (566.02,162.34) and (534.29,197.63) .. (494.82,197.94) .. controls (485.47,198.01) and (476.52,196.11) .. (468.3,192.59) -- (494.22,119.65) -- cycle ; \draw   (467.19,47.12) .. controls (475.35,43.48) and (484.27,41.44) .. (493.62,41.37) .. controls (533.09,41.07) and (565.35,75.87) .. (565.68,119.1) .. controls (566.02,162.34) and (534.29,197.63) .. (494.82,197.94) .. controls (485.47,198.01) and (476.52,196.11) .. (468.3,192.59) ;
%Shape: Ellipse [id:dp233895178341466]
\draw  [dash pattern={on 0.84pt off 2.51pt}] (82.29,150.86) .. controls (82.29,117.43) and (106.78,90.34) .. (136.99,90.34) .. controls (167.2,90.34) and (191.69,117.43) .. (191.69,150.86) .. controls (191.69,184.28) and (167.2,211.38) .. (136.99,211.38) .. controls (106.78,211.38) and (82.29,184.28) .. (82.29,150.86) -- cycle ;
%Shape: Arc [id:dp15841219591215272]
\draw  [draw opacity=0] (467.19,102.99) .. controls (475.36,99.39) and (484.28,97.37) .. (493.64,97.3) .. controls (533.38,97) and (565.86,131.8) .. (566.19,175.03) .. controls (566.52,218.27) and (534.58,253.56) .. (494.84,253.87) .. controls (485.49,253.94) and (476.53,252.07) .. (468.3,248.59) -- (494.24,175.59) -- cycle ; \draw   (467.19,102.99) .. controls (475.36,99.39) and (484.28,97.37) .. (493.64,97.3) .. controls (533.38,97) and (565.86,131.8) .. (566.19,175.03) .. controls (566.52,218.27) and (534.58,253.56) .. (494.84,253.87) .. controls (485.49,253.94) and (476.53,252.07) .. (468.3,248.59) ;
%Shape: Arc [id:dp9971906139779523]
\draw  [draw opacity=0] (465.08,192.58) .. controls (456.84,196.04) and (447.88,197.88) .. (438.53,197.75) .. controls (399.06,197.17) and (367.58,161.66) .. (368.2,118.43) .. controls (368.83,75.2) and (401.33,40.62) .. (440.8,41.19) .. controls (450.15,41.33) and (459.05,43.42) .. (467.19,47.12) -- (439.66,119.47) -- cycle ; \draw   (465.08,192.58) .. controls (456.84,196.04) and (447.88,197.88) .. (438.53,197.75) .. controls (399.06,197.17) and (367.58,161.66) .. (368.2,118.43) .. controls (368.83,75.2) and (401.33,40.62) .. (440.8,41.19) .. controls (450.15,41.33) and (459.05,43.42) .. (467.19,47.12) ;
%Shape: Arc [id:dp7433521184038936]
\draw  [draw opacity=0] (465.08,248.45) .. controls (456.84,251.91) and (447.88,253.75) .. (438.53,253.61) .. controls (399.06,253.04) and (367.57,217.53) .. (368.2,174.3) .. controls (368.83,131.07) and (401.33,96.48) .. (440.79,97.06) .. controls (450.15,97.19) and (459.05,99.29) .. (467.19,102.99) -- (439.66,175.33) -- cycle ; \draw   (465.08,248.45) .. controls (456.84,251.91) and (447.88,253.75) .. (438.53,253.61) .. controls (399.06,253.04) and (367.57,217.53) .. (368.2,174.3) .. controls (368.83,131.07) and (401.33,96.48) .. (440.79,97.06) .. controls (450.15,97.19) and (459.05,99.29) .. (467.19,102.99) ;
%Shape: Ellipse [id:dp8324751775244381]
\draw  [dash pattern={on 0.84pt off 2.51pt}] (421.69,147.78) .. controls (421.69,123.04) and (442.06,102.99) .. (467.19,102.99) .. controls (492.31,102.99) and (512.69,123.04) .. (512.69,147.78) .. controls (512.69,172.53) and (492.31,192.58) .. (467.19,192.58) .. controls (442.06,192.58) and (421.69,172.53) .. (421.69,147.78) -- cycle ;
%Curve Lines [id:da9195263463500782]
\draw    (467.19,102.99) .. controls (498,107.29) and (494,166.29) .. (465,163.29) ;
%Curve Lines [id:da9321124716717923]
\draw    (466.84,192.58) .. controls (488,187.29) and (501,138.29) .. (467.54,127.27) ;
%Curve Lines [id:da5875880994952793]
\draw    (467.19,102.99) .. controls (440,100.29) and (437,158.29) .. (465,163.29) ;
%Curve Lines [id:da11142781440119398]
\draw    (466.84,192.58) .. controls (433,175.29) and (446.53,125.07) .. (469,127.29) ;
%Straight Lines [id:da12063508276303869]
\draw  [dash pattern={on 0.84pt off 2.51pt}]  (574,34.29) -- (493.22,139.7) ;
\draw [shift={(492,141.29)}, rotate = 307.46] [color={rgb, 255:red, 0; green, 0; blue, 0 }  ][line width=0.75]    (10.93,-3.29) .. controls (6.95,-1.4) and (3.31,-0.3) .. (0,0) .. controls (3.31,0.3) and (6.95,1.4) .. (10.93,3.29)   ;
%Straight Lines [id:da5435915589116844]
\draw  [dash pattern={on 0.84pt off 2.51pt}]  (356,29.29) -- (441.8,143.69) ;
\draw [shift={(443,145.29)}, rotate = 233.13] [color={rgb, 255:red, 0; green, 0; blue, 0 }  ][line width=0.75]    (10.93,-3.29) .. controls (6.95,-1.4) and (3.31,-0.3) .. (0,0) .. controls (3.31,0.3) and (6.95,1.4) .. (10.93,3.29)   ;
% Text Node
\draw (281.55,151.93) node [anchor=north west][inner sep=0.75pt]  [font=\scriptsize]  {$\re z$};
% Text Node
\draw (288.35,138.3) node [anchor=north west][inner sep=0.75pt]   [align=left] { };
% Text Node
\draw (138.21,13.63) node [anchor=north west][inner sep=0.75pt]  [font=\scriptsize]  {$\im z$};
% Text Node
\draw (129.1,134.51) node [anchor=north west][inner sep=0.75pt]  [font=\scriptsize]  {$0$};
% Text Node
\draw (71,154.72) node [anchor=north west][inner sep=0.75pt]  [font=\scriptsize]  {$-1$};
% Text Node
\draw (195.96,152.92) node [anchor=north west][inner sep=0.75pt]  [font=\scriptsize]  {$1$};
% Text Node
\draw (248,109.46) node [anchor=north west][inner sep=0.75pt]  [font=\footnotesize,color={rgb, 255:red, 208; green, 2; blue, 27 }  ,opacity=1 ]  {$-$};
% Text Node
\draw (250,182.48) node [anchor=north west][inner sep=0.75pt]  [font=\footnotesize,color={rgb, 255:red, 208; green, 2; blue, 27 }  ,opacity=1 ]  {$+$};
% Text Node
\draw (174,96.18) node [anchor=north west][inner sep=0.75pt]  [font=\footnotesize,color={rgb, 255:red, 208; green, 2; blue, 27 }  ,opacity=1 ]  {$+$};
% Text Node
\draw (177,184.7) node [anchor=north west][inner sep=0.75pt]  [font=\footnotesize,color={rgb, 255:red, 208; green, 2; blue, 27 }  ,opacity=1 ]  {$-$};
% Text Node
\draw (138,116.1) node [anchor=north west][inner sep=0.75pt]  [font=\footnotesize,color={rgb, 255:red, 208; green, 2; blue, 27 }  ,opacity=1 ]  {$-$};
% Text Node
\draw (139,162.57) node [anchor=north west][inner sep=0.75pt]  [font=\footnotesize,color={rgb, 255:red, 208; green, 2; blue, 27 }  ,opacity=1 ]  {$+$};
% Text Node
\draw (81,91.76) node [anchor=north west][inner sep=0.75pt]  [font=\footnotesize,color={rgb, 255:red, 208; green, 2; blue, 27 }  ,opacity=1 ]  {$+$};
% Text Node
\draw (86,189.12) node [anchor=north west][inner sep=0.75pt]  [font=\footnotesize,color={rgb, 255:red, 208; green, 2; blue, 27 }  ,opacity=1 ]  {$-$};
% Text Node
\draw (27,101.71) node [anchor=north west][inner sep=0.75pt]  [font=\footnotesize,color={rgb, 255:red, 208; green, 2; blue, 27 }  ,opacity=1 ]  {$-$};
% Text Node
\draw (33,201.29) node [anchor=north west][inner sep=0.75pt]  [font=\footnotesize,color={rgb, 255:red, 208; green, 2; blue, 27 }  ,opacity=1 ]  {$+$};
% Text Node
\draw (626.73,151.9) node [anchor=north west][inner sep=0.75pt]  [font=\scriptsize]  {$\re z$};
% Text Node
\draw (633.61,138.17) node [anchor=north west][inner sep=0.75pt]   [align=left] { };
% Text Node
\draw (468.53,12.22) node [anchor=north west][inner sep=0.75pt]  [font=\scriptsize]  {$\im z$};
% Text Node
\draw (599,118.48) node [anchor=north west][inner sep=0.75pt]  [font=\footnotesize,color={rgb, 255:red, 208; green, 2; blue, 27 }  ,opacity=1 ]  {$+$};
% Text Node
\draw (600,177.46) node [anchor=north west][inner sep=0.75pt]  [font=\footnotesize,color={rgb, 255:red, 208; green, 2; blue, 27 }  ,opacity=1 ]  {$-$};
% Text Node
\draw (492,68.46) node [anchor=north west][inner sep=0.75pt]  [font=\footnotesize,color={rgb, 255:red, 208; green, 2; blue, 27 }  ,opacity=1 ]  {$-$};
% Text Node
\draw (505,211.48) node [anchor=north west][inner sep=0.75pt]  [font=\footnotesize,color={rgb, 255:red, 208; green, 2; blue, 27 }  ,opacity=1 ]  {$+$};
% Text Node
\draw (518,122.48) node [anchor=north west][inner sep=0.75pt]  [font=\footnotesize,color={rgb, 255:red, 208; green, 2; blue, 27 }  ,opacity=1 ]  {$+$};
% Text Node
\draw (522,163.46) node [anchor=north west][inner sep=0.75pt]  [font=\footnotesize,color={rgb, 255:red, 208; green, 2; blue, 27 }  ,opacity=1 ]  {$-$};
% Text Node
\draw (464,111.46) node [anchor=north west][inner sep=0.75pt]  [font=\footnotesize,color={rgb, 255:red, 208; green, 2; blue, 27 }  ,opacity=1 ]  {$-$};
% Text Node
\draw (467,166.69) node [anchor=north west][inner sep=0.75pt]  [font=\footnotesize,color={rgb, 255:red, 208; green, 2; blue, 27 }  ,opacity=1 ]  {$+$};
% Text Node
\draw (469.54,130.67) node [anchor=north west][inner sep=0.75pt]  [font=\footnotesize,color={rgb, 255:red, 208; green, 2; blue, 27 }  ,opacity=1 ]  {$+$};
% Text Node
\draw (468.84,151.18) node [anchor=north west][inner sep=0.75pt]  [font=\footnotesize,color={rgb, 255:red, 208; green, 2; blue, 27 }  ,opacity=1 ]  {$-$};
% Text Node
\draw (453.54,129.67) node [anchor=north west][inner sep=0.75pt]  [font=\footnotesize,color={rgb, 255:red, 208; green, 2; blue, 27 }  ,opacity=1 ]  {$+$};
% Text Node
\draw (454.84,149.18) node [anchor=north west][inner sep=0.75pt]  [font=\footnotesize,color={rgb, 255:red, 208; green, 2; blue, 27 }  ,opacity=1 ]  {$-$};
% Text Node
\draw (406,121.48) node [anchor=north west][inner sep=0.75pt]  [font=\footnotesize,color={rgb, 255:red, 208; green, 2; blue, 27 }  ,opacity=1 ]  {$+$};
% Text Node
\draw (410,168.46) node [anchor=north west][inner sep=0.75pt]  [font=\footnotesize,color={rgb, 255:red, 208; green, 2; blue, 27 }  ,opacity=1 ]  {$-$};
% Text Node
\draw (414,69.46) node [anchor=north west][inner sep=0.75pt]  [font=\footnotesize,color={rgb, 255:red, 208; green, 2; blue, 27 }  ,opacity=1 ]  {$-$};
% Text Node
\draw (422,210.48) node [anchor=north west][inner sep=0.75pt]  [font=\footnotesize,color={rgb, 255:red, 208; green, 2; blue, 27 }  ,opacity=1 ]  {$+$};
% Text Node
\draw (347,197.46) node [anchor=north west][inner sep=0.75pt]  [font=\footnotesize,color={rgb, 255:red, 208; green, 2; blue, 27 }  ,opacity=1 ]  {$-$};
% Text Node
\draw (341,107.48) node [anchor=north west][inner sep=0.75pt]  [font=\footnotesize,color={rgb, 255:red, 208; green, 2; blue, 27 }  ,opacity=1 ]  {$+$};
% Text Node
\draw (514.69,151.18) node [anchor=north west][inner sep=0.75pt]  [font=\scriptsize]  {$1$};
% Text Node
\draw (405.69,149.18) node [anchor=north west][inner sep=0.75pt]  [font=\scriptsize]  {$-1$};
% Text Node
\draw (568.69,150.18) node [anchor=north west][inner sep=0.75pt]  [font=\scriptsize]  {$2+\sqrt{3}$};
% Text Node
\draw (320,150.55) node [anchor=north west][inner sep=0.75pt]  [font=\scriptsize]  {$-2-\sqrt{3}$};
% Text Node
\draw (571.69,18.18) node [anchor=north west][inner sep=0.75pt]  [font=\scriptsize]  {$2-\sqrt{3}$};
% Text Node
\draw (334.69,14.4) node [anchor=north west][inner sep=0.75pt]  [font=\scriptsize]  {$-2+\sqrt{3}$};
\end{tikzpicture}
		\caption{ Signature table for $\im\theta(z)$: (a) $\hat{\xi}=2$, (b) $\hat{\xi}=-1/4$. The ``$+$'' represents where $\im \theta(z)>0$ and ``$-$'' represents where $\im \theta(z)<0$. The dashed line is the
			unit circle. }\label{fig:critfigSigntable}
	\end{center}
\end{figure}
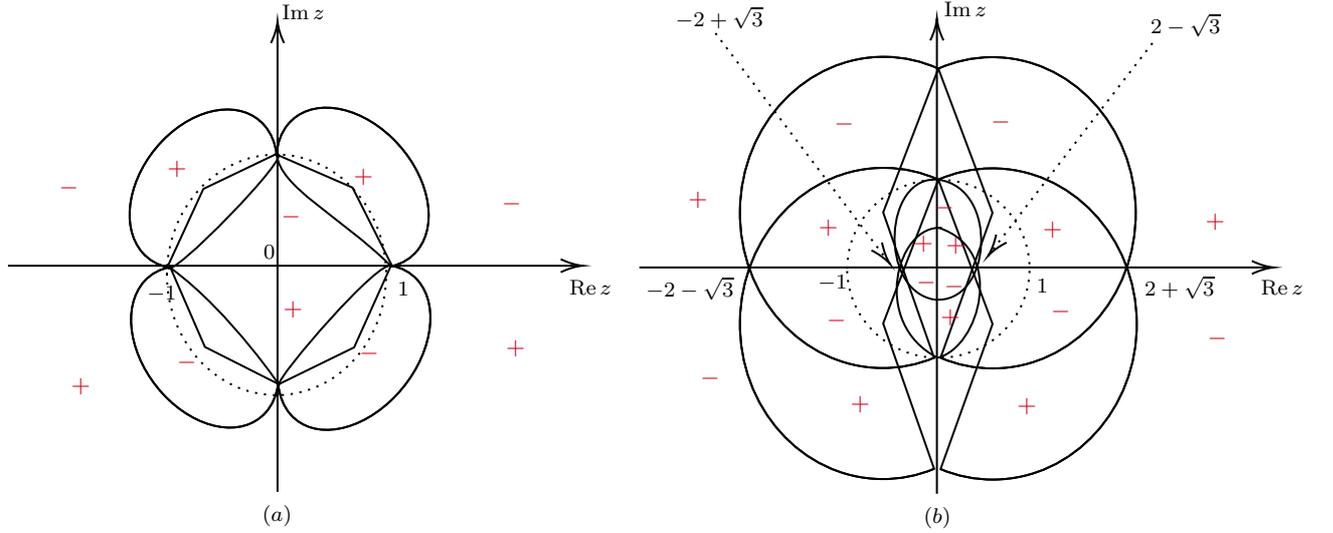

To proceed, we introduce the function
\begin{align}\label{Tfunc}
	&T(z,\hat{\xi}):=\prod_{n=1}^{2\mathcal{N}}\left(\dfrac{z-\bar{z}_n}{z-z_n}\right)\exp\left\{-\frac{1}{2\pi i}\int_{\mathrm{I}(\hat{\xi})}\frac{\log\left(1-\vert r(s)\vert^2\right)}{s-z}\dif s\right\},
\end{align}
where
\begin{align}\label{indexfunc}
	\mathrm{I}(\hat{\xi}):=\left\{ \begin{array}{ll}
		\emptyset,   &\text{if} \ \ |\hat{\xi}-2|t^{2/3}\leqslant C \ \ \text{or} \ \ 2\cdot 3^{\frac{1}{3}}<(2-\hat{\xi})t^{\frac{2}{3}}< C(\log t)^{2/3}, \\[4pt]
		\mathbb{R},   &\text{if} \ |\hat{\xi}+\frac{1}{4}|t^{2/3}\leqslant C,\\
	\end{array}\right.
\end{align}
%%Via introducing $\nu(z)=-\frac{1}{2\pi}\log(1-|r(z)|^2)$, \eqref{Tfunc} can be rewritten as
%%\begin{equation}
%%	T(z,\hat{\xi})=\prod_{n=1}^{2\mathcal{N}}\dfrac{z-\bar{z}_n}{z-z_n}\exp\left(-i\int _{\mathrm{I}(\hat{\xi})}\dfrac{\nu(\zeta ) \dif\zeta}{\zeta-z}\right).
%%\end{equation}
is an indicator function. For
\begin{align}
	\varrho:=\frac{1}{4}\min_{1\leqslant k,j\leqslant\mathcal{N}}\left\{ |z_k-z_j|,\ |\im z_k|,\ |z_k-i| \right\}, \label{varrho}	
\end{align}
we also define
$$
\mathbb{D}_n:=\{z:\ |z-z_n|\leqslant\varrho\},\qquad n=1,\ldots,2\mathcal{N},
$$
to be a small disk centered at the pole $z_n$ with radius $\varrho$.  As we shall see later, the integral term in \eqref{Tfunc} is of the great aid to factorize the jump matrix \eqref{V} for different transition zones, and the product term is useful to deal with the residue conditions. Moreover,
the auxiliary contours mentioned in item (iii) consists of the union of small circles $\partial\mathbb{D}_n$ and $\partial\mathbb{D}^*_n$ surrounding
the poles in $\mathcal{Z}\cup{\mathcal{Z}}^*$.

We now introduce the following interpolation transformation to construct a regular RH problem.
\begin{align}\label{transform:M1toM2}
	M^{(2)}(z)=M^{(2)}(z;y,t):=M^{(1)}(z)G(z)T^{\sigma_3}(z),
\end{align}
where $T$ is defined in \eqref{Tfunc} and
\begin{align}\label{funcG}
	G(z)=\left\{ \begin{array}{ll}
		\left(\begin{array}{cc}
			1 & 0\\
			-\frac{z-z_n}{c_n}e^{-2i\theta(z_n)} & 1
		\end{array}\right),   & z\in\mathbb{D}_n,\\
		\left(\begin{array}{cc}
			1 & -\frac{z-\bar{z}_n}{\bar{c}_n}e^{2i\theta(\bar{z}_n)}	\\
			0 & 1
		\end{array}\right),   & 	z\in\mathbb{D}_n^*,\\
		I, &  \text{ elsewhere}.
	\end{array}\right.
\end{align}

In view of RH problem \ref{RHP:M^{(1)}} for $M^{(1)}$, it is readily seen that $M^{(2)}$ satisfies the following RH problem.
\begin{RHP}\label{RHP:regular RH}
\hfill
	\begin{itemize}
		\item[$\bullet$] $M^{(2)}(z)$ is holomorphic for $z\in\mathbb{C}\setminus \Sigma^{(2)}$, where
		\begin{equation}
			\Sigma^{(2)}:=\mathbb{R}\cup\left[\underset{1\leqslant n\leqslant2\mathcal{N}}{\cup}\left( \partial \mathbb{D}_n\cup \partial \mathbb{D}_n^*\right)  \right]
		\end{equation}
		is shown in Figure \ref{fig:zero}.
		\item[$\bullet$] $M^{(2)}(z)=\sigma_1\overline{M^{(2)}(\bar{z})}\sigma_1=\sigma_2M^{(2)}(-z)\sigma_2$.
		\item[$\bullet$] For $z\in \Sigma^{(2)}$, we have
 %$M^{(2)}$ has continuous boundary values $M^{(2)}_\pm(z)$ on $\Sigma^{(2)}$
		\begin{equation}
			M^{(2)}_+(z)=M^{(2)}_-(z)V^{(2)}(z),
		\end{equation}
		where
		\begin{equation}\label{def:V2}
			V^{(2)}(z)=\left\{\begin{array}{ll}\left(\begin{array}{cc}
					1 & e^{2i\theta(z)}r(z)T^{-2}(z) \\
					0 & 1
				\end{array}\right)
				\left(\begin{array}{cc}
					1 & 0\\
					-e^{-2i\theta(z)}\bar{r}(z)T^2(z) & 1
				\end{array}\right),   & z\in 	\mathbb{R}\setminus \mathrm{I}(\hat{\xi}),\\[12pt]
				\left(\begin{array}{cc}
					1 & 0\\
					-\frac{e^{-2i\theta(z)}\bar{r}(z)T_-^{2}(z)}{1-|r(z)|^2} & 1
				\end{array}\right)\left(\begin{array}{cc}
					1 & \frac{e^{2i\theta(z)}r(z)T_+^{-2}(z)}{1-|r(z)|^2}\\
					0 & 1
				\end{array}\right),   & z\in \mathrm{I}(\hat{\xi}),\\[12pt]
				\left(\begin{array}{cc}
					1 & 0\\
					-c_n^{-1}(z-z_n)e^{-2i\theta(z_n)}T^2(z) & 1
				\end{array}\right),   & z\in\partial\mathbb{D}_n,\\
				\left(\begin{array}{cc}
					1 & -\bar{c}_n^{-1}(z-\bar{z}_n)e^{2i\theta(\bar{z}_n)}T^{-2}(z)	\\
					0 & 1
				\end{array}\right),   &	z\in\partial\mathbb{D}_n^*,\\
			\end{array}\right.
            \end{equation}
            for $n=1,\ldots,2\mathcal{N}$, and where $\mathrm{I}(\hat{\xi})$ is defined in \eqref{indexfunc}.
		\item[$\bullet$] As $z \to \infty$ in $\mathbb{C} \setminus \mathbb{R}$, we have $M^{(2)}(z)=I+\mathcal{O}(z^{-1})$.
	\end{itemize}
\end{RHP}
\begin{figure}[h]
	\centering	
	\begin{tikzpicture}[node distance=2cm]
		\draw[->,blue](-4,0)--(4,0)node[right]{ $\re z$};
		\draw[->](0,-2.6)--(0,2.6)node[above]{ $\im z$};
		\coordinate (I) at (0.2,0);
		\fill (I) circle (0pt) node[below] {$0$};
		\draw[red] (2,0) arc (0:360:2);
		\coordinate (I) at (2,0);
		\draw[blue] (1.7320508075688774,1) circle (0.15);
		\draw[blue] (1.7320508075688774,-1) circle (0.15);
		\draw[blue] (-1.7320508075688774,1) circle (0.15);
		\draw[blue] (-1.7320508075688774,-1) circle (0.15);
		\coordinate (J) at (1.7320508075688774,1);
		\coordinate (K) at (1.7320508075688774,-1);
		\coordinate (L) at (-1.7320508075688774,1);
		\coordinate (M) at (-1.7320508075688774,-1);
		\fill (I) circle (1pt) node[above] {$1$};
		\fill (J) circle (1pt) node[right] {$\bar{z}_m$};
		\fill (K) circle (1pt) node[right] {$z_m$};
		\fill (L) circle (1pt) node[left] {$-z_m$};
		\fill (M) circle (1pt) node[left] {$-\bar{z}_m$};
	\end{tikzpicture}
	\caption{ The red circle represents the unit circle. The blue circles around the poles together with real axis are the boundaries of $\Sigma^{(2)}$.}	\label{fig:zero}
\end{figure}
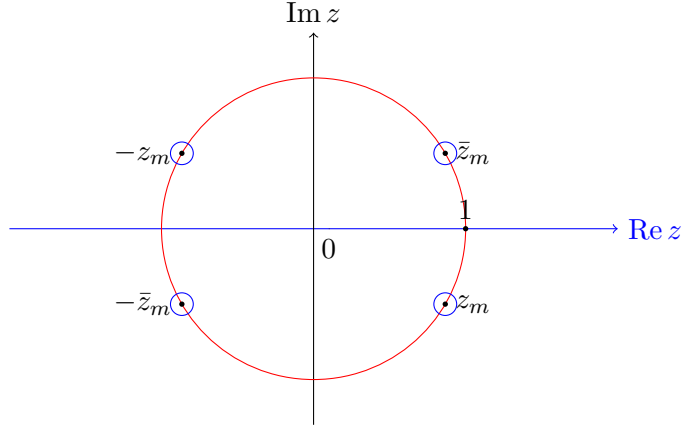

By \eqref{def:V2} and Fiugre \ref{fig:critfigSigntable}, it is readily seen that $V^{(2)}(z)\to I$ as $t \to +\infty$ for $z \in \cup_{n=1}^{2\mathcal{N}}(\partial \mathbb{D}_n\cup\partial \mathbb{D}_n^*)$ exponentially fast. Thus, RH problem \ref{RHP:regular RH} is asymptotically equivalent to the following RH problem with an error bound $\mathcal{O}(e^{-ct})$ for some constant $c>0$.
\begin{RHP}\label{RHP:regular RH asy}
	\hfill
	\begin{itemize}
		\item[$\bullet$]  $M^{(3)}(z)$ is holomorphic for $z\in\mathbb{C}\setminus \mathbb{R}$.
		\item[$\bullet$]  $M^{(3)}(z)=\sigma_1\overline{M^{(3)}(\bar{z})}\sigma_1=\sigma_2M^{(3)}(-z)\sigma_2$.
		\item[$\bullet$]  For $z\in \mathbb{R}$, we have
		\begin{equation}
			M^{(3)}_+(z)=M^{(3)}_-(z)V^{(3)}(z),
		\end{equation}
		where
		\begin{equation}\label{jump: V^{(3)}}
			V^{(3)}(z)=\left\{\begin{array}{ll}\left(\begin{array}{cc}
					1 & e^{2i\theta(z)}r(z)T^{-2}(z) \\
					0 & 1
				\end{array}\right)
				\left(\begin{array}{cc}
					1 & 0\\
					-e^{-2i\theta(z)}\bar{r}(z)T^2(z) & 1
				\end{array}\right),   & z\in 	\mathbb{R}\setminus \mathrm{I}(\hat{\xi}),\\[12pt]
				\left(\begin{array}{cc}
					1 & 0\\
					-\frac{e^{-2i\theta(z)}\bar{r}(z)T_-^{2}(z)}{1-|r(z)|^2} & 1
				\end{array}\right)\left(\begin{array}{cc}
					1 & \frac{e^{2i\theta(z)}r(z)T_+^{-2}(z)}{1-|r(z)|^2}\\
					0 & 1
				\end{array}\right),   & z\in \mathrm{I}(\hat{\xi}).\\
			\end{array}\right.
		\end{equation}
		\item[$\bullet$]  As $z\rightarrow\infty$ in $\mathbb{C}\setminus \mathbb{R}$, we have $M^{(3)}(z)=I+\mathcal{O}(z^{-1})$.
	\end{itemize}
\end{RHP}
In what follows, we will perform asymptotic analysis of RH problem \ref{RHP:regular RH asy} for different transition regions, which finally leads to the proof of Theorem \ref{mainthm}.

%%%%%%%%%%%%%%%%%%%%%%%%%%%%%%%%%%%%%%%%%%%%%%%%%%%%%%%%%%%%%%%%%%%%%%%%%%%%%%%%%%%%%%%%%%%%%%%%%%%
%%%%%%%%%%%%%%%%%%%%%%%%%%%%%%%%%%%%%%%%%%%%%%%%%%%%%%%%%%%%%%%%%%%%%%%%%%%%%%%%%%%%%%%%%%%%%%%%%%%

\section{Asymptotic analysis of the RH problem for $M^{(3)}$ in $\mathcal{R}_I$}\label{sec: 1st transition zone}

Due to the parametric form of $u$ given in \eqref{recovering u}, the analysis is actually carried out for $0\leqslant|\hat\xi-2|t^{2/3}\leqslant C$.   At the end, we will show that $\hat\xi$ is close to $\xi$ for large positive $t$. Moreover, it is also assumed that $-C\leqslant (\hat{\xi}-2)t^{2/3}\leqslant 0$ throughout this section, since the discussions for the other half region is similar.
%We only give the details of the region in this section, while the proof of can be given in an analogous way.
%In Subsection \ref{subsec:recovering1sttran}, we can make our asymptotic formulae be dependent of $(x,t)$ instead of $(y,t)$ via detailed analysis, so do not be confused that
%we use $\hat{\xi}$ instead of $\xi$ at first.

In this case, the phase function $\theta(z)$ has four saddle points $k_i$, $i=1,\ldots,4$, such that $\theta'(k_i)=0$. They are given by
\begin{align}
	k_1=2\sqrt{s_+}+\sqrt{4s_{+}+1}, \qquad k_2=-2\sqrt{s_+}+\sqrt{4s_{+}+1}, \label{equ:1sttransaddlepoints-a}\\
	k_3=2\sqrt{s_+}-\sqrt{4s_{+}+1}, \qquad k_4=-2\sqrt{s_+}-\sqrt{4s_{+}+1}, \label{equ:1sttransaddlepoints-b}
\end{align}
where
\begin{equation}\label{def:s+}
s_{+}:=\frac{1}{4\hat{\xi}}\left(-\hat{\xi}-1+\sqrt{1+4\hat{\xi}}\right).
\end{equation}
It follows from straightforward calculations that
$$k_1=1/k_2=-1/k_3=-k_4,$$
and as $\hat{\xi}\rightarrow 2^{-}$,
$$ k_{1,2} \to 1, \qquad k_{3,4}\to -1.$$
%$k_3$, $k_4$ are close to $-1$. For the transition region $-C\leqslant (\hat{\xi}-2)t^{2/3}\leqslant 0$, the signature table of Im$\theta$ is shown in Figure \ref{1stTranRegionSign}.
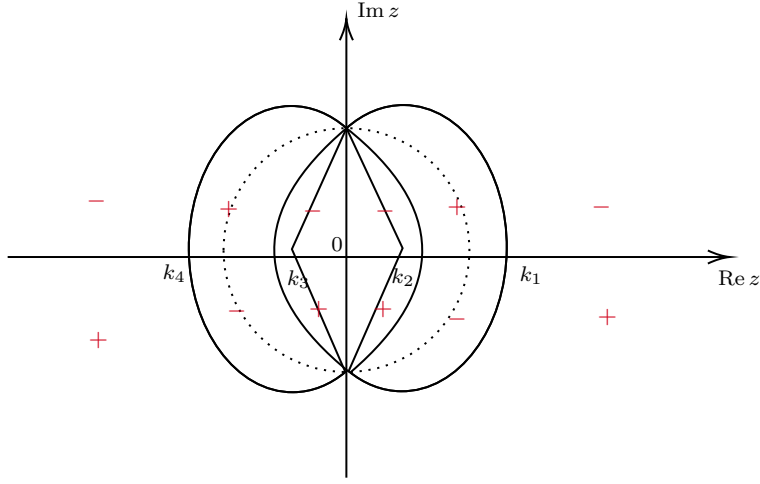
\begin{figure}[htbp]
	\begin{center}
		\tikzset{every picture/.style={line width=0.75pt}} %set default line width to 0.75pt
		\begin{tikzpicture}[x=0.75pt,y=0.75pt,yscale=-1,xscale=1]
			%uncomment if require: \path (0,300); %set diagram left start at 0, and has height of 300
			%Straight Lines [id:da6866192496016654]
			\draw    (141,151) -- (499.29,151) ;
			\draw [shift={(501.29,151)}, rotate = 180] [color={rgb, 255:red, 0; green, 0; blue, 0 }  ][line width=0.75]    (10.93,-3.29) .. controls (6.95,-1.4) and (3.31,-0.3) .. (0,0) .. controls (3.31,0.3) and (6.95,1.4) .. (10.93,3.29)   ;
			%Straight Lines [id:da1035602713425019]
			\draw    (310,33) -- (310,262) ;
			\draw [shift={(310,31)}, rotate = 90] [color={rgb, 255:red, 0; green, 0; blue, 0 }  ][line width=0.75]    (10.93,-3.29) .. controls (6.95,-1.4) and (3.31,-0.3) .. (0,0) .. controls (3.31,0.3) and (6.95,1.4) .. (10.93,3.29)   ;
			%Shape: Arc [id:dp941838739186339]
			\draw  [draw opacity=0] (310.05,86.2) .. controls (318.36,78.73) and (328.27,74.43) .. (338.88,74.56) .. controls (367.58,74.88) and (390.47,107.38) .. (390.02,147.14) .. controls (389.57,186.9) and (365.94,218.87) .. (337.25,218.55) .. controls (327.68,218.44) and (318.76,214.75) .. (311.13,208.42) -- (338.06,146.55) -- cycle ; \draw   (310.05,86.2) .. controls (318.36,78.73) and (328.27,74.43) .. (338.88,74.56) .. controls (367.58,74.88) and (390.47,107.38) .. (390.02,147.14) .. controls (389.57,186.9) and (365.94,218.87) .. (337.25,218.55) .. controls (327.68,218.44) and (318.76,214.75) .. (311.13,208.42) ;
			%Shape: Arc [id:dp05667670235134614]
			\draw  [draw opacity=0] (309.48,208.86) .. controls (301.9,215.24) and (293.02,218.94) .. (283.51,219.03) .. controls (255.11,219.31) and (231.77,187.29) .. (231.39,147.53) .. controls (231.01,107.77) and (253.72,75.31) .. (282.13,75.04) .. controls (292.37,74.94) and (301.96,79.04) .. (310.05,86.2) -- (282.82,147.04) -- cycle ; \draw   (309.48,208.86) .. controls (301.9,215.24) and (293.02,218.94) .. (283.51,219.03) .. controls (255.11,219.31) and (231.77,187.29) .. (231.39,147.53) .. controls (231.01,107.77) and (253.72,75.31) .. (282.13,75.04) .. controls (292.37,74.94) and (301.96,79.04) .. (310.05,86.2) ;
			%Shape: Parabola [id:dp30023743060759345]
			\draw   (310.05,86.2) .. controls (261.61,127.55) and (262.02,168.4) .. (311.29,208.76) ;
			%Shape: Parabola [id:dp6216716274810117]
			\draw   (312.36,208.75) .. controls (360.43,166.98) and (359.66,126.12) .. (310.05,86.2) ;
			%Shape: Circle [id:dp9017256045147721]
			\draw  [dash pattern={on 0.84pt off 2.51pt}] (248.71,147.53) .. controls (248.71,113.66) and (276.17,86.2) .. (310.05,86.2) .. controls (343.92,86.2) and (371.38,113.66) .. (371.38,147.53) .. controls (371.38,181.4) and (343.92,208.86) .. (310.05,208.86) .. controls (276.17,208.86) and (248.71,181.4) .. (248.71,147.53) -- cycle ;
			% Text Node
			\draw (494,156.4) node [anchor=north west][inner sep=0.75pt]  [font=\scriptsize]  {$\re z$};
			% Text Node
			\draw (501,143) node [anchor=north west][inner sep=0.75pt]   [align=left] { };
			% Text Node
			\draw (314,21.4) node [anchor=north west][inner sep=0.75pt]  [font=\scriptsize]  {$\im z$};
			% Text Node
			\draw (395,154.4) node [anchor=north west][inner sep=0.75pt]  [font=\scriptsize]  {$k_{1}$};
			% Text Node
			\draw (331,155.4) node [anchor=north west][inner sep=0.75pt]  [font=\scriptsize]  {$k_{2}$};
			% Text Node
			\draw (279.03,156.25) node [anchor=north west][inner sep=0.75pt]  [font=\scriptsize]  {$k_{3}$};
			% Text Node
			\draw (217,153.4) node [anchor=north west][inner sep=0.75pt]  [font=\scriptsize]  {$k_{4}$};
			% Text Node
			\draw (301,139.4) node [anchor=north west][inner sep=0.75pt]  [font=\scriptsize]  {$0$};
			% Text Node
			\draw (431,120.4) node [anchor=north west][inner sep=0.75pt]  [font=\footnotesize,color={rgb, 255:red, 208; green, 2; blue, 27 }  ,opacity=1 ]  {$-$};
			% Text Node
			\draw (434,175.4) node [anchor=north west][inner sep=0.75pt]  [font=\footnotesize,color={rgb, 255:red, 208; green, 2; blue, 27 }  ,opacity=1 ]  {$+$};
			% Text Node
			\draw (359,120.4) node [anchor=north west][inner sep=0.75pt]  [font=\footnotesize,color={rgb, 255:red, 208; green, 2; blue, 27 }  ,opacity=1 ]  {$+$};
			% Text Node
			\draw (359,176.4) node [anchor=north west][inner sep=0.75pt]  [font=\footnotesize,color={rgb, 255:red, 208; green, 2; blue, 27 }  ,opacity=1 ]  {$-$};
			% Text Node
			\draw (287,122.4) node [anchor=north west][inner sep=0.75pt]  [font=\footnotesize,color={rgb, 255:red, 208; green, 2; blue, 27 }  ,opacity=1 ]  {$-$};
			% Text Node
			\draw (323,122.4) node [anchor=north west][inner sep=0.75pt]  [font=\footnotesize,color={rgb, 255:red, 208; green, 2; blue, 27 }  ,opacity=1 ]  {$-$};
			% Text Node
			\draw (290,171.4) node [anchor=north west][inner sep=0.75pt]  [font=\footnotesize,color={rgb, 255:red, 208; green, 2; blue, 27 }  ,opacity=1 ]  {$+$};
			% Text Node
			\draw (322,171.4) node [anchor=north west][inner sep=0.75pt]  [font=\footnotesize,color={rgb, 255:red, 208; green, 2; blue, 27 }  ,opacity=1 ]  {$+$};
			% Text Node
			\draw (245,121.4) node [anchor=north west][inner sep=0.75pt]  [font=\footnotesize,color={rgb, 255:red, 208; green, 2; blue, 27 }  ,opacity=1 ]  {$+$};
			% Text Node
			\draw (249,172.4) node [anchor=north west][inner sep=0.75pt]  [font=\footnotesize,color={rgb, 255:red, 208; green, 2; blue, 27 }  ,opacity=1 ]  {$-$};
			% Text Node
			\draw (179,117.4) node [anchor=north west][inner sep=0.75pt]  [font=\footnotesize,color={rgb, 255:red, 208; green, 2; blue, 27 }  ,opacity=1 ]  {$-$};
			% Text Node
			\draw (180,187.4) node [anchor=north west][inner sep=0.75pt]  [font=\footnotesize,color={rgb, 255:red, 208; green, 2; blue, 27 }  ,opacity=1 ]  {$+$};
		\end{tikzpicture}
		\caption{ Signature table of $\im \theta$ for $-C\leqslant (\hat{\xi}-2)t^{2/3}\leqslant 0$. The ``$+$'' represents where $\im \theta(z)>0$ and ``$-$'' represents where $\im \theta(z)<0$. The dashed line is the unit circle. }\label{1stTranRegionSign}
	\end{center}
\end{figure}

By \eqref{jump: V^{(3)}}, the jump matrix for $M^{(3)}$ reads
\begin{align}\label{eq:V3inR}
	V^{(3)}(z)
    &= \left(\begin{array}{cc}
					1 & e^{2i\theta(z)}R(z) \\
					0 & 1
				\end{array}\right)\left(\begin{array}{cc}
					1 & 0\\
					-e^{-2i\theta(z)}\bar{R}(z) & 1
				\end{array}\right), \qquad z\in\mathbb{R},
\end{align}
where
\begin{equation}\label{def:R=rT}
	R(z):=R(z;\hat{\xi})=r(z)T^{-2}(z)\overset{\eqref{Tfunc}}{=}r(z)\prod_{j=1}^{2\mathcal{N}}\left(\frac{z-z_j}{z-\bar{z}_j}\right)^2.
\end{equation}
This, together with the signature table of $\im \theta$ illustrated in Figure \ref{1stTranRegionSign}, implies opening lenses around the intervals $(-\infty,k_4)\cup (k_3,k_2)\cup (k_1,+\infty)$ in what follows.
%For convenience, denote that

\subsection{Opening $\bar{\partial}$ lenses }\label{subsec:openlensTran1}
%In this subsection, we aim at removing the jump from the real axis in such a way that the new problem (a mixed $\bar{\partial}$-RH problem) takes advantage of
%the decay or growth of the exponential terms $e^{\pm 2i\theta(z)}$ for $z\notin\mathbb{R}$.

For $j=1,\ldots,4$, we define
\begin{align}
		&\Omega_{1}:=\left\lbrace z\in\mathbb{C}:\ 0 \leqslant (\arg z-k_1) \leqslant \varphi_0 \right\rbrace, \label{def:Omega1}\\
		&\Omega_{2}:=\left\lbrace z\in\mathbb{C}:\ \pi-\varphi_0  \leqslant (\arg z-k_2) \leqslant \pi ,\ |\re(z-k_2)|\leqslant (k_2-k_3)/2\right\rbrace,\\
		&\Omega_{3}:=\left\lbrace z\in\mathbb{C}:\ 0 \leqslant (\arg z-k_3) \leqslant \varphi_0,\ |\re(z-k_2)|\leqslant (k_2-k_3)/2\right\rbrace,\\
		&\Omega_{4}:=\left\lbrace z\in\mathbb{C}:\ 0 \leqslant (\arg z-k_4) \leqslant \varphi_0 \right\rbrace, \label{def:Omega4}
\end{align}
where $k_j$ are four saddle points given in \eqref{equ:1sttransaddlepoints-a} and \eqref{equ:1sttransaddlepoints-b}, and
$0<\varphi_0<\frac{\pi}{8}$ is a fixed angle such that the following conditions hold:
	\begin{itemize}
		\item[$\bullet$] $2\tan\varphi_0<k_2-k_3$,
		\item[$\bullet$] each $\Omega_{j}$  doesn't intersect the set $\left\lbrace z\in\mathbb{C}: \ \im\theta(z)=0\right\rbrace$,
		\item[$\bullet$] each $\Omega_{j}$ doesn't intersect any small disks $\mathbb{D}_n$ and $\mathbb{D}_n^*$, $n=1,\ldots, 2\mathcal{N}$.
	\end{itemize}
Moreover, we use $\Sigma_{j}$ to denote the boundary of $\Omega_{j}$ in the upper half plane and  set
\begin{equation}
\Sigma_{2,3}:=\left\lbrace z\in i\mathbb{R}^+: \ |z|\leqslant \frac{(k_2-k_3)\tan\varphi_0}{2} \right\rbrace.
\end{equation}
We refer to Figure \ref{fig:R^{(3)}} for an illustration of $\Omega_j$ and its boundary.
\begin{figure}[h]
	\centering
	\begin{tikzpicture}[node distance=2cm]
		\draw[yellow!30, fill=yellow!20] (0,0.75)--(1.5,0)--(3,0)--(4,0.5)--(4,0)--(-4,0)--(-4,0.5)--(-3,0)--(-1.5,0)--(0,0.75);
		\draw[blue!30, fill=blue!20] (0,-0.75)--(1.5,0)--(3,0)--(4,-0.5)--(4,0)--(-4,0)--(-4,-0.5)--(-3,0)--(-1.5,0)--(0,-0.75);
		\draw[dash pattern={on 0.84pt off 2.51pt}][->](-3.6,0)--(4,0)node[right]{ $\re z$};
		\draw[dash pattern={on 0.84pt off 2.51pt}][->](0,-2.5)--(0,2.5)node[above]{ $\im z$};
		\draw(3,0)--(4,0.5)node[above]{\footnotesize$\Sigma_{1}$};
		\draw[-latex](3,0)--(3.5,0.25);
		\draw[-latex](3,0)--(3.5,-0.25);
		\draw[-latex](-4,-0.5)--(-3.5,-0.25);
		\draw[-latex](-4,0.5)--(-3.5,0.25);
		\draw(3,0)--(4,-0.5)node[below]{\footnotesize$\Sigma^*_{1}$};
		\draw(-3,0)--(-4,0.5)node[above]{\footnotesize$\Sigma_{4}$};
		\draw(-3,0)--(-4,-0.5)node[below]{\footnotesize$\Sigma^*_{4}$};
		\draw[-latex](-1.5,0)--(-0.75,0.375)node[above]{\footnotesize$\Sigma_{3}$};
		\draw[-latex](-1.5,0)--(-0.75,-0.375)node[below]{\footnotesize$\Sigma^*_{3}$};
		\draw[-latex](0,0.75)--(0.75,0.375)node[above]{\footnotesize$\Sigma_{2}$};
		\draw[-latex](0,-0.75)--(0.75,-0.375)node[below]{\footnotesize$\Sigma^*_{2}$};
		\draw(0,0.75)--(1.5,0);
		\draw(0,0.75)--(0,-0.75);
		\draw(0,-0.75)--(1.5,0);
		\draw(0,0.75)--(-1.5,0);
		\draw(0,-0.75)--(-1.5,0);
		\draw[-latex](-2.1,0)--(-2.09,0);
		\draw[-latex](2.1,0)--(2.2,0);
		\draw[-latex](0,0.75)--(0,0.3);
		\draw[-latex](0,-0.75)--(0,-0.3);
		\draw(-3,0)--(-1.5,0);
		\draw(3,0)--(1.5,0);
		\coordinate (C) at (-0.2,2.2);
		\coordinate (D) at (3.45,0.15);
		\fill (D) circle (0pt) node[right] {\tiny $\Omega_{1}$};
		\coordinate (D2) at (3.45,-0.15);
		\fill (D2) circle (0pt) node[right] {\tiny $\Omega^*_{1}$};
		\coordinate (k) at (1,0.2);
		\fill (k) circle (0pt) node[left] {\tiny $\Omega_{2}$};
		\coordinate (k2) at (1,-0.2);
		\fill (k2) circle (0pt) node[left] {\tiny $\Omega^*_{2}$};
		\coordinate (k) at (0,0.2);
		\fill (k) circle (0pt) node[left] {\tiny $\Omega_{3}$};
		\coordinate (k2) at (0,-0.2);
		\fill (k2) circle (0pt) node[left] {\tiny $\Omega^*_{3}$};
		\coordinate (D3) at (-3.45,0.15);
		\fill (D3) circle (0pt) node[left] {\tiny $\Omega_{4}$};
		\coordinate (D4) at (-3.45,-0.15);
		\fill (D4) circle (0pt) node[left] {\tiny $\Omega^*_{4}$};
		\coordinate (I) at (0.2,0);
		\fill (I) circle (0pt) node[below] {$0$};
		\draw[red] (2,0) arc (0:360:2);	
		\node at (0,0.9) {\footnotesize $\Sigma_{2,3}$};
		\node at (0,-0.92) {\footnotesize $\Sigma^*_{2,3}$};
	\end{tikzpicture}
	\caption{  The jump contours of the RH problem for $M^{(4)}$. }
	\label{fig:R^{(3)}}
\end{figure}
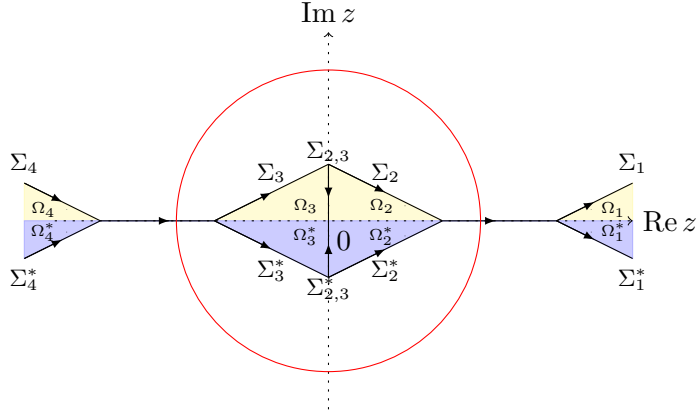

Since the function $R$ in \eqref{eq:V3inR} is not an analytic function, the idea now is to introduce the functions $d_{ j}(z):=d_{ j}(z;\hat{\xi})$, $j=1,\ldots,4$, with boundary conditions:
\begin{align}\label{equ: 1sttran defofd_j}
	&d_j(z)=\Bigg\{\begin{array}{ll}
		\bar{R}(z), & z\in \mathbb{R},\\
		\bar{R}(k_j)+\bar{R}'(k_j)^2(z-k_j),  &z\in \Sigma_{j}.\\
		\end{array}
\end{align}	
One can give an explicit construction of each $d_j$. Indeed, let $\mathcal{X}\in C^\infty_0(\mathbb{R})$ be such that
\begin{align}\label{equ:defX}
	\mathcal{X}(x):=\Bigg\{\begin{array}{ll}
		0, & x\leqslant \frac{\varphi_0}{3},\\
		1,  &x\geqslant \frac{2\varphi_0}{3}.\\
	\end{array}
\end{align}
Then,
\begin{align}\label{def:d1}
	d_1(z):=&\left[ \bar{R}(\re z)-\bar{R}(k_1)-\bar{R}'(k_1)\re(z-k_1)\right] \cos\left(\frac{\pi\arg\left(z-k_1\right)	\mathcal{X}\left(\arg\left(z-k_1\right)\right)}{2\varphi_0} \right)
  \nonumber \\ &+\bar{R}(k_1)+\bar{R}'(k_1)(z-k_1)
\end{align}
satisfies the conditions \eqref{equ: 1sttran defofd_j} for $j=1$. Some properties of $d_j$ are collected in the following proposition.
\begin{Proposition}\label{est:RandDbarR}
	For each $j=1,\ldots,4$ and $z\in \Omega_{ j}$, we have
\begin{subequations}
	\begin{align}
		&|d_j(z)|\lesssim \sin^2\left(\frac{\pi}{2\varphi_0}\arg\left(z-k_j\right)\right)+ \left(1+ \re (z)^2\right) ^{-1/2},\label{Rest}\\
		&|\bar{\partial}d_j(z)|\lesssim|\re z-k_j|^{1/2}, \label{dbarRjest}\\
		&|\bar{\partial}d_j(z)|\lesssim |\re z-k_j|^{-1/2}+\sin\left(\frac{\pi}{2\varphi_0}\arg\left(z-k_j\right)	\mathcal{X}(\arg\left(z-k_j\right)) \right),  \label{dbarRjest1}\\
		&|\bar{\partial}d_j(z)|\lesssim 1.  \label{dbarRjest2}
	\end{align}
\end{subequations}
\end{Proposition}
\begin{proof}
Without loss of generality, we assume that $j=1$, since the claims for $j=2,3,4,$ can be proved in a similar manner.

The bound in \eqref{Rest} can be obtained from the proof similar to \cite[Proposition 5]{Yang2022adv}, we omit the details here.
To show \eqref{dbarRjest}--\eqref{dbarRjest2}, note that for $z=u+vi=le^{\varphi i}+k_1\in\Omega_{1}$, we have $\bar{\partial}=\frac{e^{i\varphi}}{2}(\partial_l+il^{-1}\partial_\varphi)=\frac{1}{2}(\partial_u+\partial_v)$. Applying $\bar{\partial}$ operator to $d_1$ in \eqref{def:d1}, it is readily seen that
	\begin{align}\label{afterdbarderR+1}
		\bar{\partial}d_{1}(z)&=\frac{1}{2}\left( \bar{R}'(u)-\bar{R}'(k_1)\right)\cos\left(\frac{\pi\varphi	\mathcal{X}\left(\varphi\right)}{2\varphi_0} \right)\nonumber\\
		&~~-\frac{ie^{i\varphi}}{2l}\left[ \bar{R}(u)-\bar{R}(k_1)-\bar{R}'(k_1)(u-k_1)\right]\frac{\pi	\mathcal{X}'(\varphi)}{2\varphi_0} \sin\left(\frac{\pi}{2\varphi_0}\varphi	\mathcal{X}(\varphi) \right) .
	\end{align}
On one hand, it follows from H\"older's inequality that
	\begin{align}
		|\bar{R}'(u)-\bar{R}'(k_1)|=\Big|\int_{k_1}^u\bar{R}''(\zeta)\dif \zeta\Big|\leqslant \|\bar{R}''\|_2|u-k_1|^{1/2}, \label{afterdbarderR+1:est1}
	\end{align}
or
\begin{align}
	|\bar{R}'(u)-\bar{R}'(k_1)|=\Big|\int_{k_1}^u\bar{R}''(\zeta)\dif \zeta\Big|=\Big|\int_{k_1}^u\zeta^{-1}\zeta\bar{R}''(\zeta)\dif \zeta\Big|\leqslant \|\bar{R}''\|_{2,1}|u-k_1|^{-1/2}.\label{equ:afterdbarderR+1:est1}
\end{align}
On the other hand, since
	\begin{align}
		|\bar{R}'(u)-\bar{R}'(k_1)|\leqslant 2\|\bar{R}'\|_\infty, \label{dbarderR+1:est1}
	\end{align}
we have
	\begin{align}\label{equ: rhsastool}
		&|\bar{R}(u)-\bar{R}(k_1)-\bar{R}'(k_1)(u-k_1)|=\Big|\int_{k_1}^u\bar{R}'(\zeta)-\bar{R}'(k_1)\dif	 \zeta\Big| \leqslant 2\|\bar{R}'\|_\infty|u-k_1|,
	\end{align}
or again by H\"older's inequality,
	%\begin{align}
%		{\rm R.H.S} \ {\rm of} \ \eqref{equ: rhsastool} \leqslant 2\|\bar{R}'\|_\infty|u-k_1|. \label{afterdbarderR+1:est2}
%	\end{align}
%	Via H\"older inequality,
	\begin{align}
		&\Big|\int_{k_1}^u\bar{R}'(\zeta)-\bar{R}'(k_1)\dif	 \zeta\Big| =\Big|\int_{k_1}^u\int_{k_1}^\zeta\bar{R}''(\eta)\dif\eta\dif\zeta\Big|\leqslant \|\bar{R}''\|_2|u-k_1|^{3/2}.\label{dbarderR+12}
	\end{align}
As a consequence, on account of \eqref{afterdbarderR+1}, we obtain \eqref{dbarRjest} from  \eqref{afterdbarderR+1:est1} and \eqref{dbarderR+12}, \eqref{dbarRjest1} from \eqref{equ:afterdbarderR+1:est1} and  \eqref{equ: rhsastool}, and \eqref{dbarRjest2} from \eqref{dbarderR+1:est1} and \eqref{equ: rhsastool}.
\end{proof}

We are now ready to define the transformation
\begin{equation}\label{transtoM4}
	M^{(4)}(z)=M^{(3)}(z)R^{(3)}(z),
\end{equation}
where
\begin{equation}\label{defofR^{(3)}}
	R^{(3)}(z):=R^{(3)}(z;\hat{\xi})=\left\{\begin{array}{lll}
		\left(\begin{array}{cc}
			1 & 0\\
			d_j(z)e^{-2i\theta(z)} & 1
		\end{array}\right), & z\in \Omega_{j}, \ j=1,...,4,\\
		[12pt]
		\left(\begin{array}{cc}
			1 & d_j^*(z)e^{2i\theta(z)}\\
			0 & 1
		\end{array}\right),  &z\in \Omega^*_{j}, \ j=1,...,4,\\
		[12pt]
		I,  &{\rm elsewhere}.\\
	\end{array}\right.
\end{equation}
Then, $M^{(4)}$ satisfies the following mixed $\bar{\partial}$-RH problem.
\begin{dbar-RHP}\label{RHP:mixed RH problem}
	\hfill
	\begin{itemize}
		\item[$\bullet$] $M^{(4)}(z)$ is continuous for $z\in\mathbb{C}\setminus \Sigma^{(4)}$,
        where
        \begin{equation}\label{def:sigma4}
        \Sigma^{(4)}:=\Sigma_{2,3}\cup\Sigma^*_{2,3}\cup\left( \underset{ k=1,...,4}{\cup}\left( \Sigma_{k}\cup\Sigma^*_{k}\right)\right)
        \cup(k_4,k_3)\cup(k_2,k_1);
        \end{equation}
        see Figure \ref{fig:R^{(3)}} for an illustration.
		\item[$\bullet$] $M^{(4)}(z)=\sigma_1\overline{M^{(4)}(\bar{z})}\sigma_1=\sigma_2M^{(4)}(-z)\sigma_2$.
		\item[$\bullet$]  For $z \in \Sigma^{(4)}$,  we have
		\begin{equation}
			M^{(4)}_+(z)=M^{(4)}_-(z)V^{(4)}(z),
		\end{equation}
		where
		\begin{equation}\label{jump:V^{(4)}}
			V^{(4)}(z)=\left\{\begin{array}{ll}\left(\begin{array}{cc}
					1 & e^{2i\theta(z)}R(z) \\
					0 & 1
				\end{array}\right)
				\left(\begin{array}{cc}
					1 & 0\\
					-e^{-2i\theta(z)}\bar{R}(z) & 1
				\end{array}\right),   & z\in 	(k_4,k_3)\cup(k_2,k_1),\\[12pt]
				R^{(3)}(z)^{-1},  & z\in \Sigma_{j},\ j=1,2,3,4,\\[8pt]
				R^{(3)}(z),  & z\in \Sigma^*_{j},\ j=1,2,3,4,\\
				[8pt]
				\left(\begin{array}{cc}
					1 & 0\\
					(d_{2}(z)-d_{3}(z))e^{-2i\theta(z)} & 1
				\end{array}\right),  & z\in \Sigma_{2,3},\\
				[12pt]
				\left(\begin{array}{cc}
					1 & (d_{3}^*(z)-d_{2}^*(z))e^{2i\theta(z)}\\
					0 & 1
				\end{array}\right),  & z\in \Sigma^*_{2,3}.\\
			\end{array}\right.
		\end{equation}
		\item[$\bullet$]   As $z\rightarrow\infty$ in $\mathbb{C} \setminus \Sigma^{(4)}$, we have $M^{(4)}(z)=I+\mathcal{O}(z^{-1})$.
		\item[$\bullet$]   For $z\in\mathbb{C}$, we have the $\bar{\partial}$-derivative relation
		\begin{align}
			\bar{\partial}M^{(4)}(z)=M^{(4)}(z)\bar{\partial}R^{(3)}(z),
		\end{align}
		where
		\begin{equation}
			\bar{\partial}R^{(3)}(z)=\left\{\begin{array}{lll}
				\left(\begin{array}{cc}
					0 & 0\\
					\bar{\partial}d_j(z)e^{-2i\theta(z)} & 0
				\end{array}\right), & z\in \Omega_{j}, \ j=1,...,4,\\
				[12pt]
				\left(\begin{array}{cc}
					0 & \bar{\partial}d_j^*(z)e^{2i\theta(z)}\\
					0 & 0
				\end{array}\right),  &z\in \Omega^*_{j}, \ j=1,...,4,\\
				[12pt]
				0,  & {\rm elsewhere}.\\
			\end{array}\right.\label{DBARR1}
		\end{equation}
	\end{itemize}
\end{dbar-RHP}
The above RH problem can be decomposed into a pure RH problem under the condition $\bar\partial R^{(3)}\equiv0$ and a pure $\bar{\partial}$-problem with $\bar\partial R^{(3)}\not=0$. The next two sections are then devoted to the asymptotic analysis of these two RH problems, respectively.

%\begin{align}
%	&\Omega:=\underset{ k=1,...,4}{\cup}\left( \Omega_{ k}\cup\Omega^*_{ k}\right) ,\nonumber\\
%	&\Sigma^{(4)}(\hat{\xi}):=\Sigma_{2,3}\cup\Sigma^*_{2,3}\cup\left( \underset{ k=1,...,4}{\cup}\left( \Sigma_{k}\cup\Sigma^*_{k}\right)\right)  .\nonumber
%\end{align}
%The boundaries of $\Omega$ are shown in Figure \ref{fig:R^{(3)}}.
%Introduce

%To solve the RH problem \ref{RHP:mixed RH problem},  we obey some standard steps, e.g. \cite{Borghese2018LongTA,Cuccagna2016OnTA},
%to decompose $M^{(4)}(z)$ into a pure RH problem as $ N(z)$ under the condition $\bar\partial R^{(3)}\equiv0$,
%as well as a pure $\bar{\partial}$-problem $M^{(5)}$ with $\bar\partial R^{(3)}\not=0$.
%\begin{remark}
%The definitions of $d_2$ and $d_3$ imply $V^{(4)}=I$ near $z=0$ on $\Sigma_{2,3}$ and $\Sigma_{2,3}^*$.
%\end{remark}

\subsection{Analysis of the pure RH problem }\label{subsec:pure RH N(z)}
%In this subsection, we mainly focus on the analysis for the pure RH problem, which include three parts: global parametrix, local parametrix as well as
%small norm RH problem. Noticing that $N(z)$ is a RH problem with $\bar\partial R^{(3)}\equiv0$, thus, the RH problem for $N(z)$ is as follows:
The pure RH problem is obtained from RH problem \ref{RHP:mixed RH problem} for $M^{(4)}$ by omitting the $\bar{\partial}$-derivative part, which reads as follows.
\begin{RHP}\label{RHP: N(z) tran1}
	\hfill
	\begin{itemize}
		\item[$\bullet$]  $N(z)=N(z;\hat{\xi})$ is holomorphic for $z\in\mathbb{C}\setminus \Sigma^{(4)}$, where $\Sigma^{(4)}$ is defined in \eqref{def:sigma4}.
		\item[$\bullet$]  $N(z)=\sigma_1\overline{N(\bar{z})}\sigma_1=\sigma_2N(-z)\sigma_2$.
		\item[$\bullet$]  For $z \in \Sigma^{(4)}$, we have
		\begin{equation}
			N_+(z)=N_-(z)V^{(4)}(z),
		\end{equation}
		where $V^{(4)}(z)$ is defined in \eqref{jump:V^{(4)}}.
		\item[$\bullet$]
		As $z\rightarrow\infty$ in $\mathbb{C}\setminus \Sigma^{(4)}$, we have $N(z)=I+\mathcal{O}(z^{-1})$.
	\end{itemize}
\end{RHP}
Note that the definitions of $d_2$ and $d_3$ imply that $V^{(4)}(z)=I$ near $z=0$ on $\Sigma_{2,3}$ and $\Sigma_{2,3}^*$, and $V^{(4)}(z) \to I$  as $t \to +\infty$ on the other parts of $\Sigma^{(4)}$ (see \eqref{jump:V^{(4)}} and Figure \ref{1stTranRegionSign}), it follows that
$N$ is approximated, to the leading order, by the global parametrix $N^{(\infty)}(z)=I$. The sub-leading contribution stems from the local behaviors near the saddle points $k_j$, $j=1,\ldots,4$, which is well approximated by the Painlev\'{e} II parametrix.

%the The leading order, $N(z)$ is approximated by a global parametrix (denoted by $N^{(\infty)}$) with exponentially decaying on the jump of $N(z)$.
%In fact, it is trivial to find that $N^{(\infty)}=I$. The sub-leading contribution stems from the local behavior near those saddle points
%$k_j$, $j=1,2,3,4$. It turns out that the local parametrix (denoted by $N^{(r)}$, $N^{(l)}$ below) can be constructed in terms of the solution of the
%well-known Painlev\'e II equation \eqref{equ:standard PII equ}.

\subsubsection*{Local parametrices near $z=\pm 1$}
Let
$$U^{(r)}=\{z:|z-1|\leqslant c_0\}, \qquad U^{(l)}=\{z:|z+1|\leqslant c_0\},$$
be two small disks around $z=1$ and $z=-1$, respectively, where
\begin{align}
	c_0:=\min\{1/2,2(k_1-1)t^{\delta_1}\},\hspace*{0.5cm}\delta_1\in\left( 1/9, 1/6\right),
\end{align}
is a constant depending on $t$. For $t$ large enough, we have $k_{1,2} \in U^{(r)}$ and $k_{3,4} \in U^{(l)}$. Indeed,
if $\hat{\xi}>0$ and $-C\leqslant (\hat{\xi}-2)t^{2/3}\leqslant 0$, it follows from \eqref{def:s+} that
$0\leqslant s_+\leqslant\frac{C}{2}t^{-2/3}$, thus, by \eqref{equ:1sttransaddlepoints-a} and \eqref{equ:1sttransaddlepoints-b},
\begin{align*}
	&|k_j-1|\leqslant\sqrt{2C}t^{-1/3}, \ j=1,2, \quad \textnormal{and} \quad  |k_j+1|\leqslant\sqrt{2C}t^{-1/3}, \ j=3,4.
\end{align*}
This particularly implies that $c_0 \lesssim t^{\delta_1-1/3} \to 0$ as $t\to +\infty$, hence, $U^{(r)}$ and $U^{(l)}$ are two shrinking disks with respect to $t$.

%Denote be the neighborhood of $z=\pm1$ respectively, where the radius $c_0$ is dependent of $t$:
%\begin{align}
%	c_0:=\min\{1/2,2(k_1-1)t^\delta\},\hspace*{0.5cm}\delta\in\left( \frac{1}{9},\frac{1}{6}\right) .
%\end{align}
%Then there exists a time $T$ such that the phase points are in $U^{(r)}$ or $U^{(l)}$  when $t>T$.
%Indeed, when $\hat{\xi}>0$, $-C\leqslant (\hat{\xi}-2)t^{2/3}\leqslant 0$, it follows that $0\leqslant s_+\leqslant\frac{C}{2}t^{-2/3}$. As a consequence,
%\begin{align*}
%	&|k_j-1|\leqslant\sqrt{2C}t^{-1/3}, \ j=1,2 \quad \textnormal{and} \quad  |k_j+1|\leqslant\sqrt{2C}t^{-1/3}, \ j=3,4.
%\end{align*}
%which reveal that $c_0\to 0$ as $t\to\infty$.

For $\ell \in\{r,l\}$, we intend to solve the following local RH problem for $N^{(\ell)}$.
%define the local models $N^{(j)}(z)$, $j\in\{r,l\}$ respectively.On the contours of $\Sigma^{(j)}:=U^{(j)}\cap\Sigma^{(4)}$,
\begin{RHP}\label{RHP:Nrl}
    \hfill	
	\begin{itemize}
	\item[$\bullet$]  $N^{(\ell )}(z)$ is holomorphic for $z\in\mathbb{C}\setminus \Sigma^{(\ell)}$,
    where
      $$
      \Sigma^{(\ell)}:=U^{(\ell)}\cap \Sigma^{(4)}.
      $$

	\item[$\bullet$]  For $z \in \Sigma^{(\ell)}$, we have
		\begin{equation}
			N^{(\ell)}_+(z)=N^{(\ell)}_-(z)V^{(4)}(z),
		\end{equation}
		%where $V^{(j)}(z)=V^{(4)}(z)|_{z\in\Sigma^{(j)}}$, $j\in\{r,l\}$ (see \eqref{jump:V^{(4)}} for $V^{(4)}(z)$).
        where $V^{(4)}(z)$ is defined in \eqref{jump:V^{(4)}}.

        \item[$\bullet$]
		As $z\rightarrow\infty$ in $\mathbb{C}\setminus \Sigma^{(\ell)}$, we have $N^{(\ell)}(z)=I+\mathcal{O}(z^{-1})$.
	\end{itemize}
\end{RHP}

To solve the RH problem for $N^{(r)}$, we observe that for $z\in U^{(r)}$ and $t$ large enough,
	\begin{align}\label{eq:thetaexp1}
		\theta(z)=-\tilde{s}\tilde{k}-\frac{4}{3}\tilde{k}^3+\mathcal{O}(\tilde{k}^4t^{-1/3}),
	\end{align}
	where
	\begin{align}\label{equ:tildes1stran}
		\tilde{s}=6^{-2/3}\left(\frac{y}{t}-2\right)t^{2/3}
	\end{align}
    parametrizes the space-time region, and
	\begin{align}
		\tilde{k}=\left( \frac{9t}{2}\right) ^{1/3}(z-1)\label{equ:scaling tik}
	\end{align}
is a scaled spectral parameter. The expansion of $\theta$ in \eqref{eq:thetaexp1} invokes us to work in the $\tilde{k}$-plane, and we split the analysis into three steps.
	
%Then $\tilde{k}$ and $\breve{k}$ are both bounded.
%We obey the following steps to construct the local parametrix $N^{(r)}$ firstly.
\paragraph{Step 1: From the $z$-plane to the $\tilde{k}$-plane}
Under the change of variable \eqref{equ:scaling tik}, it is easily seen that $N^{(r)}(\tilde k)=N^{(r)}(z(\tilde k))$ is
analytic in $\mathbb{C}\setminus \tilde{\Sigma}^{(r)}$, where
\begin{equation}\label{def:tildesigmar}
	\tilde{\Sigma}^{(r)}:=\underset{j=1,2}{\cup}\left(\tilde{\Sigma}^{(r)}_j\cup\tilde{\Sigma}^{(r)*}_j\right)\cup(\tilde{k}_2, \tilde{k}_1).
\end{equation}
Here,
$\tilde{k}_j=\left( \frac{9t}{2}\right) ^{1/3}(k_j-1)$
and
\begin{align*}
\tilde{\Sigma}^{(r)}_1&=\left\lbrace \tilde{k}:\ \tilde{k}-\tilde{k}_1=le^{(\varphi_0)i},\ 0\leqslant l\leqslant c_0\left( \frac{9t}{2}\right) ^{1/3}\right\rbrace,
\\
\tilde{\Sigma}^{(r)}_2&=\left\lbrace \tilde{k}:\ \tilde{k}-\tilde{k}_2=le^{(\pi-\varphi_0)i},\ 0\leqslant l\leqslant c_0\left( \frac{9t}{2}\right) ^{1/3}\right\rbrace.
\end{align*}
Moreover, the jump of $N^{(r)}(\tilde k)$ now reads
%\begin{align*}
%	z=\left( \frac{9t}{2}\right) ^{-1/3}\tilde{k}+1.
%\end{align*}
% Therefore,  the jump matrix $V^{(r)}(z)$ transforms to the following $V^{(r)}(\tilde{k})$ in $\tilde{k}$-plane
\begin{small}
\begin{align*}
	V^{(r)}(\tilde{k})=\left\{\begin{array}{ll}\left(\begin{array}{cc}
			1 & e^{2i\theta(\left( \frac{9t}{2}\right) ^{-1/3}\tilde{k}+1)}R(\left( \frac{9t}{2}\right) ^{-1/3}\tilde{k}+1) \\
			0 & 1
		\end{array}\right)
		\left(\begin{array}{cc}
			1 & 0\\
			-e^{-2i\theta(\left( \frac{9t}{2}\right) ^{-1/3}\tilde{k}+1)}\overline{R(\left( \frac{9t}{2}\right) ^{-1/3}\tilde{k}+1)} & 1
		\end{array}\right),   & \tilde{k}\in 	(\tilde{k}_2,\tilde{k}_1),\\[14pt]
		\left(\begin{array}{cc}
			1 & 0\\
			\left(-\bar{R}(k_j)-\bar{R}'(k_j)(\left( \frac{9t}{2}\right) ^{-1/3}\tilde{k}+1-k_j)\right)e^{-2i\theta(\left( \frac{9t}{2}\right) ^{-1/3}\tilde{k}+1)} & 1
		\end{array}\right), &\tilde{k}\in \tilde{\Sigma}^{(r)}_j, \\[14pt]
		\left(\begin{array}{cc}
			1 & \left(R(k_{j})+R'(k_j)(\left( \frac{9t}{2}\right) ^{-1/3}\tilde{k}+1-k_j)\right)e^{2i\theta(\left( \frac{9t}{2}\right) ^{-1/3}\tilde{k}+1)} \\
			0 & 1
		\end{array}\right),  & \tilde{k}\in \tilde{\Sigma}^{(r)*}_j,
	\end{array}\right.
\end{align*}
\end{small}
for $j=1,2$.

\paragraph{Step 2: A model RH problem}
In view of \eqref{eq:thetaexp1} and the fact that $R(1)=r(1)\in\mathbb{R}$ (see \eqref{def:R=rT}), it is natural to expect that $N^{(r)}$ is well-approximated by the following model RH problem for $\tilde{N}^{(r)}$.
\begin{RHP}\label{RHP: Model RH problem ass N^{(r)}(z)}
%[{\bf Model RH problem $\tilde{N}^{(r)}(\tilde{k})$ associated to $N^{(r)}$}]
%	Find a matrix-valued function  $ \tilde{N}^{(r)}(\tilde{k})$ which satisfies:
\hfill
	\begin{itemize}
\item[$\bullet$]  $\tilde{N}^{(r)}(\tilde k)$ is holomorphic for $\tilde k \in\mathbb{C}\setminus \tilde{\Sigma}^{(r)}$,
    where $\tilde{\Sigma}^{(r)}$ is defined in \eqref{def:tildesigmar}.

		\item[$\bullet$]   For $\tilde{k} \in \tilde{\Sigma}^{(r)}$, we have
		\begin{equation}
			\tilde{N}^{(r)}_+(\tilde{k})=\tilde{N}^{(r)}_-(\tilde{k})\tilde{V}^{(r)}(\tilde{k}),
		\end{equation}
		where
		\begin{align}
			\tilde{V}^{(r)}(\tilde{k})=\left\{\begin{array}{ll}\left(\begin{array}{cc}
					1 & r(1)e^{-2i(\tilde{s}\tilde{k}+\frac{4}{3}\tilde{k}^3)} \\
					0 & 1
				\end{array}\right)
				\left(\begin{array}{cc}
					1 & 0\\
					-r(1)e^{2i(\tilde{s}\tilde{k}+\frac{4}{3}\tilde{k}^3)} & 1
				\end{array}\right),   & \tilde{k}\in(\tilde{k}_2,\tilde{k}_1),\\[12pt]
				\left(\begin{array}{cc}
					1 & 0\\
					-r(1)e^{2i(\tilde{s}\tilde{k}+\frac{4}{3}\tilde{k}^3)} & 1
				\end{array}\right), & \tilde{k}\in \tilde{\Sigma}^{(r)}_j, \\[12pt]
				\left(\begin{array}{cc}
					1 & r(1)e^{-2i(\tilde{s}\tilde{k}+\frac{4}{3}\tilde{k}^3)} \\
					0 & 1
				\end{array}\right),  &\tilde{k}\in \tilde{\Sigma}^{(r)*}_j,
			\end{array}\right.
        \end{align}
        for $j=1,2$.
		\item[$\bullet$] As $\tilde{k} \to \infty$ in $\mathbb{C}\setminus \tilde{\Sigma}^{(r)}$,
			we have $\tilde{N}^{(r)}(\tilde{k})=I+\mathcal{O}(\tilde{k}^{-1})$.
	\end{itemize}
\end{RHP}
In Step 3, we will solve the above RH problem explicitly with the aid of the Painelv\'{e} II parametrix. Consideration in the rest part of this step is to establish the error between the RH problems for $N^{(r)}$ and $\tilde{N}^{(r)}$ for large $t$. More precisely, note that $\tilde{N}^{(r)}$ is invertible, we define
\begin{equation}\label{def:Xi}
\Xi(\tilde{k}):=N^{(r)}(\tilde{k})\tilde{N}^{(r)}(\tilde{k})^{-1},
\end{equation}
and prove the following result.
\begin{Proposition}\label{prop:Xi}
	As $t\rightarrow +\infty$, $\Xi(\tilde{k})$ exists uniquely. Furthermore, we have
	\begin{align}\label{equ:Xiasy}
		\Xi(\tilde{k})=I+\mathcal{O}(t^{-1/3+4\delta_1}), \qquad t \to +\infty.
	\end{align}	
\end{Proposition}
\begin{proof} 	
By \eqref{def:Xi}, it is easily seen that $\Xi(\tilde{k})$ is holomorphic for $\tilde{k}\in \mathbb{C}\setminus \tilde{\Sigma}^{(r)}$ and satisfies the jump condition $$\Xi_{+}(\tilde{k})=\Xi_{-}(\tilde{k})J_{\Xi}(\tilde{k}), \qquad \tilde{k} \in \tilde{\Sigma}^{(r)},$$
where $$J_{\Xi}(\tilde{k}):=\tilde{N}^{(r)}(\tilde{k})V^{(r)}(\tilde{k}){\tilde{V}^{(r)}(\tilde{k})}^{-1}{\tilde{N}^{(r)}(\tilde{k})}^{-1}.$$
On account of the boundedness of $\tilde{N}^{(r)}(\tilde{k})$ (see \eqref{equ:mathcing local 1stcase} below),
it is sufficient to estimate the error (in $t$) between the jump matrices $V^{(r)}(\tilde{k})$ and $\tilde{V}^{(r)}(\tilde{k})$ on $(\tilde{k}_2, \tilde{k}_1)$ and $\tilde{\Sigma}^{(r)}_1$, respectively.

To proceed, we recall the following standard upper/lower triangular matrix decomposition of $V^{(r)}$ and $\tilde{V}^{(r)}$:
\begin{equation}
V^{(r)}(\tilde k)=(I-w_-(\tilde k))^{-1}(I+w_+(\tilde k)), \qquad \tilde V^{(r)}(\tilde k)=(I-\tilde w_-(\tilde k))^{-1}(I+ \tilde w_+(\tilde k)),
\end{equation}
where
\begin{align}
	w^{(r)}_+(\tilde{k})&=\left\{\begin{array}{ll}
		\left(\begin{array}{cc}
			0 & 0\\
			-e^{-2i\theta(\left( \frac{9t}{2}\right) ^{-1/3}\tilde{k}+1)}\overline{R(\left( \frac{9t}{2}\right) ^{-1/3}\tilde{k}+1)} & 0
		\end{array}\right),   & \tilde{k}\in (\tilde{k}_2,\tilde{k}_1),\\[14pt]
		\left(\begin{array}{cc}
			0 & 0\\
			\left(-\bar{R}(k_j)-\bar{R}'(k_j)(\left( \frac{9t}{2}\right) ^{-1/3}\tilde{k}+1-k_j)\right)e^{-2i\theta(\left( \frac{9t}{2}\right) ^{-1/3}\tilde{k}+1)} & 0
		\end{array}\right), &\tilde{k}\in \tilde{\Sigma}^{(r)}_j,\\[14pt]
	0,  & \text{elsewhere},
	\end{array}\right.
\label{def:wr+}
\\
w^{(r)}_-(\tilde{k})&=\left\{\begin{array}{ll}\left(\begin{array}{cc}
		0 & e^{2i\theta(\left( \frac{9t}{2}\right) ^{-1/3}\tilde{k}+1)}R(\left( \frac{9t}{2}\right) ^{-1/3}\tilde{k}+1) \\
		0 & 0
	\end{array}\right),   & \tilde{k}\in (\tilde{k}_2,\tilde{k}_1),\\[14pt]
	\left(\begin{array}{cc}
		0 & \left(R(k_{j})+R'(k_j)(\left( \frac{9t}{2}\right) ^{-1/3}\tilde{k}+1-k_j)\right)e^{\theta(\left( \frac{9t}{2}\right) ^{-1/3}\tilde{k}+1)} \\
		0 & 0
	\end{array}\right),  & \tilde{k}\in \tilde{\Sigma}^{(r)*}_j,\\[14pt]
	0,  &  \text{elsewhere},
	\end{array}\right.\\
\tilde{w}^{(r)}_+(\tilde{k})&=\left\{\begin{array}{ll}
	\left(\begin{array}{cc}
		0 & 0 \\
		-r(1)e^{2i(\tilde{s}\tilde{k}+\frac{4}{3}\tilde{k}^3)} & 0
	\end{array}\right),   &  \tilde{k}\in 	(\tilde{k}_2,\tilde{k}_1)\cup\tilde{\Sigma}^{(r)}_j, \\[12pt]
	0,  &  \text{elsewhere},
	\end{array}\right.\\
\tilde{w}^{(r)}_-(\tilde{k})&=\left\{\begin{array}{ll}
	\left(\begin{array}{cc}
		0 & r(1)e^{-2i(\tilde{s}\tilde{k}+\frac{4}{3}\tilde{k}^3)}\\
		0 & 0
	\end{array}\right),   & \tilde{k}\in 	(\tilde{k}_2,\tilde{k}_1)\cup\tilde{\Sigma}^{(r)*}_j, \\[12pt]
	0,   & \text{elsewhere},\label{def:wr-tilde}
\end{array}\right.
\end{align}
with $j=1,2$. For $\tilde{k}\in(\tilde{k}_2, \tilde{k}_1)$, we have $|e^{-2i(\tilde{s}\tilde{k}+\frac{4}{3}\tilde{k}^3)}|=|e^{2i\theta(\tilde{k})} |=1$ and
	\begin{align}
 &\left|R\left(\left( \frac{9t}{2}\right) ^{-1/3}\tilde{k}+1\right)-r(1)\right|=\left|\int_1^{\left( \frac{9t}{2}\right) ^{-1/3}\tilde{k}+1}R'(\zeta)\dif\zeta\right| \nonumber
   \\
     &\leqslant \|R'\|_\infty
		\left|\left( \frac{9t}{2}\right) ^{-1/3}\tilde{k}\right| \lesssim t^{-1/3},
  \\
		&\left|e^{2i\theta(\tilde{k})}-e^{-2i(\tilde{s}\tilde{k}+\frac{4}{3}\tilde{k}^3)}\right|=\left|\exp\left\lbrace \mathcal{O}(\tilde{k}^4t^{-1/3})\right\rbrace -1\right|\lesssim t^{-1/3}.
	\end{align}
For $\tilde{k}\in\tilde{\Sigma}^{(r)}_{1}$, since $\re\left[ 2i(\tilde{s}\tilde{k}+\frac{4}{3}\tilde{k}^3)\right] <0$, it follows that  $|e^{2i(\tilde{s}\tilde{k}+\frac{4}{3}\tilde{k}^3)}|\in L^\infty(\tilde{\Sigma}^{(r)}_1)\cap L^1(\tilde{\Sigma}^{(r)}_1)\cap L^2(\tilde{\Sigma}^{(r)}_1)$ with $\left|e^{2i\theta(\tilde{k})}-e^{-2i(\tilde{s}\tilde{k}+\frac{4}{3}\tilde{k}^3)}\right|\lesssim t^{-1/3+4\delta_1}$. Moreover,
	\begin{align*}
		&\left\vert R(k_{1})-r(1)+R'(k_1)(\left( \frac{9t}{2}\right) ^{-1/3}\tilde{k}+1-k_1)\right\vert=\left\vert\int_{k_1}^{1}(R'(\zeta)-R'(k_1))\dif\zeta+R'(k_1)\left( \frac{9t}{2}\right) ^{-1/3}\tilde{k}\right\vert\\
		&=\left\vert\int_{k_1}^{1}\int_{\zeta}^{k_1}R''(\eta)\dif\eta\dif\zeta+R'(k_1)\left( \frac{9t}{2}\right) ^{-1/3}\tilde{k}\right\vert\lesssim\left( \tilde{k}_1t^{-1/3}\right) ^{3/2}+t^{-1/2+\delta_1},
	\end{align*}
where we have made use of the fact that $R'(k_1)\sim t^{-1/6}$ in the last step. Combining the above estimates with \eqref{def:wr+}--\eqref{def:wr-tilde}, it follows that
	\begin{align}
		\|(\cdot)^i (w-\tilde{w})\|_{L^1\cap L^2\cap L^{\infty}}\lesssim t^{-1/3+4\delta_1}, \qquad w=w^{(r)}_{\pm}, \qquad i=0,1, \label{error w}
	\end{align}
which implies that $\|J_{\Xi}-I\|_{L^1\cap L^2\cap L^{\infty}}\lesssim t^{-1/3+4\delta_1}$ uniformly. Thus, the existence and uniqueness of $\Xi(\tilde{k})$ are valid by using a small norm RH problem arguments \cite{Deift2002LongtimeAF}, which also yields the estimate \eqref{equ:Xiasy}.
\end{proof}

Proposition \ref{prop:Xi} also implies that RH problem for $N^{(r)}$ is uniquely solvable for large positive $t$. Moreover, from the general theory developed by Beals and Coifman \cite{Beals1984ScatteringAI}, it follows that
\begin{align}
	N^{(r)}(\tilde{k})&=I+\frac{1}{2\pi i}\int_{\tilde{\Sigma}^{(r)}}\frac{\mu^{(r)}(\zeta)(w^{(r)}_+(\zeta)+w^{(r)}_-(\zeta))}{\zeta-\tilde{k}}\dif\zeta, \label{eq:Nrinteg} \\
	\tilde{N}^{(r)}(\tilde{k})&=I+\frac{1}{2\pi i}\int_{\tilde{\Sigma}^{(r)}}\frac{\tilde{\mu}^{(r)}(\zeta)(\tilde{w}^{(r)}_+(\zeta)+\tilde{w}^{(r)}_-(\zeta))}{\zeta-\tilde{k}}\dif\zeta, \label{eq:tildeNrinteg}
\end{align}
where $w^{(r)}_{\pm}$ and $\tilde w^{(r)}_{\pm}$ are defined in \eqref{def:wr+}--\eqref{def:wr-tilde},
\begin{align*}
	\mu^{(r)}=I+\mathcal{C}_+[\mu^{(r)}w^{(r)}_-]+\mathcal{C}_-[\mu^{(r)}w^{(r)}_+],\qquad \tilde{\mu}^{(r)}=I+\mathcal{C}_+[\tilde{\mu}^{(r)}\tilde{w}^{(r)}_-]+\mathcal{C}_-[\tilde{\mu}^{(r)}\tilde{w}^{(r)}_+],
\end{align*}
and where $\mathcal{C}_\pm$ are the Cauchy projection operators on $\tilde{\Sigma}^{(r)}$ as defined in \eqref{def:Cpm}. As $\tilde{k}\to\infty$, it is readily seen from \eqref{eq:Nrinteg} and \eqref{eq:tildeNrinteg} that
the large-$\tilde{k}$ expansion of $N^{(r)}(\tilde{k})$ and $\tilde{N}^{(r)}(\tilde{k})$ are
\begin{align}\label{eq:Nrexpansion}
	N^{(r)}(\tilde{k})=I+\frac{N^{(r)}_1}{\tilde{k}}+\frac{N^{(r)}_2}{\tilde{k}^2}+\mathcal{O}(\tilde{k}^{-3}),\qquad \tilde{N}^{(r)}(\tilde{k})=I+\frac{\tilde{N}^{(r)}_1}{\tilde{k}}+\frac{\tilde{N}^{(r)}_2}{\tilde{k}^2}+\mathcal{O}(\tilde{k}^{-3}),
\end{align}
where
\begin{align}
	&N^{(r)}_j=-\frac{1}{2\pi i}\int_{\tilde{\Sigma}^{(r)}}\zeta^{j-1}\mu^{(r)}(\zeta)(w^{(r)}_+(\zeta)+w^{(r)}_-(\zeta))\dif\zeta, \label{def:Nrj}\\
	&\tilde{N}^{(r)}_j=-\frac{1}{2\pi i}\int_{\tilde{\Sigma}^{(r)}}\zeta^{j-1}\tilde{\mu}^{(r)}(\zeta)(\tilde{w}^{(r)}_+(\zeta)+\tilde{w}^{(r)}_-(\zeta))\dif\zeta, \label{def:tildeNrj}
\end{align}
with $j=1,2$. The following proposition then follows directly from Proposition \ref{prop:Xi}.
\begin{Proposition}\label{asyN}
With $N^{(r)}_j$ and $\tilde{N}^{(r)}_j$, $j=1,2$, defined in \eqref{def:Nrj} and \eqref{def:tildeNrj}, we have, as $t\to\infty$,
	\begin{align}
		N^{(r)}_j=\tilde{N}^{(r)}_j+\mathcal{O}(t^{-1/3+4\delta_1}).
	\end{align}
\end{Proposition}

%
%Thus, it is easily seen that . The method
%we use is developed
%As mentioned in notations,
%$w^{(r)}_\pm$, $\tilde{w}^{(r)}_\pm$ are defined through the standard upper/lower triangular matrix decomposition of jump matrix $V^{(r)}$ and $\tilde{V}^{(r)}$ respectively:
%
%
%Define $\Xi(\tilde{k}):=N^{(r)}(\tilde{k})(\tilde{N}^{(r)}(\tilde{k}))^{-1}$. The RH conditions for $\Xi(\tilde{k})$ are listed below.
%\begin{RHP}Find a matrix-valued function $\Xi(\tilde{k})=\Xi(\tilde{k};y,t)$ such that
%	\begin{itemize}
%		\item[$\bullet$] $\Xi(\tilde{k})$ is analytical for $\mathbb{C}\backslash \tilde{\Sigma}^{(r)}$.
%		\item[$\bullet$] $\Xi_{+}(\tilde{k})=\Xi_{-}(\tilde{k})J_{\Xi}(\tilde{k})$ with $J_{\Xi}(\tilde{k})=\tilde{N}^{(r)}(\tilde{k})V^{(r)}(\tilde{k}){\tilde{V}^{(r)}(\tilde{k})}^{-1}{\tilde{N}^{(r)}(\tilde{k})}^{-1}$, for $\tilde{k}\in\tilde{\Sigma}^{(r)}$.
%	\end{itemize}
%\end{RHP}

%Next proposition reveals the error between the local model $N^{(r)}(\tilde{k})$ and the model problem $\tilde{N}^{(r)}(\tilde{k})$ as $t\rightarrow+\infty$.

\paragraph{Step 3: Construction of $\tilde{N}^{(r)}$}
We could solve the RH problem for $\tilde{N}^{(r)}$ explicitly by using the Painlev\'{e} II parametrix $M^{P}(z;s,\kappa)$ introduced in Appendix \ref{appendix: RHP for PII}. More precisely, define
\begin{align}\label{equ:mathcing local 1stcase}
	\tilde{N}^{(r)}(\tilde{k})=e^{\frac{\pi i}{4}\sigma_3} M^{P}(\tilde{k};\tilde s, r(1))e^{-\frac{\pi i}{4}\sigma_3}\tilde{H}(\tilde{k}),
\end{align}
where
\begin{align}
 \tilde{H}(\tilde{k})=\left\{\begin{array}{ll}
 		\left(\begin{array}{cc}
 			1 & 0\\
 			-r(1)e^{2i(\tilde{s}\tilde{k}+\frac{4}{3}\tilde{k}^3)} & 1
 		\end{array}\right), & \tilde{k}\in \tilde{\Omega},\\[12pt]
 		\left(\begin{array}{cc}
 			1 & -r(1)e^{-2i(\tilde{s}\tilde{k}+\frac{4}{3}\tilde{k}^3)} \\
 			0 & 1
 		\end{array}\right),  & \tilde{k}\in \tilde{\Omega}^*,\\[10pt]
 	I, & \text{ elsewhere,}\\
 	\end{array}\right.
\end{align}
and where the region $\tilde \Omega$ is illustrated in the left picture of Figure \ref{fig:jump of hat{N}_{a}}. We use the function
$\tilde{H}$ to transform the jump contours $\tilde{\Sigma}^{(r)}$ defined in \eqref{def:tildesigmar} to those of the Painlev\'{e} II parametrix for $|\tilde k|\leqslant c_0(9t/2)^{1/3}$; see Figure \ref{fig:jump of hat{N}_{a}} for an illustration.
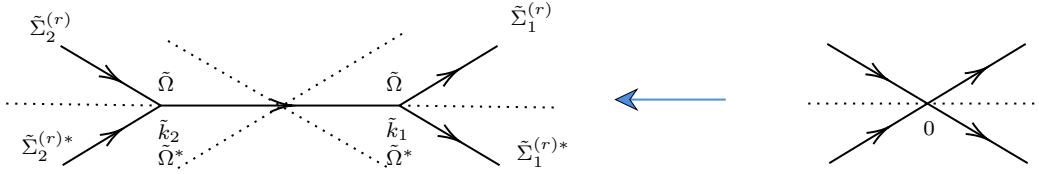
\begin{figure}[htbp]
	\begin{center}
		\tikzset{every picture/.style={line width=0.75pt}} %set default line width to 0.75pt
		\begin{tikzpicture}[x=0.75pt,y=0.75pt,yscale=-1,xscale=1]
			%uncomment if require: \path (0,300); %set diagram left start at 0, and has height of 300
			%Straight Lines [id:da3423497420414079]
			\draw    (109,144) -- (228,144) ;
			\draw [shift={(174.5,144)}, rotate = 180] [color={rgb, 255:red, 0; green, 0; blue, 0 }  ][line width=0.75]    (10.93,-3.29) .. controls (6.95,-1.4) and (3.31,-0.3) .. (0,0) .. controls (3.31,0.3) and (6.95,1.4) .. (10.93,3.29)   ;
			%Straight Lines [id:da09201758824974449]
			\draw    (228,144) -- (277,114) ;
			\draw [shift={(257.62,125.87)}, rotate = 148.52] [color={rgb, 255:red, 0; green, 0; blue, 0 }  ][line width=0.75]    (10.93,-3.29) .. controls (6.95,-1.4) and (3.31,-0.3) .. (0,0) .. controls (3.31,0.3) and (6.95,1.4) .. (10.93,3.29)   ;
			%Straight Lines [id:da8879610941244553]
			\draw    (228,144) -- (278,174) ;
			\draw [shift={(258.14,162.09)}, rotate = 210.96] [color={rgb, 255:red, 0; green, 0; blue, 0 }  ][line width=0.75]    (10.93,-3.29) .. controls (6.95,-1.4) and (3.31,-0.3) .. (0,0) .. controls (3.31,0.3) and (6.95,1.4) .. (10.93,3.29)   ;
			%Straight Lines [id:da9773409363397114]
			\draw    (59,114) -- (109,144) ;
			\draw [shift={(89.14,132.09)}, rotate = 210.96] [color={rgb, 255:red, 0; green, 0; blue, 0 }  ][line width=0.75]    (10.93,-3.29) .. controls (6.95,-1.4) and (3.31,-0.3) .. (0,0) .. controls (3.31,0.3) and (6.95,1.4) .. (10.93,3.29)   ;
			%Straight Lines [id:da3746594725280563]
			\draw    (60,174) -- (109,144) ;
			\draw [shift={(89.62,155.87)}, rotate = 148.52] [color={rgb, 255:red, 0; green, 0; blue, 0 }  ][line width=0.75]    (10.93,-3.29) .. controls (6.95,-1.4) and (3.31,-0.3) .. (0,0) .. controls (3.31,0.3) and (6.95,1.4) .. (10.93,3.29)   ;
			%Straight Lines [id:da6723019930894476]
			\draw  [dash pattern={on 0.84pt off 2.51pt}]  (432,143) -- (551,143) ;
			%Straight Lines [id:da29128583241084494]
			\draw    (491.5,143) -- (540.5,113) ;
			\draw [shift={(521.12,124.87)}, rotate = 148.52] [color={rgb, 255:red, 0; green, 0; blue, 0 }  ][line width=0.75]    (10.93,-3.29) .. controls (6.95,-1.4) and (3.31,-0.3) .. (0,0) .. controls (3.31,0.3) and (6.95,1.4) .. (10.93,3.29)   ;
			%Straight Lines [id:da6323129270849626]
			\draw    (491.5,143) -- (541.5,173) ;
			\draw [shift={(521.64,161.09)}, rotate = 210.96] [color={rgb, 255:red, 0; green, 0; blue, 0 }  ][line width=0.75]    (10.93,-3.29) .. controls (6.95,-1.4) and (3.31,-0.3) .. (0,0) .. controls (3.31,0.3) and (6.95,1.4) .. (10.93,3.29)   ;
			%Straight Lines [id:da6661920853749934]
			\draw    (441.5,113) -- (491.5,143) ;
			\draw [shift={(471.64,131.09)}, rotate = 210.96] [color={rgb, 255:red, 0; green, 0; blue, 0 }  ][line width=0.75]    (10.93,-3.29) .. controls (6.95,-1.4) and (3.31,-0.3) .. (0,0) .. controls (3.31,0.3) and (6.95,1.4) .. (10.93,3.29)   ;
			%Straight Lines [id:da7091357404665142]
			\draw    (442.5,173) -- (491.5,143) ;
			\draw [shift={(472.12,154.87)}, rotate = 148.52] [color={rgb, 255:red, 0; green, 0; blue, 0 }  ][line width=0.75]    (10.93,-3.29) .. controls (6.95,-1.4) and (3.31,-0.3) .. (0,0) .. controls (3.31,0.3) and (6.95,1.4) .. (10.93,3.29)   ;
			%Straight Lines [id:da6366646818265389]
			\draw  [dash pattern={on 0.84pt off 2.51pt}]  (31.86,143) -- (109,144) ;
			%Straight Lines [id:da09243688486496526]
			\draw  [dash pattern={on 0.84pt off 2.51pt}]  (228,144) -- (305.14,145) ;
			%Straight Lines [id:da4942049767937191]
			\draw  [dash pattern={on 0.84pt off 2.51pt}]  (112.57,111.5) -- (224.43,176.5) ;
			%Straight Lines [id:da7108279656298448]
			\draw  [dash pattern={on 0.84pt off 2.51pt}]  (117.57,175.5) -- (230.86,107) ;
			%Straight Lines [id:da7815272650764269]
			\draw [color={rgb, 255:red, 74; green, 144; blue, 226 }  ,draw opacity=1 ]   (342,141) -- (390.86,141) ;
			\draw [shift={(335.86,141)}, rotate = 0] [fill={rgb, 255:red, 74; green, 144; blue, 226 }  ,fill opacity=1 ][line width=0.08]  [draw opacity=0] (10.72,-5.15) -- (0,0) -- (10.72,5.15) -- (7.12,0) -- cycle    ;
			% Text Node
             \draw (220,125.4) node [anchor=north west][inner sep=0.75pt]  [font=\scriptsize]  {$ \tilde{\Omega}$};
            \draw (220,162.4) node [anchor=north west][inner sep=0.75pt]  [font=\scriptsize]  {$ \tilde{\Omega}^*$};
			\draw (220,148.4) node [anchor=north west][inner sep=0.75pt]  [font=\scriptsize]  {$\tilde{k}_{1}$};
			% Text Node
   		\draw (106,125.4) node [anchor=north west][inner sep=0.75pt]  [font=\scriptsize]                     {$\tilde{\Omega}$};
        	\draw (106,162.4) node [anchor=north west][inner sep=0.75pt]  [font=\scriptsize]           {$\tilde{\Omega}^*$};
			\draw (106,149.4) node [anchor=north west][inner sep=0.75pt]  [font=\scriptsize]  {$\tilde{k}_{2}$};
			% Text Node
			\draw (282,90.4) node [anchor=north west][inner sep=0.75pt]  [font=\scriptsize]  {$\tilde{\Sigma }_{1}^{(r)}$};
			% Text Node
			\draw (285,158.4) node [anchor=north west][inner sep=0.75pt]  [font=\scriptsize]  {$\tilde{\Sigma }_{1}^{(r)*}$};
			% Text Node
			\draw (42,95.4) node [anchor=north west][inner sep=0.75pt]  [font=\scriptsize]  {$\tilde{\Sigma }_{2}^{(r)}$};
			% Text Node
			\draw (38,155.4) node [anchor=north west][inner sep=0.75pt]  [font=\scriptsize]  {$\tilde{\Sigma }_{2}^{(r)*}$};
			% Text Node
			%\draw (545,83.4) node [anchor=north west][inner sep=0.75pt]  [font=\scriptsize]  {$\tilde{\Sigma }_{1}^{(r)}$};
%			% Text Node
%			\draw (543,165.4) node [anchor=north west][inner sep=0.75pt]  [font=\scriptsize]  {$\tilde{\Sigma }_{1}^{(r)*}$};
%			% Text Node
%			\draw (428,81.4) node [anchor=north west][inner sep=0.75pt]  [font=\scriptsize]  {$\tilde{\Sigma }_{2}^{(r)}$};
%			% Text Node
%			\draw (414,162.4) node [anchor=north west][inner sep=0.75pt]  [font=\scriptsize]  {$\tilde{\Sigma }_{2}^{(r)*}$};
			% Text Node
			\draw (488,150.4) node [anchor=north west][inner sep=0.75pt]  [font=\scriptsize]  {$0$};
		\end{tikzpicture}
		\caption{ The jump contours of the RH problems for $\tilde{N}^{(r)}$ (left) and $M^{P}$ (right).} \label{fig:jump of hat{N}_{a}}
	\end{center}
\end{figure}
%For this case, the jump $J^{(P)}$ of $M^{P}(\tilde{k})$ satisfies
%\begin{align}
%	J^{(P)}(\tilde{k})=\left\{\begin{array}{lll}
%		e^{-i\left(\frac{4}{3}\tilde{k}^3+\tilde{s}\tilde{k}\right)\tilde{\sigma}_3}\left(\begin{array}{cc}
%			1 & 0\\
%			ir(1) & 1
%		\end{array}\right), & \tilde{k}\in\tilde{\Sigma}^{(r)}_j, \ j=1,2, \\
%		\\
%		e^{-i\left(\frac{4}{3}\tilde{k}^3+\tilde{s}\tilde{k}\right)\tilde{\sigma}_3}\left(\begin{array}{cc}
%			1 & ir(1)\\
%			0 & 1
%		\end{array}\right),  &\tilde{k}\in\tilde{\Sigma}^{(r)*}_j, \ j=1,2.\\
%	\end{array}\right.
%\end{align}

In view of RH problem \ref{appendix: reduced PII RH}, it's readily seen that $\tilde{N}^{(r)}$ in \eqref{equ:mathcing local 1stcase} indeed solves RH problem \ref{RHP: Model RH problem ass N^{(r)}(z)}. Moreover, let $t\to +\infty$, we observe from \eqref{MPfirstexpansion}--\eqref{vMPasy} and \eqref{eq:Nrexpansion} that
\begin{align}
	&\tilde{N}^{(r)}_1=	\frac{i}{2}\left(\begin{array}{cc}
		-\int_{\tilde{s}}^{+\infty}v^2(\zeta)\dif\zeta   & v(\tilde{s})\\
		-v(\tilde{s}) & \int_{\tilde{s}}^{+\infty}v^2(\zeta)\dif\zeta
	\end{array}\right)+\mathcal{O}(e^{-ct^{3\delta_1}}),\label{equ:Nr1specific}\\
	&\tilde{N}^{(r)}_2=\frac{1}{8}\left(\begin{array}{cc}
		-\left(\int_{\tilde{s}}^{+\infty}v^2(\zeta)\dif\zeta\right)^2+v^2(\tilde{s})  & 2\left(v(\tilde{s})\int_{\tilde{s}}^{+\infty}v^2(\zeta)\dif\zeta  +v'(\tilde{s})\right) \\
		-2\left(v(\tilde{s})\int_{\tilde{s}}^{+\infty}v^2(\zeta)\dif\zeta  +v'(\tilde{s})\right)  & -\left(\int_{\tilde{s}}^{+\infty}v^2(\zeta)\dif\zeta   \right)^2+v^2(\tilde{s})
	\end{array}\right)+\mathcal{O}(e^{-ct^{3\delta_1}}),\label{equ:Nr2specific}
\end{align}
for some $c>0$, where $(')=\frac{\dif }{\dif \tilde{s}}$ and $v(\tilde{s})$ is the unique solution of Painlev\'e II equation \eqref{equ:standard PII equ},
fixed by the boundary condition
\begin{align}\label{equ:v(tildes)asy1sttran}
	v(\tilde{s})\sim r(1)\textnormal{Ai}(\tilde{s}), \qquad \tilde{s}\to +\infty.
\end{align}

%we know that the
%$\tilde{N}^{(r)}_j$, $j=1,2$ in the expansion of $\tilde{N}^{(r)}(\tilde{k})$ can be expressed in terms of the  solution
%In generic case, $|r(1)|=1, $ then $v(\tilde{s})$ is the Hastings-McLeod solution of the Painlev\'e II equation, while in non-generic case, $|r(1)|<1 $, then $v(\tilde{s})$ is the  Ablowitz-Segur solution of the Painlev\'e II equation.
%And

Finally, the RH problem for $N^{(l)}$ can be solved in a similar manner. Indeed, for $z\in U^{(l)}$ and $t$ large enough, we have
	\begin{align}
		\theta(z)=-\tilde{s}\breve{k}-\frac{4}{3}\breve{k}^3+\mathcal{O}(\breve{k}^4t^{-1/3})
	\end{align}
where $\tilde s$ is defined in \eqref{equ:tildes1stran} and
	\begin{align}
		\breve{k}=\left( \frac{9t}{2}\right) ^{1/3}(z+1)
	\end{align}
is the scaled spectral parameter in this case. Following three steps we have just performed, we could approximate $N^{(l)}$ by the RH problem for $\breve{N}^{(l)}$ on the $\breve{k}$-plane, such that, as $\breve k \to \infty$,
\begin{align}\label{eq:Nlkexp}
N^{(l)}(\breve{k})=I+\frac{N^{(l)}_1}{\breve k}+\frac{N^{(l)}_2}{\breve{k}^2}+\mathcal{O}(\breve{k}^{-3}), \qquad
	\breve{N}^{(l)}(\breve{k})=I+\frac{\breve N^{(l)}_1}{\breve{k}}+\frac{\breve N^{(l)}_2}{\breve{k}^2}+\mathcal{O}(\breve{k}^{-3})
\end{align}
where, as $t \to +\infty$,
\begin{align}
&N^{(l)}_j=\breve{N}^{(l)}_j+\mathcal{O}(t^{-1/3+4\delta_1}), \qquad j=1,2, \label{eq:Nljest}
\\
&\breve{N}^{(l)}_1=	\frac{i}{2}\left(\begin{array}{cc}
		-\int_{\tilde{s}}^{+\infty}v^2(\zeta)\dif\zeta   & -v(\tilde{s})\\
		v(\tilde{s}) & \int_{\tilde{s}}^{+\infty}v^2(\zeta)\dif\zeta
	\end{array}\right)+\mathcal{O}(e^{-ct^{3\delta_1}}),\label{equ:Nl1specific}\\
	&\breve{N}^{(l)}_2=\frac{1}{8}\left(\begin{array}{cc}
		-\left(\int_{\tilde{s}}^{+\infty}v^2(\zeta)\dif\zeta   \right)^2+v^2(\tilde{s})  & -2\left(v(\tilde{s})\int_{\tilde{s}}^{+\infty}v^2(\zeta)\dif\zeta  +v'(\tilde{s})\right) \\
		2\left(v(\tilde{s})\int_{\tilde{s}}^{+\infty}v^2(\zeta)\dif\zeta  +v'(\tilde{s})\right)  & -\left(\int_{\tilde{s}}^{+\infty}v^2(\zeta)\dif\zeta\right)^2+v^2(\tilde{s})
	\end{array}\right)+\mathcal{O}(e^{-ct^{3\delta_1}}),\label{equ:Nl2specific}
\end{align}
with the same $v$ in \eqref{equ:Nr1specific} and \eqref{equ:Nr2specific}. Here, we also need to use the fact that $r(z)=-\overline{r(-z)}=\overline{r(z^{-1})}$, which particularly implies that  $r(1)=-r(-1)\in \mathbb{R}$.

\subsubsection*{The small norm RH problem}
Let $N(z)$ and $N^{(\ell)}(z)$, $\ell \in\{r,l\}$ be the solutions of RH problems \ref{RHP: N(z) tran1} and \ref{RHP:Nrl}, we define
 \begin{equation}\label{def:E(z)1sttranregion}
 	E(z)=\left\{\begin{array}{ll}
 		N(z), & z\in \mathbb{C} \setminus\left(U^{(r)}\cup U^{(l)}\right),\\
 		N(z){N^{(r)}(z)}^{-1},  &z\in U^{(r)},\\
 		N(z){N^{(l)}(z)}^{-1},  &z\in U^{(l)}.\\
 	\end{array}\right.
 \end{equation}
It is then readily seen that $E(z)$ satisfies the following RH problem.
%The RH conditions for $E(z)$ are detailed below, which follow directly from those of $N^{(r)}(z)$ and $N^{(l)}(z)$. Let us denote
%$\Sigma^{(E)}:= \partial U^{(r)}\cup\partial U^{(l)}\cup\left( \Sigma^{(4)}\setminus (U^{(r)}\cup U^{(l)})\right) $ (see Figure \ref{figE}). The RH problem for $E(z)$ is as follows:
\begin{RHP}\label{RHP:E(z)}
	\hfill
	\begin{itemize}
	\item[$\bullet$]   $E(z)$ is holomorphic for $z\in\mathbb{C}\setminus  \Sigma^{(E)} $, where
                   $$\Sigma^{(E)}:= \partial U^{(r)}\cup\partial U^{(l)}\cup\left( \Sigma^{(4)}\setminus (U^{(r)}\cup U^{(l)})\right);$$
                   see Figure \ref{figE} for an illustration.
	\item[$\bullet$]   For $z\in \Sigma^{(E)}$, we have
	$$E_+(z)=E_-(z)V^{(E)}(z),$$
	where
	\begin{equation}
		V^{(E)}(z)=\left\{\begin{array}{llll}
			V^{(4)}(z), & z\in \Sigma^{(E)}\setminus (U^{(r)}\cup U^{(l)}),\\[4pt]
			N^{(r)}(z),  & z\in \partial U^{(r)},\\[4pt]
			N^{(l)}(z),  & z\in \partial U^{(l)},
		\end{array}\right. \label{deVE}
	\end{equation}
  and where $V^{(4)}(z)$ is defined in \eqref{jump:V^{(4)}}.
	%We choose the $\partial U^{(r)}$ and $\partial U^{(l)}$ to be clockwise oriented (see Figure \ref{figE}).
	\item[$\bullet$] As $z\to \infty$ in $\mathbb{C}\setminus \Sigma^{(E)}$, we have
	$E(z) =I+\mathcal{O}(z^{-1})$.
\end{itemize}
\end{RHP}

\begin{figure}[htbp]
	\centering
	\tikzset{every picture/.style={line width=0.75pt}} %set default line width to 0.75pt
	\begin{tikzpicture}[x=0.75pt,y=0.75pt,yscale=-1,xscale=1]
	%uncomment if require: \path (0,300); %set diagram left start at 0, and has height of 300
	%Straight Lines [id:da4200862674542387]
	\draw    (179,98) -- (238,136) ;
	\draw [shift={(213.54,120.25)}, rotate = 212.78] [color={rgb, 255:red, 0; green, 0; blue, 0 }  ][line width=0.75]    (10.93,-3.29) .. controls (6.95,-1.4) and (3.31,-0.3) .. (0,0) .. controls (3.31,0.3) and (6.95,1.4) .. (10.93,3.29)   ;
	%Straight Lines [id:da42656681430555765]
	\draw    (179,203) -- (237,165) ;
	\draw [shift={(213.02,180.71)}, rotate = 146.77] [color={rgb, 255:red, 0; green, 0; blue, 0 }  ][line width=0.75]    (10.93,-3.29) .. controls (6.95,-1.4) and (3.31,-0.3) .. (0,0) .. controls (3.31,0.3) and (6.95,1.4) .. (10.93,3.29)   ;
	%Shape: Circle [id:dp9404901221756403]
	\draw  [color={rgb, 255:red, 208; green, 2; blue, 27 }  ,draw opacity=1 ] (233,152) .. controls (233,138.19) and (244.19,127) .. (258,127) .. controls (271.81,127) and (283,138.19) .. (283,152) .. controls (283,165.81) and (271.81,177) .. (258,177) .. controls (244.19,177) and (233,165.81) .. (233,152) -- cycle ;
	%Straight Lines [id:da9325528167884771]
	\draw    (276,133) -- (334,95) ;
	\draw [shift={(310.02,110.71)}, rotate = 146.77] [color={rgb, 255:red, 0; green, 0; blue, 0 }  ][line width=0.75]    (10.93,-3.29) .. controls (6.95,-1.4) and (3.31,-0.3) .. (0,0) .. controls (3.31,0.3) and (6.95,1.4) .. (10.93,3.29)   ;
	%Straight Lines [id:da21043233869036038]
	\draw    (277,170) -- (335,201) ;
	\draw [shift={(311.29,188.33)}, rotate = 208.12] [color={rgb, 255:red, 0; green, 0; blue, 0 }  ][line width=0.75]    (10.93,-3.29) .. controls (6.95,-1.4) and (3.31,-0.3) .. (0,0) .. controls (3.31,0.3) and (6.95,1.4) .. (10.93,3.29)   ;
	%Straight Lines [id:da6871610230528629]
	\draw    (334,95) -- (394,134) ;
	\draw [shift={(369.03,117.77)}, rotate = 213.02] [color={rgb, 255:red, 0; green, 0; blue, 0 }  ][line width=0.75]    (10.93,-3.29) .. controls (6.95,-1.4) and (3.31,-0.3) .. (0,0) .. controls (3.31,0.3) and (6.95,1.4) .. (10.93,3.29)   ;
	%Straight Lines [id:da4361729573545785]
	\draw    (335,201) -- (393,163) ;
	\draw [shift={(369.02,178.71)}, rotate = 146.77] [color={rgb, 255:red, 0; green, 0; blue, 0 }  ][line width=0.75]    (10.93,-3.29) .. controls (6.95,-1.4) and (3.31,-0.3) .. (0,0) .. controls (3.31,0.3) and (6.95,1.4) .. (10.93,3.29)   ;
	%Shape: Circle [id:dp654734540665245]
	\draw  [color={rgb, 255:red, 208; green, 2; blue, 27 }  ,draw opacity=1 ] (389,150) .. controls (389,136.19) and (400.19,125) .. (414,125) .. controls (427.81,125) and (439,136.19) .. (439,150) .. controls (439,163.81) and (427.81,175) .. (414,175) .. controls (400.19,175) and (389,163.81) .. (389,150) -- cycle ;
	%Straight Lines [id:da9507007047229303]
	\draw    (432,132) -- (490,94) ;
	\draw [shift={(466.02,109.71)}, rotate = 146.77] [color={rgb, 255:red, 0; green, 0; blue, 0 }  ][line width=0.75]    (10.93,-3.29) .. controls (6.95,-1.4) and (3.31,-0.3) .. (0,0) .. controls (3.31,0.3) and (6.95,1.4) .. (10.93,3.29)   ;
	%Straight Lines [id:da03649337755817528]
	\draw    (435,164) -- (490,199) ;
	\draw [shift={(467.56,184.72)}, rotate = 212.47] [color={rgb, 255:red, 0; green, 0; blue, 0 }  ][line width=0.75]    (10.93,-3.29) .. controls (6.95,-1.4) and (3.31,-0.3) .. (0,0) .. controls (3.31,0.3) and (6.95,1.4) .. (10.93,3.29)   ;
	%Straight Lines [id:da6278559397755512]
	\draw  [dash pattern={on 0.84pt off 2.51pt}]  (113,150) -- (548,149) ;
	\draw [shift={(550,149)}, rotate = 179.87] [color={rgb, 255:red, 0; green, 0; blue, 0 }  ][line width=0.75]    (10.93,-3.29) .. controls (6.95,-1.4) and (3.31,-0.3) .. (0,0) .. controls (3.31,0.3) and (6.95,1.4) .. (10.93,3.29)   ;
	%Straight Lines [id:da2208649035960888]
	\draw    (334,97) -- (334,198) ;
	\draw [shift={(334,200)}, rotate = 270] [color={rgb, 255:red, 0; green, 0; blue, 0 }  ][line width=0.75]    (10.93,-3.29) .. controls (6.95,-1.4) and (3.31,-0.3) .. (0,0) .. controls (3.31,0.3) and (6.95,1.4) .. (10.93,3.29)   ;
	\draw [shift={(334,95)}, rotate = 90] [color={rgb, 255:red, 0; green, 0; blue, 0 }  ][line width=0.75]    (10.93,-3.29) .. controls (6.95,-1.4) and (3.31,-0.3) .. (0,0) .. controls (3.31,0.3) and (6.95,1.4) .. (10.93,3.29)   ;
	\draw  [color={rgb, 255:red, 208; green, 2; blue, 27 }  ,draw opacity=1 ] (408.76,120.06) -- (420.03,126.35) -- (407.37,128.87) ;
	\draw  [color={rgb, 255:red, 208; green, 2; blue, 27 }  ,draw opacity=1 ] (249.96,122.99) -- (262.1,127.37) -- (250.02,131.9) ;
	% Text Node
	\draw (323,151.4) node [anchor=north west][inner sep=0.75pt]  [font=\scriptsize]  {$0$};
	% Text Node
	\draw (421,149.4) node [anchor=north west][inner sep=0.75pt]  [font=\scriptsize]  {$k_{1}$};
	% Text Node
	\draw (397,149.4) node [anchor=north west][inner sep=0.75pt]  [font=\scriptsize]  {$k_{2}$};
	% Text Node
	\draw (265,152.4) node [anchor=north west][inner sep=0.75pt]  [font=\scriptsize]  {$k_{3}$};
	% Text Node
	\draw (244,152.4) node [anchor=north west][inner sep=0.75pt]  [font=\scriptsize]  {$k_{4}$};
	% Text Node
	\draw (540,158) node [anchor=north west][inner sep=0.75pt]  [font=\scriptsize] [align=left] {$\re z$};
	\end{tikzpicture}
	\caption{ The jump contour $\Sigma^{(E)}$ of the RH problem for $E$, where the two red circles are $\partial U^{(l)}$ and $\partial U^{(r)}$, repectively. }
	\label{figE}
\end{figure}
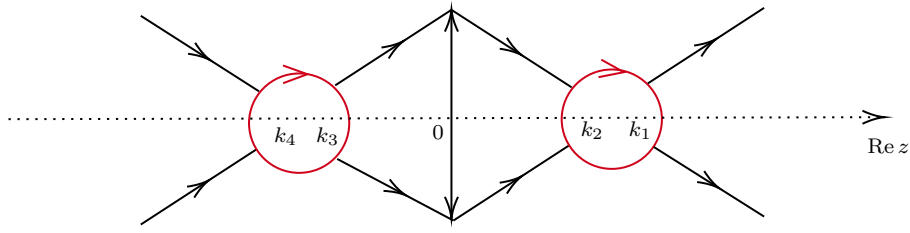

A simple calculation  shows that
\begin{equation}\label{E(z):BCrep}
	\parallel V^{(E)}(z)-I\parallel_{p}=\left\{\begin{array}{llll}
		\mathcal{O}(\exp\left\{-ct^{3\delta_1}\right\}),  & z\in  \Sigma^{(E)}\setminus  (U^{(r)}\cup U^{(l)}),\\[6pt]
		\mathcal{O}(t^{-\kappa_p}),   & z\in \partial U^{(r)}\cup\partial U^{(l)},
	\end{array}\right.
\end{equation}
for some positive $c$ with $\kappa_\infty=\delta_1$ and $\kappa_2=1/6+\delta_1/2$.
It then follows from the small norm RH problem theory \cite{Deift2002LongtimeAF} that there exists a unique solution to RH problem \ref{RHP:E(z)} for large positive $t$. Moreover, according to \cite{Beals1984ScatteringAI}, we have
\begin{equation}\label{expression:E(z)}
	E(z)=I+\frac{1}{2\pi i}\int_{\Sigma^{(E)}}\dfrac{\varpi(\zeta  ) (V^{(E)}(\zeta  )-I)}{\zeta  -z}\dif\zeta  ,
\end{equation}
where $\varpi\in I+ L^2(\Sigma^{(E)})$ is the unique solution of the Fredholm-type equation
\begin{align}
	\varpi=I+\mathcal{C}_E\varpi.\label{equ: varpi}
\end{align}
Here, $\mathcal{C}_E$: $L^2(\Sigma^{(E)})\to L^2(\Sigma^{(E)})$ is an integral operator defined by $\mathcal{C}_E(f)(z)=\mathcal{C}_-\left( f(V^{(E)}(z) -I)\right) $
with $\mathcal{C}_-$ being the Cauchy projection operator on $\Sigma^{(E)}$. Thus,
\begin{align}\label{eq:estCE}
	\parallel \mathcal{C}_E\parallel\leqslant\parallel \mathcal{C}_-\parallel \parallel V^{(E)}(z)-I\parallel_\infty \lesssim t^{-\delta_1},
\end{align}
which implies that  $1-\mathcal{C}_E$ is invertible  for   sufficiently large $t$.    So  $\varpi$  exists  uniquely with
\begin{align*}
	\varpi=I+(1-\mathcal{C}_E)^{-1}(\mathcal{C}_EI).
\end{align*}
On the other hand,
\eqref{equ: varpi} can be rewritten as
\begin{align}
	\varpi=I+\sum_{j=1}^8\mathcal{C}_E^jI+(1-\mathcal{C}_E)^{-1}(\mathcal{C}_E^9I),
\end{align}
where for $j=1,\cdots,8$, we have the estimates
\begin{align}
	&\|\mathcal{C}_E^jI\|_2\lesssim t^{-(1/6+j\delta_1-\delta_1/2)},\\
	&\| \varpi-I-\sum_{j=1}^8\mathcal{C}_E^jI\parallel_{2}\lesssim\dfrac{\parallel \mathcal{C}_E^9I\parallel_2}{1-\parallel \mathcal{C}_E\parallel}\lesssim t^{-(1/6+17\delta_1/2)}.\label{normrho}
\end{align}

For later use, we conclude this section with local behaviors of $E(z)$ at $z=i$ and $z=0$ as $t \to +\infty$. By \eqref{expression:E(z)}, it is readily seen that
\begin{align}\label{eq:Eneari}
	E(z)=E(i)+E_1(z-i)+\mathcal{O}((z-i)^2), \qquad z\to i,
\end{align}
where
\begin{align}
E(i)&=I+\frac{1}{2\pi i}\int_{\Sigma^{(E)}}\dfrac{\varpi(\zeta  ) (V^{(E)}(\zeta  )-I)}{\zeta  -i}\dif\zeta, \label{def:Ei}\\
E_1&=\frac{1}{2\pi i}\int_{\Sigma^{(E)}}\dfrac{\varpi(\zeta  ) (V^{(E)}(\zeta  )-I)}{(\zeta  -i)^2}\dif\zeta, \label{def:E1}
\end{align}
and
\begin{align}\label{def:E0}
	E(0)=I+\frac{1}{2\pi i}\int_{\Sigma^{(E)}}\frac{\varpi(\zeta  ) (V^{(E)}(\zeta  )-I)}{\zeta  }\dif\zeta.
\end{align}

\begin{Proposition}\label{asyE}
With $E(i)$, $E_1$ and $E(0)$ defined in \eqref{def:Ei}--\eqref{def:E0}, we have, as $t\to +\infty$,
\begin{subequations}
	\begin{align}
		E(i)&=I-t^{-1/3}\left( \frac{2}{9}\right)^{1/3} \left( \dfrac{N^{(r)}_1}{1-i}-\dfrac{N^{(l)}_1}{1+i}\right)
		+t^{-2/3}\frac{i}{2}\left( \frac{2}{9}\right)^{2/3}\left(N^{(r)}_2-N^{(l)}_2 \right)+\mathcal{O}(t^{-(1/3+9\delta_1)}),\label{E0t}\\
		E_1&=-t^{-1/3}\frac{i}{2}\left( \frac{2}{9}\right)^{1/3}\left(N^{(r)}_1-N^{(l)}_1 \right){+}it^{-2/3}\left( \frac{2}{9}\right)^{2/3}\left( \dfrac{N^{(r)}_2}{1-i}+\dfrac{N^{(l)}_2}{1+i}\right)+ \mathcal{O}(t^{-(1/3+9\delta_1)}).\label{equ:E1t}
        \\
	\label{equ:E(0)estastinfty}
	E(0)&=I-t^{-1/3}\left( \frac{2}{9}\right)^{1/3}\left( N^{(r)}_1-N^{(l)}_1\right)+t^{-2/3}\left( \frac{2}{9}\right)^{2/3}\left( N^{(r)}_2+N^{(l)}_2\right) +\mathcal{O}(t^{-(1/3+9\delta_1)}),
\end{align}
\end{subequations}
where $N_{j}^{(\ell)}$, $j=1,2$, $\ell \in \{l,r\}$, are defined in \eqref{eq:Nrexpansion} and \eqref{eq:Nlkexp}.
\end{Proposition}
\begin{proof}
From  \eqref{E(z):BCrep}, \eqref{normrho} and \eqref{def:Ei}, it follows that
		\begin{align*}
			E(i)&=I+\frac{1}{2\pi i}\oint_{\partial U^{(r)}}\dfrac{N^{(r)}(\zeta )-I}{\zeta-i}\dif\zeta  +\frac{1}{2\pi i}\oint_{\partial U^{(l)}}\dfrac{N^{(l)}(\zeta  )-I}{\zeta  -i}\dif\zeta  \\
			&~~~+\frac{1}{2\pi i}\oint_{\partial U^{(r)}}\dfrac{\mathcal{C}_-(\mathcal{C}_-(N^{(r)}-I)(N^{(r)}-I))(\zeta  )(N^{(r)}(\zeta  )-I)}{\zeta  -i}\dif\zeta \\
			&~~~ +\frac{1}{2\pi i}\oint_{\partial U^{(l)}}\dfrac{\mathcal{C}_-(\mathcal{C}_-(N^{(r)}-I)(N^{(r)}-I))(\zeta  )(N^{(l)}(\zeta  )-I)}{\zeta  -i}\dif\zeta
			 +\mathcal{O}(t^{-(1/3+9\delta_1)})\\
			&= I+t^{-1/3}\left( \frac{2}{9}\right)^{1/3}\frac{1}{2\pi i}\oint_{\partial U^{(r)}}\dfrac{N^{(r)}_1}{(\zeta  -i)(\zeta  -1)}\dif\zeta  +t^{-1/3}\left( \frac{2}{9}\right)^{1/3}\frac{1}{2\pi i}\oint_{\partial U^{(l)}}\dfrac{N^{(l)}_1}{(\zeta  -i)(\zeta  +1)}\dif\zeta  \\
			&~~~+t^{-2/3}\left( \frac{2}{9}\right)^{2/3}\frac{1}{2\pi i}\oint_{\partial U^{(r)}}\dfrac{N^{(r)}_2}{(\zeta  -i)(\zeta  -1)^2}\dif\zeta  +t^{-2/3}\left( \frac{2}{9}\right)^{2/3}\frac{1}{2\pi i}\oint_{\partial U^{(l)}}\dfrac{N^{(l)}_2}{(\zeta  -i)(\zeta  +1)^2}\dif\zeta  \\
			&~~~+t^{-2/3}\left( \frac{2}{9}\right)^{2/3}\frac{1}{2\pi i}\oint_{\partial U^{(r)}}\mathcal{C}_-\left(\frac{1}{(\cdot)-1}\right)(\zeta  )\dfrac{(N^{(r)}_1)^2}{(\zeta  -i)(\zeta  -1)}\dif\zeta  \\
			&~~~+t^{-2/3}\left( \frac{2}{9}\right)^{2/3}\frac{1}{2\pi i}\oint_{\partial U^{(l)}}\mathcal{C}_-\left(\frac{1}{(\cdot)+1}\right)(\zeta )\dfrac{(N^{(l)}_1)^2}{(\zeta  -i)(\zeta  +1)}\dif\zeta   \\
			&~~~+\frac{2t^{-1}}{9}\frac{1}{2\pi i}\oint_{\partial U^{(r)}}\mathcal{C}_-\left(\frac{\mathcal{C}_-\left(\frac{1}{(\cdot)+1}\right)}{(\cdot)+1}\right)(\zeta)\dfrac{(N^{(r)}_1)^3}{(\zeta  -i)(\zeta  -1)}\dif\zeta  \\
			&~~~+\frac{2t^{-1}}{9}\frac{1}{2\pi i}\oint_{\partial U^{(l)}}\mathcal{C}_-\left(\frac{\mathcal{C}_-(\frac{1}{(\cdot)+1})}{(\cdot)+1}\right)(\zeta  )\dfrac{(N^{(l)}_1)^3}{(\zeta  -i)(\zeta  +1)}\dif\zeta  +\mathcal{O}(t^{-(1/3+9\delta_1)}).
		\end{align*}	
Since $\mathcal{C}_-(\frac{1}{(\cdot)\pm1})(\zeta)=0$, an appeal to the residue theorem gives us \eqref{E0t}.

The estimates \eqref{equ:E1t} and \eqref{equ:E(0)estastinfty} can be obtained in similar manners, we omit the details here.
\end{proof}

\subsection{Analysis of the pure $\bar{\partial}$-problem }\label{subsec:dbar problem}
%It's turn to consider the pure $\bar{\partial}$-problem.
Besides the pure RH problem \ref{RHP: N(z) tran1} for $N$, the contribution to the RH problem \ref{RHP:mixed RH problem} for $M^{(4)}$ partially comes from the pure $\bar{\partial}$-problem $M^{(5)}$ defined as follows:
\begin{align}\label{transform:shengchengdbarM5}
	M^{(5)}(z)=M^{(4)}(z){N(z)}^{-1}.
\end{align}
Then $M^{(5)}$ satisfies the following $\bar{\partial}$ problem.
\begin{Dbarproblem}\label{DbarM5}
	%Find a matrix-valued function  $M^{(5)}(z)$ such that
\hfill
\begin{itemize}
	\item[$\bullet$]    $M^{(5)}(z)$ is continuous   and has sectionally continuous first partial derivatives in $\mathbb{C}$.
	\item[$\bullet$] As $z\rightarrow\infty$ in $\mathbb{C}$, we have
\begin{align}
	M^{(5)}(z) = I+\mathcal{O}(z^{-1}). \label{asymbehv7}
\end{align}
\item[$\bullet$] The $\bar{\partial}$-derivative of $M^{(5)}$ satisfies $\bar{\partial}M^{(5)}(z)=M^{(5)}(z)W^{(3)}(z)$, $z\in \mathbb{C}$ with
\begin{equation}\label{equ: W^(3)expression}
	W^{(3)}(z)=N(z)\bar{\partial}R^{(3)}(z)N(z)^{-1},
\end{equation}
where $\bar{\partial}R^{(3)}(z)$ is defined in \eqref{DBARR1}.
\end{itemize}
\end{Dbarproblem}
The solution of this pure $\bar{\partial}$-problem can be expressed in terms of the integral equation
\begin{equation}
	M^{(5)}(z)=I+\frac{1}{\pi}\iint_\mathbb{C}\dfrac{M^{(5)}(\zeta )W^{(3)} (\zeta )}{\zeta-z}\dif\mu(\zeta),\label{equ:M^{(5)}}
\end{equation}
where $\mu(\zeta)$ stands for the Lebesgue measure on $\mathbb{C}$. By introducing the left Cauchy-Green integral operator
\begin{equation}\label{def:leftCGop}
	f\mathcal{C}_z(z)=\frac{1}{\pi}\iint_{\mathbb{C}}\dfrac{f(\zeta )W^{(3)} (\zeta )}{\zeta-z}\dif \mu(\zeta),
\end{equation}
we could rewrite \eqref{equ:M^{(5)}} in an operator form
\begin{equation}
	M^{(5)}(z)=I\cdot\left(I-\mathcal{C}_z \right) ^{-1}.\label{equ:operator type of M^(5)}
\end{equation}
Aiming at estimating  $M^{(5)}(z)$, it suffices to evaluate the norm of the integral operator $\left(I-\mathcal{C}_z \right) ^{-1}$.
The first step toward this goal is the following lemma.
\begin{lemma}\label{lem:estofimtheta}
Let $\Omega_j$, $j=1,\ldots,4$, be the regions defined in \eqref{def:Omega1}--\eqref{def:Omega4}, we have the following estimates for
the imaginary part  of the phase function $\theta(z)$ given in \eqref{def:phasefunc} with $z\in\Omega_{j}\cup\Omega^*_{j}$:
	\begin{itemize}
		\item[\rm (i)]If $|\re z(1-|z|^{-2})|\leqslant 2$, one has
		\begin{subequations}
		\begin{align}
			&\im\theta(z)\leqslant - \frac{t}{2}\im z(\re z-k_j)^2,&& z\in\Omega_{j},\label{equ:estofimtheta<}\\
			&\im\theta(z)\geqslant\frac{t}{2}\im z(\re z-k_j)^2,&& z\in\Omega_{j}^*.
		\end{align}
		\end{subequations}
		\item[\rm (ii)]If $|\re z(1-|z|^{-2})|\geqslant 2$, one has
	\begin{subequations}
	 \begin{align}
	 	&\im\theta(z)\leqslant - 2t\hat{\xi}\im z,& z\in\Omega_{j}, \label{equ:estofimtheta>}\\
	 	&\im\theta(z)\geqslant 2t\hat{\xi}\im z,& z\in\Omega_{j}^*.
	 \end{align}
	\end{subequations}
	\end{itemize}
\end{lemma}
\begin{proof}
We only prove the estimates related to $z\in \Omega_1$, since the results in other regions can be obtained similarly.

For $z\in\Omega_{1}$,  we set
	\begin{align*}
		z-1/z:= u+vi,\ u,v\in\mathbb{R},\quad {\rm and } \quad  \hat{\xi}_1:=k_1-1/k_1\in[0,2).
	\end{align*}
Thus, $u=\re z(1-1/|z|^2)$, $v=\im z(1+1/|z|^2)$, and by \eqref{def:phasefunc},
\begin{equation}\label{eq:imtheta}
	\im \theta (z) = 2tv\left[ F(u,v)-\frac{\hat{\xi}}{8}\right], \qquad  F(u,v):=\frac{4-v^2-u^2}{v^4-2(4-u^2)v^2+(4+u^2)^2}.
\end{equation}
From the definition of $k_1$ in \eqref{equ:1sttransaddlepoints-a},  it follows that $\hat{\xi}=\frac{8(4-\hat{\xi}_1^2)}{(4+\hat{\xi}_1^2)^2}$.
It is easily seen that
\begin{equation*}
F(u,v)\leqslant\left\{
        \begin{array}{ll}
          F(u,0), & \hbox{$u^2\leqslant4$,} \\
          0, & \hbox{$u^2\geqslant4$.}
        \end{array}
      \right.
\end{equation*}
We further have
\begin{align*} F(u,0)-\frac{\hat{\xi}}{8}=\frac{4-u^2}{(4+u^2)^2}-\frac{4-\hat{\xi}_1^2}{(4+\hat{\xi}_1^2)^2}=-(u^2-\hat{\xi}_1^2)\frac{(4-\hat{\xi}_1^2)(8+u^2+\hat{\xi}_1^2)+(4+\hat{\xi}_1^2)^2}{(4+u^2)^2(4+\hat{\xi}_1^2)^2}.
\end{align*}
Inserting the above two formulas into \eqref{eq:imtheta}, we then obtain \eqref{equ:estofimtheta<} and \eqref{equ:estofimtheta>} from the fact that $u^2-\hat{\xi}_1^2\geqslant (\re z-k_1)^2$ and the ranges of $u$ and $\hat{\xi}$.
\end{proof}
With the help of Lemma \ref{lem:estofimtheta}, we are able to prove uniform boundedness of the norm of $\left(I-\mathcal{C}_z \right) ^{-1}$ for large $t$, as stated in the next proposition.
\begin{Proposition}\label{prop:Cz}
Let $\mathcal{C}_z$ be the operator defined in \eqref{def:leftCGop}, we have
%As $t\rightarrow\infty$
	\begin{equation}\label{eq:Czest}
		\parallel \mathcal{C}_z\parallel_{L^\infty\to L^\infty}\lesssim  t^{-1/3}, \qquad t\to +\infty,
	\end{equation}
which implies that $\|\left(I-\mathcal{C}_z \right) ^{-1}\|$ is uniformly bounded for large positive $t$.
\end{Proposition}
\begin{proof}
		For any $f\in L^\infty$,
	\begin{align}\label{inequ:fC_z}
		\parallel f\mathcal{C}_z \parallel_{\infty}&\leqslant \parallel f \parallel_{\infty}\frac{1}{\pi}\iint_{\mathbb{C}}\dfrac{|W^{(3)} (\zeta )|}{|z-\zeta|}\dif \mu(\zeta).
	\end{align}
Aiming at reaching our goal, it is then sufficient to evaluate the integral on the right-hand side of \eqref{inequ:fC_z}.
Note that $N(\zeta)$, $N(\zeta)^{-1}$ are bounded for $\zeta \in \Omega$ as well as $W^{(3)}(\zeta)\equiv 0$ for $\zeta\in\mathbb{C}\setminus\overline{\Omega}$, where
\begin{equation}
\Omega:=\underset{ k=1,...,4}{\cup}\left(\Omega_k \cup \Omega_k^* \right).
\end{equation}
In what follows, we give a detailed estimate of the integral over $\Omega_1$
%we turn
%	to consider the matrix-valued function $W^{(3)}(z)$ on $\Omega$. The details are given to control the integral term of \eqref{inequ:fC_z} for $z\in\Omega_1$.

By \eqref{equ: W^(3)expression} and \eqref{DBARR1}, it follows that
	\begin{equation}\label{eq:boundOmega1}
	\frac{1}{\pi}\iint_{\Omega_1}\dfrac{|W^{(3)} (\zeta )|}{|z-\zeta|}\dif \mu(\zeta)\lesssim \frac{1}{\pi}\iint_{\Omega_{1}}\dfrac{|\bar{\partial}d_1(\zeta ) e^{2i\theta}|}{|z-\zeta|}\dif \mu(\zeta).
	\end{equation}
We next divide $\Omega_{1}$ into two regions: $ \{z\in\Omega_{1}: \re z(1-|z|^{-2})\leqslant 2 \}$ and $ \{z\in\Omega_{1}: \re z(1-|z|^{-2})\geqslant 2 \}$.
These two regions belong to
\begin{equation}\label{def:OmegaAB}
\Omega_A=\{z\in\Omega_{1}:\re z\leqslant3 \},\qquad \Omega_B=\{z\in\Omega_{1}:\re z\geqslant 2 \},
\end{equation}
respectively. By setting
$$\zeta=u+k_1+vi, \quad z=x+yi, \qquad u, v, x, y \in\mathbb{R},$$
it is readily seen from \eqref{dbarRjest1} and \eqref{dbarRjest2} that
 \begin{align}\label{eq:sumIi}
 	\iint_{\Omega_1}\dfrac{|\bar{\partial}d_1 (\zeta )|e^{2\im \theta}}{|z-\zeta|}\dif \mu(\zeta)\lesssim I_1+I_2+I_3,	
 \end{align}
where
\begin{align*}
	&I_1:=\iint_{\Omega_A}\dfrac{e^{-tu^2v}}{|z-\zeta|}\dif u\dif v,\hspace{0.5cm}
	I_2:=\iint_{\Omega_B}\dfrac{e^{-tv}\sin\left(\frac{\pi}{2\varphi_0}\arg\left(\zeta-k_1\right)\mathcal{X}(\arg\left(\zeta-k_1\right)) \right)}{|z-\zeta|}\dif u\dif v,\\
	&I_3:=\iint_{\Omega_B}\dfrac{|u|^{-1/2}e^{-tv}}{|z-\zeta|}\dif u\dif v.
\end{align*}
Our task now is to estimate the integrals $I_i$, $i=1,2,3$, respectively.
%The estimation of $I_3$ is similar to \cite[Lemma 11]{Yang2022adv} with $I_3\lesssim t^{-1/2}$.

Using H\"older's inequality and the following basic inequalities
\begin{align*}
	\||z-\zeta|^{-1}\|_{L^q_u(v,\infty)}\lesssim |v-y|^{-1+1/q},\qquad
	\|e^{-tvu^2}\|_{L^p_u(v,\infty)}\lesssim (tv)^{-1/2p},
\end{align*}
where $1/p+1/q=1$, it follows that
\begin{align}\label{eq:estI1}
	I_1&\leqslant\int_{0}^{(3-k_1)\tan\varphi_0}\int_v^3\dfrac{e^{-tvu^2}}{|z-\zeta|}\dif u\dif v
\nonumber
\\
&\lesssim t^{-1/4} \int_{0}^{(3-k_1)\tan\varphi_0}|v-y|^{-1/2}v^{-1/4}e^{-tv^3}\dif v\lesssim t^{-1/3}.
\end{align}
To estimate $I_2$, we obtain from the definition of $\mathcal{X}$ in \eqref{equ:defX} that
\begin{align*}
	I_2\leqslant \int_{2}^{+\infty}\int_{u\tan(\varphi_0/3)}^{u\tan\varphi_0}\frac{e^{-tv}}{|z-\zeta|}\dif v\dif u.
\end{align*}
Since
\begin{align*}
	\left( \int_{u\tan(\varphi_0/3)}^{u\tan\varphi_0}e^{-ptv}\dif v\right) ^{1/p}=(pt)^{-1/p}e^{-tu\tan(\varphi_0/3)}\left( 1-e^{-ptu(\tan\varphi_0-\tan(\varphi_0/3))}\right) ^{1/p},
\end{align*}
it follows from H\"older's inequality again that
\begin{align}\label{eq:I2est}
	I_2&\lesssim\int_{2}^{+\infty}t^{-1/p}e^{-tu\tan(\varphi_0/3)}|u-x|^{-1+1/q}\dif u\lesssim t^{-1}.
\end{align}
As for $I_3$, using arguments similar to \cite[Lemma 11]{Yang2022adv}, one has
$$ I_3\lesssim t^{-1/2}.$$
This, together with \eqref{eq:boundOmega1} and \eqref{eq:sumIi}--\eqref{eq:I2est}, implies that
\begin{equation}
	\frac{1}{\pi}\iint_{\Omega_1}\dfrac{|W^{(3)} (\zeta )|}{|z-\zeta|}\dif \mu(\zeta)\lesssim t^{-1/3}.
	\end{equation}
The integral over other regions can be estimated in similar manners, which finally leads to \eqref{eq:Czest}.
%Summarizing the estimations $I_j$, $j=1,2,3$, we complete the proof.
\end{proof}

An immediate consequence of Proposition \ref{prop:Cz} is that the pure $\bar{\partial}$-problem \ref{DbarM5} for $M^{(5)}$ admits a unique solution for large positive $t$. For later use, we need the local behaviors of $M^{(5)}(z)$ at $z=i$ and $z=0$ as $t \to +\infty$. By \eqref{equ:M^{(5)}}, it is readily seen that
%To recover the solution of mCH equation, we need to evaluate $M^{(5)}$ at $z=0$ and $z=i$. Take $z=0$ in, then
\begin{equation}
	M^{(5)}(z)=M^{(5)}(i)+M^{(5)}_1(z-i)+\mathcal{O}((z-i)^2), \qquad z\to i,
\end{equation}
where
\begin{align}
	M^{(5)}(i)&=I+\frac{1}{\pi}\iint_\mathbb{C}\dfrac{M^{(5)}(\zeta )W^{(3)} (\zeta )}{\zeta-i}\dif \mu(\zeta),
   \label{def:M5i} \\
	M^{(5)}_1&=\frac{1}{\pi}\iint_{\mathbb{C}}\dfrac{M^{(5)}(\zeta )W^{(3)} (\zeta )}{(\zeta-i)^2}\dif \mu(\zeta),
\end{align}
and
\begin{equation}
	M^{(5)}(0)=I+\frac{1}{\pi}\iint_\mathbb{C}\dfrac{M^{(5)}(\zeta )W^{(3)} (\zeta)}{\zeta}\dif \mu(\zeta). \label{def:M50}
\end{equation}
%The following proposition exhibits the asymptotics for $M^{(5)}(0)$, $M^{(5)}(i)$ and $M^{(5)}_1$ as $t\to+\infty$.
\begin{Proposition}\label{prop:estdbar}
With $M^{(5)}(i),M^{(5)}_1$ and $M^{(5)}(0)$ defined in \eqref{def:M5i}--\eqref{def:M50}, we have, as $t\to +\infty$,
%There exists a large time $T>0$ such that when $t>T$, the solution $M^{(5)}(z)$  of  pure $\bar{\partial}$-problem exists unique and admits the following estimation:
	\begin{align}
		| M^{(5)}(i)-I|,\ | M^{(5)}(0)-I|,\ |M^{(5)}_1|\lesssim t^{-5/6}.\label{m3i}
	\end{align}
\end{Proposition}
\begin{proof}
%As $t\rightarrow\infty$, the existence and uniqueness of $M^{(5)}$ are obtained by the Proposition \ref{prop:Cz} and Eq.\eqref{equ:operator type of M^(5)} immediately.
As in the proof of Proposition \ref{prop:Cz}, we only present the proof for the integral over $\Omega_1$. With the regions $\Omega_A$ and $\Omega_B$ defined in \eqref{def:OmegaAB}, it follows that
%The proof of \eqref{m3i} is exhibited for $z\in\Omega_1$. Divide $\Omega_{1}$ into two regions:
	\begin{align*}
		\iint_{\Omega_{1}}\leqslant \iint_{\Omega_A}+\iint_{\Omega_B}.
	\end{align*}
%where $\Omega_A$ and $\Omega_B$ are the same as what given in Proposition \ref{prop:Cz}.
Noting that $|\zeta|^{-1}$, $|\zeta-i|^{-1}\leqslant1$ for $\zeta \in \Omega_1$ and the support of $W^{(3)}(\zeta)$ is bounded away from $\zeta=0$ and $\zeta=i$, by Proposition \ref{prop:Cz}, it's then sufficient to consider the integral
%The first step is to analyze the integral for $z\in\Omega_{A}$. When $\zeta\in \Omega_{1}$, $|\zeta|^{-1}$, $|\zeta-i|^{-1}\leqslant1$. In fact, so we can also ignore $|\zeta|^{-1}$, $|\zeta-i|^{-1}$ in other regions. Thus together with Proposition \ref{prop:Cz}, it's sufficient
%to consider the integral:
\begin{equation*}
\iint_{\Omega_{A}}M^{(5)}(\zeta )W^{(3)} (\zeta )\dif \mu(\zeta)\lesssim\iint_{\Omega_{A}}|\bar{\partial}d_1(\zeta ) e^{2i\theta}|\dif \mu(\zeta).
\end{equation*}
Let $\zeta=u+k_1+vi=k_1+le^{\varphi i}$ with $u,v,l,\varphi \in\mathbb{R}$, we observe from Lemma \ref{lem:estofimtheta} and \eqref{dbarRjest} that
\begin{multline*}
\iint_{\Omega_{A}}|\bar{\partial}d_1(\zeta ) e^{2i\theta}|\dif \mu(\zeta)\leqslant\int_{0}^{(3-k_1)\tan\varphi_0}\int_{v}^{3}(u^2+v^2)^{1/4}e^{-tu^2v}\dif u\dif v
\\
=t^{-5/6}\int_{0}^{\infty}\int_{v}^{\infty}(u^2+v^2)^{1/4}e^{-u^2v}\dif u\dif v.
\end{multline*}
Using the polar coordinates $u=l\cos\varphi$, $v=l\sin\varphi$, it is readily seen that
\begin{align*}
	&\int_{0}^{\infty}\int_{v}^{\infty}(u^2+v^2)^{1/4}e^{-u^2v}\dif u\dif v=\int_{0}^{\pi/4}\int_{0}^{+\infty}l^{3/2}e^{l^3\cos^2\varphi\sin\varphi}\dif l\dif\varphi\\
	&=\int_{0}^{\pi/4}\cos^{-5/3}\varphi\sin^{-5/6}\varphi \int_{0}^{+\infty} \frac{1}{3}e^{-w}w^{-2/3}\dif w\dif\varphi\leqslant\Gamma(5/6)B(1,1/12).
\end{align*}
where  $w=l^3\cos^2\varphi\sin\varphi$ is used in the second equality, $\Gamma(\cdot)$ and $B(\cdot,\cdot)$ are the Gamma function and Beta function, respectively. It is thereby inferred that
\begin{align*}
	&\iint_{\Omega_{A}}|\bar{\partial}d_1(\zeta ) e^{2i\theta}|\dif \mu(\zeta)\lesssim t^{-5/6}.
\end{align*}
To estimate the integral over $\zeta \in \Omega_B$, we also use
\eqref{dbarRjest2} and the facts that $|\zeta-i|^{-1},\ |\zeta-i|^{-2}\lesssim |\zeta|^{-1}$. Thus,
\begin{align*}
&\Big|\frac{1}{\pi}\iint_{\Omega_B}\dfrac{M^{(5)}(\zeta )W^{(3)} (\zeta )}{(\zeta-i)^2}\dif \mu(\zeta)\Big|,\  \Big|\frac{1}{\pi}\iint_{\Omega_B}\dfrac{M^{(5)}(\zeta )W^{(3)} (\zeta )}{\zeta-i}\dif \mu(\zeta)\Big|\\
&\lesssim\Big|\frac{1}{\pi}\iint_{\Omega_B}\dfrac{M^{(5)}(\zeta )W^{(3)} (\zeta )}{\zeta}\dif \mu(\zeta)\Big|\lesssim\Big|\frac{1}{\pi}\iint_{\Omega_B}\dfrac{\bar{\partial}d_1(\zeta ) e^{2i\theta}}{\zeta}\dif \mu(\zeta)\Big|.
\end{align*}
Since
\begin{align*}
	&\Big|\frac{1}{\pi}\iint_{\Omega_B}\dfrac{\bar{\partial}d_1(\zeta ) e^{2i\theta}}{\zeta}\dif \mu(\zeta)\Big|\leqslant\int_{2}^{+\infty}\int_{0}^{u}e^{-tv}(u^2+v^2)^{-1/2}\dif v\dif u\\
	&\leqslant t^{-1}\int_{2}^{+\infty} u^{-1}(1-e^{-tu})\dif u\lesssim t^{-1},
\end{align*}
we arrive at the desired result by combining the above three estimates.
\end{proof}

We are now ready to prove part (a) of Theorem \ref{mainthm}.

\subsection{Proof of part (a) of Theorem \ref{mainthm}}\label{subsec:recovering1sttran}
By tracing back the transformations \eqref{transform:M1toM2}, \eqref{transtoM4}, \eqref{def:E(z)1sttranregion} and \eqref{transform:shengchengdbarM5}, we conclude that, as $t\to+\infty$,
\begin{align}\label{ope}
	M^{(1)}(z)=&M^{(5)}(z)E(z)R^{(3)}(z)^{-1}T(z)^{-\sigma_3}G(z)^{-1}+\textrm{exponentially small error in $t$},
\end{align}
where $E$, $R^{(3)}$, $T$ and $G$ are defined in \eqref{def:E(z)1sttranregion}, \eqref{defofR^{(3)}}, \eqref{Tfunc} and
\eqref{funcG}, respectively.

In view of the facts that $T(0)\equiv1$, $G(0)=R^{(3)}(0)=I$, it follows from \eqref{m3i} that
\begin{align*}
	M^{(1)}(0)=E(0)+\mathcal{O}(t^{-5/6}).
\end{align*}
Since $G(z)=R^{(3)}(z)=I$ on a neighborhood of $i$, and $T(z)=T(i)+\mathcal{O}((z-i)^2)$, $z\to i$, it is readily seen from \eqref{eq:Eneari} that
\begin{align}
	M^{(1)}(z)=E(i)T(i)^{-\sigma_3}+E_1T(i)^{-\sigma_3}(z-i)+\mathcal{O}((z-i)^2)+\mathcal{O}(t^{-5/6}),
\end{align}
where $E(i)$ and $E_1$ are given in \eqref{def:Ei} and \eqref{def:E1}, respectively. Here, the error term $\mathcal{O}(t^{-5/6})$ comes from the pure $\bar{\partial}$-problem. From the reconstruction formula stated in Proposition \ref{prop: M^{(1)}},
we then obtain from Proposition \ref{asyE} that
\begin{align}	
  u(y,t)=1-t^{-2/3}2\left( 2/9\right)^{2/3}\left(  \im (N^{(r)}_1)_{11}\im (N^{(r)}_1)_{21}-\re (N_2^{(r)})_{12}\right) +\mathcal{O}(t^{-\min\{1-4\delta_1,1/3+9\delta_1\}}).
\end{align}
Together with \eqref{equ:Nr1specific}, \eqref{equ:Nr2specific}, \eqref{equ:Nl1specific}, \eqref{equ:Nl2specific} as well as Proposition \ref{asyE}, it is accomplished that
\begin{subequations}
	\begin{align}
u(x(y,t),t)&=1-(2/81)^{-1/3}t^{-2/3}v'(s) +\mathcal{O}(t^{-\min\{1-4\delta_1,1/3+9\delta_1\}}),\label{equ:u(y,t)1sttran}
\\ x(y,t)&=y-2\log(T(i))-t^{-1/3}\left(36\right)^{-1/3}\left(v(\tilde{s})+\int_{\tilde{s}}^{+\infty}v^2(\zeta)\dif\zeta\right)+\mathcal{O}(t^{-2/3}), 
\label{equ:1strelation(x-y)}
\end{align}
where
\begin{align}
&\tilde{s}=6^{-2/3}\left(\frac{y}{t}-2\right)t^{2/3},\quad \log(T(i))=\sum_{n=1}^{\mathcal{N}}\log\left( \frac{1+\im z_n}{1-\im z_n}\right).
\end{align}
\end{subequations}

If $\vert r(1)\vert\leqslant 1$, the Painlev\'e II transcendent $v$ in \eqref{equ:u(y,t)1sttran} and \eqref{equ:1strelation(x-y)} is bounded.
Taking into account the boundedness of $\log(T(i))$, it is thereby inferred that $x/t-y/t=\mathcal{O}(t^{-1})$. Therefore, $\tilde{s}-s=\mathcal{O}(t^{-1/3})$ follows from the definitions of $s$ (see \eqref{equ:s1stran}) and
$\tilde{s}$ (see \eqref{equ:tildes1stran}). Replacing $\tilde{s}$ by $s$ in \eqref{equ:u(y,t)1sttran}, taking the term $\int_{\tilde{s}}^{+\infty}v^2(\zeta)\dif\zeta$  as an example, we have
\begin{align*}
	&t^{-2/3}\int_{\tilde{s}}^{+\infty}v^2(\zeta)\dif\zeta-t^{-2/3}\int_{s}^{+\infty}v^2(\zeta)\dif\zeta\\
	&=t^{-2/3}\int_{\tilde{s}}^{s}v^2(\zeta)\dif\zeta\lesssim t^{-2/3}\left(s-\tilde{s}\right)\Vert v^2\Vert_{L^{\infty}}\lesssim t^{-1},
\end{align*}
which implies that the error term is kept. The asymptotic formula then admits the same form after replacing $\tilde{s}$ by $s$, as shown in \eqref{result:1stTran}.
\qed

%%%%%%%%%%%%%%%%%%%%%%%%%%%%%%%%%%%%%%%%%%%%%%%%%%%%%%%%%%%%%%%%%%%%%%%%%%%%%%%%%%%%%%%%%%%%%%%%%%%%%%%%%%%%%%
%%%%%%%%%%%%%%%%%%%%%%%%%%%%%%%%%%%%%%%%%%%%%%%%%%%%%%%%%%%%%%%%%%%%%%%%%%%%%%%%%%%%%%%%%%%%%%%%%%%%%%%%%%%%%%
\section{Asymptotic analysis of the RH problem for $M^{(3)}$ in $\mathcal{R}_{II}$}
\label{sec:2nd transition zone}
Due to similar reasons as in the previous section, the analysis is actually performed for $0\leqslant(\hat{\xi}+\frac{1}{4})t^{\frac{2}{3}}\leqslant C$. Moreover, we adopt the same notations (such as $R,M^{(4)},R^{(3)},M^{(5)},\ldots$) as those used before, they should be understood in different contexts, and we believe this will not lead to any confusion.

%Due to parametric form of $u$ shown in \eqref{recovering u}, the analysis is carried out for $0\leqslant|\hat\xi+\frac{1}{4}|t^{2/3}\leqslant C$.
%At the end of this section, we will exhibit that $\hat\xi$ is close to $\xi$ for large positive $t$.
%Moreover, as the previous section, it is also assumed that $0\leqslant(\hat{\xi}+\frac{1}{4})t^{\frac{2}{3}}\leqslant C$ throughout this section,
%since the discussions for the other half region is similar.

%Some notations, such as $R$, $M^{(4)}$, $R^{(3)}$ and $M^{(5)}$, that have already been introduced in the previous section
%will appear in this section; however, they represent different meanings in these two sections.
%Please interpret them according to the definitions provided in this section, and hopefully, this will not confuse the readers.

In this case, eight saddle points $k_j$, $j=1, \dots ,8$, of the phase function $\theta(z)$  such that $\theta'(k_j)=0$ are given by
\begin{align}
	k_1=-k_8=2\sqrt{s_+}+\sqrt{4s_{+}+1}, \qquad k_4=-k_5=-2\sqrt{s_+}+\sqrt{4s_{+}+1}, \label{equ:2ndtransaddlepoints-a} \\
	k_2=-k_7=2\sqrt{s_-}+\sqrt{4s_{-}+1}, \qquad k_3=-k_6=-2\sqrt{s_-}+\sqrt{4s_{-}+1}, \label{equ:2ndtransaddlepoints-b}
\end{align}
where
\begin{equation}\label{equ:s+s-2nd}
	s_{+}=\frac{1}{4\hat{\xi}}\left(-\hat{\xi}-1+\sqrt{1+4\hat{\xi}}\right),\qquad s_{-}=\frac{1}{4\hat{\xi}}\left(-\hat{\xi}-1-\sqrt{1+4\hat{\xi}}\right).
\end{equation}
It follows from a straightforward calculation that
\begin{equation*}
k_1=1/k_4=-1/k_5=-k_8, \qquad k_2=1/k_3=-1/k_6=-k_7,
\end{equation*}
and as $\hat{\xi}\rightarrow\left(-\frac{1}{4}\right)^{+}$,
\begin{equation*}
	k_{1,2}\rightarrow 2+\sqrt{3}, \quad k_{3,4}\rightarrow 2-\sqrt{3}, \quad k_{5,6}\rightarrow -2+\sqrt{3}, \quad k_{7,8}\rightarrow -2-\sqrt{3}.
\end{equation*}

\begin{figure}[h]
	\begin{center}
		\tikzset{every picture/.style={line width=0.75pt}} %set default line width to 0.75pt
		\begin{tikzpicture}[x=0.75pt,y=0.75pt,yscale=-1,xscale=1]
		%uncomment if require: \path (0,300); %set diagram left start at 0, and has height of 300
		%Straight Lines [id:da8892230263229475]
		\draw    (0,233.85) -- (626.71,231.12) ;
		\draw [shift={(628.71,231.11)}, rotate = 179.75] [color={rgb, 255:red, 0; green, 0; blue, 0 }  ][line width=0.75]    (10.93,-4.9) .. controls (6.95,-2.3) and (3.31,-0.67) .. (0,0) .. controls (3.31,0.67) and (6.95,2.3) .. (10.93,4.9)   ;
		%Shape: Ellipse [id:dp37885795448271]
		\draw   (283.22,232.48) .. controls (283.22,220.87) and (297.16,211.45) .. (314.36,211.45) .. controls (331.55,211.45) and (345.49,220.87) .. (345.49,232.48) .. controls (345.49,244.1) and (331.55,253.51) .. (314.36,253.51) .. controls (297.16,253.51) and (283.22,244.1) .. (283.22,232.48) -- cycle ;
		%Shape: Ellipse [id:dp7176717810902431]
		\draw   (338.05,152.98) .. controls (381.04,152.86) and (415.98,188.32) .. (416.08,232.2) .. controls (416.19,276.07) and (381.43,311.74) .. (338.44,311.86) .. controls (295.45,311.98) and (260.52,276.51) .. (260.41,232.63) .. controls (260.3,188.76) and (295.07,153.1) .. (338.05,152.98) -- cycle ;
		%Shape: Ellipse [id:dp5104383665466685]
		\draw   (285.99,153.03) .. controls (330.71,152.91) and (367.05,188.37) .. (367.15,232.25) .. controls (367.26,276.12) and (331.1,311.79) .. (286.38,311.92) .. controls (241.66,312.04) and (205.33,276.57) .. (205.22,232.7) .. controls (205.11,188.83) and (241.27,153.16) .. (285.99,153.03) -- cycle ;
		%Shape: Ellipse [id:dp3967453926323392]
		\draw   (173.36,232.48) .. controls (173.36,149.64) and (236.48,82.49) .. (314.36,82.49) .. controls (392.23,82.49) and (455.36,149.64) .. (455.36,232.48) .. controls (455.36,315.32) and (392.23,382.47) .. (314.36,382.47) .. controls (236.48,382.47) and (173.36,315.32) .. (173.36,232.48) -- cycle ;
		%Straight Lines [id:da5470693803759716]
		\draw    (312.22,24.05) -- (312.86,432.05) ;
		\draw [shift={(312.19,22.05)}, rotate = 89.17] [color={rgb, 255:red, 0; green, 0; blue, 0 }  ][line width=0.75]    (10.93,-4.9) .. controls (6.95,-2.3) and (3.31,-0.67) .. (0,0) .. controls (3.31,0.67) and (6.95,2.3) .. (10.93,4.9)   ;
		%Shape: Circle [id:dp6065087727256597]
		\draw  [dash pattern={on 0.84pt off 2.51pt}] (239.86,232.48) .. controls (239.86,191.34) and (273.21,157.98) .. (314.36,157.98) .. controls (355.5,157.98) and (388.86,191.34) .. (388.86,232.48) .. controls (388.86,273.63) and (355.5,306.98) .. (314.36,306.98) .. controls (273.21,306.98) and (239.86,273.63) .. (239.86,232.48) -- cycle ;
		% Text Node
		\draw (457.36,235.88) node [anchor=north west][inner sep=0.75pt]  [font=\tiny]  {$k_{1}$};
		% Text Node
		\draw (417.1,234.95) node [anchor=north west][inner sep=0.75pt]  [font=\tiny]  {$k_{2}$};
		% Text Node
		\draw (344.41,239.82) node [anchor=north west][inner sep=0.75pt]  [font=\tiny,xslant=0.02]  {$k_{4}$};
		% Text Node
		\draw (619.71,240.4) node [anchor=north west][inner sep=0.75pt]  [font=\scriptsize]  {$\re z$};
		% Text Node
		\draw (323.71,2.4) node [anchor=north west][inner sep=0.75pt]  [font=\scriptsize]  {$\im z$};
		% Text Node
		\draw (495.71,195.4) node [anchor=north west][inner sep=0.75pt]  [color={rgb, 255:red, 208; green, 2; blue, 27 }  ,opacity=1 ]  {$+$};
		% Text Node
		\draw (381.71,193.4) node [anchor=north west][inner sep=0.75pt]  [color={rgb, 255:red, 208; green, 2; blue, 27 }  ,opacity=1 ]  {$+$};
		% Text Node
		\draw (316.36,214.85) node [anchor=north west][inner sep=0.75pt]  [color={rgb, 255:red, 208; green, 2; blue, 27 }  ,opacity=1 ]  {$+$};
		% Text Node
		\draw (294.71,215.4) node [anchor=north west][inner sep=0.75pt]  [color={rgb, 255:red, 208; green, 2; blue, 27 }  ,opacity=1 ]  {$+$};
		% Text Node
		\draw (227.71,203.4) node [anchor=north west][inner sep=0.75pt]  [color={rgb, 255:red, 208; green, 2; blue, 27 }  ,opacity=1 ]  {$+$};
		% Text Node
		\draw (128.71,204.4) node [anchor=north west][inner sep=0.75pt]  [color={rgb, 255:red, 208; green, 2; blue, 27 }  ,opacity=1 ]  {$+$};
		% Text Node
		\draw (433.71,259.4) node [anchor=north west][inner sep=0.75pt]  [color={rgb, 255:red, 208; green, 2; blue, 27 }  ,opacity=1 ]  {$+$};
		% Text Node
		\draw (332.49,268.88) node [anchor=north west][inner sep=0.75pt]  [color={rgb, 255:red, 208; green, 2; blue, 27 }  ,opacity=1 ]  {$+$};
		% Text Node
		\draw (185.71,264.4) node [anchor=north west][inner sep=0.75pt]  [color={rgb, 255:red, 208; green, 2; blue, 27 }  ,opacity=1 ]  {$+$};
		% Text Node
		\draw (500.71,259.4) node [anchor=north west][inner sep=0.75pt]  [color={rgb, 255:red, 208; green, 2; blue, 27 }  ,opacity=1 ]  {$-$};
		% Text Node
		\draw (432.71,193.4) node [anchor=north west][inner sep=0.75pt]  [color={rgb, 255:red, 208; green, 2; blue, 27 }  ,opacity=1 ]  {$-$};
		% Text Node
		\draw (386.71,261.4) node [anchor=north west][inner sep=0.75pt]  [color={rgb, 255:red, 208; green, 2; blue, 27 }  ,opacity=1 ]  {$-$};
		% Text Node
		\draw (334.71,191.4) node [anchor=north west][inner sep=0.75pt]  [color={rgb, 255:red, 208; green, 2; blue, 27 }  ,opacity=1 ]  {$-$};
		% Text Node
		\draw (280.71,192.4) node [anchor=north west][inner sep=0.75pt]  [color={rgb, 255:red, 208; green, 2; blue, 27 }  ,opacity=1 ]  {$-$};
		% Text Node
		\draw (317.52,235.45) node [anchor=north west][inner sep=0.75pt]  [color={rgb, 255:red, 208; green, 2; blue, 27 }  ,opacity=1 ]  {$-$};
		% Text Node
		\draw (295.71,234.4) node [anchor=north west][inner sep=0.75pt]  [color={rgb, 255:red, 208; green, 2; blue, 27 }  ,opacity=1 ]  {$-$};
		% Text Node
		\draw (285.41,269.03) node [anchor=north west][inner sep=0.75pt]  [color={rgb, 255:red, 208; green, 2; blue, 27 }  ,opacity=1 ]  {$+$};
		% Text Node
		\draw (231.71,258.4) node [anchor=north west][inner sep=0.75pt]  [color={rgb, 255:red, 208; green, 2; blue, 27 }  ,opacity=1 ]  {$-$};
		% Text Node
		\draw (182.71,203.4) node [anchor=north west][inner sep=0.75pt]  [color={rgb, 255:red, 208; green, 2; blue, 27 }  ,opacity=1 ]  {$-$};
		% Text Node
		\draw (129.71,261.4) node [anchor=north west][inner sep=0.75pt]  [color={rgb, 255:red, 208; green, 2; blue, 27 }  ,opacity=1 ]  {$-$};
		% Text Node
		\draw (369.15,235.65) node [anchor=north west][inner sep=0.75pt]  [font=\tiny]  {$k_{3}$};
		% Text Node
		\draw (272.36,238.88) node [anchor=north west][inner sep=0.75pt]  [font=\tiny]  {$k_{5}$};
		% Text Node
		\draw (246.36,235.88) node [anchor=north west][inner sep=0.75pt]  [font=\tiny]  {$k_{6}$};
		% Text Node
		\draw (190.36,234.88) node [anchor=north west][inner sep=0.75pt]  [font=\tiny]  {$k_{7}$};
		% Text Node
		\draw (158.36,235.88) node [anchor=north west][inner sep=0.75pt]  [font=\tiny]  {$k_{8}$};
		\end{tikzpicture}		
	\caption{ Signature table of $\im \theta$ for $0\leqslant(\hat{\xi}+\frac{1}{4})t^{\frac{2}{3}}\leqslant C$. The ``$+$'' represents where $\im \theta(z)>0$ and ``$-$'' represents where $\im \theta(z)<0$.
	The dashed line is the unit circle.
	}\label{2ndTranRegionSign}
	\end{center}
	\end{figure}
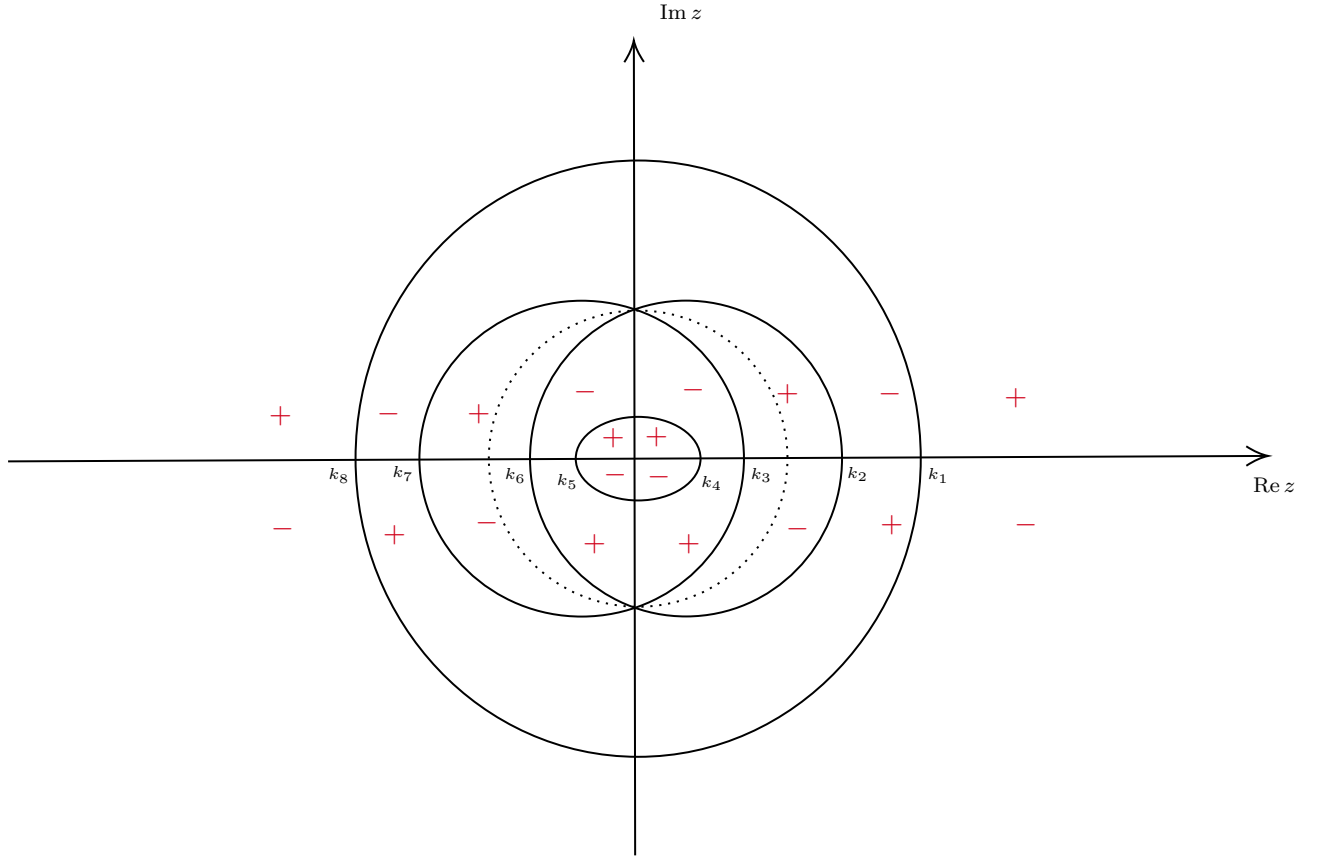

By \eqref{jump: V^{(3)}}, the jump matrix for $M^{(3)}$ in this case takes the following form
\begin{align}\label{equ:jumpV3-2ndTran}
	V^{(3)}(z)
	%&=\left(\begin{array}{cc}	1 & 0\\		-\frac{e^{-2i\theta}\bar{r}(z)T_-^{2}(z)}{1-|r(z)|^2} & 1	\end{array}\right)\left(\begin{array}{cc}		1 & \frac{e^{2i\theta}r(z)T_+^{-2}(z)}{1-|r(z)|^2}\\		0 & 1	\end{array}\right),  \nonumber\\
	&=\left(\begin{array}{cc}
		1 & 0\\
		-\overline{R(\bar{z})}e^{-2i\theta(z)} & 1
	\end{array}\right)\left(\begin{array}{cc}
		1 & R(z)e^{2i\theta(z)}\\
		0 & 1
	\end{array}\right),   \qquad  z\in \mathbb{R},
\end{align}
where
\begin{align}\label{def:Rexpressintermsof}
R(z)&:=R(z;\hat{\xi})=\frac{r(z)}{1-\vert r(z)\vert^2}T^{-2}_{+}(z)
\nonumber
\\
& \overset{\eqref{Tfunc}}{=}r(z)\prod_{j=1}^{2\mathcal{N}}\left(\frac{z-z_j}{z-\bar{z}_j}\right)^2\exp\left\{\frac{1}{\pi i}\int_{-\infty}^{+\infty}\frac{\log\left(1-\vert r(\zeta ) \vert^2\right)}{\zeta-z}\dif\zeta\right\},
\end{align}
and where the Sokhotski-Plemelj formulae is also applied for the third equality. Moreover, it is readily verified that
\begin{equation}\label{symofRonCC}
	R(z)=-\overline{R(-\bar{z})}, \qquad z\in\mathbb{C}.
\end{equation}
This, together with the signature table of $\im \theta$ illustrated in Figure \ref{2ndTranRegionSign}, implies opening lenses around the intervals
$(-\infty, k_8)\cup(k_7, k_6)\cup(k_5, k_4)\cup(k_3, k_2)\cup(k_1, +\infty)$ in what follows.

\subsection{Opening $\bar{\partial}$ lenses}\label{subsec:openlensTran2}
Since the function $R$ in \eqref{equ:jumpV3-2ndTran} is not an analytical function, as Section \ref{subsec:openlensTran1}, the idea now is to introduce the functions $d_j(z):=d_j(z,\hat{\xi})$, $j=1,\cdots, 8$, with boundary conditions:
\begin{equation}\label{equ: 2ndtran defofd_j}
	d_{j}(z,\hat{\xi})=\left\{\begin{array}{ll}
			-R(z), & z\in \mathbb{R},\\
			-R(k_j),  &z\in \Sigma_{j}, \ \ j=1,\dots,8,
		\end{array}\right.
\end{equation}
where each $\Sigma_j$ is the boundary of $\Omega_j$ in the upper half plane. We do not give precise definitions here, but refer to Figure \ref{FigjumpofM4} for an illustration.

One can give an explicit construction of each $d_j$ with the aid of $\mathcal{X}$ given in \eqref{equ:defX}. Indeed, it is readily to check that
\begin{equation}\label{equ: 2ndtran cons d_j}
	d_j(z)=-\left(R(\re z)-R(k_j)\right)\cos\left(\frac{\pi\arg(z-k_j)\mathcal{X}\left(\arg\left(z-k_j\right)\right)}{2\varphi_0}\right)-R(k_j),
%\quad j=1,\dots,8,
\end{equation}
where $\varphi_0$ is the angel between $\Sigma_1$ and the real axis. $d_j(z)$ admits similar estimates in Proposition \ref{est:RandDbarR} as well as the symmetry relation
\begin{equation}\label{symofD_j}
	d_{j}(z)=-\overline{d_j(-\bar{z})},
\end{equation}
which follows from \eqref{symofRonCC}.

\begin{figure}[htbp]
    \begin{center}
		\tikzset{every picture/.style={line width=0.75pt}} %set default line width to 0.75pt
		\begin{tikzpicture}[x=0.75pt,y=0.75pt,yscale=-1,xscale=1]
		\draw (527,145.72) -- (591,145.63) ;	\fill[yellow!30](591,145.63) --(639.95,116.56) --(639.95,145.63) ;
		\fill[blue!10](591,145.63) --(639.95,172.56) --(639.95,145.63) ;
		\draw   (591,145.63) -- (639.95,116.56) ;
		\draw    (591,145.63) -- (643.03,172.56) ;
		\fill[yellow!30](426,147.86) -- (476,111)-- (527,145.72);
		\fill[blue!10](426,147.8) -- (475,187)--  (526,145.8);
		\draw    (476,111) -- (527,145.72) ;
		\draw    (475,187) -- (526,145.72) ;
		\draw    (375,147.93) -- (426,147.86) ;
		\draw    (426,147.86) -- (476,111) ;
		\draw    (426,147.86) -- (475,187) ;
		\fill[yellow!30](275,150.07) -- (326,111)-- (375,147.93) ;
		\fill[blue!10](275,150.07) -- (326,184) -- (375,147.93) ;
		\draw    (326,111) -- (375,147.93) ;
		\draw    (326,184) -- (375,147.93) ;
		\draw    (211,150.16) -- (275,150.07) ;
		\draw    (275,150.07) -- (326,111) ;
		\draw    (275,150.07) -- (326,184) ;
		\fill[yellow!30](110,152.3) -- (159,112)-- (211,150.16) ;
		\fill[blue!10](112,153.3) -- (161,191)-- (211,150.16) ;
		\draw    (159,112) -- (211,150.16) ;
		\draw    (161,191) -- (211,150.16) ;
		\draw    (59,152.37) -- (110,152.3) ;
		\draw    (110,152.3) -- (159,112) ;
		\draw    (112,153.3) -- (161,191) ;
		\fill[yellow!30] (6.97,125.44) -- (59,152.37)--(6.97,152.37) ;
		\fill[blue!10](10.04,181.44) -- (59,152.37)--(10.04,152.37) ;
		\draw    (6.97,125.44) -- (59,152.37) ;
		\draw    (10.04,181.44) -- (59,152.37) ;
		\draw    (159,112) -- (160.5,151.23) ;
		\draw    (161,191) -- (160.5,151.23) ;
		\draw    (476,111) -- (474.97,145.62) ;
		\draw    (527,145.72) -- (591,145.63) ;
		%uncomment if require: \path (0,300); %set diagram left start at 0, and has height of 300
		%Straight Lines [id:da998056765192266]
		\draw    (527,145.72) -- (591,145.63) ;
		\draw [shift={(565,145.66)}, rotate = 179.92] [color={rgb, 255:red, 0; green, 0; blue, 0 }  ][line width=0.75]    (10.93,-3.29) .. controls (6.95,-1.4) and (3.31,-0.3) .. (0,0) .. controls (3.31,0.3) and (6.95,1.4) .. (10.93,3.29)   ;
		%Straight Lines [id:da9175293386275243]
		\draw    (591,145.63) -- (639.95,116.56) ;
		\draw [shift={(620.63,128.03)}, rotate = 149.3] [color={rgb, 255:red, 0; green, 0; blue, 0 }  ][line width=0.75]    (10.93,-3.29) .. controls (6.95,-1.4) and (3.31,-0.3) .. (0,0) .. controls (3.31,0.3) and (6.95,1.4) .. (10.93,3.29)   ;
		%Straight Lines [id:da7667001933695694]
		\draw    (591,145.63) -- (643.03,172.56) ;
		\draw [shift={(622.34,161.85)}, rotate = 207.36] [color={rgb, 255:red, 0; green, 0; blue, 0 }  ][line width=0.75]    (10.93,-3.29) .. controls (6.95,-1.4) and (3.31,-0.3) .. (0,0) .. controls (3.31,0.3) and (6.95,1.4) .. (10.93,3.29)   ;
		%Straight Lines [id:da9825748537266963]
		\draw    (476,111) -- (527,145.72) ;
		\draw [shift={(506.46,131.74)}, rotate = 214.25] [color={rgb, 255:red, 0; green, 0; blue, 0 }  ][line width=0.75]    (10.93,-3.29) .. controls (6.95,-1.4) and (3.31,-0.3) .. (0,0) .. controls (3.31,0.3) and (6.95,1.4) .. (10.93,3.29)   ;
		%Straight Lines [id:da3452756666567718]
		\draw    (475,187) -- (526,145.72) ;
		\draw [shift={(505.16,162.58)}, rotate = 141.01] [color={rgb, 255:red, 0; green, 0; blue, 0 }  ][line width=0.75]    (10.93,-3.29) .. controls (6.95,-1.4) and (3.31,-0.3) .. (0,0) .. controls (3.31,0.3) and (6.95,1.4) .. (10.93,3.29)   ;
		%Straight Lines [id:da7012830889727935]
		\draw    (375,147.93) -- (426,147.86) ;
		\draw [shift={(406.5,147.89)}, rotate = 179.92] [color={rgb, 255:red, 0; green, 0; blue, 0 }  ][line width=0.75]    (10.93,-3.29) .. controls (6.95,-1.4) and (3.31,-0.3) .. (0,0) .. controls (3.31,0.3) and (6.95,1.4) .. (10.93,3.29)   ;
		%Straight Lines [id:da9725305214723499]
		\draw    (426,147.86) -- (476,111) ;
		\draw [shift={(455.83,125.87)}, rotate = 143.6] [color={rgb, 255:red, 0; green, 0; blue, 0 }  ][line width=0.75]    (10.93,-3.29) .. controls (6.95,-1.4) and (3.31,-0.3) .. (0,0) .. controls (3.31,0.3) and (6.95,1.4) .. (10.93,3.29)   ;
		%Straight Lines [id:da8403203023764672]
		\draw    (426,147.86) -- (475,187) ;
		\draw [shift={(455.19,171.17)}, rotate = 218.62] [color={rgb, 255:red, 0; green, 0; blue, 0 }  ][line width=0.75]    (10.93,-3.29) .. controls (6.95,-1.4) and (3.31,-0.3) .. (0,0) .. controls (3.31,0.3) and (6.95,1.4) .. (10.93,3.29)   ;
		%Straight Lines [id:da5338817421663227]
		\draw    (326,111) -- (375,147.93) ;
		\draw [shift={(355.29,133.08)}, rotate = 217.01] [color={rgb, 255:red, 0; green, 0; blue, 0 }  ][line width=0.75]    (10.93,-3.29) .. controls (6.95,-1.4) and (3.31,-0.3) .. (0,0) .. controls (3.31,0.3) and (6.95,1.4) .. (10.93,3.29)   ;
		%Straight Lines [id:da8685914883088695]
		\draw    (326,184) -- (375,147.93) ;
		\draw [shift={(355.33,162.41)}, rotate = 143.64] [color={rgb, 255:red, 0; green, 0; blue, 0 }  ][line width=0.75]    (10.93,-3.29) .. controls (6.95,-1.4) and (3.31,-0.3) .. (0,0) .. controls (3.31,0.3) and (6.95,1.4) .. (10.93,3.29)   ;
		%Straight Lines [id:da4082503649068774]
		\draw    (211,150.16) -- (275,150.07) ;
		\draw [shift={(249,150.11)}, rotate = 179.92] [color={rgb, 255:red, 0; green, 0; blue, 0 }  ][line width=0.75]    (10.93,-3.29) .. controls (6.95,-1.4) and (3.31,-0.3) .. (0,0) .. controls (3.31,0.3) and (6.95,1.4) .. (10.93,3.29)   ;
		%Straight Lines [id:da7583739634433335]
		\draw    (275,150.07) -- (326,111) ;
		\draw [shift={(305.26,126.89)}, rotate = 142.54] [color={rgb, 255:red, 0; green, 0; blue, 0 }  ][line width=0.75]    (10.93,-3.29) .. controls (6.95,-1.4) and (3.31,-0.3) .. (0,0) .. controls (3.31,0.3) and (6.95,1.4) .. (10.93,3.29)   ;
		%Straight Lines [id:da7182483810332065]
		\draw    (275,150.07) -- (326,184) ;
		\draw [shift={(305.5,170.36)}, rotate = 213.64] [color={rgb, 255:red, 0; green, 0; blue, 0 }  ][line width=0.75]    (10.93,-3.29) .. controls (6.95,-1.4) and (3.31,-0.3) .. (0,0) .. controls (3.31,0.3) and (6.95,1.4) .. (10.93,3.29)   ;
		%Straight Lines [id:da8697448574753852]
		\draw    (159,112) -- (211,150.16) ;
		\draw [shift={(189.84,134.63)}, rotate = 216.27] [color={rgb, 255:red, 0; green, 0; blue, 0 }  ][line width=0.75]    (10.93,-3.29) .. controls (6.95,-1.4) and (3.31,-0.3) .. (0,0) .. controls (3.31,0.3) and (6.95,1.4) .. (10.93,3.29)   ;
		%Straight Lines [id:da8086381733781185]
		\draw    (161,191) -- (211,150.16) ;
		\draw [shift={(190.65,166.78)}, rotate = 140.76] [color={rgb, 255:red, 0; green, 0; blue, 0 }  ][line width=0.75]    (10.93,-3.29) .. controls (6.95,-1.4) and (3.31,-0.3) .. (0,0) .. controls (3.31,0.3) and (6.95,1.4) .. (10.93,3.29)   ;
		%Straight Lines [id:da9412188387192317]
		\draw    (59,152.37) -- (110,152.3) ;
		\draw [shift={(90.5,152.33)}, rotate = 179.92] [color={rgb, 255:red, 0; green, 0; blue, 0 }  ][line width=0.75]    (10.93,-3.29) .. controls (6.95,-1.4) and (3.31,-0.3) .. (0,0) .. controls (3.31,0.3) and (6.95,1.4) .. (10.93,3.29)   ;
		%Straight Lines [id:da46675578157748876]
		\draw    (110,152.3) -- (159,112) ;
		\draw [shift={(139.14,128.34)}, rotate = 140.56] [color={rgb, 255:red, 0; green, 0; blue, 0 }  ][line width=0.75]    (10.93,-3.29) .. controls (6.95,-1.4) and (3.31,-0.3) .. (0,0) .. controls (3.31,0.3) and (6.95,1.4) .. (10.93,3.29)   ;
		%Straight Lines [id:da4915496084587081]
		\draw    (112,153.3) -- (161,191) ;
		\draw [shift={(141.26,175.81)}, rotate = 217.58] [color={rgb, 255:red, 0; green, 0; blue, 0 }  ][line width=0.75]    (10.93,-3.29) .. controls (6.95,-1.4) and (3.31,-0.3) .. (0,0) .. controls (3.31,0.3) and (6.95,1.4) .. (10.93,3.29)   ;
		%Straight Lines [id:da7394897601994419]
		\draw    (6.97,125.44) -- (59,152.37) ;
		\draw [shift={(38.31,141.67)}, rotate = 207.36] [color={rgb, 255:red, 0; green, 0; blue, 0 }  ][line width=0.75]    (10.93,-3.29) .. controls (6.95,-1.4) and (3.31,-0.3) .. (0,0) .. controls (3.31,0.3) and (6.95,1.4) .. (10.93,3.29)   ;
		%Straight Lines [id:da5306582782656792]
		\draw    (10.04,181.44) -- (59,152.37) ;
		\draw [shift={(39.68,163.84)}, rotate = 149.3] [color={rgb, 255:red, 0; green, 0; blue, 0 }  ][line width=0.75]    (10.93,-3.29) .. controls (6.95,-1.4) and (3.31,-0.3) .. (0,0) .. controls (3.31,0.3) and (6.95,1.4) .. (10.93,3.29)   ;
		%Straight Lines [id:da05362410175817289]
		\draw  [dash pattern={on 0.84pt off 2.51pt}]  (1,152.45) -- (59,152.37) ;
		%Straight Lines [id:da1816688030150173]
		\draw  [dash pattern={on 0.84pt off 2.51pt}]  (110,152.3) -- (211,150.16) ;
		%Straight Lines [id:da2803047628095465]
		\draw  [dash pattern={on 0.84pt off 2.51pt}]  (275,150.07) -- (375,147.93) ;
		%Straight Lines [id:da4451084111796726]
		\draw  [dash pattern={on 0.84pt off 2.51pt}]  (426,147.86) -- (526,145.72) ;
		%Straight Lines [id:da808376128463625]
		\draw  [dash pattern={on 0.84pt off 2.51pt}]  (591,145.63) -- (649,145.55) ;
		%Straight Lines [id:da6984614631067876]
		\draw  [dash pattern={on 0.84pt off 2.51pt}]  (326,49) -- (326,111) ;
		\draw [shift={(326,47)}, rotate = 90] [color={rgb, 255:red, 0; green, 0; blue, 0 }  ][line width=0.75]    (10.93,-3.29) .. controls (6.95,-1.4) and (3.31,-0.3) .. (0,0) .. controls (3.31,0.3) and (6.95,1.4) .. (10.93,3.29)   ;
		%Straight Lines [id:da8891666688543576]
		\draw    (159,112) -- (160.5,151.23) ;
		\draw [shift={(159.98,137.61)}, rotate = 267.81] [color={rgb, 255:red, 0; green, 0; blue, 0 }  ][line width=0.75]    (10.93,-3.29) .. controls (6.95,-1.4) and (3.31,-0.3) .. (0,0) .. controls (3.31,0.3) and (6.95,1.4) .. (10.93,3.29)   ;
		%Straight Lines [id:da5649109710668039]
		\draw    (161,191) -- (160.5,151.23) ;
		\draw [shift={(160.68,165.12)}, rotate = 89.28] [color={rgb, 255:red, 0; green, 0; blue, 0 }  ][line width=0.75]    (10.93,-3.29) .. controls (6.95,-1.4) and (3.31,-0.3) .. (0,0) .. controls (3.31,0.3) and (6.95,1.4) .. (10.93,3.29)   ;
		%Straight Lines [id:da5668168226078125]
		\draw    (476,111) -- (474.97,145.62) ;
		\draw [shift={(475.3,134.31)}, rotate = 271.71] [color={rgb, 255:red, 0; green, 0; blue, 0 }  ][line width=0.75]    (10.93,-3.29) .. controls (6.95,-1.4) and (3.31,-0.3) .. (0,0) .. controls (3.31,0.3) and (6.95,1.4) .. (10.93,3.29)   ;
		%Straight Lines [id:da5505308281175518]
		\draw    (475,187) -- (474.97,145.62) ;
		\draw [shift={(474.98,160.31)}, rotate = 89.95] [color={rgb, 255:red, 0; green, 0; blue, 0 }  ][line width=0.75]    (10.93,-3.29) .. controls (6.95,-1.4) and (3.31,-0.3) .. (0,0) .. controls (3.31,0.3) and (6.95,1.4) .. (10.93,3.29)   ;
		%Straight Lines [id:da4060204703009398]
		\draw  [dash pattern={on 0.84pt off 2.51pt}]  (326,184) -- (327,273.29) ;
		%Straight Lines [id:da5499605429445973]
		\draw    (326,111) -- (325,149) ;
		\draw [shift={(325.34,136)}, rotate = 271.51] [color={rgb, 255:red, 0; green, 0; blue, 0 }  ][line width=0.75]    (10.93,-3.29) .. controls (6.95,-1.4) and (3.31,-0.3) .. (0,0) .. controls (3.31,0.3) and (6.95,1.4) .. (10.93,3.29)   ;
		%Straight Lines [id:da23233029972989283]
		\draw    (326,184) -- (325,149) ;
		\draw [shift={(325.33,160.5)}, rotate = 88.36] [color={rgb, 255:red, 0; green, 0; blue, 0 }  ][line width=0.75]    (10.93,-3.29) .. controls (6.95,-1.4) and (3.31,-0.3) .. (0,0) .. controls (3.31,0.3) and (6.95,1.4) .. (10.93,3.29)   ;
		% Text Node
		\draw (329,35.4) node [anchor=north west][inner sep=0.75pt]  [font=\scriptsize]  {$\im z$};
		% Text Node
		\draw (634,123.4) node [anchor=north west][inner sep=0.75pt]  [font=\tiny,color={rgb, 255:red, 74; green, 144; blue, 226 }  ,opacity=1 ]  {$\Omega _{1}$};
		% Text Node
		\draw (477.48,131.71) node [anchor=north west][inner sep=0.75pt]  [font=\tiny,color={rgb, 255:red, 74; green, 144; blue, 226 }  ,opacity=1 ]  {$\Omega _{2}$};
		% Text Node
		\draw (453,132.83) node [anchor=north west][inner sep=0.75pt]  [font=\tiny,color={rgb, 255:red, 74; green, 144; blue, 226 }  ,opacity=1 ]  {$\Omega _{3}$};
		% Text Node
		\draw (327.5,133.4) node [anchor=north west][inner sep=0.75pt]  [font=\tiny,color={rgb, 255:red, 74; green, 144; blue, 226 }  ,opacity=1 ]  {$\Omega _{4}$};
		% Text Node
		\draw (302.5,133.93) node [anchor=north west][inner sep=0.75pt]  [font=\tiny,color={rgb, 255:red, 74; green, 144; blue, 226 }  ,opacity=1 ]  {$\Omega _{5}$};
		% Text Node
		\draw (161.75,135.01) node [anchor=north west][inner sep=0.75pt]  [font=\tiny,color={rgb, 255:red, 74; green, 144; blue, 226 }  ,opacity=1 ]  {$\Omega _{6}$};
		% Text Node
		\draw (136.5,135.55) node [anchor=north west][inner sep=0.75pt]  [font=\tiny,color={rgb, 255:red, 74; green, 144; blue, 226 }  ,opacity=1 ]  {$\Omega _{7}$};
		% Text Node
		\draw (3,131.4) node [anchor=north west][inner sep=0.75pt]  [font=\tiny,color={rgb, 255:red, 74; green, 144; blue, 226 }  ,opacity=1 ]  {$\Omega _{8}$};
		% Text Node
		\draw (633,147.4) node [anchor=north west][inner sep=0.75pt]  [font=\tiny,color={rgb, 255:red, 74; green, 144; blue, 226 }  ,opacity=1 ]  {$\Omega _{1}^{*}$};
		% Text Node
		\draw (476.97,149.02) node [anchor=north west][inner sep=0.75pt]  [font=\tiny,color={rgb, 255:red, 74; green, 144; blue, 226 }  ,opacity=1 ]  {$\Omega _{2}^{*}$};
		% Text Node
		\draw (454,147.4) node [anchor=north west][inner sep=0.75pt]  [font=\tiny,color={rgb, 255:red, 74; green, 144; blue, 226 }  ,opacity=1 ]  {$\Omega _{3}^{*}$};
		% Text Node
		\draw (329,147.4) node [anchor=north west][inner sep=0.75pt]  [font=\tiny,color={rgb, 255:red, 74; green, 144; blue, 226 }  ,opacity=1 ]  {$\Omega _{4}^{*}$};
		% Text Node
		\draw (304,148.4) node [anchor=north west][inner sep=0.75pt]  [font=\tiny,color={rgb, 255:red, 74; green, 144; blue, 226 }  ,opacity=1 ]  {$\Omega _{5}^{*}$};
		% Text Node
		\draw (162.5,154.63) node [anchor=north west][inner sep=0.75pt]  [font=\tiny,color={rgb, 255:red, 74; green, 144; blue, 226 }  ,opacity=1 ]  {$\Omega _{6}^{*}$};
		% Text Node
		\draw (138,153.4) node [anchor=north west][inner sep=0.75pt]  [font=\tiny,color={rgb, 255:red, 74; green, 144; blue, 226 }  ,opacity=1 ]  {$\Omega _{7}^{*}$};
		% Text Node
		\draw (3,155.85) node [anchor=north west][inner sep=0.75pt]  [font=\tiny,color={rgb, 255:red, 74; green, 144; blue, 226 }  ,opacity=1 ]  {$\Omega _{8}^{*}$};
		% Text Node
		\draw (616,105.4) node [anchor=north west][inner sep=0.75pt]  [font=\tiny,color={rgb, 255:red, 208; green, 2; blue, 27 }  ,opacity=1 ]  {$\Sigma _{1}$};
		% Text Node
		\draw (612.01,159.49) node [anchor=north west][inner sep=0.75pt]  [font=\tiny,color={rgb, 255:red, 208; green, 2; blue, 27 }  ,opacity=1 ]  {$\Sigma _{1}^{*}$};
		% Text Node
		\draw (496,107.4) node [anchor=north west][inner sep=0.75pt]  [font=\tiny,color={rgb, 255:red, 208; green, 2; blue, 27 }  ,opacity=1 ]  {$\Sigma _{2}$};
		% Text Node
		\draw (438,108.4) node [anchor=north west][inner sep=0.75pt]  [font=\tiny,color={rgb, 255:red, 208; green, 2; blue, 27 }  ,opacity=1 ]  {$\Sigma _{3}$};
		% Text Node
		\draw (500,160.49) node [anchor=north west][inner sep=0.75pt]  [font=\tiny,color={rgb, 255:red, 208; green, 2; blue, 27 }  ,opacity=1 ]  {$\Sigma _{2}^{*}$};
		% Text Node
		\draw (439,167.49) node [anchor=north west][inner sep=0.75pt]  [font=\tiny,color={rgb, 255:red, 208; green, 2; blue, 27 }  ,opacity=1 ]  {$\Sigma _{3}^{*}$};
		% Text Node
		\draw (346,108.4) node [anchor=north west][inner sep=0.75pt]  [font=\tiny,color={rgb, 255:red, 208; green, 2; blue, 27 }  ,opacity=1 ]  {$\Sigma _{4}$};
		% Text Node
		\draw (344.5,169.37) node [anchor=north west][inner sep=0.75pt]  [font=\tiny,color={rgb, 255:red, 208; green, 2; blue, 27 }  ,opacity=1 ]  {$\Sigma _{4}^{*}$};
		% Text Node
		\draw (294,107.4) node [anchor=north west][inner sep=0.75pt]  [font=\tiny,color={rgb, 255:red, 208; green, 2; blue, 27 }  ,opacity=1 ]  {$\Sigma _{5}$};
		% Text Node
		\draw (293,169.4) node [anchor=north west][inner sep=0.75pt]  [font=\tiny,color={rgb, 255:red, 208; green, 2; blue, 27 }  ,opacity=1 ]  {$\Sigma _{5}^{*}$};
		% Text Node
		\draw (183,109.4) node [anchor=north west][inner sep=0.75pt]  [font=\tiny,color={rgb, 255:red, 208; green, 2; blue, 27 }  ,opacity=1 ]  {$\Sigma _{6}$};
		% Text Node
		\draw (181,174.98) node [anchor=north west][inner sep=0.75pt]  [font=\tiny,color={rgb, 255:red, 208; green, 2; blue, 27 }  ,opacity=1 ]  {$\Sigma _{6}^{*}$};
		% Text Node
		\draw (126,109.4) node [anchor=north west][inner sep=0.75pt]  [font=\tiny,color={rgb, 255:red, 208; green, 2; blue, 27 }  ,opacity=1 ]  {$\Sigma _{7}$};
		% Text Node
		\draw (122,169.98) node [anchor=north west][inner sep=0.75pt]  [font=\tiny,color={rgb, 255:red, 208; green, 2; blue, 27 }  ,opacity=1 ]  {$\Sigma _{7}^{*}$};
		% Text Node
		\draw (26,116.4) node [anchor=north west][inner sep=0.75pt]  [font=\tiny,color={rgb, 255:red, 208; green, 2; blue, 27 }  ,opacity=1 ]  {$\Sigma _{8}$};
		% Text Node
		\draw (26.52,170.31) node [anchor=north west][inner sep=0.75pt]  [font=\tiny,color={rgb, 255:red, 208; green, 2; blue, 27 }  ,opacity=1 ]  {$\Sigma _{8}^{*}$};
		% Text Node
		\draw (585,148.4) node [anchor=north west][inner sep=0.75pt]  [font=\tiny]  {$k_{1}$};
		% Text Node
		\draw (522,149.12) node [anchor=north west][inner sep=0.75pt]  [font=\tiny]  {$k_{2}$};
		% Text Node
		\draw (419,150.4) node [anchor=north west][inner sep=0.75pt]  [font=\tiny]  {$k_{3}$};
		% Text Node
		\draw (370,150.33) node [anchor=north west][inner sep=0.75pt]  [font=\tiny]  {$k_{4}$};
		% Text Node
		\draw (268,152.33) node [anchor=north west][inner sep=0.75pt]  [font=\tiny]  {$k_{5}$};
		% Text Node
		\draw (206,152.4) node [anchor=north west][inner sep=0.75pt]  [font=\tiny]  {$k_{6}$};
		% Text Node
		\draw (103,154.12) node [anchor=north west][inner sep=0.75pt]  [font=\tiny]  {$k_{7}$};
		% Text Node
		\draw (53,154.4) node [anchor=north west][inner sep=0.75pt]  [font=\tiny]  {$k_{8}$};
		% Text Node
		\draw (472,185.49) node [anchor=north west][inner sep=0.75pt]  [font=\tiny,color={rgb, 255:red, 208; green, 2; blue, 27 }  ,opacity=1 ]  {$\Sigma _{2,3}^{*}$};
		% Text Node
		\draw (465,93.49) node [anchor=north west][inner sep=0.75pt]  [font=\tiny,color={rgb, 255:red, 208; green, 2; blue, 27 }  ,opacity=1 ]  {$\Sigma _{2,3}$};
		% Text Node
		\draw (318,95.49) node [anchor=north west][inner sep=0.75pt]  [font=\tiny,color={rgb, 255:red, 208; green, 2; blue, 27 }  ,opacity=1 ]  {$\Sigma _{4,5}$};
		% Text Node
		\draw (319,180.49) node [anchor=north west][inner sep=0.75pt]  [font=\tiny,color={rgb, 255:red, 208; green, 2; blue, 27 }  ,opacity=1 ]  {$\Sigma _{4,5}^{*}$};
		% Text Node
		\draw (147,93.49) node [anchor=north west][inner sep=0.75pt]  [font=\tiny,color={rgb, 255:red, 208; green, 2; blue, 27 }  ,opacity=1 ]  {$\Sigma _{6,7}$};
		% Text Node
		\draw (152,189.49) node [anchor=north west][inner sep=0.75pt]  [font=\tiny,color={rgb, 255:red, 208; green, 2; blue, 27 }  ,opacity=1 ]  {$\Sigma _{6,7}^{*}$};
		\end{tikzpicture}		
    \caption{ The jump contours of RH problem for $M^{(4)}$ when $0\leqslant(\hat{\xi}+1/4)t^{2/3}\leqslant C$. }\label{FigjumpofM4}
    \end{center}
\end{figure}
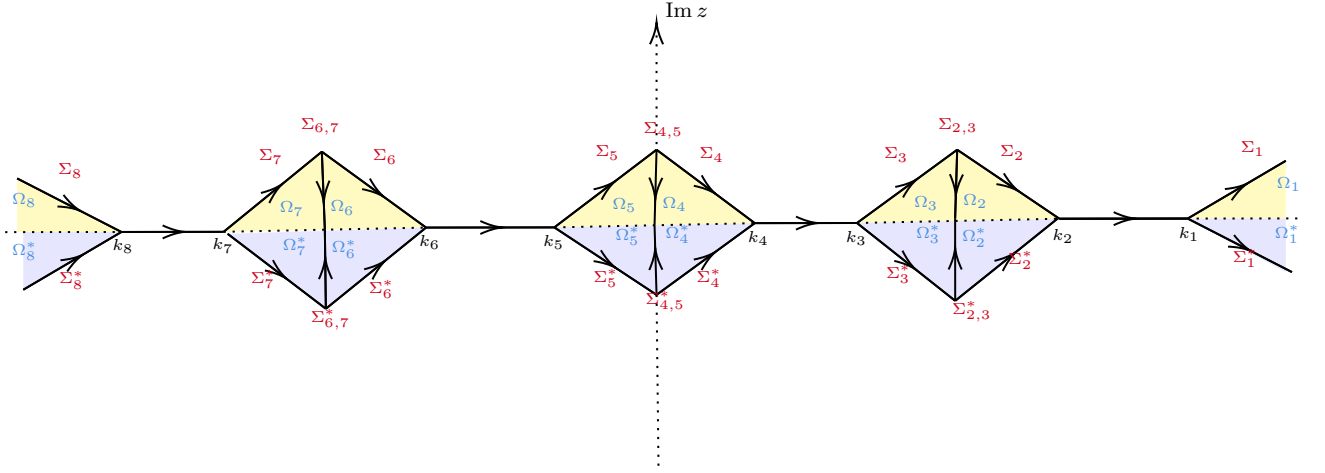

Similar to \eqref{transtoM4}, we now define
\begin{equation}\label{transform:M3toM4 2nd}
	M^{(4)}(z)=M^{(3)}(z)R^{(3)}(z),
\end{equation}
where
\begin{equation}\label{R(3)+2ndTranRegion}
	R^{(3)}(z):=R^{(3)}(z,\hat{\xi})=\left\{\begin{array}{lll}
		\left(\begin{array}{cc}
			1 & d_j(z)e^{2i\theta(z)}\\
			0 & 1
		\end{array}\right), & z\in \Omega_{j},\ \ j=1,\dots,8,\\
		[12pt]
		\left(\begin{array}{cc}
			1 & 0\\
			d_j^{*}(z)e^{-2i\theta(z)} & 1
		\end{array}\right),  &z\in \Omega_{j}^{*}, \ \ j=1,\dots,8,\\
		[12pt]
		I,  &\textrm{elsewhere.}\\
	\end{array}\right.
\end{equation}

Then, $M^{(4)}$ satisfies the following mixed $\bar{\partial}$-RH problem.
\begin{dbar-RHP}\label{RHPM^{(4)}}
  \hfill
\begin{itemize}
	\item[$\bullet$] $M^{(4)}(z)$ is continuous for $z\in\mathbb{C}\setminus\Sigma^{(4)}$, where
	\begin{equation}\label{def:Sigma4,2nd}
	\Sigma^{(4)}:=\bigcup_{j=1}^{8}(\Sigma_j\cup\Sigma_j^*)\cup\bigcup_{j=1,3,5,7}(k_{j+1},k_j)\cup\bigcup_{j=2,4,6}\left(\Sigma_{j, j+1}\cup\Sigma_{j, j+1}^*\right);
	\end{equation}
	see Figure \ref{FigjumpofM4} for an illustration.
	\item[$\bullet$] $M^{(4)}(z)=\sigma_1\overline{M^{(4)}(\bar{z})}\sigma_1=\sigma_2M^{(4)}(-z)\sigma_2$.
	\item[$\bullet$] For $z\in\Sigma^{(4)}$, we have
	\begin{equation}
		M^{(4)}_+(z)=M^{(4)}_-(z)V^{(4)}(z),
	\end{equation}
	where
	\begin{equation}\label{jumpofM4}
		V^{(4)}(z)=\left\{\begin{array}{ll}e^{i\theta(z)\hat{\sigma}_3}\left(\begin{array}{cc}
				1 & 0 \\
				-\overline{R(z)} & 1
			\end{array}\right)
			\left(\begin{array}{cc}
				1 & R(z)\\
				0 & 1
			\end{array}\right),   & z\in\underset{j=1,3,5,7}{\cup}(k_{j+1},k_j),\\[12pt]
			\begin{pmatrix}1 &  R(k_j)e^{2i\theta(z)}\\ 0 & 1 \end{pmatrix},  &z\in\Sigma_j,\ j=1,\dots,8,\\[8pt]
			\begin{pmatrix}1 & 0 \\ -\overline{R(k_j)}e^{-2i\theta(z)} & 1 \end{pmatrix},  &z\in\Sigma_j^*,\ j=1,\dots,8,\\[8pt]
			 \left(\begin{array}{cc}
			 	1 & \left(d_j(z)-d_{j+1}(z)\right)e^{2i\theta(z)}\\
			 	0 & 1
			 \end{array}\right),  & z\in \Sigma_{j,j+1}, \ j=2,4,6,\\
		 [12pt]
		 \left(\begin{array}{cc}
		 	1 & 0\\
		 	-(d^*_j(z)-d^*_{j+1}(z))e^{-2i\theta(z)} & 1
		 \end{array}\right),  &z\in \Sigma^*_{j,j+1}, \ j=2,4,6.\\
		\end{array}\right.
	\end{equation}
	\item[$\bullet$] As $z\rightarrow\infty$ in $\mathbb{C} \setminus \Sigma^{(4)}$,  we have $M^{(4)}(z) = I+\mathcal{O}(z^{-1})$.
	\item[$\bullet$] For $z\in\mathbb{C}$, we have the $\bar{\partial}$-derivative relation
	\begin{align}
		\bar{\partial}M^{(4)}=M^{(4)}\bar{\partial}R^{(3)},
	\end{align}
	where
	\begin{equation}
		\bar{\partial}R^{(3)}(z)=\left\{\begin{array}{lll}
			\left(\begin{array}{cc}
				0 & \bar{\partial}d_{j}(z)e^{2i\theta}\\
				0 & 0
			\end{array}\right), & z\in \Omega_{j},\ j=1,\dots,8,\\
			[12pt]
			\left(\begin{array}{cc}
				0 & 0\\
				\bar{\partial}d_j^*(z)e^{-2i\theta} & 0
			\end{array}\right),  &z\in \Omega_{j}^{*},\ j=1,\dots,8,\\
			[12pt]
			0,  & \mathrm{elsewhere}.\\
		\end{array}\right.
	\end{equation}
\end{itemize}
\end{dbar-RHP}
The above mixed $\bar{\partial}$-RH problem can again be decomposed into a pure RH problem under the condition $\bar\partial R^{(3)}\equiv0$ and
a pure $\bar{\partial}$-problem with $\bar\partial R^{(3)}\not=0$. The next two sections are then devoted to the asymptotic analysis of these two problems separately.

\subsection{Analysis of the pure RH problem}\label{subsec:pureN(z)2nd}
The pure RH problem is obtained from mixed $\bar{\partial}$-RH problem for $M^{(4)}$ via omitting the $\bar{\partial}$-derivative part, which reads as follows
\begin{RHP}\label{RHP: N(z) tran2}
	\hfill
	\begin{itemize}
		\item[$\bullet$]  $N(z):=N(z;\hat{\xi})$ is holomorphic for $z\in\mathbb{C}\setminus \Sigma^{(4)}$, where $\Sigma^{(4)}$ is defined in \eqref{def:Sigma4,2nd}.
		\item[$\bullet$]  $N(z)=\sigma_1\overline{N(\bar{z})}\sigma_1=\sigma_2N(-z)\sigma_2$.
		\item[$\bullet$]  For $z \in \Sigma^{(4)}$, we have
		\begin{equation}
			N_+(z)=N_-(z)V^{(4)}(z),
		\end{equation}
		where $V^{(4)}(z)$ is defined in \eqref{jumpofM4}.
		\item[$\bullet$]
		As $z\rightarrow\infty$ in $\mathbb{C}\setminus \Sigma^{(4)}$, we have $N(z)=I+\mathcal{O}(z^{-1})$.
	\end{itemize}
\end{RHP}
Note that the definitions of $d_4$ and $d_5$ imply that $V^{(4)}(z)=I$ near $z=0$ on $\Sigma_{4,5}$ and $\Sigma_{4,5}^*$, and $V^{(4)}(z) \to I$  as $t \to +\infty$ on the other parts of $\Sigma^{(4)}$ (see \eqref{jumpofM4} and Figure \ref{2ndTranRegionSign}),
it follows that $N$ is approximated, to the leading order, by the global parametrix $N^{(\infty)}(z)=I$. The sub-leading contribution stems from the local behaviors near the saddle points $k_j$, $j=1,\ldots,8$,
which is still well approximated by the Painlev\'{e} II parametrix.

\subsubsection*{Local parametrices near $z=2\pm\sqrt{3}$, $-2\pm\sqrt{3}$} \label{subsubsec:local para 2nd}
Let
\begin{align*}
	&U^{(a)}=\{z:|z-(2+\sqrt{3})|\leqslant c_0\}, \quad &&U^{(b)}=\{z:|z-(2-\sqrt{3})|\leqslant c_0\}, \\
	&U^{(c)}=\{z:|z-(-2+\sqrt{3})|\leqslant c_0\}, \quad  &&U^{(d)}=\{z:|z-(-2-\sqrt{3})|\leqslant c_0\},
\end{align*}
be four small disks around $z=2\pm\sqrt{3}$ and $z=-2\pm\sqrt{3}$ respectively, where
\begin{align}
	c_0:=\min\{1-\frac{\sqrt{3}}{2}, \ 2(k_1-2-\sqrt{3})t^{\delta_{2}}, \ 2(k_3-2+\sqrt{3})t^{\delta_{2}}\},\hspace{0.5cm}\delta_2\in(0,1/6),
\end{align}
is a constant depending on $t$. For $t$ large enough, we have $k_{1,2}\in U^{(a)}$, $k_{3,4}\in U^{(b)}$, $k_{5,6}\in U^{(c)}$ as well as $k_{7,8}\in U^{(d)}$.
Indeed, if $0\leqslant(\hat{\xi}+\frac{1}{4})t^{\frac{2}{3}}\leqslant C$, it follows from \eqref{equ:s+s-2nd} that $0<s_+\leqslant\frac{C}{2}t^{-2/3}$, $-\frac{C}{2}t^{-2/3}\leqslant s_-< 0$,
thus, by \eqref{equ:2ndtransaddlepoints-a} and \eqref{equ:2ndtransaddlepoints-b}
\begin{align*}
	&\left|k_j-(2+\sqrt{3})\right|\leqslant\sqrt{2C}t^{-1/3}, \ j=1,2, &&\left|k_j-(2-\sqrt{3})\right|\leqslant\sqrt{2C}t^{-1/3}, \ j=3,4,\\
   &\left|k_j-(-2+\sqrt{3})\right|\leqslant\sqrt{2C}t^{-1/3}, \ j=5,6, &&\left|k_j-(-2-\sqrt{3})\right|\leqslant\sqrt{2C}t^{-1/3}, \ j=7,8.
\end{align*}
This particularly implies that $c_0\lesssim t^{\delta_2-1/3}\rightarrow 0$ as $t\rightarrow\infty$, hence, $U^{(a)}$, $U^{(b)}$, $U^{(c)}$ and $U^{(d)}$ are
four shrinking disks with respect to $t$.

For $\ell\in\{a,b,c,d\}$, we intend to solve the following local RH problem for $N^{(\ell)}$.
\begin{RHP}\label{RHP:Nrell2nd}
    \hfill	
	\begin{itemize}
	\item[$\bullet$]  $N^{(\ell)}(z)$ is holomorphic for $z\in\mathbb{C}\setminus \Sigma^{(\ell)}$,
    where
      $$
      \Sigma^{(\ell)}:=U^{(\ell)}\cap \Sigma^{(4)}.
      $$
	\item[$\bullet$]  For $z \in \Sigma^{(\ell)}$, we have
		\begin{equation}
			N^{(\ell)}_+(z)=N^{(\ell)}_-(z)V^{(4)}(z),
		\end{equation}
        where $V^{(4)}(z)$ is defined in \eqref{jumpofM4}.
        \item[$\bullet$]
		As $z\rightarrow\infty$ in $\mathbb{C}\setminus \Sigma^{(\ell)}$, we have $N^{(\ell)}(z)=I+\mathcal{O}(z^{-1})$.
	\end{itemize}
\end{RHP}
To solve the RH problem for $N^{(a)}$, we observe that for $z\in U^{(a)}$ and $t$ large enough,
\begin{equation}\label{equ:appoftheta2+sqrt3}
	\theta(z)=\theta\left(2+\sqrt{3}\right)+\frac{4}{3}\hat{k}^3+\tilde{s}\hat{k}+\mathcal{O}\left(t^{-1/3}\hat{k}^4\right),
\end{equation}
where
\begin{equation}\label{equ:tilde{s}2nd}
	\tilde{s}=-\left(\frac{8}{9}\right)^{1/3}\left(\frac{y}{t}+\frac{1}{4}\right)t^{2/3},
\end{equation}
parametrizes the space-time region, and
\begin{equation}\label{equ:scaling of N_a local}
	\hat{k}=\left[\frac{9}{8}\left(26-15\sqrt{3}\right)t\right]^{1/3}\left(z-\left(2+\sqrt{3}\right)\right)
\end{equation}
is a scaled spectral parameter. The expansion of $\theta$ in \eqref{equ:appoftheta2+sqrt3} invokes us to work in the $\hat{k}$-plane,
and as we have done in Section \ref{subsec:pure RH N(z)}, the analysis is split into three steps.

\paragraph{Step 1: From the $z$-plane to the $\hat{k}$-plane}
Under the change of variable \eqref{equ:scaling of N_a local}, it is easily seen that $N^{(a)}(\hat{k})=N^{(a)}(z(\hat{k}))$ is holomorphic
for $\hat{k}\in\mathbb{C}\setminus\hat{\Sigma}^{(a)}$, where
\begin{equation}\label{def:hat{Sigma}^{(a)}}
	\hat{\Sigma}^{(a)}:=\cup_{j=1,2}\left(\hat{\Sigma}^{(a)}_{j}\cup\hat{\Sigma}^{(a)*}_j\right)\cup(\hat{k}_2, \hat{k}_1).
\end{equation}
Here, $\hat{k}_{j}=\left[\frac{9}{8}\left(26-15\sqrt{3}\right)t\right]^{1/3}\left(k_{j}-\left(2+\sqrt{3}\right)\right)$ and
\begin{align*}
	\hat{\Sigma}^{(a)}_1&=\left\lbrace \hat{k}:\ \hat{k}-\hat{k}_1=le^{(\varphi_0)i},\ 0\leqslant l\leqslant c_0\left[\frac{9}{8}\left(26-15\sqrt{3}\right)t\right]^{1/3}\right\rbrace,
	\\
	\hat{\Sigma}^{(a)}_2&=\left\lbrace \hat{k}:\ \hat{k}-\hat{k}_2=le^{(\pi-\varphi_0)i},\ 0\leqslant l\leqslant c_0\left[\frac{9}{8}\left(26-15\sqrt{3}\right)t\right]^{1/3}\right\rbrace.
\end{align*}
Moreover, the jump matrix for $N^{(a)}(\hat{k})$ reads
\begin{equation}
	V^{(a)}(\hat{k})=\left\{\begin{array}{lll}
		\left(\begin{array}{cc}
			1 & R(k_j)e^{2i\theta\left(2+\sqrt{3}+\left(\frac{9}{8}(26-15\sqrt{3})t\right)^{-1/3}\hat{k}\right)}\\
			0 & 1
		\end{array}\right), & \hat{k}\in\hat{\Sigma}^{(a)}_j, \\
		\\
		\left(\begin{array}{cc}
			1 & 0\\
			-\overline{R(k_j)}e^{-2i\theta\left(2+\sqrt{3}+\left(\frac{9}{8}(26-15\sqrt{3})t\right)^{-1/3}\hat{k}\right)} & 1
		\end{array}\right),  &\hat{k}\in\hat{\Sigma}^{(a)*}_j, \\
		\\
		e^{i\theta(z(\hat{k}))\hat{\sigma}_3}\begin{pmatrix}1 & 0\\ -\overline{R(z(\hat{k}))} & 1 \end{pmatrix}\begin{pmatrix} 1 & R(z(\hat{k}))\\ 0 & 1 \end{pmatrix}, & \hat{k}\in(\hat{k}_2,\hat{k}_1),
	\end{array}\right.
\end{equation}
for $j=1,2$.

\paragraph{Step 2: A model RH problem}
In the view of \eqref{equ:appoftheta2+sqrt3}, it is natural to expect that $N^{(a)}$ is well-approximated by the following model RH problem for $\hat{N}^{(a)}$
\begin{RHP}\label{RHP:hat{N}^{(a)}}
	\hfill
	\begin{itemize}
		\item[$\bullet$] $\hat{N}^{(a)}(\hat{k})$ is holomorphic for $\hat{k}\in\mathbb{C}\setminus\hat{\Sigma}^{(a)}$, where $\hat{\Sigma}^{(a)}$ is defined by \eqref{def:hat{Sigma}^{(a)}}.
		\item[$\bullet$] As $\hat{k}\rightarrow\infty$ in $\mathbb{C} \setminus \hat{\Sigma}^{(a)}$, $\hat{N}^{(a)}(\hat{k})=I+\mathcal{O}(\hat{k}^{-1})$.
		\item[$\bullet$] For $\hat{k}\in\hat{\Sigma}^{(a)}$, we have
		\begin{equation}
			\hat{N}^{(a)}_{+}(\hat{k})=\hat{N}^{(a)}_{-}(\hat{k})\hat{V}^{(a)}(\hat{k}),
		\end{equation}
	where
	\begin{align}\label{jump:V^{(a)}}
		\hat{V}^{(a)}(\hat{k})=\left\{\begin{array}{lll}
			e^{i\left(\frac{4}{3}\hat{k}^3+\tilde{s}\hat{k}\right)\hat{\sigma}_3}\left(\begin{array}{cc}
				1 & R_a\\
				0 & 1
			\end{array}\right), & \hat{k}\in\hat{\Sigma}^{(a)}_j, \ j=1, 2,\\
		[12pt]
			e^{i\left(\frac{4}{3}\hat{k}^3+\tilde{s}\hat{k}\right)\hat{\sigma}_3}\left(\begin{array}{cc}
				1 & 0\\
				-\overline{R_a} & 1
			\end{array}\right),  &\hat{k}\in\hat{\Sigma}^{(a)*}_j, \ j=1, 2,\\
			[12pt]
			e^{i\left(\frac{4}{3}\hat{k}^3+\tilde{s}\hat{k}\right)\hat{\sigma}_3}\begin{pmatrix}1 & 0\\ -\overline{R_a} & 1 \end{pmatrix}
			\begin{pmatrix} 1 & R_a\\ 0 & 1 \end{pmatrix}, & \hat{k}\in(\hat{k}_2,\hat{k}_1),
		\end{array}\right.
	\end{align}
		with
		\begin{multline}\label{def:originalR_a}
			R_a:=R(2+\sqrt{3})e^{2i\theta(2+\sqrt{3})}
			\\
			=r(2+\sqrt{3})\prod_{j=1}^{2\mathcal{N}}\left(\frac{2+\sqrt{3}-z_j}{2+\sqrt{3}-\bar{z}_j}\right)^2\exp\left\{\frac{1}{\pi i}\int_{-\infty}^{+\infty}\frac{\log\left(1-\vert r(\zeta) \vert^2\right)}{\zeta-(2+\sqrt{3})}\dif\zeta\right\}e^{2i\theta(2+\sqrt{3})}.
		\end{multline}
	\end{itemize}
\end{RHP}
By writing $R_a=|R_a|e^{i\phi_a}$, it's readily seen from \eqref{def:originalR_a} that
\begin{align}\label{def:Ra}
	|R_a|&=|r(2+\sqrt{3})|,\\
	\phi_a&=\arg r(2+\sqrt{3})+4\sum_{n=1}^{\mathcal{N}}\arg(2+\sqrt{3}-z_n)
\nonumber
\\
      &~~~ -\frac{1}{\pi}\int_{-\infty}^{+\infty}\frac{\log(1-|r(\zeta)|^2)}{\zeta-(2+\sqrt{3})}\dif\zeta  +2\theta(2+\sqrt{3}). \label{def:phia}
\end{align}
It is worthwhile to point out that $R_a$ is not real-valued and, due to the term $\theta(2+\sqrt{3})$, depends on $\hat{\xi}$ as well.
%\begin{remark}
%	Let us notice the features of the model RH problem for $\hat{N}^{(a)}(\hat{k})$: the parameter $R_a$ in the jump matrix is not the real-valued, and, the presence of zero order (relative to spectral parameters) term
%	$\theta(2+\sqrt{3})$ in the approximation for the phase function (see \eqref{equ:appoftheta2+sqrt3}) yields the $\hat{\xi}$ dependence in the definition of \eqref{def:originalR_a}.
%\end{remark}

To proceed, following the proof of Proposition \ref{prop:Xi}, we have the following basic estimates as $t\rightarrow +\infty$.
\begin{Proposition}
	With $R_{a}$ defined in \eqref{jump:V^{(a)}}, we have, as $t\to +\infty$
	\begin{equation}
		\left\Vert R(k_j)e^{2i\theta\left(2+\sqrt{3}+\left(\frac{9}{8}(26-15\sqrt{3})t\right)^{-1/3}\hat{k}\right)}-R_{a}\right \Vert_{(L^1\cap L^2\cap L^{\infty})({\hat{\Sigma}}^{(a)})}\lesssim t^{-1/3+4\delta_2},\ j=1,2.
	\end{equation}
Moreover, one has,
	%As $\hat{k}\rightarrow\infty$, it is readily seen that the large-$\hat{k}$ expansion of $N^{(a)}$ and $\hat{N}^{(a)}$ are, respectively,
	\begin{equation}\label{equ:expansion N^{(a)}hat{N}^{(a)}2nd}
		N^{(a)}(\hat{k})=I+\frac{N^{(a)}_1}{\hat k}+\mathcal{O}(\hat{k}^{-2}), \quad
	\hat{N}^{(a)}(\hat{k})=I+\frac{\hat N^{(a)}_1}{\hat{k}}+\mathcal{O}(\hat{k}^{-2}),\quad \hat{k} \to \infty,
	\end{equation}
	and
	\begin{equation}\label{equ:NahatNa ttoinfty}
		N^{(a)}_{1}(\hat{k})=\hat{N}^{(a)}_{1}(\hat{k})+\mathcal{O}(t^{-1/3+4\delta_2}), \qquad t\to +\infty.
	\end{equation}
\end{Proposition}
The above proposition particularly implies that $N^{(a)}(\hat{k})\hat{N}^{(a)}(\hat{k})^{-1} \to I$ as $t\to +\infty$, which guarantees the solvability of RH problem for $N^{(a)}$ for large positive $t$, provided $\hat{N}^{(a)}$ exists.

\paragraph{Step 3: Construction of $\hat{N}^{(a)}$}
The RH problem for $\hat{N}^{(a)}$ could be explicitly solved by using the Painlev\'{e} parametrix $M^{p}(z;s,\kappa)$ introduced in Appendix \ref{appendix: RHP for PII}. More precisely,
define
\begin{equation}\label{equ:mathcing local}
	\hat{N}^{(a)}(\hat{k})=e^{i\left(\frac{\phi_a}{2}-\frac{\pi}{4}\right)\sigma_3}\sigma_1M^{P}(\hat{k};\tilde{s}, |r(2+\sqrt{3})|)\hat{H}(\hat{k})\sigma_1e^{-i\left(\frac{\phi_a}{2}-\frac{\pi}{4}\right)\sigma_3},
\end{equation}
where $\phi_a$ is given in \eqref{def:phia},
\begin{align}
	\hat{H}(\hat{k})=\left\{\begin{array}{ll}
		e^{-i\left(\frac{4}{3}\hat{k}^3+\tilde{s}\hat{k}\right)\hat{\sigma}_3}\left(\begin{array}{cc}
				1 & 0\\
				i|r(2+\sqrt{3})| & 1
			\end{array}\right), & \hat{k}\in \hat{\Omega},\\[12pt]
		e^{-i\left(\frac{4}{3}\hat{k}^3+\tilde{s}\hat{k}\right)\hat{\sigma}_3}\left(\begin{array}{cc}
				1 & -i|r(2+\sqrt{3})| \\
				0 & 1
			\end{array}\right),  & \hat{k}\in \hat{\Omega}^*,\\[10pt]
		I, & \text{ elsewhere,}\\
		\end{array}\right.
   \end{align}
and where the region $\hat{\Omega}$ is illustrated in the left picture of Figure \ref{fig:jump of hat{N}_{a}2nd}.
We use the function $\hat{H}$ to transform the jump contours $\hat{\Sigma}^{(a)}$ defined in \eqref{def:hat{Sigma}^{(a)}} to those of the Painlev\'{e} II parametrix
for $\hat{k}\leqslant c_0\left[\frac{9}{8}\left(26-15\sqrt{3}\right)t\right]^{\frac{1}{3}}$; see Figure \ref{fig:jump of hat{N}_{a}2nd} for an illustration.

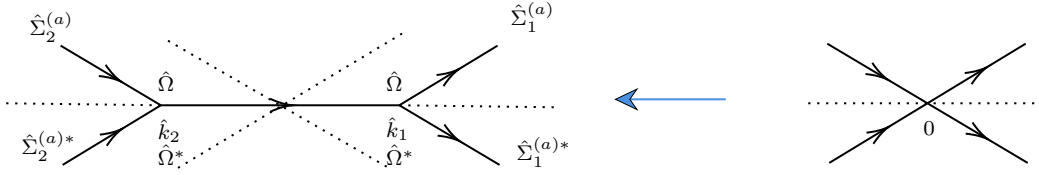
\begin{figure}[htbp]
	\begin{center}
		\tikzset{every picture/.style={line width=0.75pt}} %set default line width to 0.75pt
		\begin{tikzpicture}[x=0.75pt,y=0.75pt,yscale=-1,xscale=1]
			%uncomment if require: \path (0,300); %set diagram left start at 0, and has height of 300
			%Straight Lines [id:da3423497420414079]
			\draw    (109,144) -- (228,144) ;
			\draw [shift={(174.5,144)}, rotate = 180] [color={rgb, 255:red, 0; green, 0; blue, 0 }  ][line width=0.75]    (10.93,-3.29) .. controls (6.95,-1.4) and (3.31,-0.3) .. (0,0) .. controls (3.31,0.3) and (6.95,1.4) .. (10.93,3.29)   ;
			%Straight Lines [id:da09201758824974449]
			\draw    (228,144) -- (277,114) ;
			\draw [shift={(257.62,125.87)}, rotate = 148.52] [color={rgb, 255:red, 0; green, 0; blue, 0 }  ][line width=0.75]    (10.93,-3.29) .. controls (6.95,-1.4) and (3.31,-0.3) .. (0,0) .. controls (3.31,0.3) and (6.95,1.4) .. (10.93,3.29)   ;
			%Straight Lines [id:da8879610941244553]
			\draw    (228,144) -- (278,174) ;
			\draw [shift={(258.14,162.09)}, rotate = 210.96] [color={rgb, 255:red, 0; green, 0; blue, 0 }  ][line width=0.75]    (10.93,-3.29) .. controls (6.95,-1.4) and (3.31,-0.3) .. (0,0) .. controls (3.31,0.3) and (6.95,1.4) .. (10.93,3.29)   ;
			%Straight Lines [id:da9773409363397114]
			\draw    (59,114) -- (109,144) ;
			\draw [shift={(89.14,132.09)}, rotate = 210.96] [color={rgb, 255:red, 0; green, 0; blue, 0 }  ][line width=0.75]    (10.93,-3.29) .. controls (6.95,-1.4) and (3.31,-0.3) .. (0,0) .. controls (3.31,0.3) and (6.95,1.4) .. (10.93,3.29)   ;
			%Straight Lines [id:da3746594725280563]
			\draw    (60,174) -- (109,144) ;
			\draw [shift={(89.62,155.87)}, rotate = 148.52] [color={rgb, 255:red, 0; green, 0; blue, 0 }  ][line width=0.75]    (10.93,-3.29) .. controls (6.95,-1.4) and (3.31,-0.3) .. (0,0) .. controls (3.31,0.3) and (6.95,1.4) .. (10.93,3.29)   ;
			%Straight Lines [id:da6723019930894476]
			\draw  [dash pattern={on 0.84pt off 2.51pt}]  (432,143) -- (551,143) ;
			%Straight Lines [id:da29128583241084494]
			\draw    (491.5,143) -- (540.5,113) ;
			\draw [shift={(521.12,124.87)}, rotate = 148.52] [color={rgb, 255:red, 0; green, 0; blue, 0 }  ][line width=0.75]    (10.93,-3.29) .. controls (6.95,-1.4) and (3.31,-0.3) .. (0,0) .. controls (3.31,0.3) and (6.95,1.4) .. (10.93,3.29)   ;
			%Straight Lines [id:da6323129270849626]
			\draw    (491.5,143) -- (541.5,173) ;
			\draw [shift={(521.64,161.09)}, rotate = 210.96] [color={rgb, 255:red, 0; green, 0; blue, 0 }  ][line width=0.75]    (10.93,-3.29) .. controls (6.95,-1.4) and (3.31,-0.3) .. (0,0) .. controls (3.31,0.3) and (6.95,1.4) .. (10.93,3.29)   ;
			%Straight Lines [id:da6661920853749934]
			\draw    (441.5,113) -- (491.5,143) ;
			\draw [shift={(471.64,131.09)}, rotate = 210.96] [color={rgb, 255:red, 0; green, 0; blue, 0 }  ][line width=0.75]    (10.93,-3.29) .. controls (6.95,-1.4) and (3.31,-0.3) .. (0,0) .. controls (3.31,0.3) and (6.95,1.4) .. (10.93,3.29)   ;
			%Straight Lines [id:da7091357404665142]
			\draw    (442.5,173) -- (491.5,143) ;
			\draw [shift={(472.12,154.87)}, rotate = 148.52] [color={rgb, 255:red, 0; green, 0; blue, 0 }  ][line width=0.75]    (10.93,-3.29) .. controls (6.95,-1.4) and (3.31,-0.3) .. (0,0) .. controls (3.31,0.3) and (6.95,1.4) .. (10.93,3.29)   ;
			%Straight Lines [id:da6366646818265389]
			\draw  [dash pattern={on 0.84pt off 2.51pt}]  (31.86,143) -- (109,144) ;
			%Straight Lines [id:da09243688486496526]
			\draw  [dash pattern={on 0.84pt off 2.51pt}]  (228,144) -- (305.14,145) ;
			%Straight Lines [id:da4942049767937191]
			\draw  [dash pattern={on 0.84pt off 2.51pt}]  (112.57,111.5) -- (224.43,176.5) ;
			%Straight Lines [id:da7108279656298448]
			\draw  [dash pattern={on 0.84pt off 2.51pt}]  (117.57,175.5) -- (230.86,107) ;
			%Straight Lines [id:da7815272650764269]
			\draw [color={rgb, 255:red, 74; green, 144; blue, 226 }  ,draw opacity=1 ]   (342,141) -- (390.86,141) ;
			\draw [shift={(335.86,141)}, rotate = 0] [fill={rgb, 255:red, 74; green, 144; blue, 226 }  ,fill opacity=1 ][line width=0.08]  [draw opacity=0] (10.72,-5.15) -- (0,0) -- (10.72,5.15) -- (7.12,0) -- cycle    ;
			% Text Node
             \draw (220,125.4) node [anchor=north west][inner sep=0.75pt]  [font=\scriptsize]  {$ \hat{\Omega}$};
            \draw (220,162.4) node [anchor=north west][inner sep=0.75pt]  [font=\scriptsize]  {$ \hat{\Omega}^*$};
			\draw (220,148.4) node [anchor=north west][inner sep=0.75pt]  [font=\scriptsize]  {$\hat{k}_{1}$};
			% Text Node
   		\draw (106,125.4) node [anchor=north west][inner sep=0.75pt]  [font=\scriptsize]      {$\hat{\Omega}$};
        	\draw (106,162.4) node [anchor=north west][inner sep=0.75pt]  [font=\scriptsize]  {$\hat{\Omega}^*$};
			\draw (106,149.4) node [anchor=north west][inner sep=0.75pt]  [font=\scriptsize]  {$\hat{k}_{2}$};
			% Text Node
			\draw (282,90.4) node [anchor=north west][inner sep=0.75pt]  [font=\scriptsize]  {$\hat{\Sigma}_{1}^{(a)}$};
			% Text Node
			\draw (285,158.4) node [anchor=north west][inner sep=0.75pt]  [font=\scriptsize]  {$\hat{\Sigma}_{1}^{(a)*}$};
			% Text Node
			\draw (42,95.4) node [anchor=north west][inner sep=0.75pt]  [font=\scriptsize]  {$\hat{\Sigma}_{2}^{(a)}$};
			% Text Node
			\draw (38,155.4) node [anchor=north west][inner sep=0.75pt]  [font=\scriptsize]  {$\hat{\Sigma }_{2}^{(a)*}$};
			\draw (488,150.4) node [anchor=north west][inner sep=0.75pt]  [font=\scriptsize]  {$0$};
		\end{tikzpicture}
		\caption{ The jump contours of the RH problems for $\hat{N}^{(a)}$ (left) and $M^{P}$ (right).} \label{fig:jump of hat{N}_{a}2nd}
	\end{center}
\end{figure}
%For this case, the jump $J^{(P)}$ of $M^{P}(\hat{k})$ satisfies
%\begin{align}
%	J^{(P)}(\hat{k})=\left\{\begin{array}{lll}
%		e^{-i\left(\frac{4}{3}\hat{k}^3+\tilde{s}\hat{k}\right)\hat{\sigma}_3}\left(\begin{array}{cc}
%			1 & 0\\
%			i|r(2+\sqrt{3})| & 1
%		\end{array}\right), & \hat{k}\in\hat{\Sigma}^{(a)}_j \ j=1,2, \\
%		\\
%		e^{-i\left(\frac{4}{3}\hat{k}^3+\tilde{s}\hat{k}\right)\hat{\sigma}_3}\left(\begin{array}{cc}
%			1 & i|r(2+\sqrt{3})|\\
%			0 & 1
%		\end{array}\right),  &\hat{k}\in\hat{\Sigma}^{(a)*}_j, j=1,2.\\
%	\end{array}\right.
%\end{align}
In view of RH problem \ref{appendix: reduced PII RH}, it's readily seen that $\hat{N}^{(a)}$ in \eqref{equ:mathcing local} indeed solves RH problem \ref{RHP:hat{N}^{(a)}}.
Moreover, as $t\rightarrow+\infty$, it is observed from \eqref{MPfirstexpansion}, \eqref{vMPasy} and \eqref{equ:expansion N^{(a)}hat{N}^{(a)}2nd} that
\begin{equation}\label{equ:hat{N}^{(a)}_1}
	\hat{N}^{(a)}_1(\tilde{s})=\frac{i}{2}\begin{pmatrix}
		\int_{\tilde{s}}^{+\infty}v_{II}^2(\zeta)\dif\zeta   & -v_{II}(\tilde{s})e^{i\phi_a} \\
		v_{II}(\tilde{s})e^{-i\phi_a} & -\int_{\tilde{s}}^{+\infty}v_{II}^2(\zeta)\dif\zeta
	\end{pmatrix}+\mathcal{O}(e^{-ct^{3\delta_2}}),
\end{equation}
for some $c>0$, where $v_{II}(\tilde{s})$ is the unique solution of Painlev\'{e} equation \eqref{equ:standard PII equ}, fixed by the boundary condition
\begin{equation*}
	v_{II}(\tilde{s})\sim-\vert r(2+\sqrt{3})\vert \textnormal{Ai}(\tilde{s}), \qquad \tilde{s}\to +\infty,
\end{equation*}
with $|r(2+\sqrt{3})|<1$.

Finally, the RH problems for $N^{(b)}$, $N^{(c)}$ and $N^{(d)}$ can be solved in similar manners. Indeed, for $z\in U^{(b)}$ and $t$ large enough, we have
\begin{equation}
	\theta(z)=\theta\left(2-\sqrt{3}\right)+\frac{4}{3}\check{k}^3+\tilde{s}\check{k}+\mathcal{O}\left(t^{-1/3}\check{k}^4\right),
\end{equation}
where $\tilde{s}$ is defined by \eqref{equ:tilde{s}2nd} and
\begin{equation}
	\check{k}=\left[\frac{9}{8}\left(26+15\sqrt{3}\right)t\right]^{1/3}\left(z-\left(2-\sqrt{3}\right)\right).
\end{equation}
is the scaled spectral parameter in this case. Following the three steps we have just exhibited, we could approximate $N^{(b)}$ by the RH problem for $\check{N}^{(b)}$
on the $\check{k}$-plane, such that, as $\check{k}\rightarrow\infty$,
\begin{equation}
N^{(b)}(\check{k})=I+\frac{N^{(b)}_1}{\check k}+\mathcal{O}(\check{k}^{-2}), \qquad
\check{N}^{(b)}(\check{k})=I+\frac{\check{N}^{(b)}_1}{\check{k}}+\mathcal{O}(\check{k}^{-2}),
\end{equation}
where as $t\rightarrow+\infty$,
\begin{subequations}
\begin{align}
&{N}^{(b)}_{1}=\check{N}^{(b)}_{1}+\mathcal{O}(t^{-1/3+4\delta_2}), \label{equ:NbhatNb ttoinfty}\\
&\check{N}^{(b)}_{1}(\tilde{s})=\frac{i}{2}\begin{pmatrix}
	\int_{\tilde{s}}^{+\infty}v_{II}^2(\zeta)\dif\zeta   & -v_{II}(\tilde{s})e^{i\phi_b} \\
	v_{II}(\tilde{s})e^{-i\phi_b} & -\int_{\tilde{s}}^{+\infty}v_{II}^2(\zeta)\dif\zeta  \end{pmatrix}+\mathcal{O}(e^{-ct^{3\delta_2}}),
\end{align}
\end{subequations}
with the same $v_{II}(s)$ in \eqref{equ:hat{N}^{(a)}_1} and the argument
\begin{multline}\label{equ:argument phi_b}
	\phi_b(\tilde{s}, t)=\arg r(2-\sqrt{3})+4\sum_{n=1}^{\mathcal{N}}\arg(2-\sqrt{3}-z_n)
\\
-\frac{1}{\pi}\int_{-\infty}^{+\infty}\frac{\log(1-|r(\zeta)|^2)}{\zeta-(2-\sqrt{3})}\dif\zeta  +2\theta(2-\sqrt{3}).
\end{multline}
Here, we also  use the fact that $r(z)=\overline{r(z^{-1})}$, which particularly implies that $|r(2+\sqrt{3})|=|r(2-\sqrt{3})|$.

For $z\in U^{(c)}$ and $t$ large enough, we have
\begin{equation}
	\theta(z)=\theta\left(-2+\sqrt{3}\right)+\frac{4}{3}\tilde{k}^3+\tilde{s}\tilde{k}+\mathcal{O}\left(t^{-1/3}\tilde{k}^4\right),
\end{equation}
where $\tilde{s}$ is defined by \eqref{equ:tilde{s}2nd} and
\begin{equation}
	\tilde{k}=\left[\frac{9}{8}\left(26+15\sqrt{3}\right)t\right]^{1/3}\left(z-\left(-2+\sqrt{3}\right)\right)
\end{equation}
is the scaled spectral parameter in this case. We could approximate $N^{(c)}$ by the RH problem for $\tilde{N}^{(c)}$
on the $\tilde{k}$-plane, such that, as $\tilde{k}\rightarrow\infty$,
\begin{equation}
N^{(c)}(\tilde{k})=I+\frac{N^{(c)}_1}{\tilde k}+\mathcal{O}(\tilde{k}^{-2}), \qquad
\tilde{N}^{(c)}(\tilde{k})=I+\frac{\tilde{N}^{(c)}_1}{\tilde{k}}+\mathcal{O}(\tilde{k}^{-2}),
\end{equation}
where as $t\rightarrow+\infty$,
\begin{subequations}
\begin{align}
&{N}^{(c)}_{1}=\tilde{N}^{(c)}_{1}+\mathcal{O}(t^{-1/3+4\delta_2}), \label{equ:NchatNc ttoinfty}\\
&\tilde{N}^{(c)}_{1}(\tilde{s})=\frac{i}{2}\begin{pmatrix}
\int_{\tilde{s}}^{+\infty}v_{II}^2(\zeta)\dif\zeta   & v_{II}(\tilde{s})e^{-i\phi_b} \\
-v_{II}(\tilde{s})e^{i\phi_b} & -\int_{\tilde{s}}^{+\infty}v_{II}^2(\zeta)\dif\zeta  \end{pmatrix}+\mathcal{O}(e^{-ct^{3\delta_2}}),
\end{align}
\end{subequations}
with the same $v_{II}(s)$ in \eqref{equ:hat{N}^{(a)}_1}. Here, we also use the symmetry relation $\tilde{N}^{(c)}(-\cdot;\tilde{s})=-\sigma_2\check{N}^{(b)}(\cdot;\tilde{s})\sigma_2$.

%we also need to use the fact that $r(z)=-\overline{r(-z)}$,
%which particularly implies that $|r(2+\sqrt{3})|=|r(-2+\sqrt{3})|$, and the relation $\phi_c=\pi-\phi_b$, where
%\begin{equation}
%	\phi_c(\tilde{s},t)=\arg r(-2+\sqrt{3})+4\sum_{n=1}^{\mathcal{N}}\arg(-2+\sqrt{3}-z_n)-\frac{1}{\pi}\int_{-\infty}^{+\infty}\frac{\log(1-|r(\zeta)|^2)}{\zeta-(-2+\sqrt{3})}\dif\zeta+2\theta(-2+\sqrt{3}).
%\end{equation}

For $z\in U^{(d)}$, and $t$ large enough, we have
\begin{equation}
	\theta(z)=\theta\left(-2-\sqrt{3}\right)+\frac{4}{3}\breve{k}^3+\tilde{s}\breve{k}+\mathcal{O}\left(t^{-1/3}\breve{k}^4\right),
\end{equation}
where $\tilde{s}$ is defined by \eqref{equ:tilde{s}2nd} and
\begin{equation}
	\breve{k}=\left[\frac{9}{8}\left(26-15\sqrt{3}\right)t\right]^{1/3}\left(z-\left(-2-\sqrt{3}\right)\right)
\end{equation}
is the scaled spectral parameter in this case. We then approximate $N^{(d)}$ by the RH problem for $\breve{N}^{(d)}$
on the $\breve{k}$-plane, such that, as $\breve{k}\rightarrow\infty$,
\begin{equation}
N^{(d)}(\breve{k})=I+\frac{N^{(d)}_1}{\breve k}+\mathcal{O}(\breve{k}^{-2}), \qquad
\breve{N}^{(c)}(\breve{k})=I+\frac{\breve{N}^{(d)}_1}{\breve{k}}+\mathcal{O}(\breve{k}^{-2}),
\end{equation}
where as $t\rightarrow+\infty$,
\begin{subequations}
\begin{align}
&{N}^{(d)}_{1}=\breve{N}^{(d)}_{1}+\mathcal{O}(t^{-1/3+4\delta_2}), \label{equ:NdhatNd ttoinfty}\\
&\breve{N}^{(d)}_{1}(\tilde{s})=\frac{i}{2}\begin{pmatrix}
			\int_{\tilde{s}}^{+\infty}v_{II}^2(\zeta)\dif\zeta   & v_{II}(\tilde{s})e^{-i\phi_a} \\
			-v_{II}(\tilde{s})e^{i\phi_a} & -\int_{\tilde{s}}^{+\infty}v_{II}^2(\zeta)\dif\zeta  \end{pmatrix}+\mathcal{O}(e^{-ct^{3\delta_2}}),
\end{align}
\end{subequations}
with the same $v_{II}(s)$ in \eqref{equ:hat{N}^{(a)}_1}. Here we use the symmetry $\breve{N}^{(d)}(-\cdot;\tilde{s})=-\sigma_2\hat{N}^{(a)}(\cdot;\tilde{s})\sigma_2$.
%Here, we also need to use the fact that $r(z)=-\overline{r(-z)}$,
%which particularly implies that $|r(2+\sqrt{3})|=|r(-2-\sqrt{3})|$, and the relation $\phi_d=\pi-\phi_a$, where
%\begin{equation}
%	\phi_d(\tilde{s},t)=\arg r(-2-\sqrt{3})+4\sum_{n=1}^{\mathcal{N}}\arg(-2-\sqrt{3}-z_n)-\frac{1}{\pi}\int_{-\infty}^{+\infty}\frac{\log(1-|r(\zeta)|^2)}{\zeta-(-2-\sqrt{3})}\dif\zeta+2\theta(-2-\sqrt{3}).
%\end{equation}

%In the expansions above, we use the facts that $\theta(z)$, $\theta^{''}(z)$ are odd functions, $\theta'(z)$, $\theta^{'''}(z)$ are even functions.
%\paragraph{The other local RH models}The rest part of this subsection is to consider the other three localized RH problems. Recalling the symmetry \eqref{symofRonCC}, the following
%symmetries are valid
%\begin{equation}\label{sym:NaNb}
%\breve{N}^{(d)}(-\cdot;\tilde{s})=-\sigma_2\hat{N}^{(a)}(\cdot;\tilde{s})\sigma_2, \quad \tilde{N}^{(c)}(-\cdot;\tilde{s})=-\sigma_2\check{N}^{(b)}(\cdot;\tilde{s})\sigma_2.
%\end{equation}

\begin{remark}
It is worth noting that the local RH problems for $\hat{N}^{(a)}$ and $\check{N}^{(b)}$ (or $\tilde{N}^{(c)}$ and $\breve{N}^{(d)}$) do not admit any symmetry relation because the definition \eqref{def:Rexpressintermsof} breaks up the
relation between $z$ and $1/z$.
\end{remark}

%As the definition of $R_a$, similarly, we will use $R_j=|R_j|e^{i\phi_j}$, $j\in\{b,c,d\}$ to investigate the corresponded localized problem. Using the symmetry $r(z)=\overline{r(z^{-1})}=-\overline{r(-z)}$ for
%$z\in\mathbb{R}$, we obtain
%\begin{subequations}\label{symR_jphi_j}
%\begin{align}
%	&|R_j|=|r(2+\sqrt{3})|, \quad j\in\{a,b,c,d\},\label{|R_j|=|r(2+gen3)|}\\
%	&\phi_d=\pi-\phi_a, \quad \phi_c=\pi-\phi_b.
%\end{align}
%where
%\end{subequations}

%Similar to \eqref{N^{(a)}expansion}, taking into account \eqref{sym:NaNb} and \eqref{symR_jphi_j}, we have
%\begin{subequations}
%	\begin{align}
%		&\check{N}^{(b)}(\check{k})=I+\frac{\check{N}^{(b)}_{1}(\tilde{s})}{\check{k}}+\mathcal{O}(\check{k}^{-2}), &&\check{k}\rightarrow\infty,\\
%		&\tilde{N}^{(c)}(\tilde{k})=I+\frac{\tilde{N}^{(c)}_{1}(\tilde{s})}{\tilde{k}}+\mathcal{O}(\tilde{k}^{-2}), &&\tilde{k}\rightarrow\infty,\\
%		&\breve{N}^{(d)}(\breve{k})=I+\frac{\breve{N}^{(d)}_{1}(\tilde{s})}{\breve{k}}+\mathcal{O}(\breve{k}^{-2}), &&\breve{k}\rightarrow\infty,
%	\end{align}
%\end{subequations}
%where

\subsubsection*{The small norm RH problem}
Let $N(z)$ and $N^{(\ell)}(z)$, $\ell\in\{a,b,c,d\}$ be the solutions of RH problems \ref{RHP: N(z) tran2} and \ref{RHP:Nrell2nd}, we define
 \begin{equation}\label{def:E(z)2ndtranregion}
 	E(z)=\left\{\begin{array}{ll}
 		N(z), & z\notin \cup_{\ell\in\{a,b,c,d\}}U^{(\ell)}(z),\\
 		N(z){N^{(a)}(z)}^{-1},  &z\in U^{(a)},\\
 		N(z){N^{(b)}(z)}^{-1},  &z\in U^{(b)},\\
		N(z){N^{(c)}(z)}^{-1},  &z\in U^{(c)},\\
		N(z){N^{(d)}(z)}^{-1},  &z\in U^{(d)}.
 	\end{array}\right.
 \end{equation}
It is then readily seen that the RH conditions for $E$ are listed as follows
\begin{RHP}\label{RHP:E(z)2ndtranregion}
	\hfill
	\begin{itemize}
	\item[$\bullet$]   $E(z)$ is holomorphic for $z\in\mathbb{C}\setminus\Sigma^{(E)}$, where
	\begin{equation*}
		\Sigma^{(E)}:= \cup_{\ell\in\{a,b,c,d\}}\partial U^{(\ell)}\cup(\Sigma^{(4)}\setminus \cup_{\ell\in\{a,b,c,d\}}U^{(\ell)});
	\end{equation*}
	see Figure \ref{fig:E2ndtran} for an illustration.
	\item[$\bullet$]   For $z\in\Sigma^{(E)}$, we have
	\begin{align}
	E_+(z)=E_-(z)V^{(E)}(z),
	\end{align}
	where
	\begin{equation}
		V^{(E)}(z)=\left\{\begin{array}{llll}
			V^{(4)}(z), & z\in \Sigma^{(4)}\setminus (\cup_{\ell\in\{a,b,c,d\}}U^{(\ell)}),\\[4pt]
			{N^{(a)}(z)}^{-1},  & z\in \partial U^{(a)}, \\
			{N^{(b)}(z)}^{-1},  & z\in \partial U^{(b)}, \\
			{N^{(c)}(z)}^{-1},  & z\in \partial U^{(c)}, \\
			{N^{(d)}(z)}^{-1},  & z\in \partial U^{(d)},
		\end{array}\right.
	\end{equation}
	and where $V^{(4)}(z)$ is defined in \eqref{jumpofM4}.
	\item[$\bullet$] As $z\rightarrow\infty$ in $\mathbb{C} \setminus \Sigma^{(E)}$, we have $E(z) = I+\mathcal{O}(z^{-1})$.
\end{itemize}
\end{RHP}

\begin{figure}[htbp]
	\centering
	\tikzset{every picture/.style={line width=0.75pt}} %set default line width to 0.75pt
	\begin{tikzpicture}[x=0.75pt,y=0.75pt,yscale=-1,xscale=1]
	%uncomment if require: \path (0,300); %set diagram left start at 0, and has height of 300
	%Straight Lines [id:da2584865362269124]
	\draw    (20,80) -- (79,118) ;
	\draw [shift={(54.54,102.25)}, rotate = 212.78] [color={rgb, 255:red, 0; green, 0; blue, 0 }  ][line width=0.75]    (10.93,-3.29) .. controls (6.95,-1.4) and (3.31,-0.3) .. (0,0) .. controls (3.31,0.3) and (6.95,1.4) .. (10.93,3.29)   ;
	%Straight Lines [id:da290795151868249]
	\draw    (20,185) -- (78,147) ;
	\draw [shift={(54.02,162.71)}, rotate = 146.77] [color={rgb, 255:red, 0; green, 0; blue, 0 }  ][line width=0.75]    (10.93,-3.29) .. controls (6.95,-1.4) and (3.31,-0.3) .. (0,0) .. controls (3.31,0.3) and (6.95,1.4) .. (10.93,3.29)   ;
	%Shape: Circle [id:dp6963614417255988]
	\draw  [color={rgb, 255:red, 208; green, 2; blue, 27 }  ,draw opacity=1 ] (74,134) .. controls (74,120.19) and (85.19,109) .. (99,109) .. controls (112.81,109) and (124,120.19) .. (124,134) .. controls (124,147.81) and (112.81,159) .. (99,159) .. controls (85.19,159) and (74,147.81) .. (74,134) -- cycle ;
	%Straight Lines [id:da22134149862305086]
	\draw    (175,78) -- (234,116) ;
	\draw [shift={(209.54,100.25)}, rotate = 212.78] [color={rgb, 255:red, 0; green, 0; blue, 0 }  ][line width=0.75]    (10.93,-3.29) .. controls (6.95,-1.4) and (3.31,-0.3) .. (0,0) .. controls (3.31,0.3) and (6.95,1.4) .. (10.93,3.29)   ;
	%Straight Lines [id:da5667030271816722]
	\draw    (175,183) -- (233,145) ;
	\draw [shift={(209.02,160.71)}, rotate = 146.77] [color={rgb, 255:red, 0; green, 0; blue, 0 }  ][line width=0.75]    (10.93,-3.29) .. controls (6.95,-1.4) and (3.31,-0.3) .. (0,0) .. controls (3.31,0.3) and (6.95,1.4) .. (10.93,3.29)   ;
	%Shape: Circle [id:dp5537465980608511]
	\draw  [color={rgb, 255:red, 208; green, 2; blue, 27 }  ,draw opacity=1 ] (229,132) .. controls (229,118.19) and (240.19,107) .. (254,107) .. controls (267.81,107) and (279,118.19) .. (279,132) .. controls (279,145.81) and (267.81,157) .. (254,157) .. controls (240.19,157) and (229,145.81) .. (229,132) -- cycle ;
	%Straight Lines [id:da2655385555072687]
	\draw    (117,116) -- (175,78) ;
	\draw [shift={(151.02,93.71)}, rotate = 146.77] [color={rgb, 255:red, 0; green, 0; blue, 0 }  ][line width=0.75]    (10.93,-3.29) .. controls (6.95,-1.4) and (3.31,-0.3) .. (0,0) .. controls (3.31,0.3) and (6.95,1.4) .. (10.93,3.29)   ;
	%Straight Lines [id:da9882456185528419]
	\draw    (120,148) -- (175,183) ;
	\draw [shift={(152.56,168.72)}, rotate = 212.47] [color={rgb, 255:red, 0; green, 0; blue, 0 }  ][line width=0.75]    (10.93,-3.29) .. controls (6.95,-1.4) and (3.31,-0.3) .. (0,0) .. controls (3.31,0.3) and (6.95,1.4) .. (10.93,3.29)   ;
	%Straight Lines [id:da522095293744913]
	\draw    (272,113) -- (330,75) ;
	\draw [shift={(306.02,90.71)}, rotate = 146.77] [color={rgb, 255:red, 0; green, 0; blue, 0 }  ][line width=0.75]    (10.93,-3.29) .. controls (6.95,-1.4) and (3.31,-0.3) .. (0,0) .. controls (3.31,0.3) and (6.95,1.4) .. (10.93,3.29)   ;
	%Straight Lines [id:da8031430884165598]
	\draw    (273,150) -- (331,181) ;
	\draw [shift={(307.29,168.33)}, rotate = 208.12] [color={rgb, 255:red, 0; green, 0; blue, 0 }  ][line width=0.75]    (10.93,-3.29) .. controls (6.95,-1.4) and (3.31,-0.3) .. (0,0) .. controls (3.31,0.3) and (6.95,1.4) .. (10.93,3.29)   ;
	%Straight Lines [id:da054171317787587325]
	\draw    (330,75) -- (390,114) ;
	\draw [shift={(365.03,97.77)}, rotate = 213.02] [color={rgb, 255:red, 0; green, 0; blue, 0 }  ][line width=0.75]    (10.93,-3.29) .. controls (6.95,-1.4) and (3.31,-0.3) .. (0,0) .. controls (3.31,0.3) and (6.95,1.4) .. (10.93,3.29)   ;
	%Straight Lines [id:da8098801048927879]
	\draw    (331,181) -- (389,143) ;
	\draw [shift={(365.02,158.71)}, rotate = 146.77] [color={rgb, 255:red, 0; green, 0; blue, 0 }  ][line width=0.75]    (10.93,-3.29) .. controls (6.95,-1.4) and (3.31,-0.3) .. (0,0) .. controls (3.31,0.3) and (6.95,1.4) .. (10.93,3.29)   ;
	%Shape: Circle [id:dp5832794121739318]
	\draw  [color={rgb, 255:red, 208; green, 2; blue, 27 }  ,draw opacity=1 ] (385,130) .. controls (385,116.19) and (396.19,105) .. (410,105) .. controls (423.81,105) and (435,116.19) .. (435,130) .. controls (435,143.81) and (423.81,155) .. (410,155) .. controls (396.19,155) and (385,143.81) .. (385,130) -- cycle ;
	%Straight Lines [id:da5120836315077588]
	\draw    (486,74) -- (545,112) ;
	\draw [shift={(520.54,96.25)}, rotate = 212.78] [color={rgb, 255:red, 0; green, 0; blue, 0 }  ][line width=0.75]    (10.93,-3.29) .. controls (6.95,-1.4) and (3.31,-0.3) .. (0,0) .. controls (3.31,0.3) and (6.95,1.4) .. (10.93,3.29)   ;
	%Straight Lines [id:da28470008011973835]
	\draw    (486,179) -- (544,141) ;
	\draw [shift={(520.02,156.71)}, rotate = 146.77] [color={rgb, 255:red, 0; green, 0; blue, 0 }  ][line width=0.75]    (10.93,-3.29) .. controls (6.95,-1.4) and (3.31,-0.3) .. (0,0) .. controls (3.31,0.3) and (6.95,1.4) .. (10.93,3.29)   ;
	%Shape: Circle [id:dp08827386676520166]
	\draw  [color={rgb, 255:red, 208; green, 2; blue, 27 }  ,draw opacity=1 ] (540,128) .. controls (540,114.19) and (551.19,103) .. (565,103) .. controls (578.81,103) and (590,114.19) .. (590,128) .. controls (590,141.81) and (578.81,153) .. (565,153) .. controls (551.19,153) and (540,141.81) .. (540,128) -- cycle ;
	%Straight Lines [id:da5674474012322692]
	\draw    (428,112) -- (486,74) ;
	\draw [shift={(462.02,89.71)}, rotate = 146.77] [color={rgb, 255:red, 0; green, 0; blue, 0 }  ][line width=0.75]    (10.93,-3.29) .. controls (6.95,-1.4) and (3.31,-0.3) .. (0,0) .. controls (3.31,0.3) and (6.95,1.4) .. (10.93,3.29)   ;
	%Straight Lines [id:da49921047318119416]
	\draw    (431,144) -- (486,179) ;
	\draw [shift={(463.56,164.72)}, rotate = 212.47] [color={rgb, 255:red, 0; green, 0; blue, 0 }  ][line width=0.75]    (10.93,-3.29) .. controls (6.95,-1.4) and (3.31,-0.3) .. (0,0) .. controls (3.31,0.3) and (6.95,1.4) .. (10.93,3.29)   ;
	%Straight Lines [id:da006832114987808113]
	\draw    (583,109) -- (641,71) ;
	\draw [shift={(617.02,86.71)}, rotate = 146.77] [color={rgb, 255:red, 0; green, 0; blue, 0 }  ][line width=0.75]    (10.93,-3.29) .. controls (6.95,-1.4) and (3.31,-0.3) .. (0,0) .. controls (3.31,0.3) and (6.95,1.4) .. (10.93,3.29)   ;
	%Straight Lines [id:da07656443658326739]
	\draw    (584,146) -- (643.07,185.9) ;
	\draw [shift={(618.51,169.31)}, rotate = 214.04] [color={rgb, 255:red, 0; green, 0; blue, 0 }  ][line width=0.75]    (10.93,-3.29) .. controls (6.95,-1.4) and (3.31,-0.3) .. (0,0) .. controls (3.31,0.3) and (6.95,1.4) .. (10.93,3.29)   ;
	%Straight Lines [id:da542502053668082]
	\draw  [dash pattern={on 0.84pt off 2.51pt}]  (9.5,133.5) -- (630.5,131.51) ;
	\draw [shift={(632.5,131.5)}, rotate = 179.82] [color={rgb, 255:red, 0; green, 0; blue, 0 }  ][line width=0.75]    (10.93,-3.29) .. controls (6.95,-1.4) and (3.31,-0.3) .. (0,0) .. controls (3.31,0.3) and (6.95,1.4) .. (10.93,3.29)   ;
	%Straight Lines [id:da7499770495921136]
	\draw    (175,80) -- (175,181) ;
	\draw [shift={(175,183)}, rotate = 270] [color={rgb, 255:red, 0; green, 0; blue, 0 }  ][line width=0.75]    (10.93,-3.29) .. controls (6.95,-1.4) and (3.31,-0.3) .. (0,0) .. controls (3.31,0.3) and (6.95,1.4) .. (10.93,3.29)   ;
	\draw [shift={(175,78)}, rotate = 90] [color={rgb, 255:red, 0; green, 0; blue, 0 }  ][line width=0.75]    (10.93,-3.29) .. controls (6.95,-1.4) and (3.31,-0.3) .. (0,0) .. controls (3.31,0.3) and (6.95,1.4) .. (10.93,3.29)   ;
	%Straight Lines [id:da5479710212200013]
	\draw    (330,77) -- (330,178) ;
	\draw [shift={(330,180)}, rotate = 270] [color={rgb, 255:red, 0; green, 0; blue, 0 }  ][line width=0.75]    (10.93,-3.29) .. controls (6.95,-1.4) and (3.31,-0.3) .. (0,0) .. controls (3.31,0.3) and (6.95,1.4) .. (10.93,3.29)   ;
	\draw [shift={(330,75)}, rotate = 90] [color={rgb, 255:red, 0; green, 0; blue, 0 }  ][line width=0.75]    (10.93,-3.29) .. controls (6.95,-1.4) and (3.31,-0.3) .. (0,0) .. controls (3.31,0.3) and (6.95,1.4) .. (10.93,3.29)   ;
	%Straight Lines [id:da41733520475329633]
	\draw    (486,76) -- (486,177) ;
	\draw [shift={(486,179)}, rotate = 270] [color={rgb, 255:red, 0; green, 0; blue, 0 }  ][line width=0.75]    (10.93,-3.29) .. controls (6.95,-1.4) and (3.31,-0.3) .. (0,0) .. controls (3.31,0.3) and (6.95,1.4) .. (10.93,3.29)   ;
	\draw [shift={(486,74)}, rotate = 90] [color={rgb, 255:red, 0; green, 0; blue, 0 }  ][line width=0.75]    (10.93,-3.29) .. controls (6.95,-1.4) and (3.31,-0.3) .. (0,0) .. controls (3.31,0.3) and (6.95,1.4) .. (10.93,3.29)   ;
	\draw  [color={rgb, 255:red, 208; green, 2; blue, 27 }  ,draw opacity=1 ] (573,108) -- (560.99,103.27) -- (573.2,99.09) ;
	\draw  [color={rgb, 255:red, 208; green, 2; blue, 27 }  ,draw opacity=1 ] (416,110) -- (403.99,105.27) -- (416.2,101.09) ;
	\draw  [color={rgb, 255:red, 208; green, 2; blue, 27 }  ,draw opacity=1 ] (258,112) -- (245.99,107.27) -- (258.2,103.09) ;
	\draw  [color={rgb, 255:red, 208; green, 2; blue, 27 }  ,draw opacity=1 ] (105,114) -- (92.99,109.27) -- (105.2,105.09) ;
	% Text Node
	\draw (319,131.4) node [anchor=north west][inner sep=0.75pt]  [font=\scriptsize]  {$0$};
	% Text Node
	\draw (572,129.4) node [anchor=north west][inner sep=0.75pt]  [font=\scriptsize]  {$k_{1}$};
	% Text Node
	\draw (548,128.4) node [anchor=north west][inner sep=0.75pt]  [font=\scriptsize]  {$k_{2}$};
	% Text Node
	\draw (417,129.4) node [anchor=north west][inner sep=0.75pt]  [font=\scriptsize]  {$k_{3}$};
	% Text Node
	\draw (393,129.4) node [anchor=north west][inner sep=0.75pt]  [font=\scriptsize]  {$k_{4}$};
	% Text Node
	\draw (261,132.4) node [anchor=north west][inner sep=0.75pt]  [font=\scriptsize]  {$k_{5}$};
	% Text Node
	\draw (240,132.4) node [anchor=north west][inner sep=0.75pt]  [font=\scriptsize]  {$k_{6}$};
	% Text Node
	\draw (106,135.4) node [anchor=north west][inner sep=0.75pt]  [font=\scriptsize]  {$k_{7}$};
	% Text Node
	\draw (82,135.4) node [anchor=north west][inner sep=0.75pt]  [font=\scriptsize]  {$k_{8}$};
	% Text Node
	\draw (620,135) node [anchor=north west][inner sep=0.75pt]  [font=\scriptsize] [align=left] {$\re z$};
	\end{tikzpicture}
	\caption{ The jump contour $\Sigma^{(E)}$ of the RH problem for $E$, where the four red circles, from the right to the left, are
	$\partial U^{(a)}$, $\partial U^{(b)}$, $\partial U^{(c)}$ and $\partial U^{(d)}$, respectively.}
	\label{fig:E2ndtran}
\end{figure}
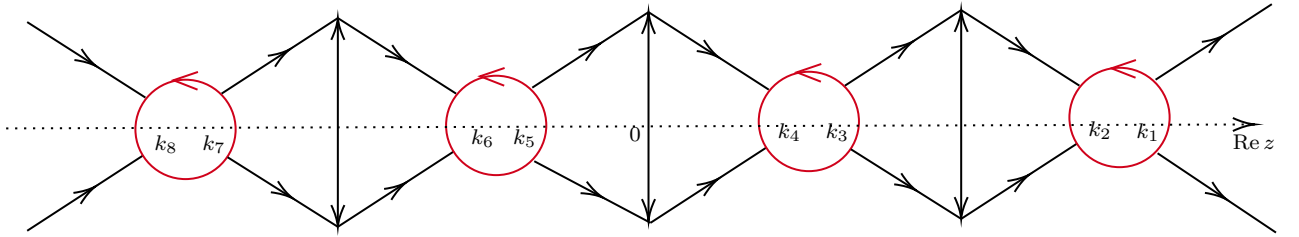

It is readily seen that
\begin{equation}\label{E(z):BCrep; 2ndtran}
	\parallel V^{(E)}(z)-I\parallel_{p}=\left\{\begin{array}{llll}
		\mathcal{O}(\exp\left\{-ct^{3\delta_2}\right\}),  & z\in  \Sigma^{(E)}\setminus  (U^{(a)}\cup U^{(b)}\cup U^{(c)}\cup U^{(d)}),\\[6pt]
		\mathcal{O}(t^{-\kappa_p}),   & z\in \partial U^{(a)}\cup\partial U^{(b)} \cup \partial U^{(c)}\cup \partial \partial U^{(d)},
	\end{array}\right.
\end{equation}
for some positive $c$ with $\kappa_\infty=\delta_2$ and $\kappa_2=1/6+\delta_2/2$.
It then follows from the small norm RH problem theory \cite{Deift2002LongtimeAF} that there exists a unique solution to RH problem \ref{RHP:E(z)2ndtranregion} for large positive $t$. Moreover, according to \cite{Beals1984ScatteringAI} again, we have
\begin{equation}\label{expression:E(z);2ndtran}
	E(z)=I+\frac{1}{2\pi i}\int_{\Sigma^{(E)}}\dfrac{\varpi(\zeta  ) (V^{(E)}(\zeta  )-I)}{\zeta  -z}\dif\zeta  ,
\end{equation}
where $\varpi\in I+ L^2(\Sigma^{(E)})$ is the unique solution of the Fredholm-type equation \eqref{equ: varpi}.

Analogous to the estimate \eqref{eq:estCE}, we have in this case that
\begin{align}
	\parallel \mathcal{C}_E\parallel\leqslant\parallel \mathcal{C}_-\parallel \parallel V^{(E)}(z)-I\parallel_\infty \lesssim t^{-\delta_2},
\end{align}
which implies that  $1-\mathcal{C}_E$ is invertible for sufficiently large $t$ in this case.
So  $\varpi$  exists  uniquely with
\begin{align*}
	\varpi=I+(1-\mathcal{C}_E)^{-1}(\mathcal{C}_EI).
\end{align*}
In addition, we have the estimates
\begin{align}
	\|\mathcal{C}_E^{j}I\|_2\lesssim t^{(-1/6+j\delta_2-\delta_2/2)},\quad \| \varpi-I-\sum_{j=1}^{4}C_E^{j}I\|\lesssim t^{-(1/6+9\delta_2)}.\label{normrho;2ndtran}
\end{align}

For later use, we conclude this section with local behaviors of $E(z)$ at $z=0$ and $z=i$ as $t\rightarrow+\infty$.
By the formulae \eqref{expression:E(z)}, it is readily seen that
\begin{equation}\label{equ:Eexpansion 2nd}
	E(z)=E(i)+E_1(z-i)+\mathcal{O}\left(\left(z-i\right)^2\right), \ z\rightarrow i,
\end{equation}
where
\begin{subequations}\label{equ:E(i)E_1 2nd}
	\begin{align}
		&E(i)=I+\frac{1}{2\pi i}\int_{\Sigma^{(E)}}\frac{\varpi(\zeta) (V^{(E)}(\zeta)-I)}{\zeta-i}\dif\zeta  , \\
		&E_1=\frac{1}{2\pi i}\int_{\Sigma^{(E)}}\frac{\varpi(\zeta) (V^{(E)}(\zeta)-I)}{(\zeta-i)^2}\dif\zeta  ,
	\end{align}
\end{subequations}
as well as
\begin{equation}\label{equ:E(0) 2nd}
	E(0)=I+\frac{1}{2\pi i}\int_{\Sigma^{(E)}}\frac{\varpi(\zeta) (V^{(E)}(\zeta)-I)}{\zeta}\dif\zeta  .
\end{equation}
\begin{Proposition}\label{Prop:E(0)Ei0Ei1}
	With $E(i)$, $E_1$ and $E(0)$ defined in \eqref{equ:E(i)E_1 2nd}--\eqref{equ:E(0) 2nd}, we have, as $t\rightarrow+\infty$,
	\begin{subequations}
		\begin{align}
			&E(0)=\begin{pmatrix}
				1 & i\left(\frac{9}{8}t\right)^{-1/3}\left(\cos\phi_a+\cos\phi_b\right)\\
				-i\left(\frac{9}{8}t\right)^{-1/3}\left(\cos\phi_a+\cos\phi_b\right) & 1
			\end{pmatrix}+\mathcal{O}\left(t^{\textnormal{max}\{{-2/3+4\delta_2, -1/3-5\delta_2}\}}\right), \label{E(0)ttoinfty}\\
			&E(i)=\begin{pmatrix}
				E_{11}(i)) & E_{12}(i)\\
				E_{21}(i)  & E_{22}(i)
			\end{pmatrix}+\mathcal{O}\left(t^{\textnormal{max}\{{-2/3+4\delta_2, -1/3-5\delta_2}\}}\right), \label{Ei0toinfty}
		\end{align}
	\end{subequations}
	where $\phi_a$ and $\phi_b$ are given in \eqref{def:phia} and \eqref{equ:argument phi_b},
		\begin{align*}
			&E_{11}(i)=1+\frac{1}{3^{2/3}}t^{-1/3}\int_{\tilde{s}}^{+\infty}v_{II}^2(\zeta)\dif\zeta,
\nonumber\\ &E_{12}(i)=-\frac{i}{3^{2/3}}t^{-1/3}v_{II}(\tilde{s})\left[\sqrt{2+\sqrt{3}}\sin\left(\phi_a-\gamma_a\right)+\sqrt{2-\sqrt{3}}\sin\left(\phi_b-\gamma_b\right)\right],
\nonumber\\ &E_{21}(i)=-\frac{i}{3^{2/3}}t^{-1/3}v_{II}(\tilde{s})\left[\sqrt{2+\sqrt{3}}\sin\left(\phi_a+\gamma_a\right)+\sqrt{2-\sqrt{3}}\sin\left(\phi_b+\gamma_b\right)\right],\nonumber\\
			&E_{22}(i)=1-\frac{1}{3^{2/3}}t^{-1/3}\int_{\tilde{s}}^{+\infty}v_{II}^2(\zeta)\dif\zeta,\nonumber\\
			&\gamma_a=\arctan(2+\sqrt{3}), \qquad \gamma_b=\arctan(2-\sqrt{3}),
		\end{align*}
		and
		\begin{align}
			E_1=&-\frac{t^{-1/3}}{2\cdot 3^{2/3}}v_{II}(\tilde{s})\begin{pmatrix}
				0 & \sin\left(\phi_a+\frac{\pi}{6}\right)-\sin\left(\phi_b-\frac{\pi}{6}\right)\\
				\sin\left(\phi_a-\frac{\pi}{6}\right)-\sin\left(\phi_b+\frac{\pi}{6}\right) & 0
			\end{pmatrix}\nonumber\\&+\mathcal{O}\left(t^{\textnormal{max}\{{-2/3+4\delta_2, -1/3-5\delta_2}\}}\right)\label{E1toinfty}.
		\end{align}
\end{Proposition}
\begin{proof}
	From \eqref{equ:NahatNa ttoinfty}, \eqref{equ:NbhatNb ttoinfty}, \eqref{equ:NchatNc ttoinfty}, \eqref{equ:NdhatNd ttoinfty}
	\eqref{normrho;2ndtran} and \eqref{equ:E(0) 2nd}, it follows that
	\begin{align}
		E(0)&=I+\sum_{\ell\in\{a,b,c,d\}}\oint_{\partial U^{(\ell)}}\frac{{N^{(\ell)}(\zeta)}^{-1}-I}{\zeta}+\mathcal{O}(t^{-1/3-5\delta_2}) \nonumber\\
		&=I-\frac{1}{2\pi i}\oint_{\partial U^{(a)}}\frac{\left(\frac{9}{8}(26-15\sqrt{3})t\right)^{-1/3}\hat{N}^{(a)}_{1}(\tilde{s})}{\zeta\left(\zeta-(2+\sqrt{3})\right)}\dif\zeta
		-\frac{1}{2\pi i}\oint_{\partial U^{(b)}}\frac{\left(\frac{9}{8}(26+15\sqrt{3})t\right)^{-1/3}\check{N}^{(b)}_{1}(\tilde{s})}{\zeta\left(\zeta-(2-\sqrt{3})\right)}\dif\zeta  \nonumber\\
		&~~~-\frac{1}{2\pi i}\oint_{\partial U^{(c)}}\frac{\left(\frac{9}{8}(26+15\sqrt{3})t\right)^{-1/3}\tilde{N}^{(c)}_{1}(\tilde{s})}{\zeta\left(\zeta-(-2+\sqrt{3})\right)}\dif\zeta
		\nonumber\\
		&~~~ -\frac{1}{2\pi i}\oint_{\partial U^{(d)}}\frac{\left(\frac{9}{8}(26-15\sqrt{3})t\right)^{-1/3}\breve{N}^{(d)}_{1}(\tilde{s})}{\zeta\left(\zeta-(-2-\sqrt{3})\right)}\dif\zeta
   +\mathcal{O}(t^{-2/3+4\delta_2})+\mathcal{O}(t^{-1/3-5\delta_2})\nonumber\\
		&=I+\frac{\left(\frac{9}{8}(26-15\sqrt{3})\right)^{-1/3}}{2+\sqrt{3}}t^{-1/3}(-\hat{N}^{(a)}_{1}+\breve{N}^{(d)}_{1}) \nonumber\\	&~~~+\frac{\left(\frac{9}{8}(26+15\sqrt{3})\right)^{-1/3}}{2-\sqrt{3}}t^{-1/3}(-\check{N}^{(b)}_{1}+\tilde{N}^{(c)}_{1})+\mathcal{O}\left(t^{\textnormal{max}\{{-2/3+4\delta_2, -1/3-5\delta_2}\}}\right) \nonumber \\
		&=I+\left(\frac{9}{8}t\right)^{-1/3} \left(-\hat{N}^{(a)}_{1}-\check{N}^{(b)}_{1}+\tilde{N}^{(c)}_{1}+\breve{N}^{(d)}_{1}\right)+\mathcal{O}\left(t^{\textnormal{max}\{{-2/3+4\delta_2, -1/3-5\delta_2}\}}\right)=\eqref{E(0)ttoinfty},
	\end{align}
	where the residue theorem is applied for the third equality.

Both the estimates \eqref{Ei0toinfty} and \eqref{E1toinfty} can be obtained in similar manners, we omit the details but point out that, to show \eqref{Ei0toinfty}, one needs the fundamental formula
	\begin{align*}
		A\sin u+B\cos u=\sqrt{A^2+B^2}\sin(u+\gamma)
	\end{align*}
with $\gamma=\arctan{\frac{B}{A}}$.
\end{proof}

\subsection{Analysis of the pure $\bar{\partial}$-problem}\label{subsec:dbarRH2ndTran}
Besides the pure RH problem for $N$, the contribution to RH problem \ref{RHPM^{(4)}} for $M^{(4)}$ partially comes from the pure $\bar{\partial}$-problem $M^{(5)}$ defined as
\begin{align}\label{transform:shengchengdbarM5 2nd}
	M^{(5)}(z)=M^{(4)}(z){N(z)}^{-1}.
\end{align}
Then $M^{(5)}$ satisfies the following $\bar{\partial}$ problem.
\begin{Dbarproblem}\label{DbarM5 2nd}
\hfill
\begin{itemize}
	\item[$\bullet$] $M^{(5)}(z)$ is continuous  and has sectionally continuous first partial derivatives in $\mathbb{C}$.
	\item[$\bullet$] As $z\rightarrow\infty$ in $\mathbb{C}$, we have
\begin{align}
	M^{(5)}(z) = I+\mathcal{O}(z^{-1}). \label{asymbehv7 2nd}
\end{align}
\item[$\bullet$] The $\bar{\partial}$-derivative of $M^{(5)}$ satisfies $\bar{\partial}M^{(5)}(z)=M^{(5)}(z)W^{(3)}(z)$, $z\in \mathbb{C}$, with
\begin{equation}\label{equ: W^(3)expression 2nd}
	W^{(3)}(z)=N(z)\bar{\partial}R^{(3)}(z)N(z)^{-1},
\end{equation}
where $\bar{\partial}R^{(3)}(z)$ is defined in \eqref{R(3)+2ndTranRegion}.
\end{itemize}
\end{Dbarproblem}
Following from the proofs of Lemma \ref{lem:estofimtheta} and Proposition \ref{prop:Cz}, it is readily seen that the pure $\bar{\partial}$ problem \ref{DbarM5 2nd} for $M^{(5)}$ exists a unique solution for large positive $t$. Moreover, the following estimates can be viewed as
an analogue of Proposition \ref{prop:estdbar} in the present case.
%via \eqref{def:M5i}-\eqref{def:M50}, it follows from the proof of Proposition \ref{prop:estdbar}
%that shown in what follows
\begin{Proposition}\label{prop:estdbar2ndtran}
Let $M^{(5)}(i)$, $M^{(5)}_{1}$ and $M^{(5)}(0)$ be defined through \eqref{def:M5i}--\eqref{def:M50} with $M^{(5)}$ and $W^{(3)}$ therein replaced by  \eqref{transform:shengchengdbarM5 2nd} and \eqref{equ: W^(3)expression 2nd},
, we have, as $t\rightarrow+\infty$,
	\begin{align}\label{equ:estdabar 2nd}
		| M^{(5)}(i)-I|,\ | M^{(5)}(0)-I|,\ |M^{(5)}_1|\lesssim t^{-2/3}.
	\end{align}
\end{Proposition}
\begin{remark}
The error bounds in \eqref{equ:estdabar 2nd} vary from those presented in Proposition \ref{prop:estdbar}. This follows from the fact that
the boundary conditions for the function $d_j(z)$ in $R^{(3)}(z,\hat{\xi})$ are different.
\end{remark}
We are now ready to prove the part (b) of Theorem \ref{mainthm}.

\subsection{Proof of part (b) of Theorem \ref{mainthm}}\label{subsec:recovering2ndtran}
By tracing back the transformations \eqref{transform:M1toM2}, \eqref{transform:M3toM4 2nd}, \eqref{def:E(z)2ndtranregion} and \eqref{transform:shengchengdbarM5 2nd}, we conclude that, as $t\rightarrow+\infty$,
\begin{align}\label{ope 2nd}
	M^{(1)}(z)=&M^{(5)}(z)E(z)R^{(3)}(z)^{-1}T(z)^{-\sigma_3}G(z)^{-1}+\textrm{exponentially small error in $t$},
\end{align}
where $E$, $R^{(3)}$, $T$ and $G$ are defined in \eqref{def:E(z)2ndtranregion}, \eqref{R(3)+2ndTranRegion}, \eqref{Tfunc} and \eqref{funcG}, respectively.

%For the later use, for $\mathrm{I}(\hat{\xi})=\mathbb{R}$ in \eqref{Tfunc}, we conclude local behaviors of $T(z)$ at $z=i$ and $z=0$ as $t\rightarrow+\infty$.
By \eqref{Tfunc}, it follows from  straightforward calculations that $T(0)=1$
and
\begin{equation}\label{equ:Texpansion 2nd}
	T(z)=T(i)+T_{1}(z-i)+\mathcal{O}\left(\left(z-i\right)^2\right), \quad z\rightarrow i,
\end{equation}
where
\begin{subequations}\label{coofTexpansionati}
\begin{align}
&T(i)=\prod_{n=1}^{2\mathcal{N}}\left(\frac{\bar{z}_n-i}{z_n-i}\right)\exp\left\{-\frac{1}{2\pi i}\int_{\mathbb{R}}\frac{\log(1-|r(\zeta)|^2)}{\zeta-i}\dif\zeta  \right\},\\
&T_1=-\prod_{n=1}^{2\mathcal{N}}\left(\frac{\bar{z}_n-i}{z_n-i}\right)\cdot\frac{1}{2\pi i}\int_{\mathbb{R}}\frac{\log(1-|r(\zeta)|^2)}{(\zeta-i)^2}\dif\zeta  \cdot\exp\left\{-\frac{1}{2\pi i}\int_{\mathbb{R}}\frac{\log(1-|r(\zeta)|^2)}{\zeta-i}\dif\zeta  \right\}\nonumber\\
&\hspace*{5em}+\sum_{n=1}^{2\mathcal{N}}\left(\frac{\bar{z}_n-i}{(z_n-i)^2}\prod_{l\neq n, l=1}^{2\mathcal{N}}\left(\frac{\bar{z}_l-i}{z_l-i}\right)\right)\exp\left\{-\frac{1}{2\pi i}\int_{\mathbb{R}}\frac{\log(1-|r(\zeta)|^2)}{\zeta-i}\dif\zeta  \right\}.
\end{align}
\end{subequations}

In view of the facts that $T(0)=1$, $G(0)=R^{(3)}(0)=I$, it follows from \eqref{equ:estdabar 2nd} that
\begin{equation}\label{Torevoverat0}
	M^{(1)}(0,t,y)=E(0)+\mathcal{O}(t^{-2/3}).
\end{equation}
Since $G(z)=R^{(3)}(z)=I$ in a neighborhood of $i$, as $z \to i$, it is readily seen
from \eqref{equ:Eexpansion 2nd} and \eqref{equ:Texpansion 2nd} that
\begin{align}\label{Torevoverati}	&M^{(1)}(z;t,y)
\nonumber
\\
&=\left(I+\mathcal{O}(t^{-2/3})\right)\left(E\left(i\right)+E_1\left(z-i\right)+\mathcal{O}\left(\left(z-i\right)^2\right)\right)\left(T\left(0\right)+T_1\left(z-i\right)+\mathcal{O}\left(\left(z-i\right)^2\right)\right)^{-\sigma_3}\nonumber\\
	&=\begin{pmatrix}
		\frac{E_{11}\left(i\right)}{T(i)} & E_{12}\left(i\right)T(i)\\
		\frac{E_{21}\left(i\right)}{T(i)} & E_{22}\left(i\right)T(i)
	\end{pmatrix}+\begin{pmatrix}
		\frac{(E_1)_{11}T(i)-E_{11}\left(i\right)T_1}{T^2(i)} &  E_{12}(i){T_1}+\left(E_1\right)_{12}T(i)\\
		\frac{(E_1)_{21}{T(i)}-E_{21}\left(i\right)T_1}{T^2{(i)}} &  E_{22}(i){T_1}+(E_1)_{22}T(i)
	\end{pmatrix}(z-i)
	\nonumber\\
	&~~~
+\mathcal{O}\left(\left(z-i\right)^2\right)+\mathcal{O}(t^{-2/3}),
\end{align}
where $E(i)$ and $E_1$ are given by \eqref{Ei0toinfty} and \eqref{E1toinfty}, respectively. Here the error term $\mathcal{O}(t^{-2/3})$ comes from the $\bar{\partial}$-problem.

Together with \eqref{Torevoverat0}, \eqref{Torevoverati}, Proposition \ref{Prop:E(0)Ei0Ei1}, \eqref{recovering u}, it follows that
\begin{subequations}\label{eq:u(y,t)RII}
	\begin{align}
		u(x(y,t),t)&=1+\frac{f_{II}(\tilde{s})}{3^{2/3}}t^{-1/3}v_{II}(\tilde{s})+\mathcal{O}\left(t^{\textnormal{max}\{{-2/3+4\delta_2, -1/3-5\delta_2}\}}\right),\label{equ:u(y,t)2ndtran}
\\
x(y,t)&=y-2\log\left(T(i)\right)
\nonumber
\\
&~~~+\frac{2}{3^{2/3}}t^{-1/3}\left(\int_{\tilde{s}}^{+\infty}v_{II}^2(\zeta)d\zeta-\sqrt{2+\sqrt{3}}\sin\left(\phi_a+\gamma_a\right)-
		\sqrt{2-\sqrt{3}}\sin\left(\phi_b+\gamma_b\right)\right)
\nonumber\\
		&~~~+\mathcal{O}\left(t^{\textnormal{max}\{{-2/3+4\delta_2, -1/3-5\delta_2}\}}\right), \label{equ:(x-y)relation2ndtran}
	\end{align}
	\end{subequations}
	where $\tilde{s}$ is given in \eqref{equ:tilde{s}2nd},
	\begin{align}
		%&\tilde{s}=-\left(\frac{8}{9}\right)^{\frac{1}{3}}\left(\frac{y}{t}+\frac{1}{4}\right)t^{2/3}, \label{equ:tildes2ndtran} \\
		f_{II}(\tilde{s})&=2\sqrt{2+\sqrt{3}}\left(\sin\phi_a\cos\gamma_a-\frac{iT_1}{T(i)}\cos\phi_a\sin\gamma_a\right)\nonumber\\
		&~~~+2\sqrt{2-\sqrt{3}}\left(\sin\phi_b\cos\gamma_b-\frac{iT_1}{T(i)}\cos\phi_b\sin\gamma_b\right)
		\nonumber
        \\
        &~~~+\sqrt{3}\cos\left(\frac{\phi_a+\phi_b}{2}\right)\sin\left(\frac{\phi_a+\phi_b}{2}\right),
	\end{align}
Since $|r(2+\sqrt{3})|<1$, the Painlev\'{e} transcendent $v_{II}(\tilde{s})$ in the formulas \eqref{equ:u(y,t)2ndtran} and \eqref{equ:(x-y)relation2ndtran} are bounded.
Taking into account the boundedness of $\log T(i)$ as well as the definition of $s$ (see \eqref{equ:s2ndtran}) and $\tilde{s}$ (see \eqref{equ:tilde{s}2nd}),
it follows from \eqref{equ:(x-y)relation2ndtran} that
\begin{equation}\label{equ:tildes-s2ndtran}
	\tilde{s}-s=\left(\frac{8}{9}\right)^{1/3}\left(x-y\right)t^{-1/3}
\end{equation}
and
\begin{equation}
	x-y=-2\log\left(T\left(i\right)\right)+\mathcal{O}(t^{-1/3}),
\end{equation}
which implies that
\begin{equation}\label{equ:tildes-s using}
	\tilde{s}-s=-\frac{4}{3^{2/3}}t^{-1/3}\log\left(T\left(i\right)\right)+\mathcal{O}\left(t^{-2/3}\right).
\end{equation}

We now estimate $\phi_a$ in \eqref{def:phia} with $\tilde{s}$ replaced by $s$. Note that
\begin{equation}
	2\theta(2+\sqrt{3})=-\sqrt{3}t\left(\frac{y}{t}-\frac{1}{2}\right)=\frac{3^{7/6}}{2}\tilde{s}t^{1/3}+\frac{3\sqrt{3}}{4}t,
\end{equation}
we obtain from \eqref{equ:tildes-s using} and \eqref{def:phia} that
\begin{equation}\label{equ:phiarelatepsia}
	\phi_a=\psi_a(s,t)+\mathcal{O}(t^{-1/3}),
\end{equation}
if $\tilde{s}$ is replaced by $s$, where $\psi_a(s,t)$ is defined through \eqref{equ:psi_a} and \eqref{equ:Lambda_a}. Similarly, since
\begin{equation}
	2\theta(2-\sqrt{3})=\sqrt{3}t\left(\frac{y}{t}-\frac{1}{2}\right)=-\frac{3^{7/6}}{2}\tilde{s}t^{1/3}-\frac{3\sqrt{3}}{4}t,
\end{equation}
we obtain from \eqref{equ:tildes-s using} and \eqref{equ:argument phi_b} that
\begin{equation}\label{equ:phibrelatepsib}
	\phi_b=\psi_b(s,t)+\mathcal{O}(t^{-1/3})
\end{equation}
with $\tilde{s}$ replaced by $s$, where $\psi_b(s,t)$ is defined through \eqref{equ:psi_a} and \eqref{equ:Lambda_b}.

After replacing $\tilde{s}$ by $s$ in \eqref{eq:u(y,t)RII}, we obtain \eqref{result:2ndTran} from \eqref{equ:tildes-s using}, \eqref{equ:phiarelatepsia} and \eqref{equ:phibrelatepsib}.
\qed

%%%%%%%%%%%%%%%%%%%%%%%%%%%%%%%%%%%%%%%%%%%%%%%%%%%%%%%%%%%%%%%%%%%%%%%%%%%%%%%%%%%%%%%%%%%%%%%%%%%%%%%%%%%%%
%%%%%%%%%%%%%%%%%%%%%%%%%%%%%%%%%%%%%%%%%%%%%%%%%%%%%%%%%%%%%%%%%%%%%%%%%%%%%%%%%%%%%%%%%%%%%%%%%%%%%%%%%%%%%

\section{Asymptotic analysis of the RH problem for $M^{(3)}$ in $\mathcal{R}_{III}$}
\label{sec: collisionless shock region}
A detailed analysis of the mCH equation in the 1st oscillatory region  can be found, for example,
in \cites{Monvel2020TheMC, Yang2022adv}, where a crucial step to establish the asymptotic behaviors is to reduce the oscillatory RH problem to several localized RH problems
near the four saddle points $k_j$, $j=1,\ldots,4$, given in \eqref{equ:1sttransaddlepoints-a} and \eqref{equ:1sttransaddlepoints-b}. It turns out the local RH problem near each $k_{j}$ is
controlled in the norm by $(1-|r(k_j)|^2)^{-1}$. In the generic case,  i.e., $|r(\pm 1)|=1$, these norms blow up as $k_{1,2}\rightarrow 1$, $k_{3,4}\rightarrow -1$. As a consequence, a new transition zone -- the collisionless shock
region $\mathcal{R}_{III}$ occurs generically between the first transition region $\mathcal{R}_I$ and the 1st oscillatory region. It is the aim of this section to carry out asymptotic analysis of the RH problem for $M^{(3)}$ for
\begin{equation}\label{def:RIIIhat}
2\cdot 3^{1/3}(\log t)^{2/3}<(2-\hat \xi)t^{2/3}< C(\log t)^{2/3}, \quad C>2\cdot 3^{1/3},
\end{equation}
which finally leads to the proof of part (c) of Theorem \ref{mainthm}. Due to a technical reason, it is also assumed that $r\in H^{s}$ with $s>5/2$ throughout this section.

%Between the $\mathcal{R}_I$ and the 1st oscillatory region, the collisionless shock
%region $\mathcal{R}_{III}$ occurs generically (i.e., $|r(\pm 1)|=1$).

%A new asymptotic formula is given by the descent analysis below, which need the condition that $r\in H^s(\mathbb{R})$ with $s>5/2$
%Between the 1st transition region and the 1st oscillatory region, the collisionless shock region (or say ``shock wave region'') occurs generically (i.e., $\vert r(\pm 1)\vert=1$).
%In this case, the diagonal entry $(1-|r|^2)^{-1}$ in the factorization of jump matrix becomes unbounded as $z\to\pm 1$, which prevents us to use the same arguments in
%the previous sections. A new asymptotic formula is given by the descent analysis below, which need the condition that $r\in H^s(\mathbb{R})$ with $s>5/2$.

\subsection{Preliminaries}\label{subsec:3rd appofthetafunc}
%\subsubsection*{Open $\bar{\partial}$ lenses: $M^{(3)}\to M^{(4)}$}
Since the function $R$ defined in \eqref{def:R=rT} is not an analytical
function, the idea now is still to introduce auxiliary functions $d_j(z)$ with
boundary conditions defined in \eqref{equ: 1sttran defofd_j}
to open $\bar{\partial}$ lenses. In the present case, the lenses start from the endpoints
of two intervals which contain the saddle points $k_{1,2}$ and $k_{3,4}$ given in \eqref{equ:1sttransaddlepoints-a} and \eqref{equ:1sttransaddlepoints-b}, respectively.
%Varing from opening $\bar{\partial}$ lenses at sadlle ponits
%$k_j$, $j=1,2,3,4$ defined in \eqref{equ:1sttransaddlepoints-a} and
%\eqref{equ:1sttransaddlepoints-b}, we will open lenses at some endpoints
%of two intervals which include $k_{1,2}$ and $k_{3,4}$, respectively.
More precisely, let
\begin{equation}\label{def:Url}
	U^{(r)}=\left\{z: |z-1|\leqslant\left(\frac{\log t}{t}\right)^{1/3}\delta(t)\right\}, \ \
	U^{(l)}=\left\{z: |z+1|\leqslant\left(\frac{\log t}{t}\right)^{1/3}\delta(t)\right\}
\end{equation}
be two disks around $z=\pm 1$, respectively, where
\begin{align}\label{def:deltat}
	\delta(t)=\mathcal{O}\left(t^{1/6-\delta_3}(\log t)^{-2/3}\right), \qquad \delta_3\in (0,1/6),
\end{align}
is a constant depending on $t$. For $t$ large enough, under the condition \eqref{def:RIIIhat},
%that $2\cdot 3^{\frac{1}{3}}(\log t)^{2/3}<(2-\hat{\xi})t^{2/3}<C(\log t)^{2/3}$,
it is readily seen that
$k_{1,2}\in U^{(r)}$ and $k_{3,4}\in U^{(l)}$.

Note that $U^{(l)} \cap \mathbb{R}=(e_l,h_l)$, $U^{(r)} \cap \mathbb{R}=(e_r,h_r)$ (see Figure \ref{fig: jumpM^{(4)} 3rd}), we then 
open $\bar{\partial}$ lenses around the intervals $(-\infty,e_l)\cup(h_l,e_r)\cup(h_r,+\infty)$ and define
\begin{equation}\label{transform: M3toM4; 3rdtrans}
	M^{(4)}(z)=M^{(3)}(z)R^{(3)}(z),
\end{equation}
where $R^{(3)}(z)$ is defined in \eqref{defofR^{(3)}} with the region $\Omega_j$ therein modified accordingly. It then follows that $M^{(4)}$ satisfies a mixed $\bar{\partial}$-RH problem similar to $\bar{\partial}$-RH problem \ref{RHP:mixed RH problem} but with jump contour illustrated in Figure \ref{fig: jumpM^{(4)} 3rd}.
Moreover, we have from  \eqref{defofR^{(3)}} that
\begin{equation}\label{equ: M4=M3 at 0 and i}
	M^{(4)}(0)=M^{(3)}(0),\qquad   M^{(4)}(i)=M^{(3)}(i).
\end{equation}

\begin{figure}[htbp]
	\begin{center}
\tikzset{every picture/.style={line width=0.75pt}} %set default line width to 0.75pt
\begin{tikzpicture}[x=0.75pt,y=0.75pt,yscale=-1,xscale=1]
%uncomment if require: \path (0,300); %set diagram left start at 0, and has height of 300
%Straight Lines [id:da1130194326977223]
\draw    (227.43,150.29) -- (288.43,150.29) ;
\draw [shift={(227.43,150.29)}, rotate = 0] [color={rgb, 255:red, 0; green, 0; blue, 0 }  ][fill={rgb, 255:red, 0; green, 0; blue, 0 }  ][line width=0.75]      (0, 0) circle [x radius= 3.35, y radius= 3.35]   ;
%Straight Lines [id:da6542405021773341]
\draw    (355.43,150.29) -- (416.43,150.29) ;
\draw [shift={(416.43,150.29)}, rotate = 0] [color={rgb, 255:red, 0; green, 0; blue, 0 }  ][fill={rgb, 255:red, 0; green, 0; blue, 0 }  ][line width=0.75]      (0, 0) circle [x radius= 3.35, y radius= 3.35]   ;
%Straight Lines [id:da6007950155802195]
\draw    (288.43,150.29) -- (355.43,150.29) ;
\draw [shift={(355.43,150.29)}, rotate = 0] [color={rgb, 255:red, 0; green, 0; blue, 0 }  ][fill={rgb, 255:red, 0; green, 0; blue, 0 }  ][line width=0.75]      (0, 0) circle [x radius= 3.35, y radius= 3.35]   ;
\draw [shift={(288.43,150.29)}, rotate = 0] [color={rgb, 255:red, 0; green, 0; blue, 0 }  ][fill={rgb, 255:red, 0; green, 0; blue, 0 }  ][line width=0.75]      (0, 0) circle [x radius= 3.35, y radius= 3.35]   ;
%Straight Lines [id:da011159212238893268]
\draw  [dash pattern={on 0.84pt off 2.51pt}]  (416.43,150.29) -- (496.43,150.29) ;
\draw [shift={(498.43,150.29)}, rotate = 180] [color={rgb, 255:red, 0; green, 0; blue, 0 }  ][line width=0.75]    (10.93,-3.29) .. controls (6.95,-1.4) and (3.31,-0.3) .. (0,0) .. controls (3.31,0.3) and (6.95,1.4) .. (10.93,3.29)   ;
%Straight Lines [id:da09056351352122438]
\draw  [dash pattern={on 0.84pt off 2.51pt}]  (145.43,150.29) -- (227.43,150.29) ;
%Straight Lines [id:da9228155183726048]
\draw    (416.43,150.29) -- (477.43,119.29) ;
\draw [shift={(452.28,132.07)}, rotate = 153.06] [color={rgb, 255:red, 0; green, 0; blue, 0 }  ][line width=0.75]    (10.93,-3.29) .. controls (6.95,-1.4) and (3.31,-0.3) .. (0,0) .. controls (3.31,0.3) and (6.95,1.4) .. (10.93,3.29)   ;
%Straight Lines [id:da02994200487788623]
\draw    (416.43,150.29) -- (479.43,179.29) ;
\draw [shift={(453.38,167.29)}, rotate = 204.72] [color={rgb, 255:red, 0; green, 0; blue, 0 }  ][line width=0.75]    (10.93,-3.29) .. controls (6.95,-1.4) and (3.31,-0.3) .. (0,0) .. controls (3.31,0.3) and (6.95,1.4) .. (10.93,3.29)   ;
%Straight Lines [id:da07450988211745035]
\draw    (166.43,181.29) -- (227.43,150.29) ;
\draw [shift={(202.28,163.07)}, rotate = 153.06] [color={rgb, 255:red, 0; green, 0; blue, 0 }  ][line width=0.75]    (10.93,-3.29) .. controls (6.95,-1.4) and (3.31,-0.3) .. (0,0) .. controls (3.31,0.3) and (6.95,1.4) .. (10.93,3.29)   ;
%Straight Lines [id:da5774778687404671]
\draw    (164.43,121.29) -- (227.43,150.29) ;
\draw [shift={(201.38,138.29)}, rotate = 204.72] [color={rgb, 255:red, 0; green, 0; blue, 0 }  ][line width=0.75]    (10.93,-3.29) .. controls (6.95,-1.4) and (3.31,-0.3) .. (0,0) .. controls (3.31,0.3) and (6.95,1.4) .. (10.93,3.29)   ;
%Straight Lines [id:da7935789307691992]
\draw    (322,116) -- (288.43,150.29) ;
\draw [shift={(310.11,128.14)}, rotate = 134.4] [color={rgb, 255:red, 0; green, 0; blue, 0 }  ][line width=0.75]    (10.93,-3.29) .. controls (6.95,-1.4) and (3.31,-0.3) .. (0,0) .. controls (3.31,0.3) and (6.95,1.4) .. (10.93,3.29)   ;
%Straight Lines [id:da6564113745972637]
\draw    (355.43,150.29) -- (323,186) ;
\draw [shift={(343.92,162.96)}, rotate = 132.24] [color={rgb, 255:red, 0; green, 0; blue, 0 }  ][line width=0.75]    (10.93,-3.29) .. controls (6.95,-1.4) and (3.31,-0.3) .. (0,0) .. controls (3.31,0.3) and (6.95,1.4) .. (10.93,3.29)   ;
%Straight Lines [id:da20373975567519786]
\draw    (355.43,150.29) -- (322,116) ;
\draw [shift={(343.6,138.15)}, rotate = 225.73] [color={rgb, 255:red, 0; green, 0; blue, 0 }  ][line width=0.75]    (10.93,-3.29) .. controls (6.95,-1.4) and (3.31,-0.3) .. (0,0) .. controls (3.31,0.3) and (6.95,1.4) .. (10.93,3.29)   ;
%Straight Lines [id:da2257938159502304]
\draw    (323,186) -- (288.43,150.29) ;
\draw [shift={(310.58,173.17)}, rotate = 225.93] [color={rgb, 255:red, 0; green, 0; blue, 0 }  ][line width=0.75]    (10.93,-3.29) .. controls (6.95,-1.4) and (3.31,-0.3) .. (0,0) .. controls (3.31,0.3) and (6.95,1.4) .. (10.93,3.29)   ;
%Straight Lines [id:da006063603480674917]
\draw    (322,116) -- (321.93,150.29) ;
\draw [shift={(321.95,139.14)}, rotate = 270.12] [color={rgb, 255:red, 0; green, 0; blue, 0 }  ][line width=0.75]    (10.93,-3.29) .. controls (6.95,-1.4) and (3.31,-0.3) .. (0,0) .. controls (3.31,0.3) and (6.95,1.4) .. (10.93,3.29)   ;
%Straight Lines [id:da7334496410156579]
\draw    (323,186) -- (321.93,150.29) ;
\draw [shift={(322.28,162.15)}, rotate = 88.28] [color={rgb, 255:red, 0; green, 0; blue, 0 }  ][line width=0.75]    (10.93,-3.29) .. controls (6.95,-1.4) and (3.31,-0.3) .. (0,0) .. controls (3.31,0.3) and (6.95,1.4) .. (10.93,3.29)   ;
%Straight Lines [id:da7551837369483863]
\draw  [dash pattern={on 0.84pt off 2.51pt}]  (322.43,86.29) -- (322.43,122.29) ;
\draw [shift={(322.43,84.29)}, rotate = 90] [color={rgb, 255:red, 0; green, 0; blue, 0 }  ][line width=0.75]    (10.93,-3.29) .. controls (6.95,-1.4) and (3.31,-0.3) .. (0,0) .. controls (3.31,0.3) and (6.95,1.4) .. (10.93,3.29)   ;
%Straight Lines [id:da0016558132794155522]
\draw  [dash pattern={on 0.84pt off 2.51pt}]  (321.43,178.29) -- (321.43,213.29) ;
%Shape: Circle [id:dp4629854987404509]
\draw  [dash pattern={on 0.84pt off 2.51pt}] (355.43,150.29) .. controls (355.43,133.44) and (369.08,119.79) .. (385.93,119.79) .. controls (402.77,119.79) and (416.43,133.44) .. (416.43,150.29) .. controls (416.43,167.13) and (402.77,180.79) .. (385.93,180.79) .. controls (369.08,180.79) and (355.43,167.13) .. (355.43,150.29) -- cycle ;
%Shape: Circle [id:dp009635430393651001]
\draw  [dash pattern={on 0.84pt off 2.51pt}] (227.43,150.29) .. controls (227.43,133.44) and (241.08,119.79) .. (257.93,119.79) .. controls (274.77,119.79) and (288.43,133.44) .. (288.43,150.29) .. controls (288.43,167.13) and (274.77,180.79) .. (257.93,180.79) .. controls (241.08,180.79) and (227.43,167.13) .. (227.43,150.29) -- cycle ;
% Text Node
\draw (352.43,155.69) node [anchor=north west][inner sep=0.75pt]  [font=\tiny]  {$e_{r}$};
% Text Node
\draw (410.43,155.69) node [anchor=north west][inner sep=0.75pt]  [font=\tiny]  {$h_{r}$};
% Text Node
\draw (500.43,153.69) node [anchor=north west][inner sep=0.75pt]  [font=\scriptsize]  {$\re z$};
% Text Node
\draw (322.43,67.69) node [anchor=north west][inner sep=0.75pt]  [font=\scriptsize]  {$\im z$};
% Text Node
\draw (484,106.4) node [anchor=north west][inner sep=0.75pt]  [font=\scriptsize]  {$\Sigma _{1}$};
% Text Node
\draw (193,109.4) node [anchor=north west][inner sep=0.75pt]  [font=\scriptsize]  {$\Sigma {_{4}}$};
% Text Node
\draw (191.93,169.19) node [anchor=north west][inner sep=0.75pt]  [font=\scriptsize]  {$\Sigma _{4}^{*}$};
% Text Node
\draw (486,180.4) node [anchor=north west][inner sep=0.75pt]  [font=\scriptsize]  {$\Sigma _{1}^{*}$};
% Text Node
\draw (338,117.4) node [anchor=north west][inner sep=0.75pt]  [font=\scriptsize]  {$\Sigma _{2}$};
% Text Node
\draw (335.2,170.4) node [anchor=north west][inner sep=0.75pt]  [font=\scriptsize]  {$\Sigma _{2}^{*}$};
% Text Node
\draw (285,115.4) node [anchor=north west][inner sep=0.75pt]  [font=\scriptsize]  {$\Sigma _{3}$};
% Text Node
\draw (290,167.4) node [anchor=north west][inner sep=0.75pt]  [font=\scriptsize]  {$\Sigma _{3}^{*}$};
% Text Node
\draw (276.43,155.69) node [anchor=north west][inner sep=0.75pt]  [font=\tiny]  {$h_{l}$};
% Text Node
\draw (222.43,155.69) node [anchor=north west][inner sep=0.75pt]  [font=\tiny]  {$e_{l}$};
% Text Node
\draw (316,103.4) node [anchor=north west][inner sep=0.75pt]  [font=\scriptsize]  {$\Sigma _{23}$};
% Text Node
\draw (315,186.4) node [anchor=north west][inner sep=0.75pt]  [font=\scriptsize]  {$\Sigma _{23}^{*}$};
% Text Node
\draw (464,129.4) node [anchor=north west][inner sep=0.75pt]  [font=\scriptsize]  {$\Omega _{1}$};
% Text Node
\draw (465,152.4) node [anchor=north west][inner sep=0.75pt]  [font=\scriptsize]  {$\Omega _{1}^{*}$};
% Text Node
\draw (322.18,130) node [anchor=north west][inner sep=0.75pt]  [font=\scriptsize]  {$\Omega _{2}$};
% Text Node
\draw (323,150.4) node [anchor=north west][inner sep=0.75pt]  [font=\scriptsize]  {$\Omega _{2}^{*}$};
% Text Node
\draw (302.18,132.69) node [anchor=north west][inner sep=0.75pt]  [font=\scriptsize]  {$\Omega _{3}$};
% Text Node
\draw (303.18,150.69) node [anchor=north west][inner sep=0.75pt]  [font=\scriptsize]  {$\Omega _{3}^{*}$};
% Text Node
\draw (160.18,132.69) node [anchor=north west][inner sep=0.75pt]  [font=\scriptsize]  {$\Omega _{4}$};
% Text Node
\draw (159.18,151.69) node [anchor=north west][inner sep=0.75pt]  [font=\scriptsize]  {$\Omega _{4}^{*}$};
\end{tikzpicture}
\caption{ The jump contour of the $\bar{\partial}$-RH problem for $M^{(4)}$.}\label{fig: jumpM^{(4)} 3rd}
\end{center}
\end{figure}
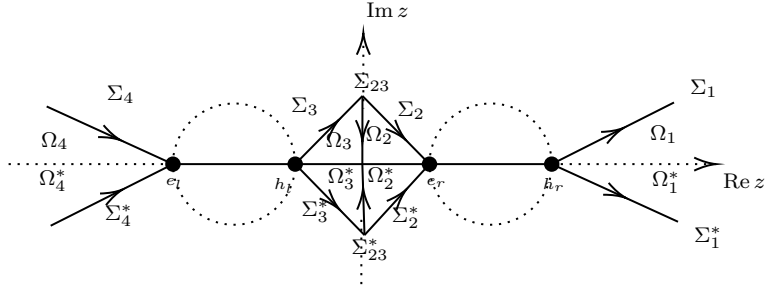

The above mixed $\bar{\partial}$-RH problem can again be decomposed into a pure RH problem for $N$  and a pure $\bar{\partial}$-problem. As $t \to +\infty$, the main contribution to $M^{(4)}$ comes from $N$. Moreover, the jump matrix of $N$ tends to the identity matrix exponentially fast expect for that on the intervals $(e_r,h_r)$ and $(e_l,h_l)$, which leads us to consider the following RH problem for  $N^{(j)}$, $j\in\{r,l\}$.

\begin{RHP}
	\hfill
	\begin{itemize}
		\item[$\bullet$] $N^{(j)}(z)$ is holomorphic for $z\in\mathbb{C}\setminus [e_j,h_j]$.
		\item[$\bullet$] For $z \in (e_j, h_j)$, we have
		\begin{equation}
			N^{(j)}_{+}(z)=N^{(j)}_{-}(z)V^{(j)}(z),
		\end{equation}	
		where
		\begin{equation*}
			V^{(j)}(z)=\left(\begin{array}{cc}
				1-|R(z)|^2 & R(z)e^{2i\theta(z)}\\
				-\bar{R}(z)e^{-2i\theta(z)} & 1
			\end{array}\right).
		\end{equation*}
		\item[$\bullet$]
        As $z\rightarrow\infty$ in
		$\mathbb{C}\setminus [e_j,h_j]$, we have
		$N^{(j)}(z)=I+\mathcal{O}(z^{-1})$.
	\end{itemize}
\end{RHP}
%$N^{(j)}(z)$, $j\in\{r,l\}$ could be reduced to some model RH problems below. Except the jump condition of the intervals $(e_r, h_r)$ and $(e_l, h_l)$ on $\mathbb{R}$, the jump matrices exponentially decay to $I$ as $t\to\infty$.
%The radius of the local is $\left(\frac{\log t}{t}\right)^{1/3}\delta(t)$, where
%\begin{align*}
%\delta(t)=o\left(t^{1/6}(\log t)^{-2/3} \right) .
%\end{align*}
%For later use, it is straightforward seen that the following relation holds
In conclusion, we have
\begin{align}\label{equ: M4}
	M^{(4)}(z)=\left(N^{(r)}(z)+N^{(l)}(z)\right)\left(I+o(1)\right)-I, \qquad t \to +\infty.
\end{align}
We will give a detailed asymptotic analysis of the RH problem for $N^{(r)}$ below, which also fits to $N^{(l)}$.
\begin{remark}
We simply use $o(1)$ as the error bound in \eqref{equ: M4}, which essentially consists of the $\bar{\partial}$ error from  the pure $\bar{\partial}$ problem as well as the exponential error by opening $\bar{\partial}$ lenses.  From now on, we ignore delicate error analysis and use $o(1)$ to represent the error term. The reason is that, on one hand, it limits the length of the paper; on the other hand, we are mainly concerned
with the subsequent term after the background wave in the asymptotic formula. Be that as it may, we still point out the sources of the error term whenever possible.
\end{remark}

\subsection{Asymptotic analysis of the RH problem for $N^{(r)}$}
To solve the RH problem for $N^{(r)}$ as $t\to +\infty$, we introduce the scaled spectral variable
\begin{align}\label{equ:3rdtranhat{k}scale}
	\hat{k}=\sqrt{\frac{12p}{q}}(2-\hat{\xi})^{-1/2}(z-1),
\end{align}
where $p$ and $q$ are any two fixed positive constants. It is then readily seen that if $2\cdot 3^{\frac{1}{3}}(\log t)^{2/3}<(2-\hat{\xi})t^{2/3}<C(\log t)^{2/3}$, $C>2\cdot 3^{\frac{1}{3}}$,
\begin{align}
	\theta(z(\hat{k}))=\tau\hat{\theta}(\hat{k})+\mathcal{O}\left( (\log t)^{4/3}t^{-1/3}\hat{k}^2\right), \qquad t\to +\infty,
 \label{equ:theta5.3}
\end{align}
where
\begin{align}\label{equ:tau}
	\tau:=t(2-\hat{\xi})^{3/2}\sqrt{\frac{q}{48p^3}}
\end{align}
parametrizes the space-time region and
\begin{align}\label{def:hattheta}
	 \hat{\theta}(\hat{k})=p\hat{k}-q\hat{k}^3
\end{align}

By \eqref{def:RIIIhat}, \eqref{def:Url}, \eqref{def:deltat}, \eqref{equ:3rdtranhat{k}scale} and  \eqref{equ:theta5.3}, we have
\begin{align}
	\theta(z(\hat{k}))=\tau\hat{\theta}(\hat{k})+\mathcal{O}\left( t^{-2\delta_3}\right),
\end{align}
and
\begin{align}
	\hat{e}_r:&=\sqrt{\frac{12p}{q}}(2-\hat{\xi})^{-1/2}(e_r-1)=\mathcal{O}\left(t^{1/6-\delta_3}(\log t)^{-2/3}\right),
\\
\hat{h}_r:&=\sqrt{\frac{12p}{q}}(2-\hat{\xi})^{-1/2}(h_r-1)=\mathcal{O}\left(t^{1/6-\delta_3}(\log t)^{-2/3}\right).
\end{align}
Thus, $N^{(r)}(\hat{k})=N^{(r)}(z(\hat{k}))$ is approximated by the
following RH problem for $N^{(r1)}(\hat{k})$.
\begin{RHP}\label{RHP:Nr1}
	\hfill
	\begin{itemize}
		\item[$\bullet$] $N^{(r1)}(\hat{k})$ is holomorphic for
		$\hat{k}\in\mathbb{C}\setminus \mathbb{R}$.
		\item[$\bullet$] For $\hat{k}\in \mathbb{R}$, we have
		\begin{equation}
		N^{(r1)}_{+}(\hat{k})=N^{(r1)}_{-}(\hat{k})\hat{V}^{(r1)}(\hat{k}),
		\end{equation}
		where
		\begin{align}
			\hat{V}^{(r1)}(\hat{k})=\left(\begin{array}{cc}
				1-|R(z(\hat{k}))|^2 & R(z(\hat{k}))e^{2i\tau\hat{\theta}(z(\hat{k}))}\\
				-\bar{R}(z(\hat{k}))e^{-2i\tau\hat{\theta}(z(\hat{k}))} & 1
			\end{array}\right).
		\end{align}
		\item[$\bullet$] As $\hat{k}\rightarrow\infty$ in $\mathbb{C}\setminus \mathbb{R}$, we have $N^{(r1)}(\hat{k})=I+\mathcal{O}(\hat{k}^{-1})$.
	\end{itemize}
\end{RHP}
Moreover, it is readily seen that
%As $2-\hat{\xi}\to0$, $\tau\to+\infty$, it follows from \eqref{equ:theta5.3} that
\begin{align}
	N^{(r)}(\hat{k})=N^{(r1)}(\hat{k})\left(I+o(1)\right), \qquad t\to +\infty.
\label{equ:trans Nr1}
\end{align}
We next introduce the so-called $g$-function to further analyze the RH problem for $N^{(r1)}$.

\subsubsection*{The $g$-function mechanism}\label{subsec: g-function}
%Main analytical tool used to solve the RH problem for $N^{(r1)}$
%is the $g$-function mechanism in what follows.

We start with a hyperelliptic surface $\mathcal{M}$ defined as the set of all points $P:=(w,\hat{k})\in\mathbb{C}^2$ such that
\begin{align}\label{eq:algeb}
	w^2(\hat{k})=(\hat{k}^2-a^2)(\hat{k}^2-b^2),
\end{align}
where $0\leqslant a\leqslant b$. Together with two points $\infty^+$ and $\infty^-$ at infinity, $\mathcal{M}$ is a compact Riemann surface of genus 1. Its branch cuts and the canonical homology
basis $\{a_1, b_1\}$ ($a_1$-cycle and $b_1$-cycle) are illustrated in Figure \ref{fig: fig g}.
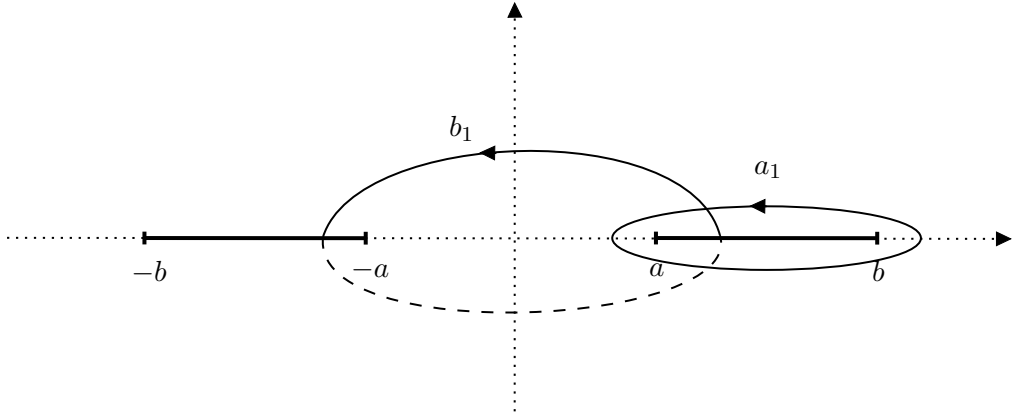
\begin{figure}[htbp]
\centering
\tikzset{every picture/.style={line width=0.75pt}} %set default line width to 0.75pt
\begin{tikzpicture}[x=0.75pt,y=0.75pt,yscale=-0.8,xscale=0.8]
	%uncomment if require: \path (0,317); %set diagram left start at 0, and has height of 317
	%Straight Lines [id:da27738470004786975]
	\draw  [dash pattern={on 0.84pt off 2.51pt}]  (9.8,170.87) -- (632.8,171.53) ;
	\draw [shift={(635.8,171.53)}, rotate = 180.06] [fill={rgb, 255:red, 0; green, 0; blue, 0 }  ][line width=0.08]  [draw opacity=0] (8.93,-4.29) -- (0,0) -- (8.93,4.29) -- cycle    ;
	%Straight Lines [id:da9086027715275589]
	\draw  [dash pattern={on 0.84pt off 2.51pt}]  (326.4,280) -- (326.44,26) ;
	\draw [shift={(326.4,23.2)}, rotate = 89.17] [fill={rgb, 255:red, 0; green, 0; blue, 0 }  ][line width=0.08]  [draw opacity=0] (8.93,-4.29) -- (0,0) -- (8.93,4.29) -- cycle    ;
	%Straight Lines [id:da05289794097147382]
	\draw [line width=1.5]    (95.5,171) -- (233.5,171.1) ;
	\draw [shift={(233.5,171.1)}, rotate = 180.04] [color={rgb, 255:red, 0; green, 0; blue, 0 }  ][line width=1.5]    (0,4.02) -- (0,-4.02)   ;
	\draw [shift={(95.5,171)}, rotate = 180.04] [color={rgb, 255:red, 0; green, 0; blue, 0 }  ][line width=1.5]    (0,4.02) -- (0,-4.02)   ;
	%Straight Lines [id:da283651499397896]
	\draw [line width=1.5]    (414.5,171) -- (552.5,171.1) ;
	\draw [shift={(552.5,171.1)}, rotate = 180.04] [color={rgb, 255:red, 0; green, 0; blue, 0 }  ][line width=1.5]    (0,4.02) -- (0,-4.02)   ;
	\draw [shift={(414.5,171)}, rotate = 180.04] [color={rgb, 255:red, 0; green, 0; blue, 0 }  ][line width=1.5]    (0,4.02) -- (0,-4.02)   ;
	%Shape: Ellipse [id:dp05965180939622017]
	\draw   (387.25,171.05) .. controls (387.25,160) and (430.34,151.05) .. (483.5,151.05) .. controls (536.66,151.05) and (579.75,160) .. (579.75,171.05) .. controls (579.75,182.1) and (536.66,191.05) .. (483.5,191.05) .. controls (430.34,191.05) and (387.25,182.1) .. (387.25,171.05) -- cycle ;
	%Straight Lines [id:da10014974352481887]
	\draw    (476.5,150.91) -- (483.5,151.05) ;
	\draw [shift={(473.5,150.85)}, rotate = 1.15] [fill={rgb, 255:red, 0; green, 0; blue, 0 }  ][line width=0.08]  [draw opacity=0] (8.93,-4.29) -- (0,0) -- (8.93,4.29) -- cycle    ;
	%Curve Lines [id:da4220771340534617]
	\draw    (206.9,170.2) .. controls (228.4,102.2) and (434.4,94.7) .. (454.4,169.2) ;
	%Straight Lines [id:da5181469444759406]
	\draw    (308,117.41) -- (315,117.55) ;
	\draw [shift={(305,117.35)}, rotate = 1.15] [fill={rgb, 255:red, 0; green, 0; blue, 0 }  ][line width=0.08]  [draw opacity=0] (8.93,-4.29) -- (0,0) -- (8.93,4.29) -- cycle    ;
	%Curve Lines [id:da6990136332654757]
	\draw  [dash pattern={on 4.5pt off 4.5pt}]  (206.9,170.2) .. controls (198.4,240.7) and (471.4,226.7) .. (454.4,169.2) ;
	% Text Node
	\draw (85.5,183.5) node [anchor=north west][inner sep=0.75pt]    {$-b$};
	% Text Node
	\draw (222.5,183) node [anchor=north west][inner sep=0.75pt]    {$-a$};
	% Text Node
	\draw (408.4,185) node [anchor=north west][inner sep=0.75pt]    {$a$};
	% Text Node
	\draw (547.4,183.5) node [anchor=north west][inner sep=0.75pt]    {$b$};
	% Text Node
	\draw (473.9,120.9) node [anchor=north west][inner sep=0.75pt]    {$a_{1}$};
	% Text Node
	\draw (283.6,92.4) node [anchor=north west][inner sep=0.75pt]    {$b_{1}$};
\end{tikzpicture}
	\caption{ The homology basis of the Riemann surface $\mathcal{M}$ associated with $w(\hat{k})=\sqrt{(\hat{k}^2-a^2)(\hat{k}^2-b^2)}$.}
	\label{fig: fig g}.
\end{figure}

The $g$-function is then defined by
\begin{align}\label{equ:g func}
	g(\hat{k}):=-3q\int_{b}^{\hat{k}} w(\zeta)\dif\zeta+B_1/4,
\end{align}
where
\begin{equation}\label{def:B1}
	  B_1=-3q\oint_{b_1}w(\zeta)\dif\zeta.
\end{equation}
It is easily seen that $g'(a)=g'(b)=0$. To determine the parameters $a$ and $b$, we require that $g(\hat{k}) \sim \theta(\hat{k})$ as $\hat{k}\rightarrow\infty$, which yields
\begin{equation}
	a^2+b^2=\frac{2p}{3q}.\label{equ: a,b con1}
\end{equation}
%Observe that for
%\begin{align}
%	0\leqslant a\leqslant b, \ a^2+b^2=\frac{2p}{3q},
%\end{align}
%the expression
The other condition is given by the equation
\begin{align}
	\int_a^b\sqrt{(\zeta^2-a^2)(b^2-\zeta^2)}\dif\zeta=-\frac{2\sqrt{3}p^{3/2}\log (2-\hat{\xi})}{3q^{3/2}(2-\hat{\xi})^{3/2}t}. \label{equ:a,b con2}
\end{align}
By \eqref{def:RIIIhat}, it's readily seen that
\begin{equation}
	0<-\frac{2\sqrt{3}p^{3/2}\log (2-\hat{\xi})}{3q^{3/2}(2-\hat{\xi})^{3/2}t}<\sqrt{\frac{2p^3}{81q^3}}.
\end{equation}
With the aid of \eqref{equ: a,b con1},  $\int_a^b\sqrt{(\zeta^2-a^2)(b^2-\zeta^2)}\dif\zeta$ is a monotone function with
respect to $a$, which decreases from $\sqrt{\frac{2p^3}{81q^3}}$ to $0$ when $a$ increases from $0$ to $\sqrt{\frac{p}{3q}}$. Thus, we can indeed determine the non-negative constants $a$ and $b$ uniquely from \eqref{equ: a,b con1} and \eqref{equ:a,b con2}.
%Thus \eqref{equ:a,b con2} defines the parameter $a$ uniquely under the condtion
%$2\cdot 3^{\frac{1}{3}}<(2-\hat{\xi})t^{2/3}<C(\log t)^{2/3}$.
%The existence of $g$-function is assured by the following equality
%\begin{align}
%	\int_a^b\sqrt{(\zeta^2-a^2)(b^2-\zeta^2)}\dif\zeta=-\frac{2\sqrt{3}p^{3/2}\log (2-\hat{\xi})}{3q^{3/2}(2-\hat{\xi})^{3/2}t}.\label{equ:a,b con2}
%\end{align}
%Furthermore, it is observed that the expression $\int_a^b\sqrt{(\zeta^2-a^2)(b^2-\zeta^2)}\dif\zeta$ is a monotone function with
%respect to $a$, which decreases from $\sqrt{\frac{2p^3}{81q^3}}$ to $0$ when $a$ increases from $0$ to $\sqrt{\frac{p}{3q}}$, i.e.,
%\begin{equation}
%	0<-\frac{2\sqrt{3}p^{3/2}\log (2-\hat{\xi})}{3q^{3/2}(2-\hat{\xi})^{3/2}t}<\sqrt{\frac{2p^3}{81q^3}}, \ a\in\left(0, \sqrt{\frac{p}{3q}}\right)
%\end{equation}
%which yields the bound
%\begin{equation}
%	2\cdot 3^{\frac{1}{3}}<(2-\hat{\xi})t^{2/3}<C(\log t)^{2/3}.
%\end{equation}
%Hence \eqref{equ:a,b con2} can determine $a$ uniquely.
\begin{remark}\label{remark a,b}
It is clear that both $a$ and $b$ depend on the choice of $p$ and $q$. If $(\tilde{p},\tilde{q})$ is the other pair of positive constants, one can check that
	$$ \tilde{a}=\left( \frac{q\tilde{p}}{p\tilde{q}}\right)^{1/2} a, \qquad \tilde{b}=\left( \frac{q\tilde{p}}{p\tilde{q}}\right)^{1/2} b$$
satisfies equations \eqref{equ: a,b con1} and \eqref{equ:a,b con2} with $(p,q)$ therein replaced by $(\tilde{p},\tilde{q})$.
% Then $(\tilde{a},\tilde{b})$  admits \eqref{equ: a,b con1} under $(\tilde{p},\tilde{q})$. On the other hand, under denoting $w=\left( \frac{q\tilde{p}}{p\tilde{q}}\right)^{1/2}\zeta$, it appears that
%	\begin{align*}
%		&\int_a^b\sqrt{(\zeta^2-a^2)(b^2-\zeta^2)}\dif\zeta=\int_{\tilde{a}}^{\tilde{b}}\sqrt{(\left( \frac{p\tilde{q}}{q\tilde{p}}\right)w^2-a^2)(b^2-\left( \frac{p\tilde{q}}{q\tilde{p}}\right)w^2)}\left( \frac{p\tilde{q}}{q\tilde{p}}\right)^{1/2}\dif w\\
%		&=\left( \frac{p\tilde{q}}{q\tilde{p}}\right)^{3/2}\int_{\tilde{a}}^{\tilde{b}}\sqrt{(w^2-{\tilde{a}}^2)({\tilde{b}}^2-w^2)}\dif w,
%	\end{align*}
%which implies that
%	$(\tilde{a},\tilde{b})$ also satisfies
%	 \eqref{equ:a,b con2}  under $(\tilde{p},\tilde{q})$.
\end{remark}

From the definition of the $g$-function, the following proposition is immediate.
\begin{Proposition}\label{prop:gfunc}
The function $g$ defined in \eqref{equ:g func} satisfies the following properties.
	\begin{itemize}
		\item[$\bullet$] $g(\hat{k})$ is holomorphic for $\hat{k}\in\mathbb{C}\setminus\left[-b,b\right]$.
		\item[$\bullet$] $g(\hat{k})$ satisfies the jump relations
		\begin{subequations}
			\begin{align}	
				&g_+(\hat{k})+g_-(\hat{k})=B_1/2, &\hat{k}\in(a,b),\label{equ; jump g1}\\
				&g_+(\hat{k})+g_-(\hat{k})=-B_1/2, &\hat{k}\in(-b,-a),\label{equ; jump g2}\\
				&g_+(\hat{k})-g_-(\hat{k})=A_1, &\hat{k}\in(-a,a),\label{equ; jump g3}
			\end{align}
		\end{subequations}
where
\begin{equation}\label{def:A1}
A_1=-3q\oint_{a_1}w(\zeta)\dif\zeta
\end{equation}
and $B_1$ is given in \eqref{def:B1}.
		\item[$\bullet$] As $\hat{k}\rightarrow\infty$ in $\mathbb{C}\setminus [-b,b]$, we have $g(\hat{k})=\hat{\theta}(\hat{k})+\mathcal{O}( \hat{k}^{-1} )$,
where $\hat{\theta}(\hat{k})$ is defined in \eqref{def:hattheta}.
	\end{itemize}
\end{Proposition}
%The signature table for $\im g(\hat{k})$ is depicted in Figure \ref{fig: sign g}.

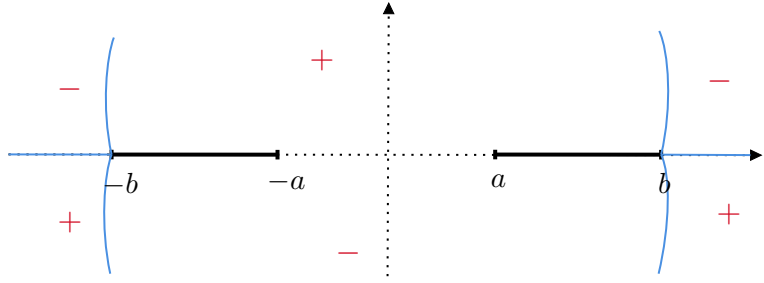
\begin{figure}[htbp]
\begin{center}
\tikzset{every picture/.style={line width=0.75pt}} %set default line width to 0.75pt
\begin{tikzpicture}[x=0.75pt,y=0.75pt,yscale=-0.6,xscale=0.6]
	%uncomment if require: \path (0,290); %set diagram left start at 0, and has height of 290
	%Straight Lines [id:da8487215616385209]
	\draw  [dash pattern={on 0.84pt off 2.51pt}]  (9.8,170.87) -- (632.8,171.53) ;
	\draw [shift={(635.8,171.53)}, rotate = 180.06] [fill={rgb, 255:red, 0; green, 0; blue, 0 }  ][line width=0.08]  [draw opacity=0] (8.93,-4.29) -- (0,0) -- (8.93,4.29) -- cycle    ;
	%Straight Lines [id:da7468507230687398]
	\draw  [dash pattern={on 0.84pt off 2.51pt}]  (325,273) -- (325.99,47) ;
	\draw [shift={(326,44)}, rotate = 90.25] [fill={rgb, 255:red, 0; green, 0; blue, 0 }  ][line width=0.08]  [draw opacity=0] (8.93,-4.29) -- (0,0) -- (8.93,4.29) -- cycle    ;
	%Straight Lines [id:da4809582370006369]
	\draw [line width=1.5]    (95.5,171) -- (233.5,171.1) ;
	\draw [shift={(233.5,171.1)}, rotate = 180.04] [color={rgb, 255:red, 0; green, 0; blue, 0 }  ][line width=1.5]    (0,4.02) -- (0,-4.02)   ;
	\draw [shift={(95.5,171)}, rotate = 180.04] [color={rgb, 255:red, 0; green, 0; blue, 0 }  ][line width=1.5]    (0,4.02) -- (0,-4.02)   ;
	%Straight Lines [id:da7088975242434341]
	\draw [line width=1.5]    (414.5,171) -- (552.5,171.1) ;
	\draw [shift={(552.5,171.1)}, rotate = 180.04] [color={rgb, 255:red, 0; green, 0; blue, 0 }  ][line width=1.5]    (0,4.02) -- (0,-4.02)   ;
	\draw [shift={(414.5,171)}, rotate = 180.04] [color={rgb, 255:red, 0; green, 0; blue, 0 }  ][line width=1.5]    (0,4.02) -- (0,-4.02)   ;
	%Curve Lines [id:da30563138638626786]
	\draw [color={rgb, 255:red, 74; green, 144; blue, 226 }  ,draw opacity=1 ]   (97.4,72.9) .. controls (91.4,88.9) and (87,131.2) .. (95.5,171) ;
	%Curve Lines [id:da45459661227450066]
	\draw [color={rgb, 255:red, 74; green, 144; blue, 226 }  ,draw opacity=1 ]   (95.5,171) .. controls (89.5,187) and (85.9,231.1) .. (94.4,270.9) ;
	%Straight Lines [id:da21412767921703924]
	\draw [color={rgb, 255:red, 74; green, 144; blue, 226 }  ,draw opacity=1 ]   (9.8,170.87) -- (95.5,171) ;
	%Straight Lines [id:da9016458185027922]
	\draw [color={rgb, 255:red, 74; green, 144; blue, 226 }  ,draw opacity=1 ]   (552.5,171.1) -- (626.76,171.46) ;
	%Curve Lines [id:da47417568724538484]
	\draw [color={rgb, 255:red, 74; green, 144; blue, 226 }  ,draw opacity=1 ]   (550.9,67.4) .. controls (558.9,84.4) and (561.9,131) .. (552.5,171.1) ;
	%Curve Lines [id:da27713717850612696]
	\draw [color={rgb, 255:red, 74; green, 144; blue, 226 }  ,draw opacity=1 ]   (552.5,171.1) .. controls (560.5,188.1) and (559.8,230.8) .. (550.4,270.9) ;
	% Text Node
	\draw (85.5,183.5) node [anchor=north west][inner sep=0.75pt]  [xscale=1,yscale=1]  {$-b$};
	% Text Node
	\draw (222.5,183) node [anchor=north west][inner sep=0.75pt]  [xscale=1,yscale=1]  {$-a$};
	% Text Node
	\draw (408.4,185) node [anchor=north west][inner sep=0.75pt]  [xscale=1,yscale=1]  {$a$};
	% Text Node
	\draw (547.4,183.5) node [anchor=north west][inner sep=0.75pt]  [xscale=1,yscale=1]  {$b$};
	% Text Node
	\draw (257.9,80.4) node [anchor=north west][inner sep=0.75pt]  [font=\large,xscale=1,yscale=1]  {$\textcolor[rgb]{0.82,0.01,0.11}{+}$};
	% Text Node
	\draw (596.4,209.4) node [anchor=north west][inner sep=0.75pt]  [font=\large,xscale=1,yscale=1]  {$\textcolor[rgb]{0.82,0.01,0.11}{+}$};
	% Text Node
	\draw (48.4,216.4) node [anchor=north west][inner sep=0.75pt]  [font=\large,xscale=1,yscale=1]  {$\textcolor[rgb]{0.82,0.01,0.11}{+}$};
	% Text Node
	\draw (47.9,104.4) node [anchor=north west][inner sep=0.75pt]  [font=\large,xscale=1,yscale=1]  {$\textcolor[rgb]{0.82,0.01,0.11}{-}$};
	% Text Node
	\draw (279.4,242.1) node [anchor=north west][inner sep=0.75pt]  [font=\large,xscale=1,yscale=1]  {$\textcolor[rgb]{0.82,0.01,0.11}{-}$};
	% Text Node
	\draw (588.4,98.4) node [anchor=north west][inner sep=0.75pt]  [font=\large,xscale=1,yscale=1]  {$\textcolor[rgb]{0.82,0.01,0.11}{-}$};
\end{tikzpicture}
	\caption{ Signature table for $\im g(\hat{k})$. The ``$+$'' represents where $\im g(\hat{k})>0$ and ``$-$'' represents where $\im g(\hat{k})<0$. The blue curve
	is the critical line where $\im g(\hat{k})=0$.}\label{fig: sign g}
\end{center}
\end{figure}

With the help of the $g$-function, we now define
\begin{equation}\label{equ: transtoNr2}
	N^{(r2)}(\hat{k})=N^{(r1)}(\hat{k})e^{i\tau(\theta(\hat{k})-g(\hat{k}))\sigma_3}.
\end{equation}
From RH problem \ref{RHP:Nr1} for $N^{(r1)}$ and Proposition \ref{prop:gfunc}, it is readily seen that
$N^{(r2)}$ satisfies the following RH problem.
\begin{RHP}
	\hfill
	\begin{itemize}
		\item[$\bullet$] $N^{(r2)}(\hat{k})$ is holomorphic for $\hat{k}\in\mathbb{C}\setminus \mathbb{R}$.
		\item[$\bullet$] For $\hat{k}\in \mathbb{R}$, we have
		\begin{align}
			N^{(r2)}_{+}(\hat{k})=N^{(r2)}_{-}(\hat{k})\hat{V}^{(r2)}(\hat{k}),
		\end{align}
		where
		%$N^{(r2)}(\hat{k})$ satisfies the jump relation $N^{(r2)}_{+}(\hat{k})=N^{(r2)}_{-}(\hat{k})\hat{V}^{(r2)}(\hat{k})$, where
		%\begin{align}
		%	\hat{V}^{(r2)}(\hat{k})=\begin{pmatrix}
		%		(1-|R|^2)e^{i\tau\left(g_{+}-g_{-}\right)} & Re^{-i\tau\left(g_{+}+g_{-}\right)} \\
		%		-\bar{R}e^{i\tau \left(g_{+}+g_{-}\right)} & e^{-i\tau\left(g_{+}-g_{-}\right)}
		%	\end{pmatrix}, \hat{k}\in\mathbb{R}.
		%\end{align}
		%Specifically,
		\begin{align}\label{equ: jumpV(r2)}
			\hat{V}^{(r2)}(\hat{k})
&=\begin{pmatrix}
				(1-|R(z(\hat{k}))|^2)e^{-i\tau\left(g_{+}(\hat{k})-g_{-}(\hat{k})\right)} & R(z(\hat{k}))e^{i\tau\left(g_{+}(\hat{k})+g_{-}(\hat{k})\right)} \\
				-\bar{R}(z(\hat{k}))e^{-i\tau \left(g_{+}(\hat{k})+g_{-}(\hat{k})\right)} & e^{i\tau\left(g_{+}(\hat{k})-g_{-}(\hat{k})\right)}
			\end{pmatrix}
\nonumber
\\
&=\left\{\begin{array}{llll}
			 \left(\begin{array}{cc}
			(1-|R(z(\hat{k}))|^2)e^{-i\tau(g_+(\hat{k})-g_-(\hat{k}))} & R(z(\hat{k}))e^{i\tau B_1/2}\\
			-\bar{R}(z(\hat{k}))e^{-i\tau B_1/2} & e^{i\tau(g_+(\hat{k})-g_-(\hat{k}))}
			\end{array}\right), & \hat{k}\in(a, b),\\[10pt]
			\left(\begin{array}{cc}
			(1-|R(z(\hat{k}))|^2)e^{-i\tau(g_+(\hat{k})-g_-(\hat{k}))} & R(z(\hat{k}))e^{-i\tau B_1/2}\\
			-\bar{R}\left(z(\hat{k})\right)e^{i\tau B_1/2} & e^{i\tau\left(g_{+}(\hat{k})-g_{-}(\hat{k})\right)}
			\end{array}\right), & \hat{k}\in (-b, -a),\\[10pt]
			\left(\begin{array}{cc}
			(1-|R(z(\hat{k}))|^2)e^{-i\tau A_1} & R(z(\hat{k}))e^{i\tau \left(g_{+}(\hat{k})+g_{-}(\hat{k})\right)}\\
			-\bar{R}(z(\hat{k}))e^{-i\tau \left(g_{+}(\hat{k})+g_{-}(\hat{k})\right)} & e^{i\tau A_1}
			\end{array}\right), & \hat{k}\in (-a, a), \\[10pt]
			\left(\begin{array}{cc}
			1-|R(z(\hat{k}))|^2 & R(z(\hat{k}))e^{2i\tau g(\hat{k})}\\
			-\bar{R}(z(\hat{k}))e^{-2i\tau g(\hat{k})} & 1
			\end{array}\right), & |\hat{k}|>b.
			\end{array}\right.
		\end{align}
		\item[$\bullet$]As $\hat{k}\rightarrow\infty$ in $\mathbb{C}\setminus \mathbb{R}$, we have $N^{(r2)}(\hat{k})=I+\mathcal{O}(\hat{k}^{-1})$.
	\end{itemize}
\end{RHP}
\begin{subequations}
The jump matrix $\hat{V}^{(r2)}(\hat{k})$ admits the
standard upper/lower triangular factorizations:
	\begin{align}
		&\left(\begin{array}{cc}
			(1-|R(\hat{k})|^2)e^{-i\tau(g_+(\hat{k})-g_-(\hat{k}))} & R(\hat{k})e^{i\tau(g_+(\hat{k})+g_-(\hat{k}))}\\
			-\bar{R}(\hat{k})e^{-i\tau(g_+(\hat{k})+g_-(\hat{k}))} & e^{i\tau(g_+(\hat{k})-g_-(\hat{k}))}
		\end{array}\right)\nonumber\\
	&=\left(\begin{array}{cc}
		1 & R(\hat{k})e^{2i\tau g_-(\hat{k})}\\
			0 & 1
		\end{array}\right)e^{-i\tau(g_+(\hat{k})-g_-(\hat{k}))\sigma_3}\left(\begin{array}{cc}
			1 & 0\\
			-\bar{R}(\hat{k})e^{-2i\tau g_+(\hat{k})} & 1
		\end{array}\right),
%\label{equ: fenjie1}
\\
	&=\left(\begin{array}{cc}
	1 & 0\\
		\frac{-\bar{R}(\hat{k})e^{-2i\tau g_-(\hat{k})} }{1-|R(\hat{k})|^2}& 1
	\end{array}\right)(	1-|R(\hat{k})|^2)^{\sigma_3}e^{-i\tau(g_+(\hat{k})-g_-(\hat{k}))\sigma_3}\left(\begin{array}{cc}
		1 &\frac{ R(\hat{k})e^{2i\tau g_+(\hat{k})}}{1-|R(\hat{k})|^2}\\
		0 & 1
	\end{array}\right).
%\label{equ: fenjie2}
	\end{align}
\end{subequations}
This, together with the signature table of $\im g(\hat{k})$ shown in Figure \ref{fig: sign g}, suggests us to open $\bar{\partial}$ lenses as illustrated in Figure \ref{fig: jumpN^{(r4)}}.  Following the same strategy used in Section \ref{subsec:openlensTran1}, we arrive at a new mixed $\bar{\partial}$-RH problem for $N^{(r3)}$, which is well approximated by a pure RH problem for $N^{(r4)}$, i.e.,
\begin{equation}\label{equ: transtoNr4}
	N^{(r3)}(\hat{k})=N^{(r4)}(\hat{k})\left(I+o(1)\right).
\end{equation}
The jump contours of the RH problem for $N^{(r4)}$ are shown in Figure \ref{fig: jumpN^{(r4)}}, and the associated jump matrix $\hat{V}^{(r4)}$ is given by
\begin{align}\label{equ: jumpV(r4)}
	\hat{V}^{(r4)}(\hat{k})=&\left\{\begin{array}{llll}
		(1-|R(\hat{k})|^2)^{\sigma_3}e^{-i\tau(g_+(\hat{k})-g_-(\hat{k}))\sigma_3}, & \hat{k}\in (-a, a), \\[10pt]
		\hat{V}^{(r2)}(\hat{k}),& \hat{k}\in(a, b)\cup (-b, -a),\\[8pt]
	\left(\begin{array}{cc}
		1 & 0\\
		-\bar{R}(\pm b)e^{-2i\tau g(\hat{k})} & 1
	\end{array}\right), & \hat{k}\in\Sigma_{\pm b},\\[12pt]
	\left(\begin{array}{cc}
		1 & R(\pm b)e^{2i\tau g(\hat{k})}\\
		0 & 1
	\end{array}\right),& \hat{k}\in\Sigma_{\pm b}^*,\\[12pt]
	\left(\begin{array}{cc}
		1 &\frac{ R(\pm a)e^{2i\tau g(\hat{k})}}{1-|R(\pm a)|^2}\\
		0 & 1
	\end{array}\right),& \hat{k}\in\Sigma_{\pm a},\\[12pt]
	\left(\begin{array}{cc}
		1 & 0\\
		\frac{-\bar{R}(\pm a)e^{-2i\tau g(\hat{k})} }{1-|R(\pm a)|^2}& 1
	\end{array}\right),& \hat{k}\in\Sigma_{\pm a}^*,\\[12pt]
	\left(\begin{array}{cc}
	1 & (d_{a}(\hat{k})-d_{-a}(\hat{k}))e^{2i\tau g(\hat{k})}\\
	0 & 1
	\end{array}\right), & \hat{k}\in\Sigma_{0},\\[12pt]
	\left(\begin{array}{cc}
	1 & 0\\
	(d_{a}^*(\hat{k})-d_{-a}^*(\hat{k}))e^{-2i\tau g(\hat{k})} & 1
	\end{array}\right),& \hat{k}\in\Sigma_{0}^*,
	\end{array}\right.
	\end{align}
	where the auxiliary function $d_{\pm a}(\hat{k})$ can be similarly defined by \eqref{equ: 2ndtran cons d_j} with the
	boundary condition
	\begin{equation*}
		d_{\pm a}(\hat{k})=\left\{\begin{array}{ll}
			-\frac{ R(\hat{k})}{1-|R(\hat{k})|^2}, & \hat{k}\in \mathbb{R},\\
			-\frac{ R(\pm a)}{1-|R(\pm a)|^2},  &\hat{k}\in \Sigma_{\pm a}.
		\end{array}\right.
	\end{equation*}

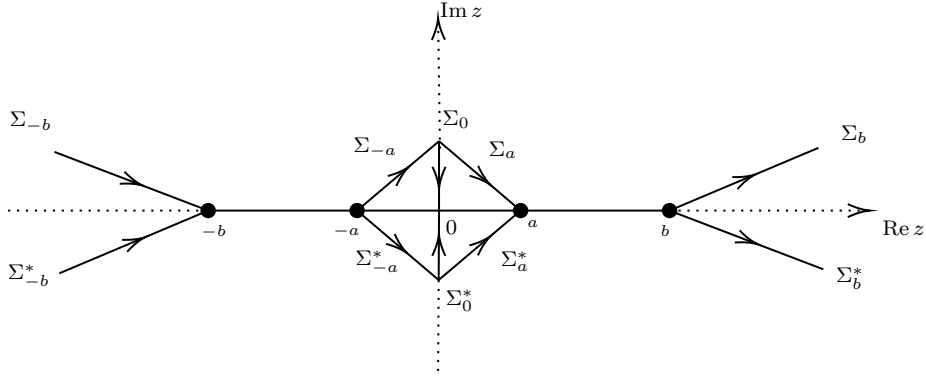
\begin{figure}[htbp]
\begin{center}
\tikzset{every picture/.style={line width=0.75pt}} %set default line width to 0.75pt
\begin{tikzpicture}[x=0.75pt,y=0.75pt,yscale=-1,xscale=1]
%uncomment if require: \path (0,300); %set diagram left start at 0, and has height of 300
%Straight Lines [id:da5170638898118751]
\draw    (205.29,133.84) -- (279.58,133.84) ;
\draw [shift={(205.29,133.84)}, rotate = 0] [color={rgb, 255:red, 0; green, 0; blue, 0 }  ][fill={rgb, 255:red, 0; green, 0; blue, 0 }  ][line width=0.75]      (0, 0) circle [x radius= 3.35, y radius= 3.35]   ;
%Straight Lines [id:da11500966937983792]
\draw    (361.18,133.84) -- (435.47,133.84) ;
\draw [shift={(435.47,133.84)}, rotate = 0] [color={rgb, 255:red, 0; green, 0; blue, 0 }  ][fill={rgb, 255:red, 0; green, 0; blue, 0 }  ][line width=0.75]      (0, 0) circle [x radius= 3.35, y radius= 3.35]   ;
%Straight Lines [id:da8292514446194768]
\draw    (279.58,133.84) -- (361.18,133.84) ;
\draw [shift={(361.18,133.84)}, rotate = 0] [color={rgb, 255:red, 0; green, 0; blue, 0 }  ][fill={rgb, 255:red, 0; green, 0; blue, 0 }  ][line width=0.75]      (0, 0) circle [x radius= 3.35, y radius= 3.35]   ;
\draw [shift={(279.58,133.84)}, rotate = 0] [color={rgb, 255:red, 0; green, 0; blue, 0 }  ][fill={rgb, 255:red, 0; green, 0; blue, 0 }  ][line width=0.75]      (0, 0) circle [x radius= 3.35, y radius= 3.35]   ;
%Straight Lines [id:da3713026746213546]
\draw  [dash pattern={on 0.84pt off 2.51pt}]  (435.47,133.84) -- (533.34,133.84) ;
\draw [shift={(535.34,133.84)}, rotate = 180] [color={rgb, 255:red, 0; green, 0; blue, 0 }  ][line width=0.75]    (10.93,-3.29) .. controls (6.95,-1.4) and (3.31,-0.3) .. (0,0) .. controls (3.31,0.3) and (6.95,1.4) .. (10.93,3.29)   ;
%Straight Lines [id:da47318562319666646]
\draw  [dash pattern={on 0.84pt off 2.51pt}]  (105.43,133.84) -- (205.29,133.84) ;
%Straight Lines [id:da17290825909268714]
\draw    (435.47,133.84) -- (509.76,102.28) ;
\draw [shift={(478.14,115.71)}, rotate = 156.98] [color={rgb, 255:red, 0; green, 0; blue, 0 }  ][line width=0.75]    (10.93,-3.29) .. controls (6.95,-1.4) and (3.31,-0.3) .. (0,0) .. controls (3.31,0.3) and (6.95,1.4) .. (10.93,3.29)   ;
%Straight Lines [id:da7656938514659242]
\draw    (435.47,133.84) -- (512.2,163.38) ;
\draw [shift={(479.43,150.77)}, rotate = 201.05] [color={rgb, 255:red, 0; green, 0; blue, 0 }  ][line width=0.75]    (10.93,-3.29) .. controls (6.95,-1.4) and (3.31,-0.3) .. (0,0) .. controls (3.31,0.3) and (6.95,1.4) .. (10.93,3.29)   ;
%Straight Lines [id:da4789587943128817]
\draw    (131,165.41) -- (205.29,133.84) ;
\draw [shift={(173.67,147.28)}, rotate = 156.98] [color={rgb, 255:red, 0; green, 0; blue, 0 }  ][line width=0.75]    (10.93,-3.29) .. controls (6.95,-1.4) and (3.31,-0.3) .. (0,0) .. controls (3.31,0.3) and (6.95,1.4) .. (10.93,3.29)   ;
%Straight Lines [id:da446258997849901]
\draw    (128.57,104.31) -- (205.29,133.84) ;
\draw [shift={(172.53,121.23)}, rotate = 201.05] [color={rgb, 255:red, 0; green, 0; blue, 0 }  ][line width=0.75]    (10.93,-3.29) .. controls (6.95,-1.4) and (3.31,-0.3) .. (0,0) .. controls (3.31,0.3) and (6.95,1.4) .. (10.93,3.29)   ;
%Straight Lines [id:da7500943879807052]
\draw    (320.47,98.93) -- (279.58,133.84) ;
\draw [shift={(305.35,111.84)}, rotate = 139.5] [color={rgb, 255:red, 0; green, 0; blue, 0 }  ][line width=0.75]    (10.93,-3.29) .. controls (6.95,-1.4) and (3.31,-0.3) .. (0,0) .. controls (3.31,0.3) and (6.95,1.4) .. (10.93,3.29)   ;
%Straight Lines [id:da3784082122549388]
\draw    (361.18,133.84) -- (320.47,98.93) ;
\draw [shift={(346.14,120.94)}, rotate = 220.62] [color={rgb, 255:red, 0; green, 0; blue, 0 }  ][line width=0.75]    (10.93,-3.29) .. controls (6.95,-1.4) and (3.31,-0.3) .. (0,0) .. controls (3.31,0.3) and (6.95,1.4) .. (10.93,3.29)   ;
%Straight Lines [id:da8055111564161896]
\draw  [dash pattern={on 0.84pt off 2.51pt}]  (320.03,39) -- (320.99,105.33) ;
\draw [shift={(320,37)}, rotate = 89.17] [color={rgb, 255:red, 0; green, 0; blue, 0 }  ][line width=0.75]    (10.93,-3.29) .. controls (6.95,-1.4) and (3.31,-0.3) .. (0,0) .. controls (3.31,0.3) and (6.95,1.4) .. (10.93,3.29)   ;
%Straight Lines [id:da9561944718781281]
\draw  [dash pattern={on 0.84pt off 2.51pt}]  (320.3,168.76) -- (320,216) ;
%Straight Lines [id:da8728004255355859]
\draw    (361.18,133.84) -- (320.3,168.76) ;
\draw [shift={(346.06,146.76)}, rotate = 139.5] [color={rgb, 255:red, 0; green, 0; blue, 0 }  ][line width=0.75]    (10.93,-3.29) .. controls (6.95,-1.4) and (3.31,-0.3) .. (0,0) .. controls (3.31,0.3) and (6.95,1.4) .. (10.93,3.29)   ;
%Straight Lines [id:da7253095935496872]
\draw    (320.47,98.93) -- (320.47,134.57) ;
\draw [shift={(320.47,122.75)}, rotate = 270] [color={rgb, 255:red, 0; green, 0; blue, 0 }  ][line width=0.75]    (10.93,-3.29) .. controls (6.95,-1.4) and (3.31,-0.3) .. (0,0) .. controls (3.31,0.3) and (6.95,1.4) .. (10.93,3.29)   ;
%Straight Lines [id:da1651704683811197]
\draw    (320.3,168.76) -- (320.47,134.57) ;
\draw [shift={(320.41,145.67)}, rotate = 90.29] [color={rgb, 255:red, 0; green, 0; blue, 0 }  ][line width=0.75]    (10.93,-3.29) .. controls (6.95,-1.4) and (3.31,-0.3) .. (0,0) .. controls (3.31,0.3) and (6.95,1.4) .. (10.93,3.29)   ;
%Straight Lines [id:da5150770703423206]
\draw    (279.58,133.84) -- (320.3,168.76) ;
\draw [shift={(304.49,155.21)}, rotate = 220.62] [color={rgb, 255:red, 0; green, 0; blue, 0 }  ][line width=0.75]    (10.93,-3.29) .. controls (6.95,-1.4) and (3.31,-0.3) .. (0,0) .. controls (3.31,0.3) and (6.95,1.4) .. (10.93,3.29)   ;
% Text Node
\draw (363.18,137.24) node [anchor=north west][inner sep=0.75pt]  [font=\tiny]  {$a$};
% Text Node
\draw (429.58,139.46) node [anchor=north west][inner sep=0.75pt]  [font=\tiny]  {$b$};
% Text Node
\draw (540.28,137.39) node [anchor=north west][inner sep=0.75pt]  [font=\scriptsize]  {$\re z$};
% Text Node
\draw (319.61,27.81) node [anchor=north west][inner sep=0.75pt]  [font=\scriptsize]  {$\im z$};
% Text Node
\draw (519.62,89.29) node [anchor=north west][inner sep=0.75pt]  [font=\scriptsize]  {$\Sigma _{b}$};
% Text Node
\draw (104.87,82.38) node [anchor=north west][inner sep=0.75pt]  [font=\scriptsize]  {$\Sigma {_{-b}}$};
% Text Node
\draw (103.91,159.29) node [anchor=north west][inner sep=0.75pt]  [font=\scriptsize]  {$\Sigma _{-b}^{*}$};
% Text Node
\draw (517.05,160.71) node [anchor=north west][inner sep=0.75pt]  [font=\scriptsize]  {$\Sigma _{b}^{*}$};
% Text Node
\draw (343.81,97.5) node [anchor=north west][inner sep=0.75pt]  [font=\scriptsize]  {$\Sigma _{a}$};
% Text Node
\draw (349.74,152.7) node [anchor=north west][inner sep=0.75pt]  [font=\scriptsize]  {$\Sigma _{a}^{*}$};
% Text Node
\draw (276.26,95.46) node [anchor=north west][inner sep=0.75pt]  [font=\scriptsize]  {$\Sigma _{-a}$};
% Text Node
\draw (277.35,152.44) node [anchor=north west][inner sep=0.75pt]  [font=\scriptsize]  {$\Sigma _{-a}^{*}$};
% Text Node
\draw (266.28,139.46) node [anchor=north west][inner sep=0.75pt]  [font=\tiny]  {$-a$};
% Text Node
\draw (200.51,139.46) node [anchor=north west][inner sep=0.75pt]  [font=\tiny]  {$-b$};
% Text Node
\draw (320.67,82.24) node [anchor=north west][inner sep=0.75pt]  [font=\scriptsize]  {$\Sigma _{0}$};
% Text Node
\draw (322.3,172.16) node [anchor=north west][inner sep=0.75pt]  [font=\scriptsize]  {$\Sigma _{0}^{*}$};
% Text Node
\draw (322.38,137.24) node [anchor=north west][inner sep=0.75pt]  [font=\scriptsize]  {$0$};
\end{tikzpicture}
\caption{The jump contours of the RH problems for $N^{(r4)}$.}\label{fig: jumpN^{(r4)}}
\end{center}
\end{figure}

\subsubsection*{Reduction to a model RH problem}\label{subsec:3rd transition transform to model RH problem}
%\subsubsection*{Reduction $N^{(r)}$ for a model RH problem}
%The main aim of this subsection is to transform $N^{(r4)}$ to a model RH problem
%associated to Jacobi theta function.
As $t\to +\infty$, it follows from \eqref{equ:3rdtranhat{k}scale} and direct calculations
that
%exists a positive constant $C_R=\frac{q}{12p}\re{r''}(1)>0$
%such that
%depending on $R$ (hence, on the reflection coefficient and the discrete spectrum) such that
\begin{align}\label{equ:R1}
R(z(\hat{k}))\sim 1,  \qquad	1-|R(z(\hat{k}))|^2\sim C_R(2-\hat{\xi})\hat{k}^2,
\end{align}
where
$$
C_R=\frac{q}{12p}\re{r''}(1)>0.
$$
Here, we have made use of the assumption $r\in H^s(\mathbb{R})$ with $s>5/2$  in the second estimate.

On account of the signs of $\im g$ (see Figure \ref{fig: sign g}), it is readily seen from \eqref{equ: jumpV(r4)} that
$\hat{V}^{(r4)} \to I$ exponentially fast as $t \to +\infty$ expect on the interval $(-b,b)$. This, together with \eqref{equ:R1}, implies that $N^{(r4)}$ is approximated by the following RH problem for $N^{(r5)}$ for $t$ large enough.
%and \eqref{equ: fenjie1}--\eqref{equ: fenjie2} that the jump matrices of $N^{(r4)}$
%on the contours $\Sigma_j\cup\Sigma_{j}^{*}$, $j\in\{\pm a, \pm b\}$
%(see also the Figure \ref{fig: jumpN^{(r4)}}) exponentially
%decay to identity as $t\rightarrow\infty$.
%Then, on the account of \eqref{equ:R1}, $N^{(r4)}(\hat{k})$
%could be approxiamted the following RH problem for $N^{(r5)}(\hat{k})$ as $t$ large enough
\begin{RHP}
	\hfill
	\begin{itemize}
	\item[$\bullet$] $N^{(r5)}(\hat{k})$ is holomorphic for $\hat{k}\in\mathbb{C}\setminus [-b,b]$.
	\item[$\bullet$] For $\hat{k}\in (-b,b)$, we have
	\begin{equation}
	N_{+}^{(r5)}(\hat{k})=N_{-}^{(r5)}(\hat{k})\hat{V}^{(r5)}(\hat{k}),
	\end{equation}
	where
		\begin{align}\label{def:hatVr5}
			\hat{V}^{(r5)}(\hat{k})=\left\{\begin{array}{llll}
			 \begin{pmatrix}
				C_R(2-\hat{\xi})\hat{k}^2e^{-2i\tau g_{+}(\hat k)+i\tau B_1/2} & e^{i\tau B_1/2}\\
				-e^{-i\tau B_1/2} & e^{2i\tau g_{+}(\hat k)-i\tau B_1/2}
			\end{pmatrix}, & \hat{k}\in(a,b),\\ [10pt]
			\begin{pmatrix}
			C_R(2-\hat{\xi})\hat{k}^2e^{-2i\tau g_{+}(\hat k)-i\tau B_1/2} & e^{-i\tau B_1/2}\\
			-e^{i\tau B_1/2} & e^{2i\tau g_{+}(\hat k)+i\tau B_1/2}
			\end{pmatrix}, & \hat{k}\in (-b,-a), \\[10pt]
			\begin{pmatrix}
				C_R(2-\hat{\xi})\hat{k}^2e^{-i\tau A_1} & 0\\
				0 & \left(C_R(2-\hat{\xi})\hat{k}^2e^{-i\tau A_1}\right)^{-1}
			\end{pmatrix}, & \hat{k}\in (-a,a).
			\end{array}\right.
		\end{align}
%=I, & \vert\hat{k}\vert>b, \\
%			&=\begin{pmatrix}
%				C_R(2-\hat{\xi})\hat{k}^2e^{2i\tau g_{+}+i\tau B_1/2} & e^{i\tau B_1/2}\\
%				-e^{-i\tau B_1/2} & e^{-2i\tau g_{+}-i\tau B_1/2}
%			\end{pmatrix}, & \hat{k}\in(a,b),\label{jumpVr5a}\\
%			&=\begin{pmatrix}
%			C_R(2-\hat{\xi})\hat{k}^2e^{i\tau2g_{+}-i\tau B_1/2} & e^{-i\tau B_1/2}\\
%			-e^{i\tau B_1/2} & e^{-2i\tau g_{+}+i\tau B_1/2}
%			\end{pmatrix}, & \hat{k}\in (-b,-a), \label{jumpVr5b}\\	
%			&=\begin{pmatrix}
%				C_R(2-\hat{\xi})\hat{k}^2e^{i\tau A_1} & 0\\
%				0 & \left(C_R(2-\hat{\xi})\hat{k}^2e^{i\tau A_1}\right)^{-1}
%			\end{pmatrix}, & \hat{k}\in (-a,a).\label{jumpVr5c}
%		\end{align}
%	\end{subequations}
	\item[$\bullet$] As $\hat{k}\to \infty$ in $\mathbb{C}\setminus [-b,b]$, we have
	$N^{(r5)}(\hat{k})=I+\mathcal{O}(\hat k^{-1})$.
	\end{itemize}
\end{RHP}
As $t\to +\infty$, we have  $2-\hat{\xi}\to0$, $\tau\to+\infty$ and
\begin{equation}\label{equ: asytoNr5}
	N^{(r4)}(\hat{k})=N^{(r5)}(\hat{k})\left(I+o(1)\right).
\end{equation}

A key observation here is, by \eqref{equ:a,b con2} and \eqref{def:A1},
%Aimming at reducing to a model RH problem, it follows from \eqref{equ:a,b con2} and \eqref{equ; jump g3} that the precise condition
\begin{align}\label{equ: int}
	(2-\hat{\xi})e^{-i\tau A_1}=(2-\hat{\xi})\exp\left\lbrace 6\tau q\int_b^a\sqrt{(\zeta^2-a^2)(b^2-\zeta^2)}\dif\zeta\right\rbrace \equiv 1.
\end{align}
%Noticing the jump conditions on $(-a, a)$ of $\hat{V}^{(r5)}(\hat{k})$, the precise condition \eqref{equ:a,b con2}
%\begin{align}\label{equ: int}
%	(2-\hat{\xi})\exp\left\lbrace -6\tau q\int_b^a\sqrt{(\zeta^2-a^2)(b^2-\zeta^2)}\dif\zeta\right\rbrace =1
%\end{align}
%is required to assure that $N^{(r5)}$ could be transformed to a model RH problem.
Since the diagonal entries of $\hat{V}^{(r5)}(\hat k)$ in \eqref{def:hatVr5} tend to $0$ as $t \to +\infty$ for $\hat k \in (-b,-a)\cup (a,b)$, we are lead to consider the following RH problem,  of which the jump matrix
is the limiting form of $V^{(r5)}$.
%is given as follows
%Indeed, as $\tau\rightarrow\infty$, the $(2,2)$ entries of \eqref{jumpVr5a} and \eqref{jumpVr5b} converge to $0$, and under the condition \eqref{equ: int}, the $(1,1)$ entries
%also decay to zero. Thus a new RH problem for $N^{(r6)}(\hat{k})$, of which the jump matrices $V^{r6}(\hat{k})$
%is the limiting form of $V^{(r5)}(\hat{k})$ is given as follows
\begin{RHP}\label{RHP:  Nr6}
	\hfill
	\begin{itemize}
		\item[$\bullet$] $N^{(r6)}(\hat{k})$ is holomorphic for $\hat{k}\in\mathbb{C}\setminus [-b,b]$.
		\item[$\bullet$] For $\hat{k}\in(-b,b)$, we have
		\begin{equation}
		N^{(r6)}_{+}(\hat{k})=N^{(r6)}_{-}(\hat{k})\hat{V}^{(r6)}(\hat{k}),
		\end{equation}
		where
		\begin{align}
			\hat{V}^{(r6)}(\hat{k})=\left\{\begin{array}{llll}
				\left(\begin{array}{cc}
				0 & e^{i\tau B_1/2}\\
				-e^{-i\tau B_1/2} & 0
			\end{array}\right), & \hat{k}\in (a,b),\\[10pt]
		\left(\begin{array}{cc}
				C_R\hat{k}^2 & 0\\
				0 & \frac{1}{C_R\hat{k}^2}
			\end{array}\right), & \hat{k}\in (-a,a),\\[10pt]
			\left(\begin{array}{cc}
			0 & e^{-i\tau B_1/2}\\
				-e^{i\tau B_1/2} & 0
			\end{array}\right), & \hat{k}\in (-b,-a).
		\end{array}\right.
		\end{align}
		\item[$\bullet$] As $\hat{k}\rightarrow\infty$ in $\mathbb{C}\setminus [-b,b]$, $N^{(r6)}(\hat{k})=I+\mathcal{O}(\hat{k}^{-1})$.
	\end{itemize}
\end{RHP}
In what follows, we solve the above RH problem explicitly by using the Jacobi theta function.
%It will turn out that RH problem \ref{RHP: Nr6} can be solve explicitly in terms of the Jacobi theta function.
%Intending to transform $N^{(r6)}$ to a model RH problem, we introduce a scalar function $h(\hat{k})$ which can remove the jump conditions on the interval $(-a,a)$.

\subsubsection*{Construction of $N^{(r6)}$}
To proceed, we introduce a scalar function
\begin{align}\label{equ:h func}
	h(\hat{k}):=\frac{w(\hat{k})}{2\pi i}\left[ \int_{-b}^{-a}\frac{-\Delta_0}{(\zeta-\hat{k})w_+(\zeta)}\dif\zeta+\int_{-a}^a\frac{i\log (C_R\zeta^2)}{(\zeta-\hat{k})w(\zeta)}\dif\zeta +\int_{a}^{b}\frac{\Delta_0}{(\zeta-\hat{k})w_+(\zeta)}\dif\zeta\right],
\end{align}
where $w$ is defined through \eqref{eq:algeb} and
\begin{align}\label{def:Dleta0}
\Delta_0:=\left(\int_b^a\frac{1}{w_+(\zeta)}\dif\zeta\right) ^{-1}\int_0^a\frac{2i\log (C_R\zeta^2)}{w(\zeta)}\dif\zeta
\end{align}
is a real number. By using the Sokhotski-Plemelj formula and a direct caculation, the following proposition for $h$ is immediate.
\begin{Proposition}\label{prop:h}
The scalar function $h$ defined in \eqref{equ:h func} satisfies the following properties.
\begin{itemize}
	\item[$\bullet$] $h(\hat{k})$ is holomorphic for $\hat{k}\in\mathbb{C} \setminus [-b,b]$.
	\item[$\bullet$] $h(\hat{k})$ satisfies the jump relations
	\begin{subequations}
		\begin{align}
			&h_+(\hat k)-h_-(\hat k)=i\log (C_{R}\hat{k}^2), & \hat{k}\in(-a,a),\\
			&h_+(\hat k)+h_-(\hat k)=\Delta_0,  & \hat{k}\in(a,b),\\
			&h_+(\hat k)+h_-(\hat k)=-\Delta_0, & \hat{k}\in(-b,-a),
		\end{align}
	\end{subequations}
where $\Delta_0$ is given in \eqref{def:Dleta0}.
	\item[$\bullet$] As $\hat{k}\to\infty$ in $\mathbb{C}\setminus [-b,b]$, we have $h(\hat{k})=\mathcal{O}(\hat{k}^{-1})$.
\end{itemize}
\end{Proposition}

We then eliminate the jump of $N^{(r6)}$ on $(-a,a)$ by defining
\begin{align}
		N^{(r7)}(\hat{k})=N^{(r6)}(\hat{k})e^{-ih(\hat{k})\sigma_3}.\label{equ:trans Nr7}
\end{align}
By RH problem \ref{RHP:  Nr6} and Proposition \ref{prop:h}, $N^{(r7)}$ satisfies the following RH problem.
\begin{RHP}\label{RHP:  model Nr7}
	\hfill
	\begin{itemize}
		\item[$\bullet$] $N^{(r7)}(\hat{k})$ is holomorphic for $\hat{k}\in\mathbb{C}\setminus([-b,-a]\cup [a,b])$.
		\item[$\bullet$] For $\hat{k}\in(-b,-a)\cup(a,b)$, we have
		\begin{equation}
		N^{(r7)}_{+}(\hat{k})=N^{(r7)}_{-}(\hat{k})\hat{V}^{(r7)}(\hat{k}),
		\end{equation}
		where
		\begin{align}
			\hat{V}^{(r7)}(\hat{k})=\left\{\begin{array}{llll}
				\left(\begin{array}{cc}
					0 & e^{i\tau B_1/2+\Delta_0}\\
					-e^{-i\tau B_1/2-\Delta_0} & 0
				\end{array}\right), & \hat{k}\in (a,b),\\[10pt]
				\left(\begin{array}{cc}
					0 & e^{-i\tau B_1/2-\Delta_0}\\
					-e^{i\tau B_1/2+\Delta_0} & 0
				\end{array}\right), & \hat{k}\in (-b,-a).
			\end{array}\right.
		\end{align}
		\item[$\bullet$]
As $\hat{k}\rightarrow\infty$ in $\mathbb{C}\setminus([-b,-a]\cup[a,b])$, we have $N^{(r7)}(\hat{k})=I+\mathcal{O}(\hat{k}^{-1})$.
	\end{itemize}
\end{RHP}
The above RH problem $N^{(r7)}$ can be solved explicitly by using the Jacobi theta function
\begin{align}\label{equ:Jacobi Theta func}
	\Theta(s):=\sum_{n\in\mathbb{Z}}e^{2\pi ins+\varkappa\pi i n^2},
\end{align}
where
\begin{align}\label{equ: varkappa expression}
	\varkappa=\left(\int_b^a\frac{1}{w_+(\zeta)}\dif\zeta\right) ^{-1}\int_a^{-a}\frac{1}{w(\zeta)}\dif\zeta.
\end{align}
It is easily seen that $\Theta(s)$ is an even function and satisfies
\begin{align*}
		\Theta(s+1)=\Theta(s), \qquad	\Theta(s+\varkappa)=e^{-2\pi is-\pi i\varkappa}	\Theta(s).
\end{align*}
Moreover, $\Theta(s)$ vanishes at the lattice of half periods, i.e.,
\begin{align*}
	\Theta(s)=0, \qquad  s=\frac{1}{2}+\frac{\varkappa}{2}+\mathbb{Z}+\varkappa\mathbb{Z}.
\end{align*}

It then follows from a straightforward calculation (cf. \cite{Deift1994TheCS}) that
\begin{align}
	N^{(r7)}(\hat{k}):=\left(\begin{array}{cc}
		\frac{\nu+\frac{1}{\nu}}{2}\frac{\Theta(A(\hat{k})-\frac{\varkappa}{4}+\frac{\phi}{\pi})\Theta(A(\infty)-\frac{\varkappa}{4})}{\Theta(A(\hat{k})-\frac{\varkappa}{4})\Theta(A(\infty)-\frac{\varkappa}{4}+\frac{\phi}{\pi})} & e^{-i\phi}\frac{\nu-\frac{1}{\nu}}{2i}\frac{\Theta(-A(\hat{k})-\frac{\varkappa}{4}+\frac{\phi}{\pi})\Theta(A(\infty)-\frac{\varkappa}{4})}{\Theta(-A(\hat{k})-\frac{\varkappa}{4})\Theta(A(\infty)-\frac{\varkappa}{4}+\frac{\phi}{\pi})}\\
		e^{i\phi}\frac{\nu-\frac{1}{\nu}}{-2i}\frac{\Theta(A(\hat{k})+\frac{\varkappa}{4}+\frac{\phi}{\pi})\Theta(-A(\infty)+\frac{\varkappa}{4})}{\Theta(A(\hat{k})+\frac{\varkappa}{4})\Theta(-A(\infty)+\frac{\varkappa}{4}+\frac{\phi}{\pi})} & \frac{\nu+\frac{1}{\nu}}{2}\frac{\Theta(-A(\hat{k})+\frac{\varkappa}{4}+\frac{\phi}{\pi})\Theta(-A(\infty)+\frac{\varkappa}{4})}{\Theta(-A(\hat{k})+\frac{\varkappa}{4})\Theta(-A(\infty)+\frac{\varkappa}{4}+\frac{\phi}{\pi})}
	\end{array}\right)
\end{align}
solves RH problem \ref{RHP:  model Nr7}. Here, $A(\hat{k})$ defined in \eqref{def:Ahatk}
%\begin{align}
%	A(\hat{k})=\left(2 \int_b^a\frac{1}{w_+(\zeta)}\dif\zeta\right) ^{-1}\int_b^{\hat{k}}\frac{1}{w(\zeta)}\dif\zeta,\qquad \hat{k}\in\mathbb{C}\setminus [-b,b],
%\end{align}
is the Abelian integral with the base point $b$ and the contour of integration lies on the first sheet of the Riemann surface $\mathcal{M}$,
\begin{align}
	\nu(\hat{k})=\left[ \frac{(\hat{k}-a)(\hat{k}+b)}{(\hat{k}+a)(\hat{k}-b)}\right]^{1/4}, \qquad \hat{k}\in\mathbb{C}\setminus ([-b,-a]\cup[a,b]),
\end{align}
such that $\nu(\infty)=1$, and
\begin{align}\label{equ: phi expression}
	\phi=\frac{\tau B_1}{2}-i\Delta_0.
\end{align}

%To obtian the concrete expression for $N^{(r7)}$, the abelian integral with the base point $b$
%is required as follows
%
%of which contour integration is on the first sheet of the Riemann surface $\mathcal{M}$.
%
%To proceed, let
%\begin{align}
%	\nu(\hat{k})=\left[ \frac{(\hat{k}-a)(\hat{k}+b)}{(\hat{k}+a)(\hat{k}-b)}\right]^{1/4}
%\end{align}
%with branch cuts $(-b,-a)\cup(b,a)$ and normalization condition $\nu(\infty)=1$.

%where
\begin{remark}\label{RK:indeppq}
By Remark \ref{remark a,b}, one can check that the quantities $A(\infty)$, $\varkappa$ and $\phi$ are independent of the choice of $p$ and $q$.
\end{remark}
%Therefore, as $\hat{k}\to\infty$,
For later use, we note that
\begin{align}\label{eq:expNr7}
	N^{(r7)}(\hat{k})=I+\frac{	N^{(r7)}_1}{\hat{k}}+\frac{	N^{(r7)}_2}{\hat{k}^2}+\mathcal{O}\left( \frac{1}{\hat{k}^3}\right), \qquad \hat{k} \to \infty
\end{align}
where
\begin{align}
N^{(r7)}_1=\begin{pmatrix} 0 & \left(N^{(r7)}_1\right)_{12} \\ \left(N^{(r7)}_1\right)_{21} & 0\end{pmatrix},\qquad N^{(r7)}_2=\begin{pmatrix} \left(N^{(r7)}_2\right)_{11} & \left(N^{(r7)}_2\right)_{12} \\ \left(N^{(r7)}_2\right)_{21} & \left(N^{(r7)}_2\right)_{22}\end{pmatrix}
\end{align}
with
\begin{subequations}\label{equ: (N^(r7)_1)_{12&21}}
	\begin{align} &\left(N^{(r7)}_1\right)_{12}=-\left(N^{(r7)}_1\right)_{21}=\frac{b-a}{2i}\cdot\frac{\Theta\left(A(\infty)-\frac{\varkappa}{4}\right)\Theta\left(-A(\infty)-\frac{\varkappa}{4}+\frac{\phi}{\pi}\right)}{\Theta\left(-A(\infty)-\frac{\varkappa}{4}\right)\Theta\left(A(\infty)-\frac{\varkappa}{4}+\frac{\phi}{\pi}\right)}e^{-i\phi},\\ &\left(N^{(r7)}_2\right)_{12}=\frac{b-a}{2i}\cdot\frac{\Theta\left(A(\infty)-\frac{\varkappa}{4}\right)}{\Theta\left(A(\infty)-\frac{\varkappa}{4}+\frac{\phi}{\pi}\right)}\cdot\left( \frac{\Theta\left(-A(\hat{k})-\frac{\varkappa}{4}+\frac{\phi}{\pi}\right)}{\Theta\left(-A(\hat{k})-\frac{\varkappa}{4}\right)}\right)'\Big|_{\hat{k}=\infty} e^{i\phi}.
	\end{align}
\end{subequations}

\subsubsection*{Local analysis near $\hat k= \pm a, \pm b$ and the small norm RH problem}
In each neighborhood of $\hat k= \pm a, \pm b$, we still need to construct a local parametrix $N^{(r4, j)}$, $j\in\{\pm a, \pm b\}$.
These parametrices can be built with the aid of the Airy functions; cf. \cite{FLQDNLS2021}. Since they only contribute to a higher order correction compared with that of $N^{(r7)}$, we omit the details here.

Finally, we define
\begin{equation}\label{def:E3rdtranregion}
	E(\hat{k})=\left\{\begin{array}{ll}
		N^{(r4)}(\hat{k}){N^{(r7)}(\hat{k})}^{-1}, & \hat{k}\notin \underset{{j\in\{\pm a, \pm b\}}}\bigcup U^{(r4,j)}, \\
		N^{(r4)}(\hat{k}){N^{(r4, j)}(\hat{k})}^{-1},  &\hat{k}\in U^{(r4, j)},\\
	\end{array}\right.
\end{equation}
where $j\in\{\pm a, \pm b\}$ and $U^{(r4,j)}$ is a small neighborhood of $j$. As in Section \ref{subsec:pure RH N(z)}, it is then readily seen that $E$ satisfies a small norm RH problem and $E(\hat{k})=I+o(1)$ for large positive $t$.

\subsection{Proof of part (c) of Theorem \ref{mainthm}}\label{subsec:recovering3sttran}
By tracing back the transformations \eqref{equ:trans Nr1}, \eqref{equ: transtoNr2}, \eqref{equ: transtoNr4},
\eqref{equ: asytoNr5}, \eqref{equ:trans Nr7} and \eqref{def:E3rdtranregion}, we conclude that, as $t\to +\infty$,
\begin{align}\label{equ:recover N^{(r)}}
N^{(r)}(\hat{k})=N^{(r7)}(\hat{k})e^{ih(\hat{k})\sigma_3}e^{-i\tau(\hat{\theta}(\hat{k})-g(\hat{k}))\sigma_3}\left(I+o(1)\right),
\end{align}
where $\hat \theta$, $g$ and $h$ are defined in \eqref{def:hattheta}, \eqref{equ:g func} and \eqref{equ:h func}, respectively.

As $\hat{k}\to\infty$,  $N^{(r)}$ admits the expansion
\begin{align*}%\label{equ: N^{(r)} expansion}
N^{(r)}(\hat{k})=I+\frac{N^{(r)}_1}{\hat{k}}+\frac{N^{(r)}_2}{\hat{k}^2}+\mathcal{O}\left( \frac{1}{\hat{k}^3}\right)
\end{align*}
with
\begin{align*}
	N^{(r)}_1=\left(\begin{array}{cc}	
	 \left(N^{(r)}_1\right)_{11} & \left(N^{(r)}_1\right)_{12} 
 \\[8pt]
\overline{\left(N^{(r)}_1\right)_{12}} & \overline{\left(N^{(r)}_1\right)_{11}} \end{array}\right)
	\quad \textrm{and} \quad 
	N^{(r)}_2=\left(\begin{array}{cc}	\left(N^{(r)}_2\right)_{11} & \left(N^{(r)}_2\right)_{12} \\[8pt] \overline{\left(N^{(r)}_2\right)_{12}} & \overline{\left(N^{(r)}_2\right)_{11}} \end{array}\right),
\end{align*}
where the symmetry condition $M^{(4)}(z)=\sigma_1\overline{M^{(4)}(\bar{z})}\sigma_1$ is additionally used.

In view of \eqref{equ: M4}, we still need analogous results for $N^{(l)}$. To that end, note that
%Aimming at obtaining the asymptotics, it follows from the formula \eqref{equ: M4}
%that the expression for $N^{(l)}$ is also required.
%Indeed, the RH problem for $N^{(l)}$ can be solved in a similar manner
%as $N^{(r)}$. As $2-\hat{\xi}\rightarrow 0$ and $\tau\rightarrow+\infty$, we have

\begin{align}\label{equ:checktheta}
	\theta(z(\check{k}))=\tau\check{\theta}(\check{k})+\mathcal{O}\left( (\log t)^{4/3}t^{-1/3}\check{k}^2\right), \qquad t\to +\infty,
\end{align}
where $\tau$ is defined in \eqref{equ:tau} and
\begin{equation}
	\check{k}=\sqrt{\frac{12p}{q}}(2-\hat{\xi})^{-1/2}(z+1)
\end{equation}
is the scaled spectral variable in this case. Following asymptotic analysis of the RH problem for
$N^{(r)}$, it follows that
\begin{equation}
N^{(l)}(\check{k})=I+\frac{N^{(l)}_1}{\check{k}}+\frac{N^{(l)}_2}{\check{k}^2}+\mathcal{O}\left(\frac{1}{\check{k}^3}\right), \quad \check{k}\rightarrow\infty
%N^{(l7)}(\check{k})=I+\frac{N^{(l7)}_1}{\check{k}}+\mathcal{O}\left(\frac{1}{\check{k}^2}\right),\quad
\end{equation}
with 
\begin{align*}
	N^{(l)}_1=\begin{pmatrix} -\overline{\left(N^{(r)}_1\right)_{11}} & \overline{\left(N^{(r)}_1\right)_{12}} \\ \left(N^{(r)}_1\right)_{12} & -\left(N^{(r)}_1\right)_{11} \end{pmatrix},\qquad	N^{(l)}_2=\begin{pmatrix} \overline{\left(N^{(r)}_2\right)_{11}} & -\overline{\left(N^{(r)}_2\right)_{12}}\\ -\left(N^{(r)}_2\right)_{12}  & \left(N^{(r)}_2\right)_{11} \end{pmatrix},
\end{align*}
where the other symmetry relation $M^{(4)}(z)=\sigma_2M^{(4)}(-z)\sigma_2$ is applied.
%As for $N^{(l)}$, for any fixed positive constants $p$, $q$, we introduce the scaled spectral variable:
%\begin{align}
%	\check{k}=\sqrt{\frac{12p}{q}}(2-\hat{\xi})^{-1/2}(z+1).
%\end{align}
%Obeying the same strategy used for $N^{(r)}$, as $2-\hat{\xi}\rightarrow 0$, $\tau\rightarrow+\infty$, it follows that
%\begin{subequations}
%\begin{align}
%	&N^{(l7)}_1=\begin{pmatrix} 0 & \left(N^{(r7)}_1\right)_{21} \\ \left(N^{(r7)}_1\right)_{12} & 0\end{pmatrix},\\
%	&N^{(l)}_1=\begin{pmatrix}i\check{h}_1+i\check{\tau}\check{g}_0 & \left(N^{(r7)}_1\right)_{21} \\ \left(N^{(r7)}_1\right)_{12} & -i\check{h}_1-i\check{\tau} \check{g}_0\end{pmatrix}\left(I+o(1)\right),
%\end{align}
%\end{subequations}
%where $\check{\tau}$, $\check{h}_1$ and $\check{g}_0$ are the corresponded quantities to $\tau$, $h_1$, $g_0$.

To proceed, observe that
\begin{align}
	&\frac{1}{\hat{k}}=-\sqrt{\frac{q(2-\hat{\xi})}{12p}}+\mathcal{O}(z),&z\to 0, \label{equ:1/hat{k}zto0}\\
	&\frac{1}{\hat{k}}=\sqrt{\frac{q(2-\hat{\xi})}{12p}}\frac{1+i}{2}-\sqrt{\frac{q(2-\hat{\xi})}{12p}}\frac{i}{2}(z-i)+\mathcal{O}\left((z-i)^2\right), &z\to i,\label{equ:1/hat{k}ztoi}
\end{align}
we then obtain from \eqref{transform:M1toM2}, \eqref{transform: M3toM4; 3rdtrans},
\eqref{equ: M4=M3 at 0 and i} and \eqref{equ: M4} that as $t \to +\infty$,
\begin{subequations}
\begin{align}\label{equ: M^{(1)}(0); 3rdtran}
	M^{(1)}(0)=\left( I+\sqrt{\frac{2-\hat{\xi}}{12}}(q/p)^{1/2}\left(N^{(l)}_1-N^{(r)}_1\right)+\frac{q(2-\hat{\xi})}{12p}(N^{(l)}_2+N^{(r)}_2)\right)  \left(I+o(1)\right),
\end{align}
and
\begin{align}\label{equ: M^{(1)}(i); 3rdtran}
	M^{(1)}(z)&=I+\sqrt{\frac{2-\hat{\xi}}{12}}(q/p)^{1/2}\left(\frac{1+i}{2}N^{(l)}_1-\frac{1-i}{2}N^{(r)}_1\right)+\frac{q(2-\hat{\xi})i}{24p}(N^{(l)}_2-N^{(r)}_2)\nonumber\\
	\hspace*{2em}&+\left( \sqrt{\frac{2-\hat{\xi}}{12}}(q/p)^{1/2}\frac{i}{2}\left(N^{(l)}_1-N^{(r)}_1\right)+\frac{q(2-\hat{\xi})}{12p}\left(\frac{1+i}{2i}N^{(l)}_2+\frac{1-i}{2i}N^{(r)}_2\right)\right) (z-i)\nonumber\\
	\hspace*{2em}&+\mathcal{O}\left((z-i)^2\right), \qquad z\to i.
\end{align}
\end{subequations}
This, together with the reconstruction formula in Proposition \ref{prop: M^{(1)}}, implies that
\begin{align} u(x(y,t),t)&=1-\frac{(2-\hat{\xi})q}{6p}\left(\im\left(N^{(r)}_1\right)_{11}\im\left(N^{(r)}_1\right)_{12}+\re\left(N^{(r)}_2\right)_{12}\right)+o\left(2-\hat{\xi}\right),
\label{eq:UxyRIII}
\\
	x(y,t)&=y+\sqrt{\frac{2-\hat{\xi}}{3}}(q/p)^{1/2}\left(\frac{i-1}{2}\left(N^{(r)}_1\right)_{11}+\frac{i+1}{2}\overline{\left(N^{(r)}_1\right)_{11}}-2\im\left(N^{(r)}_1\right)_{12}\right)\nonumber\\ &~~~+o\left(\sqrt{2-\hat{\xi}}\right).\label{equ:x=y+}
\end{align}

By \eqref{equ:recover N^{(r)}} and \eqref{eq:expNr7}, we obtain
%\begin{align}
%	N^{(r)}_1=\begin{pmatrix}ih_1+i\tau g_0 & \left(N^{(r7)}_1\right)_{12} \\ \left(N^{(r7)}_1\right)_{21} & -ih_1-i\tau g_0\end{pmatrix}\left(I+o(1)\right), \qquad t\to+\infty
%\end{align}
\begin{align}
	\left(N^{(r)}_1\right)_{11}=(ih_1+i\tau g_0)\left(1+o(1)\right),\qquad\left(N^{(r)}_1\right)_{12}=\left(N^{(r7)}_1\right)_{12}\left(1+o(1)\right),
\end{align}
with $\left(N^{(r7)}_1\right)_{12}$ given in \eqref{equ: (N^(r7)_1)_{12&21}} and
\begin{align}\label{def:h1g0}
	h_1:=\lim_{\hat{k}\to\infty}\left(\hat{k}h(\hat{k}) \right), \qquad
	g_0:=\lim_{\hat{k}\to\infty}\left(\hat{k}(g(\hat{k})-\hat{\theta}(\hat{k}))\right).
\end{align}
%and
%\begin{align}
%	%&N^{(l7)}_1=\begin{pmatrix} 0 & \left(N^{(r7)}_1\right)_{21} \\ \left(N^{(r7)}_1\right)_{12} & 0\end{pmatrix},\\
%	&N^{(l)}_1=\begin{pmatrix}i\check{h}_1+i\tau\check{g}_0 & \left(N^{(r7)}_1\right)_{21} \\ \left(N^{(r7)}_1\right)_{12} & -i\check{h}_1-i\tau \check{g}_0\end{pmatrix}\left(I+o(1)\right).
%\end{align}
%Here, we have made use of the symmetry relation $M^{(4)}(z)=\sigma_2M^{(4)}(-z)\sigma_2$, and $\check{h}_1$, $\check{g}_0$ are defined similarly as $h_1$, $g_0$ given in \eqref{def:h1g0}.

Inserting \eqref{equ: (N^(r7)_1)_{12&21}} into \eqref{eq:UxyRIII}, it is accomplished that
\begin{align} u(x(y,t),t)&=1-\frac{(2-\hat{\xi})(a-b)q}{12p}\left(-i\cdot\frac{\Theta\left(A(\infty)-\frac{\varkappa}{4}\right)}{\Theta\left(A(\infty)-\frac{\varkappa}{4}+\frac{\phi}{\pi}\right)}\cdot\left( \frac{\Theta\left(-A(\hat{k})-\frac{\varkappa}{4}+\frac{\phi}{\pi}\right)}{\Theta\left(-A(\hat{k})-\frac{\varkappa}{4}\right)}\right)'\Big|_{\hat{k}=\infty} e^{i\phi}\right.\nonumber
	\\	&~~~\left. +\frac{\Theta\left(A(\infty)-\frac{\varkappa}{4}\right)\Theta\left(-A(\infty)-\frac{\varkappa}{4}+\frac{\phi}{\pi}\right)}{\Theta\left(-A(\infty)-\frac{\varkappa}{4}\right)\Theta\left(A(\infty)-\frac{\varkappa}{4}+\frac{\phi}{\pi}\right)}e^{i\phi}\right)\left(1+o(1)\right).
\end{align}
Finally, it can be checked directly that one can replace $x(y,t)$
by $x$ in the above formula, which leads to the asymptotic formula \eqref{result: u(x,t); part (c)}.
\qed

%straightforward verified that the asymptotics for $u(x,t)$ stated in
%part (c) of Theorem \ref{mainthm} holds after replacing the variable.

%%%%%%%%%%%%%%%%%%%%%%%%%%%%%%%%%%%%%%%%%%%%%%%%%%%%%%%%%%%%%%%%%%%%%%%%%%%%%%%%%%%%%%%%%%%%%%%%%%%%%%%%%%%%%%
%%%%%%%%%%%%%%%%%%%%%%%%%%%%%%%%%%%%%%%%%%%%%%%%%%%%%%%%%%%%%%%%%%%%%%%%%%%%%%%%%%%%%%%%%%%%%%%%%%%%%%%%%%%%%%
\appendix
\section{The Painlev\'e II parametrix} \label{appendix: RHP for PII}
A result due to Hastings and McLeod \cite{Hastings1980ABV} asserts that, for any $\kappa \in \mathbb{R}$, there exists a unique solution to the homogeneous Painlev\'e II equation \eqref{equ:standard PII equ} which behaves like $\kappa \textnormal{Ai}(s)$ for large positive $s$. This one-parameter family of solutions are characterized by the following RH problem; cf. \cite{FIKN}.

\begin{RHP}\label{appendix: reduced PII RH}
\hfill
	%In particular, if $\mathcal{S}=(s_1,0,-s_1)$ where $s_1\in i\mathbb{R}$, the RH problem \ref{RHP: classical PII} degenerates
%to the case we use in our present paper (see Figure \ref{fig:jump of reduced M^p}).
\begin{itemize}
  \item $M^{P}(z; s,\kappa)$ is holomorphic for $z\in \mathbb{C}\setminus \Gamma$, where $\Gamma:=\cup_{i=1,3,4,6}\Gamma_i$ with
  \begin{equation}
		\Gamma_i:=\left\{z\in\mathbb{C}: \arg z=\frac{\pi}{6}+\frac{\pi}{3}(i-1)\right\};
	\end{equation}
see Figure \ref{fig:jump of reduced M^p} for an illustration.
  \item $M^{P}$ satisfies the jump condition
  $$
  M^{P}_{+}(z; s,\kappa)=M^{P}_{-}(z; s,\kappa)e^{-i\left(\frac{4}{3}z^3+sz\right)\hat{\sigma}_3}S_i, \quad z \in \Gamma_i\backslash\{0\}, \quad i=1,3,4,6,
  $$
where the matrix $S_i$ depending on $\kappa$ for each ray $\Gamma_i$ is shown in Figure \ref{fig:jump of reduced M^p}
  \item As $z\to \infty$ in $\mathbb{C} \setminus \Gamma$, we have $M^{P}(z;s,\kappa)=I+\mathcal{O}(z^{-1})$.
   \item As $z\to 0 $, we have $M^{P}(z;s,\kappa)=\mathcal{O}(1)$.
\end{itemize}
%\begin{subequations}
%	\begin{align}
%		&M^{P}_{+}(\tilde{s},z)=M^{P}_{-}(\tilde{s},z)e^{-i\left(\frac{4}{3}z^3+\tilde{s}z\right)\hat{\sigma}_3}S_n, &&z\in P_n\backslash\{0\}, \ n=1,3,4,6,\\
%		&M^{P}(\tilde{s},z)=I+\mathcal{O}(z^{-1}), &&z\rightarrow\infty,\\
%		&M^{P}(\tilde{s},z)=\mathcal{O}(1), &&z\rightarrow 0,
%	\end{align}
%\end{subequations}
\end{RHP}
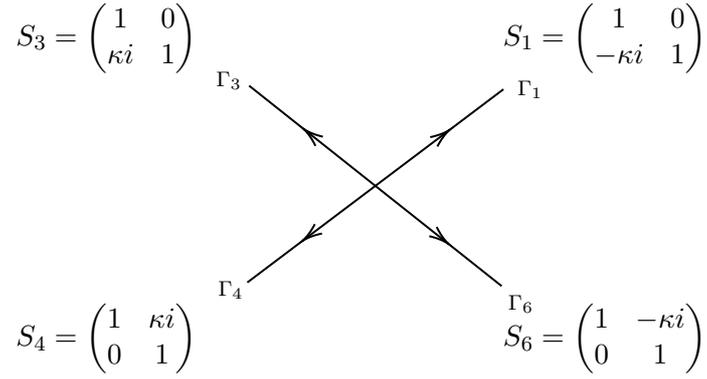
\begin{figure}[htbp]
	\begin{center}
\tikzset{every picture/.style={line width=0.75pt}} %set default line width to 0.75pt
\begin{tikzpicture}[x=0.75pt,y=0.75pt,yscale=-0.9,xscale=0.9]
%uncomment if require: \path (0,300); %set diagram left start at 0, and has height of 300
%Straight Lines [id:da4842262815787848]
\draw    (361.86,77) -- (290.86,131) ;
\draw [shift={(331.93,99.76)}, rotate = 142.74] [color={rgb, 255:red, 0; green, 0; blue, 0 }  ][line width=0.75]    (10.93,-3.29) .. controls (6.95,-1.4) and (3.31,-0.3) .. (0,0) .. controls (3.31,0.3) and (6.95,1.4) .. (10.93,3.29)   ;
%Straight Lines [id:da10064781210526919]
\draw    (290.86,131) -- (219.86,185) ;
\draw [shift={(250.58,161.63)}, rotate = 322.74] [color={rgb, 255:red, 0; green, 0; blue, 0 }  ][line width=0.75]    (10.93,-3.29) .. controls (6.95,-1.4) and (3.31,-0.3) .. (0,0) .. controls (3.31,0.3) and (6.95,1.4) .. (10.93,3.29)   ;
%Straight Lines [id:da9121370955753667]
\draw    (290.86,131) -- (220.86,75) ;
\draw [shift={(251.17,99.25)}, rotate = 38.66] [color={rgb, 255:red, 0; green, 0; blue, 0 }  ][line width=0.75]    (10.93,-3.29) .. controls (6.95,-1.4) and (3.31,-0.3) .. (0,0) .. controls (3.31,0.3) and (6.95,1.4) .. (10.93,3.29)   ;
%Straight Lines [id:da24960493969604625]
\draw    (360.86,187) -- (290.86,131) ;
\draw [shift={(331.32,163.37)}, rotate = 218.66] [color={rgb, 255:red, 0; green, 0; blue, 0 }  ][line width=0.75]    (10.93,-3.29) .. controls (6.95,-1.4) and (3.31,-0.3) .. (0,0) .. controls (3.31,0.3) and (6.95,1.4) .. (10.93,3.29)   ;
% Text Node
\draw (368,70.4) node [anchor=north west][inner sep=0.75pt]  [font=\scriptsize]  {$\Gamma_{1}$};
% Text Node
\draw (201,65.4) node [anchor=north west][inner sep=0.75pt]  [font=\scriptsize]  {$\Gamma_{3}$};
% Text Node
\draw (202,182.4) node [anchor=north west][inner sep=0.75pt]  [font=\scriptsize]  {$\Gamma_{4}$};
% Text Node
\draw (362.86,190.4) node [anchor=north west][inner sep=0.75pt]  [font=\scriptsize]  {$\Gamma_{6}$};
% Text Node
\draw (360,26.4) node [anchor=north west][inner sep=0.75pt]  {$S_{1} =\begin{pmatrix}
1 & 0\\
-\kappa i & 1
\end{pmatrix}$};
% Text Node
\draw (90,194.4) node [anchor=north west][inner sep=0.75pt]    {$S_{4} =\begin{pmatrix}
1 & \kappa i\\
0 & 1
\end{pmatrix}$};
% Text Node
\draw (90,26.4) node [anchor=north west][inner sep=0.75pt]    {$S_{3} =\begin{pmatrix}
1 & 0\\
\kappa i & 1
\end{pmatrix}$};
% Text Node
\draw (360,194.4) node [anchor=north west][inner sep=0.75pt]    {$S_{6} =\begin{pmatrix}
1 & -\kappa i \\
0 & 1
\end{pmatrix}$};
\end{tikzpicture}	
	\caption{ The jump contours $\Gamma_i$ and the corresponding matrices $S_i$, $i=1,3,4,6$, in the RH problem for $M^{P}$.} \label{fig:jump of reduced M^p}
	\end{center}
\end{figure}

The above RH problem admits a unique solution. Moreover, there exist smooth functions $\{M^{p}_j(s)\}_{j=1}^{\infty}$ such that, for each integer $N \geqslant 0$, we have
\begin{equation}
	M^{P}(s,z)=I+\sum_{j=1}^{N}\frac{M^{P}_j(s)}{z^j}+\mathcal{O}(z^{-N-1}),\quad z\rightarrow\infty,
\end{equation}
where
\begin{align}\label{MPfirstexpansion}
	&M^{P}_1(s)=\frac{1}{2}\begin{pmatrix}
		-i\int_{s}^{+\infty}v^2(\zeta)\dif\zeta   & v(s) \\
		v(s) & i\int_{s}^{+\infty}v^2(\zeta)\dif\zeta
	\end{pmatrix},
\\
\label{MPsecondexpansion}
	&M^{P}_2(s)=\frac{1}{8}\begin{pmatrix}
		-\left(\int_{s}^{+\infty}v^2(\zeta)\dif\zeta\right)^2+v^2(s)   &  2i\left(v(s)\int_{s}^{+\infty}v^2(\zeta)\dif\zeta+v'(s)\right)  \\
		-2i\left(v(s)\int_{s}^{+\infty}v^2(\zeta)\dif\zeta+v'(s)\right) & -\left(\int_{s}^{+\infty}v^2(\zeta)\dif\zeta\right)^2+v^2(s)
	\end{pmatrix}.
\end{align}
%where
%\begin{align*}
%	&\left(M^{P}_{2}\right)_{11}(\tilde{s})=-\frac{1}{8}\left(\left(\int_{\tilde{s}}^{+\infty}v^2(\zeta)\dif\zeta\right)^2-v^2(\tilde{s})\right),\\
%	&\left(M^{P}_{2}\right)_{12}(\tilde{s})=\frac{i}{4}\left(v(\tilde{s})\int_{\tilde{s}}^{+\infty}v^2(\zeta)\dif\zeta+v'(\tilde{s})\right).
%\end{align*}
%For each $C>0$,
%\begin{equation}
%	\underset{\tilde{s}\geqslant-C_1}{sup} \ \underset{z\in\mathbb{C}\backslash P}{sup}|M^{P}(\tilde{s},z)|<\infty.
%\end{equation}
The function $v$ in \eqref{MPfirstexpansion} then solves the Painlev\'e II equation \eqref{equ:standard PII equ} with the boundary condition
\begin{equation}\label{vMPasy}
	v(s)\sim \kappa\textnormal{Ai}(s), \qquad s\to +\infty.
\end{equation}
We call RH problem \ref{appendix: reduced PII RH} the Painlev\'e II parametrix.

\section*{Acknowledgements}
We are grateful to Engui Fan and Jian Xu for useful discussions and comments. Yiling Yang is partially supported by National Science Foundation of China udner grant number 12247182 and China Post-Doctoral Science Foundation. Lun Zhang is partially supported by National Natural Science Foundation of China under grant numbers 12271105, 11822104, and ``Shuguang Program'' supported by Shanghai Education Development Foundation and Shanghai Municipal Education Commission.

\end{document}